\documentclass[10pt]{amsart}

\usepackage{amsmath,amsrefs, amssymb,xcolor,hyperref}
\usepackage{graphicx}
\usepackage{dsfont}
\definecolor{darkblue}{HTML}{004C93} 
\definecolor{MainRed}{rgb}{.6, .1, .1}

\numberwithin{equation}{section}

\newcommand{\R}{\mathbb{R}}

\newcommand{\N}{\mathbb{N}}

\def\al{{\alpha}}        
\def\La{{\Lambda}} 
\def\be{{\beta}}         
\def\de{{\delta}}         
         
\def\om{{\omega}}         
\def\Om{{\Omega}}         
\def\la{{\lambda}}

\def\ep{{\varepsilon}}

\def\phi{{\varphi}}

\DeclareMathAlphabet{\doba}{U}{msb}{m}{n}

\gdef\mN{\doba{N}}

\gdef\mR{\doba{R}}         
\gdef\mS{\doba{S}}

\def\Vol{{\mathop{\rm Vol}}}    
\def\exp{{\mathop{\rm exp}}} 
\def\ker{{\mathop{\rm Ker}}}
\def\span{{\mathop{\rm Span}}}

\newcommand{\ve}{\varepsilon}
\newcommand{\vp}{\varphi}
\newcommand{\bal}{\begin{aligned}}
\newcommand{\eal}{\end{aligned}}
\newcommand{\ben}{\begin{equation}}
\newcommand{\een}{\end{equation}}
\newcommand{\ee}{\end{equation*}}

\def \ep{\epsilon}

\theoremstyle{plain}

\newtheorem{theo}{Theorem}[section]
\newtheorem{prop}[theo]{Proposition}
\newtheorem{lemme}{Lemma}[section]
\newtheorem{corol}[theo]{Corollary}
\theoremstyle{remark}
\newtheorem{rem}{Remark}[section]

\begin{document}

\title[Extremising eigenvalues of the GJMS operators]{Extremising eigenvalues of the GJMS operators in a fixed conformal class}

\author{Emmanuel Humbert}

\address{Emmanuel Humbert, LMPT CNRS UMR 6083, Facult\'e des Sciences et Techniques Univerist\'e Fran\c{c}ois Rabelais, Parc de Grandmont 37200 Tours, France}
\email{humbert@univ-tours.fr}

\author{Romain Petrides}

\address{Romain Petrides, IMJ-PRG, Universit\'e Paris Cit\'e, B\^atiment Sophie Germain 75205 Paris Cedex 13, France}
\email{romain.petrides@imj-prg.fr }

\author{Bruno Premoselli}

\address{Bruno Premoselli, Universit\'e Libre de Bruxelles, Service d'analyse, CP 218, Boulevard du Triomphe, B-1050 Bruxelles, Belgique}
\email{bruno.premoselli@ulb.be}

\thanks{E.H. is supported by the project Einstein-PPF (\href{https://anr.fr/Project-ANR-23-CE40-0010}{ANR-23-CE40-0010}), funded by the French National Research Agency. 
B.P. is supported by the Fonds Th\'elam, by an ARC Avanc\'e 2020 grant and by an EoS FNRS grant.}

\date{May 29, 2025}

	\begin{abstract}
Let $(M,g)$ be a closed Riemannian manifold of dimension $n\geq 3$. If $s$ is a positive integer satisfying $2s<n$, we let $P_g^s$ be the GJMS operator of order $2s$ in $M$. We investigate in this paper the extremal values taken by fixed eigenvalues of $P_h^s$ as $h$ runs through the whole conformal class $[g]$. Because of the conformal covariance of $P_g^s$, two variational problems are of particular interest: the problem of maximising negative eigenvalues of $P_h^s$ and the one of minimising positive eigenvalues of $P_h^s$, as $h \in [g]$. These extremal values -- that we call throughout the paper \emph{conformal eigenvalues} -- are conformal invariants of $(M,g)$ and optimisers for these problems, when they exist, are known to not be smooth metrics in general. In this paper we develop a general framework that allows us to address the the existence theory for extremals of conformal eigenvalues. We define and investigate eigenvalues for singular conformal metrics, that we call \emph{generalised eigenvalues}. We develop a new variational framework for renormalised eigenvalues of any index over the set of admissible (singular) conformal factors: we obtain semi-continuity results and Euler-Lagrange equations for local extremals. Using this framework we prove, under mild assumptions on $(M,g)$ and $s$, several new (non)-existence results for extremals of renormalised eigenvalues over $[g]$. These include, among other results, a maximisation result for negative eigenvalues, the minimisation of the principal eigenvalue of $P_g^s$ and the analysis of the conformal eigenvalues of the round sphere $(\mathbb{S}^n, g_0)$. We also establish a strong connection between the existence of optimisers and (nodal) solutions of prescribed $Q$-curvature equations. Our analysis allows any order $s \ge 1$ and allows $P_g^s$ to have kernel. Previous results only covered the cases $s=1,2$ and $k=1,2$. Our work strongly generalises these results to any $s \ge 1$ and to eigenvalues of any order. 
	\end{abstract}

\maketitle

\tableofcontents

\section{Introduction and Statement of the results} \label{settings} \label{sec:intro}

\subsection{Conformal eigenvalues}

Let $(M,g)$ be a closed Riemmanian manifold of dimension $n \geq 3$, that is compact without boundary. Throughout this paper $s\in \N^*$ will denote a positive integer satisfying $s < \frac{n}{2}$, and we will denote by $P_g^s$ the celebrated GJMS operator of order $2s$ in $M$. It is an elliptic self-adjoint differential operator of order $2s$ in $M$ which was discovered by Graham, Jenne, Mason and Sparling in \cite{GJMS} and plays a central role in conformal geometry. $P_g^s$ is conformally covariant in the following sense: if $u \in C^\infty(M)$, $u>0$ and $\hat{g}  = u^{\frac{4}{n-2s}}g \in [g]$, then 
\begin{equation*} 
P_{\hat g}^s (f)  =  u^{- \frac{n+2s}{n-2s}} P_g^s(u f) \quad \text{ for all } f \in C^\infty(M). 
\end{equation*}
When $s=1$, $P_g^1$ is the celebrated conformal laplacian
$$ P_g^1 = \Delta_g + \frac{n-2}{4(n-1)} S_g,$$
where we have let $\Delta_g = - \text{div}_g(\nabla \cdot)$ and where $S_g$ denotes the scalar curvature of $(M,g)$, while for $s=2$ the operator $P_g^2$ is the so-called Paneitz-Branson operator \cite{Branson, Paneitz}. We refer to subsection \ref{GJMS} below for more detailed properties of these operators. 
Since $M$ is closed the spectrum of $P_g^s$ is discrete and consists in a non-decreasing sequence of eigenvalues that we will denote by 
$$\la_1(g) \leq \la_2(g) \leq ... \leq \la_k(g) \rightarrow_{k \to +\infty} +\infty.$$
These eigenvalues obey the famous Rayleigh characterisation: 
\begin{equation*} 
 \la_k (g)= \inf_{\underset{\dim V = k}{V \subset H^s(M)}}\max_{v \in V \setminus \{0\}}  \frac{\int_M v P_g^s v \, dv_g}{\int_M v^2 \, dv_g},
 \end{equation*}
and they possess a Hilbert basis of $L^2(M)$-normalised associated eigenfunctions satisfying $P_g^s \vp_k = \lambda_k(g) \vp_k$ in $M$ for every $k \ge 1$. Throughout this paper we will use the following notation: 
$$k_- = \max \big\{k \geq 1,  \la_k(g)<0 \big\} \; \hbox{ and } k_+= \min  \big \{k \geq 1, \la_k(g)>0 \big\}.$$
It is well-known (see e.g. \cite{CanzaniGoverJakobsonPonge} and also \cite{ElSayed}) that $k_-$ and $k_+$ are conformal invariants, hence only depend on $[g]$. With these notations we have $\dim \ker (P_g^s) = k_+-k_- - 1$ and, if $\la_1(g) \geq 0$, we set by convention $k_-=0$.

\medskip

In this article we address the following question: how do the eigenvalues of the operators $P_h^s$, for $s \ge 1$, behave as $h$ runs through the whole conformal class $[g]$? More precisely, we are interested in the problem of determining the extremal values taken by eigenvalues of $P_h^s$ when $h$ runs through  $[g]$. The case of zero eigenvalues, when they exist, is easily dealt with: if $k_+ - k_- \ge 2$ and $k_- <k < k_+ $ is fixed, the conformal covariance of $P_g^s$ shows that $\lambda_k(h) = 0$ for every $h \in [g]$, and that $h \in [g] \mapsto \lambda_k(h)$ is constant null. We will thus focus in this paper in the non-trivial cases $k \le k_-$ and $k \ge k_+$ of, respectively, negative and positive eigenvalues of $P_g^s$.  As a simple observation, the Rayleigh characterisation of the eigenvalues shows that for any metric $h$ in $M$ and for $\mu >0$ we have $ \la_k(\mu h) \Vol(M,\mu h)^{\frac{2s}{n}} = \la_k(h) \Vol(M, h)^{\frac{2s}{n}}$. In order to avoid trivial degenerations to $0$ or $\pm \infty$ the relevant quantity that we will investigate are the \emph{renormalised} eigenvalues given by
 $$h \in [g] \mapsto \la_k(h) \Vol(M,h)^{\frac{2s}{n}} \quad \text{ for some }  k \ge 1. $$ 
For each $k \ge 1$ fixed, one may wonder what the least and largest value of this functional in $[g]$ are, and this leads to two distinct optimisation problems. A surprising observation is that one of these two problems is always easily solved: we indeed have
\begin{equation} \label{eq:AmmannJammes}
 \begin{aligned}
                            \sup_{h \in [g]} \la_k(h) \Vol(M,h)^{\frac{2s}{n}} = +\infty & \quad  \hbox{ if } &   k \geq k_+,   \\
                             \inf_{h \in [g]} \la_k(h) \Vol(M,h)^{\frac{2s}{n}} = -\infty & \quad \hbox{ if } & k \leq k_- .
                         \end{aligned} 
                           \end{equation}
In other words, positive eigenvalues of $P_h^s$ are unbounded from above, and negative eigenvalues of $P_h^s$ are unbounded from below, as $h$ runs through $[g]$. Property \eqref{eq:AmmannJammes} is an involved consequence of the conformal covariance of $P_g^s$: we prove it in Proposition~\ref{prop:unbounded} by combining the results in \cite{AmmannJammes} and \cite{CaseMalchiodi}. In view of \eqref{eq:AmmannJammes} we are thus left with considering the problems of minimising positive eigenvalues and of maximising negative eigenvalues of $P_h^s$ as  $h \in [g]$. We introduce for this the following variational problems: 
\begin{equation} \label{defLambdak}
\La_k^s(M,[g]) = \left \{ \begin{aligned}
                            \inf_{h \in [g]} \la_k(h) \Vol(M,h)^{\frac{2s}{n}} & \quad \hbox{ if } &   k \geq k_+ ,\\
                             \sup_{h \in [g]} \la_k(h) \Vol(M,h)^{\frac{2s}{n}} & \quad \hbox{ if } & k \leq k_- .
                           \end{aligned} \right. 
                           \end{equation}
We will say that $\Lambda_k(M,[g])$ is the \emph{$k$-th conformal eigenvalue} of $(M,[g])$. Equivalently, since the renormalised eigenvalues are scale-invariant,
$$\La_k^s(M,[g]) = \left \{ \begin{aligned}
                            \inf_{h \in [g],  \Vol(M,h)=1} \la_k(h)  &\quad \hbox{ if } &   k \geq k_+ , \\
                             \sup_{h \in [g],  \Vol(M,h)=1} \la_k(h)  & \quad  \hbox{ if } & k \leq k_- .
                           \end{aligned} \right. $$
                           In the case where $k_+-k_- \ge 2$ and $k_- < k < k_+$ we may similarly define $\Lambda_k(M,[g]) = 0$, although this is of little interest: this invariant is trivially attained by the previous discussion. For every $k \ge 1$, $\Lambda_k^s(M,[g])$ is a conformal invariant of $(M,g)$. In the following we will thus restrict our attention to the cases $k \ge k_+$ or $k \le k_-$. In this paper we thoroughly investigate \eqref{defLambdak}: we aim at determining conditions under which $\La_k^s(M,[g])$ is attained as well as precisely describing the geometric properties of their extremals, when they exist. 
                           
\medskip

Problem \eqref{defLambdak} is inspired by similar eigenvalue-optimisation problems in conformal geometry, which have attracted a lot of attention in recent years. Most of the available litterature deals with the problem of maximising eigenvalues of the Laplace-Beltrami operator in a conformal class: this was investigated for instance in \cites{Petrides3, Petrides2, Petrides4, NadirashviliSire1, NadirashviliSire2, KNPP, KSminmax} on surfaces (that is, when $n=2$) or more recently in \cite{KS22,Petrides6} on closed manifolds of dimension $n \ge 3$. Concerning the optimisation of eigenvalues of the GJMS operator $P_g^s$ few results on the existence of extremals \eqref{defLambdak} are available in the literature and they only deal with the cases $s=1,2$ and  $k=1,2$: see \cite{AmmannHumbert, BenBou, ElSayed, GurskyPerez, PremoselliVetois3}. No results are available when $s \ge 3$ or when $k \ge 3$. Due to the conformal covariance of $P_g^s$ problem \eqref{defLambdak} is of a different nature than the eigenvalue-optimisation of Laplace-Beltrami eigenvalues (investigated e.g. in \cite{KS22,Petrides6}) and exhibits very different features : \eqref{eq:AmmannJammes} shows, for instance, that \emph{positive} eigenvalues of $P_h^s$ are not bounded from above in $[g]$, in stark contrast with the eigenvalues of the Laplace-Beltrami operator (see \cite{YangYau, LiYau, korevaar, HASSANNEZHAD, ColboisElSoufi, ElSoufiIlias}). In another direction, \cite{AmmannHumbert} shows that extremals for the second (positive) conformal eigenvalue of $P_g^1$, when they exist, are always simple; this is something that cannot happen in general for eigenvalues of the Laplace-Beltrami operator.

\medskip

 In this paper we perform a systematic investigation of problem \eqref{defLambdak}. We develop a general variational formalism that allows us to tackle \eqref{defLambdak} at any order $s \ge 1$ and for any eigenvalue $k \ge 1$ and takes into account possibly singular metrics. Our aim is twofold: on the one hand we aim at understanding the specific features of problem \eqref{defLambdak} in the broader context of eigenvalue-optimisation problems in conformal classes in dimensions $n \ge 3$; we are confident that the techniques that we develop here will prove useful in related contexts. On the other hand we aim at studying the properties of $\Lambda_k^s(M,[g])$ from the perspective of conformal geometry: the sequence $(\Lambda_k^s(M,[g]))_{k \ge1}$ provides a new family of conformal invariants of $(M,g)$, and we believe that the spectral-theoretic perspective that we propose in this paper may provide additional insights in the conformal geometry of $(M,g)$.

\subsection{Generalised eigenvalues of GJMS operators}

Throughout the paper we will assume that $P_g^s$ satisfies the following: 
\begin{equation} \label{eq:unique:continuation} 
\begin{aligned}
\text{ If } & \vp \in C^\infty(M)  \text{ satisfies }  P_g^s \vp = 0 \text{ and vanishes on a set }  \\
& \qquad \qquad \text{ of positive measure, then } \vp = 0. 
\end{aligned}
\end{equation}
Property \eqref{eq:unique:continuation} is a unique continuation property for kernel elements of $P_g^s$. It is trivially satisfied when $\ker(P_g^s) = \{0\}$: this is the case, for instance, if $P_g^s $ is coercive. When $s=1$, notably, \eqref{eq:unique:continuation} is true regardless of whether $\ker(P_g^1)$ is trivial or not: this follows from classical unique continuation results for second-order operators (see e.g. \cite{HardtSimon}).  In Proposition~\ref{prop:unique:continuation} below we give instances of $(M,g,s)$ for which \eqref{eq:unique:continuation} is satisfied: when $s \ge 2$ these include the case of analytic Riemannian manifolds and of locally conformally Einstein manifolds. To the best of our knowledge, however, \eqref{eq:unique:continuation} is not known to hold true in general when $s \ge 2$. The results that we prove in this paper all assume \eqref{eq:unique:continuation}: all our results will in particular be true under the following assumptions (which imply \eqref{eq:unique:continuation}) :

\medskip

\begin{center}
 \emph{when $s=1$, regardless of the geometry of $(M,g)$, and  [when $s \ge 1$ and $\ker(P_g^s) = \{0\}$ ]}.
 \end{center}  

\medskip

It is well-known that extremals for eigenvalue-optimisation problems, if they exist, may not be smooth metrics. When $n=2$ it was highlighted in \cite{Kokarev,Petrides2} that extremal metrics for the Laplace-Beltrami operator on surfaces may have a finite number of conical singularities. When $n \ge 3$, for the conformal Laplacian $P_g^1$ in the case $s=1$, it was observed in \cite{AmmannHumbert, GurskyPerez} that the singular set of extremal metrics can be much wilder. In order to develop a comprehensive variational theory for eigenvalue functionals, that we want to rely on to determine when the invariants $\Lambda_k^s(M,[g])$ are attained, we thus first need to generalise the notion of eigenvalues of $P_g^s$ to non-smooth metrics $g$, in order to include the possible limiting cases of non-smooth metrics. We do this in the following way: if $k \ge 1$ is an integer and $\beta \in L^{\frac{n}{2s}}(M) \backslash \{0\}$ is a nonnegative function, we will let in what follows 
\begin{equation*} 
 \la_k (\beta)= \inf_{V \in \mathcal{G}^{\beta}_k(C^{\infty}(M))}\max_{v \in V \setminus \{0\}}  \frac{\int_M v P_g^s v \, dv_g}{\int_M \beta v^2 \, dv_g}.
 \end{equation*}
Here $\mathcal{G}^{\beta}_k(C^{\infty}(M))$ denotes the set of linear subspaces $V = \text{Span}(v_1, \cdots, v_k)$ of smooth functions such that $\beta^{\frac12}v_1, \cdots, \beta^{\frac12} v_k$ are linearly independent. When $\beta = 1$, this is simply the classical min-max definition of eigenvalues of $P_g^s$. But as $\beta$ varies this gives rise to a subtler definition: since we do not assume that $\beta >0$ a.e. in $M$, $\lambda_k(\beta)$ may very well be equal to $-\infty$ when $k \le k_-$ . We will say that $\lambda_k(\beta)$ is \emph{the $k$-th generalised eigenvalue associated to $\beta$}. Its geometric interpretation is as follows: if $u \in C^\infty(M)$, $u>0$, the conformal covariance properties of $P_g^s$ and the Rayleigh characterisation of eigenvalues show that
$$  \la_k (u^{\frac{4}{n-2s}}g)= \inf_{V \subset H^s(M), \dim V = k}\max_{v \in V \setminus \{0\}}  \frac{\int_M v P_g^s v \, dv_g}{\int_M u^{\frac{4s}{n-2s}} v^2 \, dv_g}. $$ 
By a direct analogy, $\lambda_k(\beta)$ should then be formally understood as the $k$-th eigenvalue of $P_h^s$ for a formal metric $h = \beta^{\frac{1}{s}} g$. We thoroughly investigate the properties of $\lambda_k(\beta)$ in Section  \ref{sec:defvp} below: we show that it is well-defined provided \eqref{eq:unique:continuation} holds, and that one can associate generalised eigenfunctions to $\lambda_k(\beta)$. In Proposition~\ref{prop:preliminaryvarset} below we prove in particular that the following holds: 
\begin{equation} \label{def_conf_eigen}
 \La_k^s(M,[g]) = \left\{ \begin{aligned}
                             \inf_{\beta \in L^{\frac{n}{2s}}(M) \backslash \{0\}, \beta \ge 0} \la_k (\beta) \Vert \beta \Vert_{L^{\frac{n}{2s}}} 
                              & \quad \hbox{ if }  k  \geq k_+ , \\
                            \sup_{\beta \in L^{\frac{n}{2s}}(M) \backslash \{0\}, \beta \ge 0} \la_k (\beta) \Vert \beta \Vert_{L^{\frac{n}{2s}}}  
                            & \quad \hbox{ if }  k \leq k_-.
                           \end{aligned} \right.
\end{equation}
When $\beta$ positive and smooth we have $\Vol(M,\beta h)^{\frac{2s}{n}} = \Vert \beta \Vert_{L^{\frac{n}{2s}}}$: thus, \eqref{def_conf_eigen} is simply a natural reformulation of problem \eqref{defLambdak} that allows $ \La_k^s(M,[g])$ to be attained by singular metrics. In the rest of the paper we investigate whether $\Lambda_k^s(M,[g])$ is attained or not by solely focusing on problem \eqref{def_conf_eigen}. Following \eqref{def_conf_eigen} we will say that $\Lambda_k^s(M,[g])$ is \emph{attained} if there exists $\beta \in L^{\frac{n}{2s}}(M)\backslash \{0\}$, $ \beta \ge 0$, such that
$$ \Lambda_k^s(M,[g]) = \lambda_k \big( \beta \big)\Vert \beta \Vert_{L^{\frac{n}{2s}}}. $$
We will equivalently say that $\Lambda_k^s(M,[g])$ is attained at the generalised metric $\beta^{\frac{1}{s}} g$. If $\beta$ vanishes somewhere in $M$, this metric becomes singular on $\{\beta = 0\}$.

\subsection{Statement of our main results}

Our main results address the existence of extremals for $ \La_k^s(M,[g])$, and we show that the situation is radically different between negative and positive eigenvalues. 

\medskip

We first consider the case $k \le k_-$, where we maximise negative generalised eigenvalues. Our first result shows that, in this case,  a possibly singular maximal metric always exists in $[g]$: 

\begin{theo} \label{theo:vp:negatives}
Let $s \in \mathbb{N}^*$ and let $(M,g)$ be a closed Riemannian manifold of dimension $n > 2s$. We let $P_g^s$ be the GJMS operator of order $2s$ in $M$ and we assume that \eqref{eq:unique:continuation} holds and that $k_- \ge 1$. Let $k \le k_{-}$ be an integer. Then  $\Lambda_k^s(M,[g])<0$ and $\Lambda_k(M,[g])$ is attained at a generalised metric $\beta^{\frac{1}{s}} g$, where $\beta$ is a nonnegative nonzero function and $\beta \in C^{0,\alpha}(M)$ for some $0< \alpha < 1$. 
\end{theo}
A simple consequence of Theorem \ref{theo:vp:negatives} is in particular that $\Lambda_{k_-}^s(M,[g])$ is attained when $k_- \ge 1$ and \eqref{eq:unique:continuation} holds. The assumption $k_- \ge 1$ is necessary in order for negative eigenvalues to exist. Theorem \ref{theo:vp:negatives} was previously proven in the case $s=1$ and $k=2$ in \cite{GurskyPerez}, under the additional assumption $\ker(P_g^s) = \{0\}$. Our analysis only assumes \eqref{eq:unique:continuation} and when $s=1$ we are able to deal with every eigenvalue of index $k \le k_-$ without assuming that $P_g^1$ has kernel. We point out that $\Lambda_k^s(M,[g])<0$ is a consequence of our analysis, and not an assumption.

\medskip

We now turn to the case  $k \ge k_+$ and we consider the problem of minimising positive eigenvalues in $[g]$. It turns out that this case is harder since concentration phenomena may occur. In order to rule them out, we introduce an additional conformal invariant. If $(\mathbb{S}^n, g_0)$ denotes the round sphere we will denote, throughout this paper, $\Lambda_k^s(\mathbb{S}^n, [g_0])$ simply by $\Lambda_k^s(\mathbb{S}^n)$. For $k \geq k_+$ and  $s \in \N^*$ we define 
 \begin{equation} \label{definition_X}
  X_k^s(M,[g]) = \min \Bigg\{ \Big( \La_{\ell_0}^s(M,[g])^{\frac{n}{2s}} + \La_{\ell_1}^s(\mS^n)^{\frac{n}{2s}} +  \cdots + 
 \La_{\ell_r}^s(\mS^n)^{\frac{n}{2s}} \Big)^{\frac{2s}{n}} \Bigg\},
 \end{equation}
where the minimum is taken over the set of all indexes $r,\ell_0 \in \mN$ and $\ell_1,\cdots,\ell_r \in \mN^*$ such that 
\begin{enumerate}
\item  $\ell_0 \in \{ 0 \} \cup \{ k_+ ,\cdots, k - 1 \}$, where by convention we let $\Lambda_0^s(M,[g])= 0$;
 \item $\ell_0  + \cdots + \ell_r =k$ if $\ell_0 \geq k_+ $ and $\ell_1+ \cdots + \ell_r= k-k_+ +1$ if $\ell_0 =0$ ; 
 \item $\La_{\ell_i}(\mS^n)$, $1 \le i \le r$ are attained, and $\La_{\ell_0}^s(M,[g])$ is attained if $\ell_0 >0$.  
\end{enumerate}
The main result of our paper establishes that $\Lambda_k^s(M,[g])$, when $k \ge k_+$, is attained provided it lies strictly below $X_k^s(M,[g])$. It states as follows: 

\begin{theo} \label{theo:vp:positives}
Let $(M,g)$ be a closed manifold of dimension $n \ge 3$ and let $s \in \mathbb{N}^*$, $s < \frac{n}{2}$. We let $P_g^s$ be the GJMS operator of order $2s$ in $M$. Let $k \geq k_+$ be an integer. Then:
\begin{enumerate}
\item We have $\La_k^s(M,[g]) \leq  X_k^s(M,[g])$.
\item Assume that $P_g^s$ satisfies \eqref{eq:unique:continuation}. Then $\Lambda_k(M,[g])>0$. If in addition $\La_k^s(M,[g]) <  X_k^s(M,[g])$ then $\La_k^s(M,[g])$ is attained at a generalised metric $\beta^{\frac{1}{s}} g$, where $\beta$ is a nonnegative nonzero function and $\beta \in C^{0,\alpha}(M)$ for some $0< \alpha < 1$. 
\end{enumerate}
\end{theo}

The gap condition $\La_k^s(M,[g]) <  X_k^s(M,[g])$ in Theorem \ref{theo:vp:positives} is reminiscent of similar existence results for conformal eigenvalues \cite{Petrides2,Petrides3,Petrides4, Petrides6, NadirashviliSire1, KSminmax}. In the proof of Theorem \ref{theo:vp:positives} we prove that the invariant $X_k^s(M,[g])$ appears as an energy threshold for suitably constructed minimising sequences for $\Lambda_k^s(M,[g])$: more precisely, the assumption $\La_k^s(M,[g]) <  X_k^s(M,[g])$, as is now well-understood, prevents these sequences to blow-up. Theorem \ref{theo:vp:positives} greatly generalises all the available existence results for \eqref{defLambdak} and had previously only been proven in very specific settings: in \cite{AmmannHumbert} when $s=1$, $k=2$ and when $\lambda_1(g) \ge 0$, in \cite{ElSayed} when $s=1$, $k=2$ and $\lambda_1(g) < 0$, and in  \cite{BenBou} when $s=2$, $k=2$, $\lambda_1(g) \ge 0$ and $(M,g)$ is Einstein. Note also that Theorem \ref{theo:vp:positives} implies the results in \cite{SireXu}. One of the main novelties of our approach in Theorem \ref{theo:vp:positives} is that $X_k(M,[g])$ only involves invariants $\Lambda_{\ell}^s(M,[g])$ and $\Lambda_{\ell}^s(\mS^n)$ of lower-order that are \emph{themselves} attained. This is an important observation and it will allow us to obtain by induction many new (non)-existence results for extremals of the invariants $\Lambda_k(M,[g])$ that we state in the next subsections (see e.g. Theorems \ref{incr_prop}, \ref{gamma} and \ref{noncf}). An immediate consequence of Theorems \ref{theo:vp:negatives} and \ref{theo:vp:positives}, that we state for clarity, is the following:
\begin{corol} \label{Lambdak=0}
Let $(M,g)$ be a closed manifold of dimension $n \ge 3$ and let $s \in \mathbb{N}^*$, $s < \frac{n}{2}$. We let $P_g^s$ be the GJMS operator of order $2s$ in $M$. If $P_g^s$ satisfies (\ref{eq:unique:continuation}) then 
$$\Lambda_k^s(M,[g])=0 \Longleftrightarrow k_- < k <k_+.$$
\end{corol}

\subsection{Consequences of Theorems  \ref{theo:vp:negatives} and   \ref{theo:vp:positives}.}

We state in this subsection many important consequences of Theorems  \ref{theo:vp:negatives} and   \ref{theo:vp:positives}. We obtain (non)-existence results on the invariants $\La_k^s(M,[g])$ -- in particular when $(M,g)$ is the round sphere -- as well as structural properties on the behavior of possible extremals for $\Lambda_k^s(M,[g])$.

\subsubsection{Simplicity of extremals in the threshold cases $k=k_-$ and $k=k_+$.} The cases $k = k_-$ and $k= k_+$ in Theorems \ref{theo:vp:negatives} and \ref{theo:vp:positives}, corresponding respectively to the largest negative and least positive eigenvalue, are of particular interest. In these cases we can prove that $\Lambda_k^s(M,[g])$, when attained, is \emph{simple}, and we make the statement of Theorems \ref{theo:vp:negatives} and \ref{theo:vp:positives} more precise. We first consider the case $k = k_-$:
\begin{corol} \label{corol:k:moins}
Let $s \in \mathbb{N}^*$ and let $(M,g)$ be a closed Riemannian manifold of dimension $n > 2s$. We let $P_g^s$ be the GJMS operator of order $2s$ in $M$ and we assume that \eqref{eq:unique:continuation} holds and that $k_- \ge 1$.  Then the generalised eigenspace associated to $\Lambda_{k_-}^s(M,[g])$ is one-dimensional and there is $\varphi \in C^{2s,\alpha}(M)$, $0< \alpha < 1$, a generalised eigenfunction which satisfies $\Vert \vp \Vert_{L^\frac{2n}{n-2s}} = 1$ and  
$$P_g^s\varphi =  \Lambda_{k_-}^s(M,[g])|\varphi|^{\frac{4s}{n-2s}} \phi \quad { in }  \quad M. $$
Furthermore, $ \Lambda_{k_-}^s(M,[g])$ is attained at the generalised metric $|\vp|^{\frac{4}{n-2s}} g$. If in addition $s=1$ and $k_- \geq 2$, $\varphi$ changes sign. \end{corol}

If $s=1, k_- \ge 2$ and $P_g^1$ satisfies \eqref{eq:unique:continuation}, Corollary \ref{corol:k:moins} shows in particular that there exists a sign-changing solutions of the constant (negative) prescribed scalar curvature equation in $M$. This generalises the results in \cite{GurskyPerez}. The case $k = k_+$ is similar, provided we assume that $\Lambda_{k_+}^s(M,[g])$ is attained: 

\begin{corol} \label{corol:k:plus}
Let $s \in \mathbb{N}^*$ and let $(M,g)$ be a closed Riemannian manifold of dimension $n > 2s$. We let $P_g^s$ be the GJMS operator of order $2s$ in $M$ and we assume that \eqref{eq:unique:continuation} holds. Assume that $\La_{k_+}^s(M,[g]) <  \Lambda_1^s(\mathbb{S}^n)$. Then $\La_{k_+}^s(M,[g])$ is attained, the generalised eigenspace associated to it is one-dimensional and there is $\varphi \in C^{2s,\alpha}(M)$, $0< \alpha < 1$, a generalised eigenfunction which satisfies $\Vert \vp \Vert_{L^\frac{2n}{n-2s}} = 1$ and  
$$P_g^s\varphi =  \Lambda_{k_+}^s(M,[g])|\varphi|^{\frac{4s}{n-2s}} \phi \quad { in }  \quad M. $$
Furthermore, $ \Lambda_{k_+}^s(M,[g])$ is attained at the generalised metric $|\vp|^{\frac{4}{n-2s}} g$. If in addition $s=1$ and $k_+ \geq 2$, $\varphi$ changes sign. 
\end{corol}

The precise definition of generalised eigenfunctions will be given in Section~\ref{sec:defvp} below. That the functions $\varphi$ obtained in Corollaries \ref{corol:k:moins} and \ref{corol:k:plus} satisfy the (constant) prescribed $Q$-curvature equation should not be surprising. In the case where $P_g^s$ is coercive, indeed, minimising $\Lambda_1^s(M,[g])$ is equivalent to minimising the Rayleigh quotient associated to $P_g^s$ (see Proposition \ref{prop:Lambda1:Yamabe} below).  The problem of attaining the invariants $\Lambda_k^s(M,[g])$ can thus be seen as an involved generalisation of the classical Yamabe problem for the $Q$-curvature, and Theorems \ref{theo:vp:negatives} and \ref{theo:vp:positives} show that there is a direct analogy: negative eigenvalues are unconditionally attained (provided $P_g^s$ satisfies \eqref{eq:unique:continuation}), while positive ones require the strict inequality $\La_k^s(M,[g]) <  X_k^s(M,[g])$, which should be understood as a generalisation of the celebrated Aubin threshold criterion for the Yamabe equation \cite{AubinYamabe}. This provides yet another motivation to investigate the invariants $\Lambda_k^s(M,[g])$ in addition to their spectral theoretic relevance. We point out that the simplicity of $\Lambda_{k_\pm}^s(M,[g])$, when attained, is an indirect consequence of the conformal invariance of $P_g^s$ and is specific to \eqref{defLambdak}: by contrast, if one eigenvalue of the Laplace-Beltrami operator is attained by a maximal metric in the conformal class, then it is never simple for this metric.

\subsubsection{Conformal eigenvalues of the round sphere}

The conformal eigenvalues $\Lambda_k^s(\mathbb{S}^n)$ of the round sphere $(\mathbb{S}^n, g_0)$ are of paramount importance since, in view of \eqref{definition_X} and Theorem~\ref{theo:vp:positives}, they arise as energy thresholds for the existence theory of extremals of $\Lambda_k^s(M,[g])$ on a general manifold $(M,g)$. In $(\mathbb{S}^n, g_0)$ it is well-known that $\lambda_1(g_0) >0$: thus $k_- = 0$, $k_+=1$, and \eqref{eq:unique:continuation} holds for $P_{g_0}^s$. For any $s \ge 1$, by the conformal covariance of $P_{g_0}^s$, it is easily seen that $\Lambda_1^s(\mathbb{S}^n)$ is the optimal Sobolev constant in $\R^n$ and is attained at any round metric in $[g_0]$ (see Proposition \ref{prop:Lambda1:Yamabe} below). Suprisingly, however, $\Lambda_2^s(\mS^n)$ is never attained: 
 
\begin{theo} \label{lambda2:sphere}
Let $n \ge 3$ and $1 \le s < \frac{n}{2}$. Then 
$$ \Lambda_2^s(\mathbb{S}^n)^{\frac{n}{2s}} = 2 \Lambda_1^s(\mathbb{S}^n)^{\frac{n}{2s}} $$
and $\Lambda_2^s(\mathbb{S}^n)$ is not attained.  
\end{theo}
In view of Theorem \ref{lambda2:sphere} it may be tempting to believe that the set of indexes $k \ge 1$ for which $\Lambda_k^s(\mathbb{S}^n)$ is attained reduces to $\{1\}$. If $n \ge 2s+5$ we are able to prove that this is not the case. This is the content of the next result:
\begin{theo} \label{gamma}
Let $s \ge 1$ and assume that $n \ge 2s + 5$. There exists $k \in \{3, \cdots, n+1\}$ such that $\Lambda_k^s(\mathbb{S}^n)$ is attained. Furthermore, $\La_{k}^s(\mS^n)$ is not attained at a round metric.
\end{theo}
Estimating precisely the value of the minimal such $k$ for any $n$ seems difficult. As the dimension $n$ grows to infinity we are nevertheless able to obtain a much better bound than $k \le n+1$ and to prove that it remains bounded. We refer to Proposition \ref{gamma:meilleur} below. The question of determining precisely the set of indexes $k \ge 1$ for which $\Lambda_k^s(\mathbb{S}^n)$ is attained remains widely open.

\subsubsection{Extremals of  $\Lambda_{k_+}^s(M,[g])$.}

We now focus on the special case $k=k_+$ of the least positive eigenvalue. If $(M,g)$ is a closed manifold, we provide sufficient conditions on $(M,g)$ for $\La_{k_+}^s(M,[g])$ to be attained. We first assume that $P_g^s$ has no kernel: 
\begin{theo} \label{prop:kplus:atteint}
Let $s \in \mathbb{N}^*$ and $(M,g)$ be a closed Riemannian manifold of dimension $n > 2s$. We let $P_g^s$ be the GJMS operator of order $2s$ in $M$. We assume that $\ker P_g^s = \{0\}$ and that one of the following conditions is satisfied:
\begin{itemize}
\item $n \ge 2s + 4$ and $(M,g)$ is not locally conformally flat
\item~[\,$2s+1 \le n \le 2s+3$ or $(M,g)$ is locally conformally flat]   and there exists $\xi \in M$ such that $m(\xi) >0$, where $m(\xi)$ is the mass of the Green's function of $P_g^s$ at $\xi$.
\end{itemize}
Then $ \Lambda_{k_+}^s(M,[g]) < \Lambda_1^s(\mathbb{S}^n)$ and  $\Lambda_{k_+}^s(M,[g])$ is attained. 
\end{theo}
When Theorem \ref{prop:kplus:atteint} applies, Corollary \ref{corol:k:plus} then shows that $ \Lambda_{k_+}^s(M,[g])$ has a one-dimensional generalised eigenspace. When $2s+1 \le n \le 2s+3$ or when $(M,g)$ is locally conformally flat, the so-called \emph{mass} at $\xi \in M$ is defined as the constant term in the expansion of the Green's function of $P_g^s$ at $\xi$: we refer to \eqref{eq:def:mass} below for its definition. When $k_+ \ge 2$, Theorem \ref{prop:kplus:atteint} is new except in the special case $s=1$ and $k_+ =2$ where it was proven in \cite{ElSayed}. When $k_+=1$, that is when $P_g^s$ is coercive, Theorem \ref{prop:kplus:atteint} amounts to minimising the Rayleigh quotient associated to $P_g^s$ (see Proposition \ref{prop:Lambda1:Yamabe}). As such, it had already been proven in the context of the Yamabe equation when $s=1$ in \cite{AubinYamabe, SchoenYamabe}, in the context of the Paneitz equation when $s=2$ in \cite{DjadliHebeyLedoux, EspositoRobert, GurskyHangLin, GurskyMalchiodi, HangYang} and for $s \ge 2$ in  \cite{MazumdarVetois}. The novelty in Theorem \ref{prop:kplus:atteint} consists in dealing with every eigenvalue of index $k_+ \ge 1$: we develop for this a new approach by test-function computations. When $k_+ = 1, s=1$, and provided $3 \le n \le 5$ or $(M,g)$ is locally conformally flat and not conformal to $(\mathbb{S}^n, g_0)$, the celebrated positive mass theorem (\cite{SchoenYau, LeeParker}) asserts that the mass function is positive at every point. When $s=2$ and $5 \le n \le 7$ or $(M,g)$ is locally conformally flat,  the mass function is positive in $M$ provided $k_+=1$, $g$ has positive Yamabe invariant and semi-positive $Q$-curvature and $(M,g)$ is not conformally diffeomorphic to $(\mathbb{S}^n, g_0)$, see \cite{HangYang} (and previous results by \cite{GurskyMalchiodi, HumbertRaulot, Michel}). To our knowledge there are no results that ensure positivity of the mass of the Green's function of $P_g^s$ in $M$ when $s \ge 3$. We point out however that Theorem \ref{prop:kplus:atteint} only requires the positivity of the mass at one point and can thus be applied in specific examples. 

\medskip

We now consider the case where $P_g^s$ is allowed to have kernel, which fits into the framework of Theorem \ref{theo:vp:positives} provided \eqref{eq:unique:continuation} is satisfied. We prove the following analogue of Theorem \ref{prop:kplus:atteint}: 
\begin{theo} \label{prop:kplus:atteint:noyau}
Let $s \in \mathbb{N}^*$ and $(M,g)$ be a closed Riemannian manifold of dimension $n > 2s$. We let $P_g^s$ be the GJMS operator of order $2s$, and we assume that $\ker (P_g^s) \neq \{0\}$. Assume that $n \ge 4s+5$ and that $(M,g)$ is not locally conformally flat. Then 
$$ \Lambda_{k_+}^s(M,[g]) < \Lambda_1^s(\mathbb{S}^n, [g_0]).$$
As a consequence, if $P_g^s$ satisfies \eqref{eq:unique:continuation}, $\Lambda_{k_+}^s(M,[g])$ is attained. 
\end{theo}
When $P_g^s$ has kernel test-function computations for $ \Lambda_{k_+}^s(M,[g]) $ become more complicated to perform, and this is the reason why the assumptions of Theorem \ref{prop:kplus:atteint:noyau} are stronger than those of Theorem \ref{prop:kplus:atteint}. Theorem \ref{prop:kplus:atteint:noyau} is the first existence result for minimisers of $\Lambda_{k}^s(M,[g])$ for any $k \ge k_+$ and for a general $\ker(P_g^s) \neq \{0\}$, and is new even in the case $s=1$.

\subsubsection{Further results} 

We finally state a few additional consequences of Theorems \ref{theo:vp:negatives} and   \ref{theo:vp:positives}. The first one is the following strict monotonicity result for the sequence $(\Lambda_k^s(M,[g]))_k$, which can be seen as a generalised spectral gap for the conformal eigenvalues of $(M,g)$:

\begin{theo} \label{incr_prop}
Let $s \in \mathbb{N}^*$, $(M,g)$ be a closed Riemannian manifold of dimension $n >2s$ and let $P_g^s$ be the GJMS operator of order $2s$. Assume that $P_g^s$ satisfies \eqref{eq:unique:continuation}. Let $k, k' \in \N^*$ with $k<k'$ and such that $\Lambda^s_k(M,[g]) \Lambda^s_{k'}(M,[g]) \not=0$.  Then  $\Lambda^s_k(M,[g])< \Lambda^s_{k'}(M,[g])$. 
\end{theo}
Together with Corollary \ref{Lambdak=0}, this result shows that the map $k \mapsto \Lambda^s_k(M,[g])$ is strictly increasing in $\{ 1 \le k \le k_-\}$ and in $\{k \geq k_+ \}$. Note that we do not require the invariants $\Lambda^s_k(M,[g]) $ to be attained in Theorem \ref{incr_prop}. 

\medskip

The second result addresses the case of coercive operators. Let $(M,g)$ a closed Riemannian manifold of dimension $n \ge 3$. If $\lambda_1(g) >0$ (or, equivalently, if $k_+ = 1$) we will say that $P_g^s$ is coercive; in this case $\ker(P_g^s) = \{0\}$ and \eqref{eq:unique:continuation} is trivially satisfied.  If $s \ge1 $ and $n \ge 2s+5$ we denote by $k_n \in \{3, \cdots, n+1\}$ the smallest integer $k$ for which   $\La_{k}^s(\mS^n)$ is attained, and whose existence follows from Theorem \ref{gamma}.  As a surprising consequence of Theorem \ref{gamma} and of the recursive approach enabled by Theorem \ref{theo:vp:positives} we prove that, in large dimensions, $\Lambda_k^s(M,[g])$ is attained  for any $k < k_n$ if $(M,[g])$ has nontrivial geometry and $P_g^s$ is coercive:
\begin{theo} \label{noncf}
Let $s \in \mathbb{N}^*$, $(M,g)$ be a closed Riemannian manifold of dimension $n > 2s$ and let $P_g^s$ be the GJMS operator of order $2s$ in $M$. Assume that $k_+ = 1$, that $n \ge 2s+9$ and that $(M,g)$ is not locally conformally flat. Then $\La_k^s(M,[g])$ is attained for every $k\le k_n-1$. 
\end{theo}
Note that, under the assumptions of Theorem \ref{noncf}, it already followed from Theorem \ref{prop:kplus:atteint} that $\Lambda_1^s(M,[g])$ is attained. Since $k_n \ge 3$, a consequence of Theorem \ref{noncf} is in particular that, if $P_g^s$ is coercive, $(M,g)$ is not locally conformally flat and $n \ge 2s+9$, $\Lambda_1^s(M,[g])$ and $\Lambda_2^s(M,[g])$ are both attained. It has been recently proven in \cite{PremoselliVetois4} that, when $s=1$, Theorem \ref{noncf} is sharp for $k=2$: the result of  \cite{PremoselliVetois4} indeed shows that when $3 \le n \le 10$ there is an open neighbourhood $U$ of the round metric $g_0$ in $\mathbb{S}^n$ in a strong topology such that for $g \in U$, $\Lambda_2^1(\mathbb{S}^n, [g])$ is never attained and satisfies $\Lambda_2^1(\mathbb{S}^n, [g])^{\frac{n}{2}} =\Lambda_1^1(\mathbb{S}^n, [g])^{\frac{n}{2}} + \Lambda_1^1(\mathbb{S}^n)^{\frac{n}{2}}$. The result of \cite{PremoselliVetois4} is to this day the only non-existence result for extremals of $\Lambda_k^s(M,[g])$ when $(M,g)$ is not conformally equivalent to $(\mathbb{S}^n, g_0)$.

\medskip

\subsection{Strategy of proof and outline of the paper}

We introduce in Section \ref{sec:generalites} several technical results that are used throughout the paper: we define the weighted function spaces for our analysis, introduce the GJMS operators and state some of their properties. In Proposition~\ref{prop:unique:continuation} we provide conditions under which \eqref{eq:unique:continuation} is satisfied.  We also prove a key lemma that is repeatedly applied in the paper, Lemma \ref{lem:mainlemma}, that describes the asymptotic behaviour of sequences of generalised eigenfunctions. Since $\ker(P_g^s)$ may be nontrivial, sequences of generalised eigenfunctions need not be bounded in $H^s(M)$. Lemma  \ref{lem:mainlemma} describes their possible behaviors, and assumption \eqref{eq:unique:continuation} crucially comes into play to treat the $H^s(M)$-unbounded case. 

\medskip

In Section \ref{sec:defvp} we define the generalised eigenvalues and prove their main properties. We investigate the renormalised eigenvalue functionals 
$$\beta \in \mathcal{A}_p  \mapsto \bar{\lambda}_k^p(\beta) = \lambda_k(\beta)\Vert \beta \Vert_{L^p},$$
where $\mathcal{A}_p = \{ \beta \in L^{p}(M)\setminus \{0\} ; \beta \geq 0\}$ is the set of admissible weights and $p \geq \frac{n}{2s}$. These functionals are dilation-invariant and can thus be investigated over $\{ \beta \in \mathcal{A}_p, \Vert \beta \Vert_{L^p} \geq 1 \}$, which is a complete metric space for the $L^p$ distance. The geometric problem \eqref{defLambdak} corresponds to $p = \frac{n}{2s}$, but we also investigate the subcritical setting $p > \frac{n}{2s}$.  The main challenge that we face -- in particular when $P_g^s$ is not coercive -- is to define generalised eigenvalues when $\beta$ may vanish on a set of positive measure. We prove that for $\beta \in \mathcal{A}_p$ all but a finite number of generalised eigenvalues exist (Propositions \ref{prop:wellposedmu_1} and \ref{prop:deficontinuityeigen}) and investigate their continuity properties (Propositions \ref{prop:uppersemicontinuity}, \ref{prop:lipeigen} and \ref{prop:preliminaryvarset}). All these results rely on \eqref{eq:unique:continuation}. Proposition \ref{prop:preliminaryvarset} in particular proves the equivalence of \eqref{defLambdak} and \eqref{def_conf_eigen}, which justifies our choice to study the new variational problem \eqref{def_conf_eigen}. 

\medskip

In Section \ref{sec:theorievariationnellevp} we develop a variational theory for generalised eigenvalues which is of independent interest. The renormalised eigenvalue functionals are neither convex nor differentiable in general: we therefore develop and \emph{ad hoc} analysis based on so-called \emph{non-negative variations}, which first appeared in dimension 2 in \cite{Petrides5,Petrides4}. Precisely, in Proposition \ref{prop:firstderivative} below we compute right derivatives at $0$ of 
$$t \mapsto \bar{\lambda}_k^p(\beta + t b)$$ for non-negative variations $b \in \mathcal{A}_p$, so that $ \beta + t b \in \mathcal{A}_p$. This allows us, in Proposition \ref{euler_minimizer} below, to show that \emph{if} $\beta \in \mathcal{A}_p$ is a local extrema of $ \bar{\lambda}_k^p$ there is a family of generalised eigenfunctions $(v_1,\cdots,v_m)$ associated to $\lambda_k(\beta)$ such that 
\begin{equation}\label{eq:euler-lagrange}
\sum_{i=1}^m v_i^2 = \beta^{p-1} \quad \text{ a.e. in } M. 
\end{equation}
Equation \eqref{eq:euler-lagrange} is the Euler-Lagrange equation associated to local extrema of $\beta \in \mathcal{A}_p \mapsto \bar{\lambda}_k^p(\beta)$ and we crucially use it throughout the paper. An important byproduct of our analysis is that we are able to bound the number $m$ in \eqref{eq:euler-lagrange} solely in terms of the index of the eigenvalue: this was observed in \cite{PetridesTewodrose} and generalises previous results for $k=1,2$ in \cite{AmmannHumbert,GurskyPerez}. A novelty of our approach is also that \eqref{eq:euler-lagrange} holds regardless of the measure of the set $\{ \beta = 0 \}$: this is a notable improvement with respect to previous results in \cite{AmmannHumbert,GurskyPerez}\footnote{In these references partial variational theories for $\bar{\lambda}_k^p$, $k=1,2$, were established by considering internal variations of the form $t \mapsto \bar{\lambda}_k^p(\beta (1+t \phi))$ for $\phi \in L^\infty(M)$ and differentiating them at $t=0$. This, however, only proves \eqref{eq:euler-lagrange} in $\{ \beta \neq 0 \}$, which is not known a priori to be of full measure}. We more generally consider in Section \ref{sec:theorievariationnellevp} a notion of approximated extremals for $\bar{\lambda}_k^p$, which generalises the classical notion of Palais-Smale sequence for regular functionals. By adapting the approach of  \cite{Petrides5, Petrides4} we use Ekeland's variational principle to construct in Propositions \ref{prop:PS} and \ref{prop:PS2}  families of almost-optimisers for $\bar{\lambda}_k^p$. We then show that, when $p > \frac{n}{2s}$ is fixed, these families of almost-optimisers converge to extremals of $\beta \in \mathcal{A}_p \mapsto \bar{\lambda}_k^p(\beta)$ (Proposition \ref{subcritical_theorem}). This is another advantage of our new approach: the problems $\beta \mapsto \bar{\lambda}_k^p(\beta)$ for $p \ge \frac{n}{2s}$ are a natural family of relaxations of \eqref{def_conf_eigen} that fit into the same variational framework and do not require us to penalise $\{\beta = 0 \}$ as was done in \cite{GurskyPerez}.

\medskip

In Section \ref{sec:vpnegatives} we prove Theorem \ref{theo:vp:negatives}. The proof goes through an asymptotic analysis of almost maximisers for $\Lambda_k^s(M,[g])$ constructed in Proposition \ref{prop:PS2}. Possible blow-up of such sequences is easily ruled out by Lemma \ref{lem:mainlemma} since $\Lambda_k^s(M,[g])< 0$.

\medskip

Section \ref{proof_theorems} contains the proof of our main result, Theorem \ref{theo:vp:positives}. The proof goes again through an asymptotic analysis of a suitable minimising sequence: for $p > \frac{n}{2s}$ we consider the minimiser $\beta_p$ of $\beta \mapsto \bar{\lambda}_k^p(\beta)$ given by Proposition \ref{subcritical_theorem}, normalised so that $\Vert \beta_p \Vert_p = 1$. The family $(\beta_p)_{p > \frac{n}{2s}}$ satisfies \eqref{eq:euler-lagrange} and provides, as $p \to \frac{n}{2s}$, a well-behaved approximating sequence for $\Lambda_k^s(M,[g])$,
whose asymptotic behavior we investigate as $p \to \frac{n}{2s}$. Let $v_p = (v_1^p,\cdots,v_m^p)$ be the generalised eigenfunctions given by \eqref{eq:euler-lagrange} and $|v_p|^2 = \sum_{i=1}^m (v_i^p)^2$. Then, as $p \to \frac{n}{2s}$, we have $\beta_p = \vert v_p \vert^{\frac{2}{p-1}}$ and $v_p$ satisfies
\begin{equation*} 
P_g^s v_p  = \lambda_p \vert v_p \vert^{\frac{2}{p-1}} v_p \quad \text{ in } M,  \text{ and } \quad \int_M  \vert v_p \vert^{\frac{2p}{p-1}} dv_g = 1, 
\end{equation*}
where $\lambda_p = \lambda_k(\beta_p)$. This system is asymptotically critical since $\frac{2}{p-1} \to \frac{4s}{n-2s}$ as $p \to \frac{n}{2s}$: since in addition $\lambda_p >0$ (because $k\geq k_+$) bubbling phenomena for $v_p$, and thus for $\beta_p$, may occur as $p \to \frac{n}{2s}$. We perform an asymptotic analysis of $v_p$ and show that if it blows-up then $\Lambda_k^s(M,[g]) \ge X_k^s(M,[g])$, where $X_k^s(M,[g])$ is as in \eqref{definition_X}, which contradicts the assumptions of Theorem  \ref{theo:vp:positives}. Since $k \ge k_+$ is any integer, there can be arbitrarily many bubbles involved and a general multi-bubble analysis can be difficult to manage. We bypass this difficulty by proving a one-bubble extraction result that splits the energy in $\Lambda_k^s(M,[g])$ into a weak-limit part in $(M,g)$ (that may be equal to zero) and a bubbling part in $(\mathbb{S}^n,g_0)$. This is proven in Theorem \ref{prop:bubbling} below and is the core of the analysis in the proof of Theorem \ref{theo:vp:positives}; it is also one of the originalities of our approach. The crucial point here is that the energy of the weak limit part, if non-zero, is given by some $\Lambda_{\ell_0}^s(M,[g])$, with $\ell_0 \le k-1$, \emph{which is itself attained}. This justifies the definition of  $X_k^s(M,[g])$ in \eqref{definition_X}. We then conclude the proof by recursively applying our bubble-extraction result to the bubbling part itself in $(\mathbb{S}^n,g_0)$. Again, a challenge in the proof of Theorem \ref{prop:bubbling} is the possible kernel of $P_g^s$, and the unboundedness of $v_p$ in $H^s(M)$ that may ensue. This situation may only occur when $\beta_p \rightharpoonup 0 $ in $L^{\frac{n}{2s}}(M)$ and we analyse it separately. If $w_i^p$ denotes the orthogonal projection of $v_i^p$ onto $\ker(P_g^s)^\perp$, our main observation is that $(w_1^p, \cdots, w_m^p)$ provides an asymptotically better competitor for $\Lambda_k^s(M,[g])$: we therefore perform an asymptotic analysis on $w_i^p$ in this case. 

\medskip

Section \ref{sec:vpsphere} focuses on the case where $P_g^s$ satisfies the maximum principle. We prove the simplicity of the first generalised eigenfunction in this case (Proposition \ref{prop:vp1:simple}) and, if $\beta$ is an extremal of $\beta \in \mathcal{A}_p \mapsto \bar{\lambda}_k^p(\beta)$, we provide conditions ensuring that some generalised eigenspaces are one-dimensional (Proposition \ref{prop:espace1dim}). These results strongly rely on the Euler-Lagrange relation \eqref{eq:euler-lagrange}. As a consequence of these results and of Theorems  \ref{theo:vp:negatives} and \ref{theo:vp:positives} we prove Corollaries \ref{corol:k:moins} and \ref{corol:k:plus} and 
 Theorems \ref{lambda2:sphere} and Theorem \ref{gamma}. 

\medskip

Section \ref{sec:monotonie} is devoted to the proof of Theorem \ref{incr_prop}. We show that the strict monotonicity of the sequence $(\Lambda_k^s(M,[g]))_k$ follows from \eqref{eq:euler-lagrange} and from a recursive argument that relies on Theorem \ref{theo:vp:positives}. 

\medskip

Finally, Section \ref{sec:fonctionstest} is devoted to the proof of Theorems \ref{prop:kplus:atteint}, \ref{prop:kplus:atteint:noyau} and \ref{noncf}. In each case we prove that the strict inequality in the statement of Theorem \ref{theo:vp:positives} holds. We do this for arbitrary values of $k$ and $k_+$ by developing new test-functions arguments. We adapt the test-function approaches developed for the prescribed $Q$-curvature equation (see for instance \cite{MazumdarVetois}) to the setting of a generalised eigenvalue of arbitrary index.

\section{Technical results} \label{sec:generalites}

\subsection{Function spaces} \label{function_spaces}

We now introduce the function spaces that we will need to define eigenvalues of generalised metrics. Let $p\geq 1$. If $\beta \in L^p(M) \backslash \{0\}$ satisfies $\beta \geq 0$ a.e we let
$$ \begin{aligned}
L^2(M, \beta dv_g) 
  &= \Big \{ v: M \to \R \text{ measurable such that }\int_{M}  v^2 \beta dv_g < + \infty \Big \},
\end{aligned} $$
that we endow with the semi-norm
$$ \Vert v \Vert_{L^2_\beta} = \left(\int_{M}  v^2 \beta dv_g\right)^{\frac{1}{2}} $$
We now set 
\begin{equation} \label{def:Kbeta}
K_\beta = \Big \{ v: M \to \R \text{ measurable such that }  \beta^{\frac{1}{2}} v = 0 \text{ a.e. in } M \Big \} 
\end{equation}
Then $L^2_\beta(M)$, defined by
$$ \begin{aligned}
L^2_\beta(M) = L^2(M, \beta dv_g) / K_\beta
 \end{aligned}  $$
and endowed with the scalar product $(v,w)_{L^2_\beta(M)} = \int_{M}  vw \beta dv_g$ associated to $\Vert \cdot \Vert_{L^2_\beta}$, defines a Hilbert space. Let $s\in \N^*$ be a non-zero integer such that $n>2s$. We will let throughout this paper
$$ \Delta_g = - \text{div}_g (\nabla \cdot ) $$
be the Laplace-Beltrami operator. Let $H^s(M)$ be the closure of $C^\infty(M)$ for the norm 
\begin{equation} \label{eq:normeHs}
\Vert u \Vert_{H^s} = \left( \int_M \big|\Delta_g^{\frac{s}{2}} u \big|^2 dv_g +\int_M u^2 dv_g \right)^{\frac{1}{2}} , \quad u \in C^\infty(M)
\end{equation}
where we have let 
$$\Delta_g^{\frac{s}{2}} u = \left \{ 
\begin{aligned}
& \Delta_g^{m } u & \text{ if } s = 2m \text{ is even}, \\
& \nabla  \Delta_g^{m } u & \text{ if } s = 2m+1 \text{ is odd}.
\end{aligned} \right. $$
It is well-known that $H^s(M)$ defines a Hilbert space whose norm is given by the following scalar product: 
\begin{equation} \label{eq:scal:Hs}
\langle u, v \rangle_{H^s} =\langle u , \Delta_g^s v \rangle_{L^2} + \langle u, v\rangle_{L^2} .
\end{equation}
Equivalently, $H^s(M)$ is also the completion of $C^\infty(M)$ with respect to the norm 
 $$\Vert u \Vert_{W^{s,2}} = \left( \sum_{\ell=0}^s \int_M \big|\Delta_g^{\frac{\ell}{2}} u \big|^2 dv_g  \right)^{\frac{1}{2}} , \quad u \in C^\infty(M),$$
 and the latter is equivalent to \eqref{eq:normeHs}. For simplicity we will always work with \eqref{eq:normeHs}. For $p \ge 1$ we define the following metric space:
\begin{equation} \label{def:Ap}
 \mathcal{A}_p = \{ \beta \in L^p(M) \setminus \{0\}  ,\beta \geq 0 \text{ a.e } \} ,
 \end{equation}
endowed with the distance associated to the $L^p$ norm and the associated topology. Sobolev embeddings ensure that $H^s(M)$ continuously embeds into $L^{\frac{2n}{n-2s}}(M)$ when $n >2s$. As a consequence, if  $p \ge \frac{n}{2s}$, $\beta \in \mathcal{A}_p$ and $v \in H^s(M)$, we have $\beta v^2 \in L^1(M)$ by H\"older's inequality. For $p \ge \frac{n}{2s}$ and $\beta \in \mathcal{A}_p$ we may thus define 
$$\Vert u \Vert_{H^s_\beta}= \left(\int_M \big|\Delta_g^{\frac{s}{2}} u \big|^2 dv_g +\int_M \beta u^2 dv_g \right)^{\frac{1}{2}} \quad \text{ for } u \in H^s(M). $$

\begin{prop} \label{prop:eqnormscompactness}
Let $s \in \N^*$, $2s < n$, and assume that $p \geq \frac{n}{2s}$.
\begin{enumerate}
\item For all $\beta \in \mathcal{A}_p$, there is an open neighborhood $U_\beta$ of $\beta$ in $\mathcal{A}_p$ and a constant $C_\beta = C(n,s,p,g,\beta) >0$ such that 
$$ \forall \tilde{\beta}\in U_\beta, C_\beta^{-1} \Vert u \Vert_{H^s} \leq  \Vert u \Vert_{H^s_{\tilde{\beta}}} \leq C_\beta \Vert u \Vert_{H^s} \quad \text{ for all } u \in H^s(M). $$
\item If $(v_i)_{\geq 0}$ is a bounded sequence in $H^s(M)$ and $\beta_i \to \beta$ in $\mathcal{A}_p$ then, up to a subsequence, $(v_i)_{i\geq 0}$ weakly converges to $v$ in $H^s(M)$ and
\begin{equation} \label{eq:cvforte}
\int_{M} (v_i - v)^q \beta_i dv_g \to 0 \quad \text{ as } i \to + \infty. 
\end{equation}
for any $1 \leq q \le 2$.
\item Let $\beta \in \mathcal{A}_p$ such that $\beta>_{a.e}0 $. If $T$ is a differential operator of order at most $2s-1$ with smooth coefficients on $M$ and $\ep_0 >0$, then there is an open neighborhood $V_\beta$ of $\beta$ in $\mathcal{A}_p$ and a constant $K = K(p,T,\ep_0,\beta)$ such that for all $\tilde{\beta}\in V_\beta$,
$$ \forall v \in H^s(M), \left\vert \int_{M} vTv dv_g \right\vert \leq \ep_0 \int_M \big|\Delta_g^{\frac{s}{2}} v \big|^2 dv_g + K \int_M \tilde{\beta} v^2 dv_g $$
\end{enumerate}
\end{prop}

If $\beta \in \mathcal{A}_p$ for some $p \ge \frac{n}{2s}$, Proposition \ref{prop:eqnormscompactness} shows that $u \mapsto \Vert u \Vert_{H^s_{\beta}}$ is a norm on $H^s(M)$. We can thus  define a space $H^s_{\beta}(M)$ as the closure of $C^\infty(M)$ for $\Vert \cdot \Vert_{H^s_\beta}$ and Proposition  \ref{prop:eqnormscompactness} show that $H^s(M) = H^s_\beta(M)$ and that the norms $\Vert \cdot \Vert_{H^s}$ and $\Vert \cdot \Vert_{H^s_\beta}$ are equivalent on $H^s(M)$. Clearly, $H^s_\beta(M)$ is a Hilbert space for the scalar product  
$$\langle \cdot,\rangle_{H^s_\beta}=\langle \cdot , \Delta_g^s \cdot \rangle_{L^2} + \langle \cdot ,  \cdot \rangle_{L^2_\beta} .$$

\begin{proof}
We first prove (2). Let $(v_i)_{i\geq 0}$ be a bounded sequence in $H^s(M)$ and $\beta_i \in \mathcal{A}_p$ such that $\beta_i \to \beta$ in $\mathcal{A}_p$. Up to taking a subsequence, $(v_i)_{i\geq 0}$ weakly converges to some function $v$ in $H^s(M)$. Let $R>0$. Then, if $1 \le q \le 2$ is fixed,
$$ \begin{aligned}
\int_M \beta_i(v_i-v)^q = & \int_{M} \left(\beta_i-\beta\right) (v_i - v)^qdv_g + \int_{\{\beta \le R \}} \beta(v_i - v)^qdv_g \\
&+ \int_{\{ \beta>R\}} \beta(v_i - v)^qdv_g.
\end{aligned} $$
Up to passing to a subsequence again $v_i \to v$ in $L^q(M)$, so the second integral in the right-hand side converges to $0$. The first and last integrals are estimated with the continuous embedding $H^s(M) \subset L^{\frac{2n}{n-2s}}(M)$. Since $\Vert v_i \Vert_{H^s} $ is bounded we obtain
$$ \int_{M} \left(\beta_i-\beta\right) (v_i - v)^qdv_g \leq C \Big( \int_{ M } (\beta-\beta_i)^{\frac{n}{2s}}dv_g \Big)^{\frac{2s}{n}} \to 0 $$
as $i \to + \infty$ since $\beta_i \to \beta$ in $L^p(M)$ and $p \ge \frac{n}{2s}$. As a consequence, 
$$ \limsup_{i \to + \infty}\int_{M} \beta(v_i - v)^qdv_g \le C \Big( \int_{\{ \beta >R\}} \beta^{\frac{n}{2s}}dv_g \Big)^{\frac{2s}{n}} $$
for some $C>0$ independent of $R$. Letting $R \to + \infty$ proves \eqref{eq:cvforte}.

\medskip

We now prove (1). Let $\beta, \tilde{\beta} \in \mathcal{A}_p $ be such that $\Vert\tilde{\beta}-\beta \Vert_{L^{p}} \leq \Vert \beta \Vert_{L^{p}} $, then for $v \in H^s(M)$
$$ \int_{M} \tilde{\beta} v^2 \leq \Vert \tilde{\beta} \Vert_{L^{\frac{n}{2s}}} \Vert v \Vert_{L^{\frac{2n}{n-2s}}}^2 \leq C \Vert \beta \Vert_{L^{\frac{n}{2s}}} \Vert v \Vert_{L^{\frac{2n}{n-2s}}}^2 \leq C \Vert \beta \Vert_{L^{\frac{n}{2s}}} \Vert v \Vert_{H^{s}}^2 $$
where we used the classical Sobolev embedding $H^s(M) \subset L^{\frac{2n}{n-2s}}(M) $. Therefore, 
$$ \Vert v \Vert_{H^s_{\tilde{\beta}}}^2 \leq \left( 1+ C\Vert \beta \Vert_{L^{p}}\right) \Vert v \Vert_{H^{s}}^2 $$
Let's prove the reverse inequality. We assume by contradiction that there is $\beta_i \to \beta$ and a sequence $v_i \in H^s(M)$ such that $\Vert v_i \Vert_{H^s} = 1 $ and $\Vert v_i \Vert_{H^s_{\beta_i}} \to 0 $ as $i\to +\infty$. Since $(v_i)_{i\geq 0}$ is bounded in $H^s(M)$, up to the extraction of a subsequence it converges to $v$ weakly in $H^s(M)$ and strongly in $L^2(M)$. The assumption $\Vert v_i \Vert_{H^s_{\beta_i}} \to 0 $  implies that 
$$ \Vert \Delta_g^{\frac{s}{2}} v_i \Vert_{L^2}^2 \to 0, $$
which in turn implies that $\int_M v_i ^2 dv_g \to 1$ and hence that $\int_{M} v^2 dv_g = 1$. This implies that $v$ is a non-zero constant function. Notice that by \eqref{eq:cvforte}, we also have that up to a subsequence,
$$ \int_M (v-v_i)^2 \beta_i dv_g \to 0 $$
as $i\to +\infty$, which implies that $\int_M \beta v^2 dv_g = 0$, and hence that $\int_M \beta dv_g = 0$, which is a contradiction. 

\medskip

We finally prove (3) by contradiction. Assume that there is a sequence $(\beta_i)_{i\geq 0}$ in $\mathcal{A}_p$ such that $\beta_i \to \beta$ in $L^p(M)$ and a sequence $(v_i)_{i\geq 0}$ in $H^s(M)$ such that
\begin{equation} \label{eq:contradictT}\left\vert \int_{M} v_iTv_i dv_g \right\vert > \ep_0 \int_M \big|\Delta_g^{\frac{s}{2}} v_i \big|^2 dv_g + i \int_M \beta_i v_i^2 dv_g. \end{equation}
up to dividing $v_i$ by $\Vert v_i \Vert_{H^s}$, we can assume that $\Vert v_i \Vert_{H^s} = 1$. Then, up to taking a subsequence $(v_i)_{i\geq 0}$ weakly converges to $v$ in $H^s(M)$ and by compact Sobolev embeddings and by (2),
$$ \int_M v_i T v_i dv_g \to \int_M v T v dv_g \text{ and } \int_M \beta_i v_i^2 dv_g \to  \int_M \beta v^2 dv_g $$
as $i \to +\infty$. Passing to the limit in \eqref{eq:contradictT}, we obtain that  $ \int_M \beta v^2 = 0$.  Since $\beta>_{a.e}0$, this means that $v = 0$. Then by \eqref{eq:contradictT} and compact Sobolev embeddings,
$$ \ep_0 \int_M \big|\Delta_g^{\frac{s}{2}} v_i \big|^2 dv_g \leq  \int_M v_i T v_i dv_g \to 0  $$
as $i\to +\infty$, so that $v_i$ converges strongly to $0$ in $H^s(M)$, contradicting $\Vert v_i \Vert_{H^s} = 1$.
\end{proof}

\subsection{Generalities on GJMS operators} \label{GJMS}
Let $(M,g)$ be a closed Riemannian manifold of dimension $n \ge 3$, that is compact without boundary, and let $s \in \mathbb{N}^* = \mathbb{N}\backslash \{0\}$ with $s < \frac{n}{2}$. In the following we will denote by $P_g^s$ be the GJMS operator of order $2s$ in $M$. It is a differential operator in $M$, introduced in Graham-Jenne-Sparkling-Mason \cite{GJMS}, which is conformally covariant in the following sense: if $u \in C^\infty(M)$, $u>0$ and $\hat{g}  = u^{\frac{4}{n-2s}}g$ is a metric conformal to $g$,
\begin{equation} \label{eq:confinv}
P_{\hat g}^s (f)  =  u^{- \frac{n+2s}{n-2s}} P_g^s(u f) \quad \text{ for all } f \in C^\infty(M). 
\end{equation}
The operators $P_g^s$ are self-adjoint in $L^2(M)$ (see e.g. \cite{GrahamZworski}) and their leading-order term is the $s$-th power of $\Delta_g$: precisely (see again \cite{GJMS}) we have
\begin{equation} \label{eq:def:Ag}
 P_g^s = \Delta_g^s + A_g^s
 \end{equation}
 where $\Delta_g = - \text{div}_g(\nabla \cdot)$ and where $A_g^s$ is a differential operator with smooth coefficients of order at most $2s-1$ which ensures the conformal invariance. If $s=1$ and $n \ge 3$, $P_g^1$ is the celebrated conformal Laplacian:
$$ P_g^1 = \Delta_g + \frac{n-2}{4(n-1)} S_g, $$
where $S_g$ is the scalar curvature of $(M,g)$. If $s=2$ and $n \ge 5$, $P_g^2$ is the Paneitz-Branson operator \cite{Branson, Paneitz}, whose expression is 
$$ P_g^2u = \Delta_g^2 u - \text{div}_g \left[  \left( a_n S_g g  - \frac{4}{n-2} \text{Ric}_g \right)(\nabla u, \cdot) \right] + \frac{n-4}{2} Q_g u $$
for $u \in C^\infty(M)$. Here we have let $a_n =\frac{(n-2)^2+4}{2(n-1)(n-2)}$, $\text{Ric}_g$ denotes the Ricci tensor of $(M,g)$ and $Q_g$ is the so-called $Q$-curvature whose expression is given by 
$$ Q_g = \frac{1}{2(n-1)} \Delta_g S_g + \frac{n^3-4n^2+16n-16}{8(n-1)^2(n-2)^2} S_g^2 - \frac{2}{(n-2)^2}|\text{Ric}_g|_g^2. $$
\medskip
An expression of $P_g^3$ can be found in \cite{Juhl}, but explicit formulae for $P_g^s, s\ge 4$, for a general metric $g$ do not exist,  since the lower-order terms in \eqref{eq:def:Ag} grow exponentially complex in terms of the geometry of $(M,g)$. Recursive formulae have however been obtained in  \cite{Juhl}. An explicit formula exists if we assume that $(M,g)$ is an Einstein manifold: it is indeed proven in Fefferman-Graham \cite[Proposition 7.9]{FeffermanGraham} that if $\text{Ric}(g) =  \frac{S}{n} g$, then  
\begin{equation} \label{factor:Pg}
 P_{g}^{s} = \prod_{j=1}^s \Big( \Delta_{g} + \frac{(n+2j-2)(n-2j)}{4n(n-1)} S \Big). 
 \end{equation}

\subsection{Unique continuation for GJMS operators}

As mentioned in the introduction we will assume throughout this paper that the operator $P_g^s$ satisfies the unique continuation property \eqref{eq:unique:continuation}. We recall the definition of the kernel of $P_g^s$:
$$ \ker (P_g^s) = \big \{ v \in H^s(M), P_g^s v = 0 \big \}. $$
By standard elliptic theory, any $v \in \ker (P_g^s)$ is smooth in $M$. As will be explained in Section~\ref{sec:defvp} (see e.g. Proposition~\ref{prop:wellposedmu_1} below) assumption \eqref{eq:unique:continuation} is crucial in order to be able to properly define the generalised eigenvalues $\lambda_k(\beta)$ when $\beta \in L^{\frac{n}{2s}}(M) \backslash \{0 \}$ is only a nonnegative function. Assumption \eqref{eq:unique:continuation}  is obviously true when $\ker(P_g^s) = \{0\}$, which happens for instance if $P_g^s$ is coercive (see e.g. \eqref{def:coercivite} below). When $\ker(P_g^s) \neq \{0\}$, \eqref{eq:unique:continuation} ensures that non-zero functions of $\ker (P_g^s)$ vanish on a set of measure zero in $M$. The following result lists several important situations where $P_g^s$ may have kernel and where \eqref{eq:unique:continuation} holds true:

\begin{prop} \label{prop:unique:continuation}
Let $(M,g)$ be a smooth closed Riemannian manifold of dimension $n \ge 3$ and let $s \in \mathbb{N}^*$, with $2s < n$. We let $P_g^s$ be the GJMS operator introduced in  subsection \ref{GJMS}. Assume that one of the following assumptions is satisfied: 
\begin{itemize}
    \item $s=1$ 
    \item $s \ge 1$ and $\ker(P_g^s) = \{0\}$ 
    \item $s \ge 1$ and $M$ and $g$ are analytic 
   \item $s \ge 1$ and $(M,g)$ is locally conformally Einstein.
\end{itemize}    
Then $P_g^s$ satisfies \eqref{eq:unique:continuation}. 
    \end{prop}
    Remark that the locally conformally Einstein case covers the case where $(M,g)$ is locally conformally flat. 
    
\begin{proof}
If $\ker(P_g^s) = \{0\}$ \eqref{eq:unique:continuation} is obvious. If $s=1$, $P_g^1 = \Delta_g + \frac{n-2}{4(n-1)}S_g$ is the conformal laplacian and \eqref{eq:unique:continuation} follows from the unique continuation property for second-order operators which is proven for instance in  Hardt-Simon \cite[Theorem 1.10]{HardtSimon}. If $s \ge 1$ and $M$ and $g$ are analytic this follows from the results in \cite{Kukavica}. 

We finally assume that $s \ge 1$ and $(M,g)$ is locally conformally Einstein. We let $\vp \in \ker(P_g^s)$ and assume that $\{ \vp = 0\}$ has positive measure in $M$. Since $M$ is compact there exists $x \in \{ \vp = 0 \}$ such that $\{\vp = 0 \} \cap B_g(x, \delta)$ has positive measure for all $\delta >0$. Let $U$ be a neighbourhood of $x$ such that $\{\vp = 0 \} \cap U$ has positive measure and such that $g$ is conformal to an Einstein metric in $U$: $g = u^{ \frac{4}{n-2s}} \tilde{g}$ in $U$ for some positive $u \in C^\infty(U)$, where $\tilde g$ is Einstein. By the conformal covariance \eqref{eq:confinv} of $P_g^s$,  $v = u \vp$ satisfies $P_{\tilde{g}}^s v = 0$ in $U$ and $|U \cap \{v = 0\} | >0$.  Since $\tilde{g}$ is an Einstein metric, $P_{\tilde{g}}^s$ is the product of second-order Schr\"odinger operators by \eqref{factor:Pg}. It is then easily seen that \eqref{eq:unique:continuation} follows from the next lemma, which we now prove: 

\begin{lemme} 
Let $(M,g)$ be a closed Riemannian manifold, let $m\in \mathbb{N}^*$ and let $\Omega \subseteq M$ be an open set. We set $P_m = L_m \circ \cdots \circ L_1$ where for all $1\leq i \leq m$, $L_i = \Delta_g + T_i$ and $T_i$ is a differential operator of order $\leq 1$ with smooth coefficients. Let $u \in C^\infty(\Omega)$ be such that $P_m u = 0$ in $\Omega$ and such that $\{ u = 0 \} $ has positive measure. Then $u \equiv 0$ in $\Omega$. 
\end{lemme}

\begin{proof}
By a simple covering argument we may assume that $\Om$ is a chart domain in $M$, and that the operators $L_i$ can be written in coordinates as elliptic operators with smooth coefficients in open sets of $\R^n$. We prove the result by induction on $m$. For $m=1$ the result follows from the unique continuation results in Hardt-Simon \cite[Theorem 1.10]{HardtSimon}. Let $m \ge 2$ and assume that the property holds for $m-1$. Let $u \in C^\infty(\Omega)$ be such that $P_m u = 0$ in $\Omega$ and $|\{u = 0\}| >0$. 
A standard consequence of the chain rule for Sobolev functions is that $\nabla u = 0$ a.e. in $\{u=0\}$ and thus, iteratively, that $\nabla^{2m-2}u = 0$ a.e. in $\{u=0\}$. 
Therefore, setting
$$ v = (L_{m-1} \circ \cdots \circ L_1)u ,$$
we get that $v$ vanishes in $\{u=0\}$, which has positive measure in $\Omega$. The assumption $P_mu = 0$ implies that $ L_m v = 0$ in $\Omega$, which implies that $v\equiv0$ in $\Omega$, again by \cite[Theorem 1.10]{HardtSimon}. Hence $ (L_{m-1} \circ \cdots \circ L_1)u = 0 $ in $\Omega$ and the induction assumption applied to $u$ shows that $u\equiv 0$ in $\Omega$.
\end{proof}
\end{proof}

\begin{rem}    
It is likely that \eqref{eq:unique:continuation} also holds true for any closed manifold $(M,g)$, when $s \in \{1,2,3,4\}$ and $2s < n$. Indeed, the conformal covariance \eqref{eq:confinv} of $P_g^s$ shows that the zero sets of kernel elements are conformal invariants of $(M,g)$, and it is enough to prove property \eqref{eq:unique:continuation} for any fixed metric in $[g]$. Around any point of $M$, and up to a conformal change of $g$, one can use conformal normal coordinates (as introduced in \cite{LeeParker}) where $\det g \equiv 1$. In this coordinate chart $P_g^s$ rewrites as $P_g^s = \Delta_g^s + T$ where $T$ is a differential operator of order $\leq 2s-2$: this is proven for instance in \cite[Equation (2.7)]{MazumdarVetois}. Now, in open sets of $\R^n$ there is a well-established unique continuation theory for polyharmonic operators, due to \cite{Protter} (see also \cite{LinUC}). It states that, if $T$ is a differential operator with smooth coefficients of order $\le  [\frac{3s}{2}]$, solutions $u$ of equations of the form $\Delta_\xi^s u + T u = 0$ in a ball $B$ that vanish at infinite order at a point must satisfy $u=0$ in $B$. By analogy with these results we may expect that the same unique continuation property will hold true for $P_g^s$ provided the assumptions of \cite{LinUC, Protter} are satisfied, that is provided $2s-2 \le [\frac{3s}{2}]$, which is equivalent to $s \in \{1,2,3,4\}$. Note however that the analysis in \cite{LinUC, Protter} is only written for the flat Laplacian and one would have to carefully check that the arguments adapt for general Laplace-Beltrami operators to conclude.
    \end{rem}

\subsection{Convergence of generalised eigenfunctions}
In the whole paper we will need to study the behavior of sequences of functions $(v_i)_{i\ge0}$ in $H^s(M)$ satisfying eigenfunction-type equations like
\begin{equation*}
 P_g^s v_i = \lambda_i \beta_i v_i  \quad \text{ in } M, 
 \end{equation*}
where $P_g^s$ is the GJMS operator of order $2s$ in $M$ (see subsection \ref{GJMS}), $(\lambda_i)_{i \ge0}$ is a sequence of real numbers and $(\beta_i)_{i \ge0}$ a sequence of functions that are uniformly bounded in some $L^p(M)$ space. We prove in this subsection two technical results that will be used throughout the paper. 

\medskip

The first result addresses the \emph{a priori} boundedness of $(v_i)_{i \ge 0}$ in $H^s(M)$. This question is easily answered when $P_g^s$ has no kernel but the problem becomes more subtle in the presence of a non-trivial kernel. We recall that 
$$ \ker (P_g^s) = \big \{ v \in H^s(M), P_g^s v = 0 \big \}. $$
By standard elliptic theory, any $v \in \ker (P_g^s)$ is smooth in $M$. In the following $\ker(P_g^s)^{\perp}$ will denote the orthogonal of $\ker(P_g^s)$ with respect to the $H^s(M)$-scalar product defined by \eqref{eq:scal:Hs}.

\begin{lemme} \label{lem:mainlemma} Let $s \in \mathbb{N}^*$, $2s <n$, $(M,g)$ be a closed manifold of dimension $n \ge 3$ and $P_g^s$ be the GJMS operator of order $2s$ on $(M,g)$. We assume that \eqref{eq:unique:continuation} holds true. Let $(v_i)_{i \ge0}$ be a sequence in $H^s(M)$, $(\lambda_i)_{i\ge0}$ be a sequence of real numbers, $(p_i)_{i \ge0}$ be a sequence of positive numbers satisfying $p_i \ge \frac{n}{2s}$ for all $i \ge 0$ and, for any $i \ge0$, let $\beta_i \in \mathcal{A}_{p_i}$ be a nonnegative nonzero function. Assume that we have, for all $i \ge 0$, 
\begin{equation}  \label{eq:vecteursvi}
 P_g^s v_i = \lambda_i \beta_i v_i  \quad \text{ in } M, 
 \end{equation}
 $ \int_M \beta_i v_i^2 = 1$, $\int_{M} \beta_i^{p_i} dv_g = 1$, $\lambda_i\neq 0$ and $\limsup_{i\to +\infty} \vert \lambda_i \vert < +\infty$. For all $i \ge 0$ we decompose $v_i$ as 
 $$ v_i = k_i + w_i $$
where $k_i \in \ker(P_g^s)$ and $w_i \in \ker(P_g^s)^{\perp}$. Then 
\begin{enumerate}
\item $ ( w_i )_{i \ge0} $ is bounded in $H^s(M)$
\item There is $c>0$ such that, for all $i\in \mathbb{N}$, $\left\vert \lambda_i \right\vert \geq c$ and $\Vert w_i \Vert_{H^s} \ge c$
\item If $\Vert v_i \Vert_{H^s} \to +\infty$ as $i\to +\infty$ then, $p_i \to \frac{n}{2s}$ from above, $\lambda_i \geq c$ for $i$ large enough, $(w_i)_{i \ge0}$ weakly converges to $0$ in $H^s(M)$, $(\beta_i)_{i \ge0}$ weakly converges to $0$ in $L^q(M)$ for all $q \leq \frac{n}{2s}$ and we have up to the extraction of a subsequence that
$$ v_i = \gamma_i \big(K + o(1) \big) \quad \text{ as } i \to + \infty,$$
where $K$ is a non-zero element of $\ker(P_g^s)$, where $\gamma_i = \Vert v_i \Vert_{H^s}$ and where $o(1)$ denotes a sequence that strongly converges to $0$ in $H^s(M)$.
\end{enumerate}
\end{lemme}

\begin{proof}
A simple observation, that easily follows from standard elliptic theory, is that there exists $C >0$  such that, for any $u \in \ker(P_g^s)^{\perp}$, 
\begin{equation} \label{eq:inv:noyau}
 \Vert u \Vert_{H^s} \le C \Vert P_g^s u \Vert_{H^{-s}},
 \end{equation}
 where $H^{-s}(M)$ denotes the dual space of $(H^s(M), \Vert \cdot \Vert_{H^s})$ and $\Vert \cdot \Vert_{H^{-s}}$ is the norm canonically defined on it. Since $k_i \in \ker(P_g^s)$, $w_i \in \ker(P_g^s)^{\perp}$ satisfies 
\begin{equation} \label{eq:wi}
 P_g^s w_i = \lambda_i \beta_i v_i \quad \text{ in } M 
 \end{equation}
 for all $i \ge 0$, so that by \eqref{eq:inv:noyau} we have 
 \begin{equation} \label{eq:dual}
 \Vert w_i \Vert_{H^s} \le |\lambda_i| \Vert  \beta_i v_i \Vert_{H^{-s}}.
 \end{equation} Since $\beta_i \in L^{\frac{n}{2s}}(M)$, Sobolev's inequality  shows that, for all $i \ge 0$, $\lambda_i \beta_i v_i \in L^{\frac{2n}{n+2s}}(M) \subset H^{-s}(M)$ and that $  \Vert \beta_i v_i \Vert_{H^{-s}} \le C \Vert \beta_i v_i \Vert_{L^{\frac{2n}{n+2s}}}$. H\"older's inequality and the assumptions of the lemma then show that, for any $ i \ge 0$, 
 \begin{equation*}
   \Vert \beta_i v_i \Vert_{L^{\frac{2n}{n+2s}}}^{\frac{2n}{n+2s}}   \le \left( \int_M \beta_i^{\frac{n}{2s}}dv_g \right)^{\frac{2s}{n+2s}} \left( \int_M \beta_i v_i^2 dv_g \right)^{\frac{n}{n+2s}} \le C. 
   \end{equation*} 
Going back to   \eqref{eq:dual}, and since $(\lambda_i)_{i\ge 0}$ is bounded, proves that $(w_i)_{i\ge0}$ is bounded in $H^s(M)$. This proves the first bullet of the lemma.

We now claim that we have for some positive constant $c>0$
\begin{equation} \label{eq:wi:nonzero}
 \Vert w_i \Vert_{H^s} \geq c.
\end{equation}
We first integrate \eqref{eq:vecteursvi} against $k_i$. Since $k_i \in \ker(P_g^s)$ and $P_g^s$ is self-adjoint we have 
\begin{equation*}
 0 = \int_M (P_g^s k_i) v_i dv_g = \int_M (P_g^s v_i) k_i dv_g = \lambda_i \int_M \beta_i v_i k_i dv_g, 
 \end{equation*}
from which we deduce that 
\begin{equation} \label{eq:ortho:noyau}
\int_M \beta_i v_i k_i dv_g = 0
 \end{equation}
 for all $i \ge 0$ since $\lambda_i \neq 0$ by assumption. Hence, since $w_i = v_i - k_i$, we have 
\begin{equation} \label{eq:intwi}
  \int_M \beta_i w_i^2 dv_g = \int_M \beta_i v_i^2dv_g + \int_M \beta_i k_i^2dv_g \ge 1 
  \end{equation}
by assumption and since $\beta_i \ge 0$. By H\"older and Sobolev inequalities and since $p_i \ge \frac{n}{2s}$ and $\int_M \beta_i^{p_i}dv_g = 1$ this proves \eqref{eq:wi:nonzero}. A straightforward consequence of \eqref{eq:dual} is also that, up reducing $c$, $ |\lambda_i| \geq c$. This proves the second bullet of the lemma. 
 
For the last part of the Lemma we assume, $\Vert v_i \Vert_{H^s} \to +\infty$ as $i\to +\infty$. We let, for $i \ge 0$, $V_i = \frac{v_i}{\Vert v_i\Vert_{H^s}}, K_i = \frac{k_i}{\Vert v_i\Vert_{H^s}}$ and $W_i = \frac{w_i}{\Vert v_i\Vert_{H^s}}$.  Since $(w_i)_{i \ge 0}$ is bounded in $H^s(M)$, $W_i \to 0$ strongly in $H^s(M)$. Hence $\Vert K_i \Vert_{H^s} =1 + o(1)$ as $i \to + \infty$. The sequence $(K_i)_{i \ge0}$ is bounded in $H^s(M)$ and belongs to $\ker(P_g^s)$ which is finite-dimensional so, up to the extraction of a subsequence, it strongly converges towards $K \in \ker(P_g^s)$ satisfying $\Vert K \Vert_{H^s} = 1$. 
 The previous arguments show that $V_i = K + o(1)$ in $H^s(M)$ as $i \to + \infty$. 
 Coming back to \eqref{eq:ortho:noyau} and dividing by $\Vert v_i \Vert_{H^s}^2$ then shows, since $(\beta_i)_{i \ge0}$ is bounded in $L^{\frac{2n}{n-2s}}(M)$ and by Sobolev's inequality, that 
$$\int_M \beta_i K^2dv_g = o(1) $$
as $i \to + \infty$. We let $\beta$ denote the weak limit of $(\beta_i)_{i\ge0}$ in $L^{\frac{n}{2s}}(M)$, up to taking a subsequence. The latter shows that $\int_M \beta K^2dv_g = 0$. Since $K \in \ker(P_g^s)$ and $\Vert K \Vert_{H^s}= 1$ it is nonzero a.e. by \eqref{eq:unique:continuation}, and hence $\beta = 0$. We now prove that 
\begin{equation} \label{eq:wiconv}
w_i \rightharpoonup 0 \quad \text{  in } H^s(M)
\end{equation}
as $i \to + \infty$, up to taking a subsequence. Since $(w_i)_{i \ge0}$ is bounded in $H^s(M)$ we can let $w$ be its weak limit, up to taking a subsequence. Let $\vp \in C^{\infty}(M)$. Testing \eqref{eq:wi} against $\vp$ yields 
$$ \int_M P_g^s w \vp dv_g + o(1) = \lambda_i \int_M \beta_i v_i \vp dv_g. $$
Since $(\lambda_i)_{i \ge0}$ is bounded we have, by H\"older's inequality,
$$ \left|  \lambda_i \int_M \beta_i v_i \vp dv_g\right|^2 \le C\Vert \vp \Vert_{L^\infty}^2 \int_{M} \beta_i dv_g \int_M \beta_i v_i^2 dv_g = o(1)$$
as $i \to + \infty$, since $\beta_i \rightharpoonup 0$ in $L^{\frac{n}{2s}}(M)$. Hence $P_g^s w = 0$ and thus, since $w \in \ker(P_g^s)^{\perp}$, $w=0$, which proves \eqref{eq:wiconv}. By uniqueness of the subsequential  weak limits of $(w_i)_{i \ge 0}$ and $(\beta_i)_{i\geq 0}$, we obtain that $(w_i)_{i\geq 0}$ weakly converges to $0$ in $H^s(M)$ and $(\beta_i)_{i\geq 0}$ weakly converges to $0$ in $L^q(M)$ for all $q\leq \frac{n}{2s}$.

We now prove that $p_i \to \frac{n}{2s}$. Assume by contradiction that, up to taking a subsequence, $ \lim_{i \to + \infty} p_i > \frac{n}{2s}$ and let $p \in (\frac{n}{2s},  \lim_{i \to + \infty} p_i)$. Then $\frac{2p}{p-1} < \frac{2n}{n-2s}$ and by \eqref{eq:wiconv} and Sobolev's inequality, $w_i \to 0$ in $L^{\frac{2p}{p-1}}(M)$. Hence, 
$$  \int_M \beta_i w_i^2 dv_g \le \Vert w_i \Vert_{L^\frac{2p}{p-1}}^{2} \Vert \beta_i \Vert_{L^{p_i}} = o(1), $$
which contradicts \eqref{eq:intwi}. 

To conclude the proof of the Lemma it remains to show that if $(v_i)_{i \ge0}$ is not bounded in $H^s(M)$, then $\lambda_i \geq c$ for $i$ large enough (recall that $\vert \lambda_i \vert \geq c$ by (2)). We proceed by contradiction and assume that, up to the extraction of a subsequence, $\lambda_i  \le 0$ for all $i \ge 0$. Integrating \eqref{eq:vecteursvi} against $w_i$ and using \eqref{eq:def:Ag} and \eqref{eq:wiconv} we then obtain
$$ \int_M \big|\Delta_g^{\frac{s}{2}} w_i \big|^2 dv_g + o(1) =  \lambda_i \int_M \beta_i v_i w_i dv_g \le 0, $$
where we used \eqref{eq:ortho:noyau} and \eqref{eq:intwi} to write that 
$$ \int_M \beta_i v_i w_i dv_g = \int_M \beta_i w_i^2 dv_g \ge 1.$$
This gives in the end 
$$   \int_M \big|\Delta_g^{\frac{s}{2}} w_i \big|^2 dv_g  = o(1), $$
which, by \eqref{eq:wiconv},  is a contradiction with \eqref{eq:wi:nonzero}. This concludes the proof of Lemma \ref{lem:mainlemma}. 
   \end{proof}

The second result that we prove deals with the possible blow-up behavior of $H^s(M)$-bounded sequences $(v_i)_{i \ge0}$ satisfying \eqref{eq:vecteursvi}. Before stating it we introduce some notations. We let $K_{n,s}$ be the optimal constant for Sobolev's inequality in $\R^n$, which is given by 
\begin{equation} \label{defKns}
\frac{1}{K_{n,s}^2} = \inf_{u \in C^\infty_c(\R^n) \backslash \{0\}} \frac{\int_{\R^n} \big| \Delta_{\xi}^{\frac{s}{2}}u \big|^2 \, dv_\xi}{\left( \int_{\R^n} |u|^{\frac{2n}{n-2s}} \, dv_\xi \right)^{\frac{n-2s}{n}}},
\end{equation}
where $\Delta_\xi = - \sum_{i=1}^n \partial_i^2$ is the nonnegative Euclidean laplacian.  Let $(\be_i)_{i \ge0}$ be a sequence of nonnegative functions that is uniformly bounded in $L^{\frac{n}{2s}}(M)$. We define the {\em set of concentration points}  of $(\be_i)_{i \ge 0}$ by 
\begin{equation} \label{conc_points}
A=  \Biggl\{ x \in M \,\Big|\, \forall \de>0, \;  \limsup_{i\to + \infty} \int_{B_g(x,\de)}
  \be_i^{\frac{n}{2s}}  \,dv_g > \frac12 K_{n,s}^{-\frac{n}{s}} \Biggr\},
  \end{equation}
where $K_{n,s}$ is given by \eqref{defKns} and where $B_g(x,\de)$ denotes the Riemannian ball for $g$ of center $x$ and radius $\de$. Since $(\be_i)_{i \ge0}$ is bounded in $L^{\frac{n}{2s}}(M)$, $A$ is finite and may be empty. 

\begin{lemme} \label{convergence_appendix}
Let $s \in \mathbb{N}^*$, $2s <n$, $(M,g)$ be a closed manifold of dimension $n \ge 3$ and $P_g^s$ be the GJMS operator of order $2s$ on $(M,g)$. Let $(\be_i)_{i \ge0}$ be a sequence of nonnegative functions that is bounded in $L^{\frac{n}{2s}}(M)$ and let $A$ be given by \eqref{conc_points}. Let $(v_i)_{i \ge0}$ be a sequence of functions in $H^s(M)$ that satisfy, for all $i \ge 0$, 
\begin{equation} \label{eqvk}
P_g^s  v_i = \be_i  v_i  \quad \text{ in } M. 
\end{equation}
Assume that $(v_i)_{i \ge0}$ is bounded in $H^s(M)$, and up to the extraction of a subsequence, let $v$ be its weak limit. Then, $v_i \to v$ strongly in $H^s_{loc}(M \backslash A)$ as $i \to + \infty$. 
\end{lemme}

\begin{proof}
Up to passing to a subsequence we may assume that $v_i$ converges to $v$ weakly in $H^s(M)$ and strongly in $H^{s-1}(M)$ and that $\beta_i$ converges weakly to $\beta$ in $L^{\frac{n}{2s}}(M)$. By Sobolev embeddings $v_i \to v$ in $L^{\frac{n}{n-2s}}(M)$, so that passing \eqref{eqvk} to the limit we obtain that $v$ satisfies 
\begin{equation} \label{eqvk:1}
P_g^s  v = \be v  \quad \text{ in } M
\end{equation}
in a weak sense. Let $x \in M \backslash A$ and let $\delta >0$ be small enough so that $B_g(x, \de) \subset M \backslash A$ and 
\begin{equation} \label{eq:vk11}
\limsup_{i \to + \infty} \int_{B_g(x,\de)}  \be_i^{\frac{n}{2s}}  \,dv_g \le \frac12 K_{n,s}^{-\frac{n}{s}}. 
\end{equation}
Let $\eta \in C^\infty_c(B_g(x,\de))$ and let in what follows 
$$ w_i = \eta(v_i - v). $$
Straightforward computations show that for any $\ell \ge 1$ we have, pointwise, 
$$\big| \nabla^\ell w_i \big|_g = \eta \big| \nabla^\ell(v_i-v_0) \big|_g + O \big( \sum_{1 \leq r \leq \ell-1} \big| \nabla^r (v_i-v)\big|_g \big),$$
so that in particular $w_i \to 0$ strongly in $H^{s-1}(M)$ as $i \to + \infty$, and 
\begin{equation} \label{eqvk:2}
 P_g^s w_i = \eta P_g^s (v_i - v) + O \big( \sum_{1 \leq r \leq 2s-1} \big| \nabla^r (v_i-v)\big|_g \big) \quad \text{ in } M. 
 \end{equation}
 Using \eqref{eqvk} and \eqref{eqvk:1} we have 
 $$ P_g^s (v_i-v) = \beta_i(v_i-v) + (\beta_i - \beta) v, $$
 so that integrating \eqref{eqvk:2} against $w_i$ and using the latter and \eqref{eq:def:Ag} gives, since $w_i \to 0$ in $H^{s-1}(M)$, 
 \begin{equation} \label{eqvk:3}
  \int_M  \big|\Delta_g^{\frac{s}{2}} w_i \big|^2 dv_g + o(1) = \int_M \beta_i w_i^2 dv_g + \int_M \eta (\beta_i - \beta) v w_i dv_g. 
\end{equation}
Since $\beta \in L^{\frac{n}{2s}}(M)$, \eqref{eqvk:1} and standard regularity results (see for instance Mazumdar \cite[Theorem 5]{Mazumdar1}) show that $v \in L^q(M)$ for all $q \ge 1$. Independently, by compact Sobolev's embeddings, $w_i$ strongly converges to $0$ in $L^p(M)$ for all $p < \frac{2n}{n-2s}$. H\"older's inequality then shows that $v w_i$ strongly converges to $0$ in $L^{\frac{n}{n-2s}}(M)$ as $i \to + \infty$. Since $\beta_{i} - \beta$ is bounded in $L^{\frac{n}{2s}}(M)$ we thus have 
 \begin{equation} \label{eqvk:4}
 \int_M \eta (\beta_i - \beta) v w_i dv_g = o(1) 
\end{equation}
as $ i \to + \infty$. We now estimate the first integral in the right-hand side in \eqref{eqvk:3}. We first recall the following Sobolev's inequality for $P_g^s$ in $M$, proven in Mazumdar \cite[Theorem 2]{Mazumdar1}: for all $\ep >0$ there exists $B_\ep >0$ such that, for any $u \in H^s(M)$, 
\begin{equation} \label{eqvk:5}
\left(  \int_M |u|^{\frac{2n}{n-2s}} dv_g \right)^{\frac{n-2s}{n}} \le \big( K_{n,s}^2 + \ep \big) \int_M  \big|\Delta_g^{\frac{s}{2}} u \big|^2 dv_g + B_\ep \Vert u \Vert_{H^{s-1}}^2,
\end{equation}
where $K_{n,s}$ is given by \eqref{defKns}. Note that one can also choose $\ep = 0$ as very recently proven in \cite{Carletti} but we will not need this here. Let $\ep  = \frac12 K_{n,s}^2$ 
 and apply \eqref{eqvk:5} to $u = w_i$: this gives 
$$ \left(  \int_M |w_i|^{\frac{2n}{n-2s}} dv_g \right)^{\frac{n-2s}{n}} \le \frac32 K_{n,s}^{2} \int_M  \big|\Delta_g^{\frac{s}{2}} w_i \big|^2 dv_g + o(1) $$
as $i \to + \infty$. H\"older's inequality and \eqref{eq:vk11} then imply that 
\begin{equation} \label{eqvk:6}
\begin{aligned}
 \int_M \beta_i w_i^2 dv_g & \le  \left( \int_{B_g(x,\de)}  \be_i^{\frac{n}{2s}}  \,dv_g \right)^{\frac{2s}{n}} \Vert w_i \Vert_{\frac{2n}{n-2s}}^{2} \\
 & \le \frac34 \int_M  \big|\Delta_g^{\frac{s}{2}} w_i \big|^2 dv_g + o(1) 
\end{aligned} 
\end{equation}
as $i \to + \infty$. Combining \eqref{eqvk:6} with \eqref{eqvk:3} and \eqref{eqvk:4} then shows that 
$$ \int_M  \big|\Delta_g^{\frac{s}{2}} w_i \big|^2 dv_g \to 0$$
as $i \to + \infty$, and hence that $v_i \to v$ in $B_g(x,\delta)$. A simple covering argument then implies that $v_i \to v$ in $H^s_{loc}(M \backslash A)$. 
\end{proof}
\begin{rem}
A careful inspection of the proofs of Lemmas \ref{lem:mainlemma} and \ref{convergence_appendix} shows that the results remain true if $P_g^s$ is replaced by any differential operator of the form
$$ P = \Delta_g^s + A,$$
where $A$ is a self-adjoint operator of order less than $2s-1$. Conformal invariance is not needed here. 
\end{rem}

\section{Generalised eigenvalues of $P_g^s$} \label{generalised_eigen} \label{sec:defvp}

\subsection{Definition of generalised eigenvalues.} In this subsection we define generalised eigenvalues of $P_g^s$. As mentioned in the introduction, these generalised eigenvalues extend the definition of eigenvalues of $P_h^s$ when $h$ is only formally a metric defined by $h = \beta g$ for some $\beta \in \mathcal{A}_{\frac{n}{2s}}$. Throughout this section $(M,g)$ is again a closed Riemannian manifold of dimension $n \ge 3$ and $s \in \mathbb{N}^*$ satisfies $2s <n$. 
If $\be \in L^{\frac{n}{2s}}(M)$,  $v,w \in H^s(M)$, we define
\begin{equation} \label{eq:definitions}
\begin{aligned}
 G(v) & = \int_M  v P^s_g v dv_g   \\
 \Gamma(v,w) &= \int_M  v P^s_g w dv_g = \int_M  w P^s_g v dv_g \\
  Q(\beta,v) & = \int_M \beta v^2 dv_g  \\
B(\beta,v,w) &= \int_M \beta v w dv_g  \\
 \mathcal{R}(\beta,v) & =  \frac{G(v)}{Q(\beta,v)} \text{ when } Q(\beta,v) \neq 0.
\end{aligned} \end{equation}
Let $\beta \in \mathcal{A}_p$ for some $p \ge \frac{n}{2s}$, where $\mathcal{A}_{p}$ is as in \eqref{def:Ap}, and let $V$ be a finite-dimensional subspace of $L^2(M, \beta dv_g)$. We let 
$$ \dim_\beta V = \dim \text{Span} \{ \beta^{\frac{1}{2}}v ; v \in V \} $$
which, equivalently, is the dimension of the projection of $V$ in $L^2_\beta(M)$ (see definition after \eqref{def:Kbeta}). Note in particular that $\dim_{\beta} V \le \dim V$. Let $k \geq 0$ be an integer. We denote by $\mathcal{G}_k^\beta(H^s(M))$ the set of linear subspaces $V$ of $H^s(M)$ such that $\dim_\beta(V) = k$. We define the $k$-th generalised eigenvalue of $P_g^s$ as follows: 
\begin{equation} \label{eq:def:lambdak}
 \lambda_k(\beta)= \inf_{V \in \mathcal{G}_k^\beta(H^s(M))} \max_{v\in V\setminus \{0\}}  \mathcal{R}(\beta,v),
 \end{equation}
 where $\mathcal{R}(\beta,v)$ is as in \eqref{eq:definitions}. We briefly explain the heuristic behind definition \eqref{eq:def:lambdak}. Assume that $\hat{g} = u^{\frac{4}{n-2s}}g$ for some positive $u \in C^\infty(M)$. If $k \ge 1$, the classical Rayleigh characterisation of eigenvalues shows that 
 $$\begin{aligned} 
  \lambda_k(P_{\hat{g}}^s) & =  \inf_{V \in \mathcal{G}_k(H^s(M))} \max_{v\in V\setminus \{0\}} \frac{\int_M v P_{\hat{g}}^s v dv_{\hat{g}} }{\int_M v^2 dv_{\hat{g}}}  \\
  & =  \inf_{V \in \mathcal{G}_k(H^s(M))} \max_{v\in V\setminus \{0\}} \frac{\int_M v P_{g}^s v dv_{g} }{\int_M \beta v^2 dv_{g}}, \\
  \end{aligned} $$ 
 where we have used the conformal covariance of $P_g^s$ given by \eqref{eq:confinv} and where we have let $\beta = u ^{\frac{4s}{n-2s}}$. Definition \eqref{eq:def:lambdak} thus formally extends the definition of $\lambda_k(P_h^s)$ when $h$ is a singular metric of the form $h = \beta g$, and the assumption $\beta \in L^{\frac{n}{2s}}(M)$ ensures that $h$ has finite volume.  Equivalently, \eqref{eq:def:lambdak} can be seen as the $k$-th eigenvalue of $P_g^s$ over the weighted space $L^2_\beta(M)$. The first and second generalised eigenvalues were first defined in \cite{AmmannHumbert} when $s=1$ and $P_g^1$ is coercive, and were used in \cite{GurskyPerez} when $s=1$ and $\beta > 0 $ a.e. We provide here a general framework for every value of $s \ge1$ and every $\beta \in \mathcal{A}_p$. If $K_\beta$ is as in \eqref{def:Kbeta} we will let 
$$ K_{s, \beta} = H^s(M) \cap K_\beta.$$
We define its orthogonal space with respect to the bilinear form $\Gamma$ in \eqref{eq:definitions} as follows:
$$ H_{\beta} = K_{s,\beta}^{\perp_{\Gamma}} = \{ v \in H^s(M) ; \forall w \in K_{s,\beta}, \Gamma(v,w) = 0 \} .$$

\medskip 

As a first result, we prove that if $\beta \in \mathcal{A}_p$ for some $p\geq \frac{n}{2s}$, the assumption $\lambda_1(\beta) > -\infty$ is enough to guarantee that every  $k$-th generalised eigenvalue $\lambda_k(\beta)$, for $k \ge 1$, is attained, and to construct generalised eigenfunctions associated to $\lambda_k(\beta)$. 

\begin{prop} \label{prop:wellposedmu_1}
Let $s \in \mathbb{N}^*, 2s <n$ and let $P_g^s$ be the GJMS operator of order $2s$ in $(M,g)$. Let $\beta \in \mathcal{A}_p$ for some $p\geq \frac{n}{2s}$. We assume either that $\beta >0$ a.e or that $P_g^s$ satisfies \eqref{eq:unique:continuation}. Then, the following assumptions are equivalent:
\begin{itemize}
\item[(i)]$ \forall v \in H^s(M)\setminus\{0\}, Q(\beta,v) = 0 \Rightarrow G(v) > 0 $
\item[(ii)] There is $\Lambda \geq 0$ such that if we let, for all $v \in H^s(M)$, 
$$ N_\beta (v)^2= \Lambda Q(\beta,v) + G(v) ,$$
$N_\beta$ defines a norm on $H^s(M)$ which is equivalent to $\Vert \cdot \Vert_{H^s}$ .
\item[(iii)] $\lambda_1(\beta)>-\infty$.
\end{itemize}
If any of these three assumptions is satisfied 
we have
\begin{equation} \label{eq:directsumKbetaHbeta} K_{s,\beta} \oplus H_\beta = H^s(M), \end{equation}
where this splitting is orthogonal for the scalar product associated to $N_{\beta}$, and we have the following min-max characterization of $\lambda_k(\beta)$
\begin{equation} \label{eq:newminmu1} \lambda_k(\beta) = \inf_{V \in \mathcal{G}_k(H_\beta)} \max_{v \in V \setminus \{0\}} \mathcal{R}(\beta,v). \end{equation}
In addition, for any $k \ge 1$, there exists $\vp \in H_\beta$ that satisfies
\begin{equation}\label{eq:eulerlagrangefirsteigen} P_g^s \vp = \lambda_k(\beta) \beta \vp \quad \text{ in } M, \end{equation}
and the vector space $E_k(\beta)$ of solutions of \eqref{eq:eulerlagrangefirsteigen} is finite-dimensional. Finally if, for any $k \ge 1$, $(\vp_{k,\ell}(\beta))_{1 \le \ell \le \dim E_k(\beta)}$ is an orthonormal family of $E_k(\beta)$ with respect to the scalar product associated to $N_{\beta}$ then $\big( \vp_{k,\ell}(\beta)\big)_{k \ge 1, 1 \le \ell \le \dim E_k(\beta)}$ is a Hilbert basis of $H_\beta$ for the same scalar product.
\end{prop}

\begin{rem}
Notice that if $\beta > 0$ a.e, assumption (i) is obviously satisfied since condition $Q(\beta,v) = 0$ implies $v=0$.
\end{rem}

Functions satisfying \eqref{eq:eulerlagrangefirsteigen} will be called generalised eigenfunctions associated to $\beta$. In the rest of this paper we will often work with families $(\vp_i(\beta))_{i \ge 0}$ of generalised eigenfunctions normalised as follows:
$$ B(\beta, \vp_i, \vp_j) = \int_M \beta \vp_i \vp_j dv_g = \delta_{ij}. $$
Remark that in the particular case $k=1$ definition \eqref{eq:def:lambdak} rewrites as follows: 
$$ \lambda_1(\beta)= \inf_{v \in H^s(M), \beta^{\frac12} v \neq 0}   \mathcal{R}(\beta,v).$$

\begin{proof}
It is clear that $(ii)\Rightarrow (i)$. 

Let's prove $(i)\Rightarrow (iii)$: we assume that $(iii)$ does not hold, so that $\lambda_1(\beta)=-\infty$. We can thus let $(v_i)_{i\ge0} $ be a sequence in $H^s(M)$ such that $Q(\beta,v_i) = 1$ and $G(v_i) \to -\infty$ as $i\to +\infty$. By Proposition \ref{prop:eqnormscompactness}    we have 
$$ \begin{aligned}
\left\vert \int_M v_iA_g^s(v_i)dv_g \right\vert   
&\le   C   \left(Q(\beta,v_i) +\int_M \big|\Delta_g^{\frac{s}{2}} v_i \big|^2 dv_g\right) \\
&= C \left(1 + \int_M \big|\Delta_g^{\frac{s}{2}} v_i \big|^2 dv_g\right) 
\end{aligned} $$
for some $C>0$ that is independent of $i$, where $A_g^s$ is given by \eqref{eq:def:Ag}. Since by \eqref{eq:def:Ag} we have 
$$ \int_M v_iA_g^s(v_i)dv_g = G(v_i) - \int_M \big|\Delta_g^{\frac{s}{2}} v_i \big|^2 dv_g, $$
the assumption $G(v_i) \to - \infty$ thus implies that we have $\int_M \big|\Delta_g^{\frac{s}{2}} v_i \big|^2 dv_g \to + \infty$ as $i \to + \infty$ and since $Q(\beta, v_i)=1$, that $ \Vert v_i \Vert_{H^s_{\beta}}^2 = (1+o(1)) \int_M \big|\Delta_g^{\frac{s}{2}} v_i \big|^2 dv_g$. We can thus let, for $i \ge 0$, $N_i= \big(\int_M \big|\Delta_g^{\frac{s}{2}} v_i \big|^2 dv_g\big)^{\frac12}$ and  $\tilde{v}_i = \frac{v_i}{N_i}$. Then $(\tilde{v}_i)_{i \ge0}$ is bounded in $H^s(M)$ and by Proposition \ref{prop:eqnormscompactness}, up to a subsequence, there is a function $\tilde{v} \in H^s(M)$ such that $\tilde{v}_i$ converges to $\tilde{v}$ weakly in $H^s(M)$ and strongly both in $H^{s-1}(M)$ and in $L^2_\beta(M)$. Passing to the limit the normalisation assumptions gives $ Q(\beta,\tilde{v}) = 0 $
so that $\tilde{v}\in K_\beta$ and
$$ G(\tilde{v}) \leq \liminf_{i\to+\infty} G(\tilde{v}_i) \leq 0. $$
Independently, for any $i \ge 0$ we have, by \eqref{eq:def:Ag},
$$  \int_M \tilde{v}_i A_g^s(\tilde{v}_i)dv_g = G(\tilde{v}_i) - 1 \leq -1  $$
so that $ \int_M \tilde{v} A_g^s(\tilde{v})dv_g \leq -1 $ and $\tilde{v}\neq 0$. This implies that $(i)$ does not hold.

\medskip

We now prove that $(iii)\Rightarrow (ii)$. We assume $(iii)$.
Then we immediately have
$$\forall v \in H^s(M) \setminus K_{s,\beta}, -\lambda_1(\beta)Q(\beta,v) + G(v) \geq 0. $$
Since $H^s(M) \setminus K_{s,\beta}$ is dense in $H^s(M)$ we obtain
\begin{equation} \label{eq:inequalitymu_1} \forall v \in H^s(M), -\lambda_1(\beta) Q(\beta,v) + G(v) \geq 0.  \end{equation}
Let $v \in H^s(M)$ that attains equality in \eqref{eq:inequalitymu_1}. Let's first prove that $v \in K_{s,\beta} \Rightarrow v=0$. This is clear if $\beta>_{a.e}0$. Assume now that $P_g^s$ satisfies \eqref{eq:unique:continuation} and that $v \in K_{s,\beta}$, so that $G(v) = 0$ by \eqref{eq:inequalitymu_1}. Then we have for any $w \in H^s(M) \setminus K_{s,\beta}$, and $t \in \R\setminus \{0\}$ 
$$ \lambda_1(\beta) \leq \frac{G(v+tw)}{Q(\beta,v+tw)}. $$
Since $\beta^{\frac12}v = 0$ a.e and $G(v) = 0$ the latter becomes 
$$ t^2 \lambda_1(\beta) Q(\beta,w) \leq  2t \Gamma(v,w) + t^2 G(w). $$
Choosing first $t>0$, dividing by $t$ and letting $t\to 0$ gives $\Gamma(v,w) \ge 0$. Choosing then $t<0$, dividing by $t$ and letting $t\to 0$ gives $\Gamma(v,w) \le0$.
Thus
$$  \forall w \in H^s(M) \setminus K_{s,\beta}, \quad \Gamma(v,w) = 0 .$$
Again by density of $H^s(M) \setminus K_{s,\beta}$ this equality passes to the limit and gives 
$$ P_g^s v = 0 \quad \text{ in } M. $$
If $v \neq 0$, by \eqref{eq:unique:continuation} it only vanishes on a set of measure zero. Since we assumed $Q(\beta,v) = 0$ and since $\beta$ is nonnegative, $\beta$ is supported in $v^{-1}(\{0\})$, which is impossible since $\beta \in L^{\frac{n}{2s}} \backslash \{0\}$. This shows that $v=0$. We have thus shown that if equality holds in  \eqref{eq:inequalitymu_1} for some $v \neq 0$, then $Q(\beta,v)\neq 0$, so that (ii) is satisfied for $\Lambda = -\lambda_1(\beta)+1$ and $N_\beta(v) = (- \lambda_1(\beta)+1) Q(\beta, v) + G(v)$. To conclude the proof of (ii) we need to prove the equivalence of the norms $\Vert \cdot \Vert_{H^s}$ and $N_\beta$. First, by proposition \ref{prop:eqnormscompactness}, there exists $C>0$ (depending on $\beta$ and $\Lambda$) such that
$$ N_\beta \leq C \Vert \cdot \Vert_{H^s} . $$
To prove the reverse inequality we assume by contradiction that
$$ \inf_{v\in H^s(M)\setminus \{0\}} \frac{N_\beta(v)}{\Vert v \Vert_{H^s}} = 0$$ 
and we let $(v_i)_{i \ge0}$ be a sequence of functions in $H^s(M)$ such that $\Vert v_i \Vert_{H^s} = 1$ and $ N(v_i) \to 0$
as $i\to +\infty$. Since $(v_i)_{i\ge0}$ is bounded in $H^s(M)$ there is a function $v\in H^s(M)$ such that $v_i$ converges to $v$ weakly in $H^s(M)$ and, by Proposition \ref{prop:eqnormscompactness}, strongly both in $H^{s-1}(M)$ and in $L^2_\beta(M)$, up to a subsequence. By lower semi-continuity of $N_\beta$, we have
$$ N_\beta(v) \leq \liminf_{i\to+\infty} N_\beta(v_i) = 0 ,$$
and hence $v = 0$ since $N_\beta$ is a norm. Then $v_i \to 0$ in $H^{s-1}(M)$ and thus 
$$ 1 = \Vert v_i \Vert_{H^s}^2 =  \int_M \big|\Delta_g^{\frac{s}{2}} v_i \big|^2 dv_g + o(1) $$
as $i\to +\infty$. Since $A_g^s$ is an operator of order at most $2s-1$ this leads to
$$ 0 = \lim_{i\to+\infty} \int_M v_i A_g^s v_i dv_g = \lim_{i\to+\infty} \left( N_\beta(v_i)^2 - \int_M \big|\Delta_g^{\frac{s}{2}} v_i \big|^2 dv_g - \Lambda Q(\beta,v_i) \right) \le -1 $$
which is a contradiction. 
Thus $N_\beta$ and $\Vert \cdot \Vert_{H^s}$ are equivalent and this concludes the proof of $(i)\Leftrightarrow(ii)\Leftrightarrow(iii)$.

\medskip

Assume now that any of the three properties (i), (ii) or (iii) are satisfied. First, we easily see that $N_\beta$ is a Hilbert norm on $H^s(M)$. By definition of $H_\beta$ we also immediately deduce \eqref{eq:directsumKbetaHbeta} from the following equality:
$$ H_\beta = K_{s,\beta}^{\perp_\Gamma} = K_{s,\beta}^{\perp_{N_\beta}}, $$
where with a slight abuse of notations we denoted by $K_{s,\beta}^{\perp_{N_\beta}}$ the orthogonal of $K_{s,\beta}$ for the bilinear form associated to $N_\beta$. We now prove that \eqref{eq:newminmu1} holds. We denote by $\tilde{\lambda}_k(\beta)$ the expression in the right-hand side of \eqref{eq:newminmu1}. We let $\pi_\beta : H^s(M) \to H_\beta$ be the orthonormal projection with respect to the scalar product associated to $N_\beta$. For any subspace $V$ of $H^s(M)$ we then have
$$ V\in \mathcal{G}_k^\beta(H^s(M)) \Leftrightarrow  \pi_\beta(V) \in \mathcal{G}_k(H_\beta). $$
Observe first that 
$$ \forall v \in H^s(M) , \mathcal{R}(\beta,v) = \mathcal{R}(\beta,\pi_\beta(v)) + G \big( v - \pi_\beta(v) \big) \ge  \mathcal{R}(\beta,\pi_\beta(v)) $$
where the last inequality follows from (i). Let $V \in \mathcal{G}_k^\beta(H^s(M))$. As a consequence of the latter inequality we have
$$ \max_{v \in V \backslash \{0\}}  \mathcal{R}(\beta,v) \ge \max_{v \in V \backslash \{0\}} \mathcal{R}(\beta,\pi_\beta(v)) = \max_{v \in \pi_{\beta}(V) \backslash \{0\}} \mathcal{R}(\beta,v) \ge \tilde{\lambda}_k(\beta) $$
by  \eqref{eq:newminmu1}. Taking the infimum over $V$ yields $\lambda_k(\beta) \ge \tilde{\lambda}_k(\beta)$. The reverse inequality simply follows from the definition of $\lambda_k(\beta)$ in \eqref{eq:def:lambdak} since $H_\beta \subseteq H^s(M)$ and any subspace $V$ of $H^s(M)$ which belongs to $\mathcal{G}_k(H_\beta)$ is obviously in $\mathcal{G}_k^\beta(H^s(M))$. This proves \eqref{eq:newminmu1}.

\medskip

We finally prove \eqref{eq:eulerlagrangefirsteigen}. Observe first that (ii) implies that the operator $P_g + \Lambda \beta: H^s(M) \to H^{-s}(M)$ is invertible, where we endow $H^{-s}(M)$ with the dual norm induced by $N_\beta$ (which is equivalent to the dual norm induced by $\Vert \cdot \Vert_{H^s}$). We define, for $v \in L^2_\beta(M)$, 
$$Lv : = \big( P_g + \Lambda \beta \big)^{-1}( \beta v).$$
Remark that $L$ is well-defined since for any $v \in K_{\beta}$, $\beta v \equiv 0$ by definition. Let $w \in H^s(M)$. 
Then, by definition of $N_\beta$, we have 
\begin{equation} \label{eq:tout:L}
 (Lv, w)_{N_\beta} = \int_M w (P_g^s + \Lambda \beta)Lv dv_g = \int_M \beta vw dv_g,
 \end{equation}
where with a slight abuse of notations we have denoted by $(\cdot, \cdot)_{N_\beta}$ the scalar product in $H^s(M)$ associated to $N_\beta$. By choosing $w \in K_{s,\beta}$, which satisfies $\beta w \equiv 0$, we obtain with \eqref{eq:tout:L} that $\text{Im} L \subseteq H_\beta$. Cauchy-Schwarz inequality together with \eqref{eq:tout:L}, Proposition~\ref{prop:eqnormscompactness} and point (iii) also shows that $\Vert Lv \Vert_{N_\beta} \le C \Vert v \Vert_{L^2_\beta}$ for some $C>0$, so that $L: (L^2_\beta(M), \Vert \cdot \Vert_{L^2_\beta}) \to (H_\beta, N_\beta)$ is continuous. Clearly $L$ is self-adjoint because of \eqref{eq:tout:L}. Since $(H_\beta, N_\beta)$ embeds continuously and compactly  in $(L^2_\beta(M),  \Vert \cdot \Vert_{L^2_\beta})$ by Proposition~\ref{prop:eqnormscompactness}, (2), we can think of $L$ as a compact self-adjoint linear operator from $(L^2_\beta(M), \Vert \cdot \Vert_{L^2_\beta})$ to itself. Equation \eqref{eq:eulerlagrangefirsteigen}  then follows from the spectral theorem applied to $L$ and from the min-max characterisation of the eigenvalues of $L$ in $L^2_\beta(M)$. 
\end{proof}

\begin{rem}
Notice that in the case where $\ker(P_g^s) \neq \{0\}$, assumption \eqref{eq:unique:continuation} is crucial in order to have Proposition \ref{prop:wellposedmu_1}.
\end{rem}

\subsection{Continuity of generalised eigenvalues and existence of eigenfunctions}
In this subsection we generalise proposition \ref{prop:wellposedmu_1} to the case where some generalised eigenvalues are allowed to be equal to $-\infty$. If $\beta \in \mathcal{A}_p$ for some $p \ge \frac{n}{2s}$ is such that  $\lambda_{m-1}(\beta) = -\infty$ and $\lambda_m(\beta) >- \infty$ for some $m \geq 2$ we similarly prove the existence of eigenfunctions associated to $\lambda_k(\beta)$ for $k\geq m$. We prove in addition that the only possible generalised eigenfunctions of $P_g^s$ are associated to well-defined generalised eigenvalues $\lambda_k(\beta)$ for $k \ge m$. We first prove the following upper semi-continuity result:

\begin{prop} \label{prop:uppersemicontinuity}
Let $p \ge \frac{n}{2s}$ and let  $(\beta_i)_{i\ge0}$ be a sequence of functions in $\mathcal{A}_p$ that converges to some $\beta \in \mathcal{A}_p$ in $L^p(M)$. Then
$$ \limsup_{i\to +\infty} \lambda_k(\beta_i) \leq \lambda_k(\beta) $$
\end{prop}

\begin{proof} We first assume that $\lambda_k(\beta) = -\infty$. Let $\epsilon>0$. By \eqref{eq:def:lambdak} there is a finite-dimensional subspace $V$ of $H^s(M)$, satisfying $\dim_{\beta} V = k$, such that 
$$ 
\sup_{v\in V\setminus \{0\}} \mathcal{R}(\beta,v) \leq -\frac{1}{\epsilon} ,$$
where $\mathcal{R}(\beta,v)$ is as in \eqref{eq:definitions}. We claim that 
\begin{equation} \label{eq:upper1}
 \sup_{v \in V \backslash \{0\}, Q(\beta,v) = 1 } \left\vert Q(\beta_i,v) - 1 \right\vert = o(1) 
 \end{equation}
as $i \to + \infty$. Let indeed $(v_j)_{1 \le j \le m} \subset H^s(M)$, for some $m \ge k$, be a generating family of $V$ such that 
$$B(\beta, v_j, v_\ell) = \delta_{j\ell} \quad \text{ for all }  1 \le j, \ell \le m.$$
First, since $\beta_i \to \beta$ in $L^p(M)$ and $p\ge \frac{n}{2s}$ we have by Sobolev's inequality 
$$B(\beta_i, v_j, v_\ell) = \delta_{j\ell} + o(1) $$ 
as $i \to + \infty$ for all fixed $1 \le j,\ell \le m$, so that $\dim_{\beta_i} V= k$ for $i$ large enough. If $v \in V \backslash \{0\}$ satisfies $Q(\beta,v) = 1$ we can write it as $v = \sum_{j=1}^m \alpha_j v_j$ with $\sum_{j=1}^m \alpha_j^2 = 1$, and \eqref{eq:upper1} follows. With \eqref{eq:upper1} we now have, for $i$ large enough,
$$\sup_{v\in V\setminus \{0\}} \mathcal{R}(\beta_i,v)  = \sup_{v\in V\setminus \{0\}, Q(\beta,v) = 1} \mathcal{R}(\beta_i,v) \le (1+o(1))\sup_{v\in V\setminus \{0\}} \mathcal{R}(\beta,v).$$
For $i$ large enough we thus have 
$$ \lambda_k(\beta_i) \leq \sup_{v\in V\setminus \{0\}}  \mathcal{R}(\beta_i,v) \leq -\frac{2}{\epsilon}. $$ 
Letting $i\to +\infty$ and then $\epsilon \to 0$ shows that $\limsup_{i \to + \infty} \lambda_k(\beta_i) = - \infty$. 

\medskip

We now assume that $\lambda_k(\beta) > -\infty$. Let $\epsilon>0$. By \eqref{eq:def:lambdak} there is a finite-dimensional subspace $V$ of $H^s(M)$, satisfying $\dim_{\beta} V = k$, such that
$$\sup_{v\in V\setminus \{0\}} \mathcal{R}(\beta,v) \leq \lambda_k(\beta) +\epsilon $$
With \eqref{eq:upper1} we again obtain that, for $i$ large enough, 
$$\begin{aligned}
 \lambda_k(\beta_i) \leq \sup_{v\in V\setminus \{0\}}  \mathcal{R}(\beta_i,v)  \le (1+o(1))\sup_{v\in V\setminus \{0\}} \mathcal{R}(\beta,v) \le \lambda_k(\beta)+\epsilon + o(1). 
 \end{aligned} $$ 
 Letting $i\to +\infty$ and then $\epsilon \to 0$ completes the proof.
\end{proof}
A simple yet useful consequence of Proposition \ref{prop:uppersemicontinuity} is also that, if $\beta_i \to \beta$ in $L^p(M)$ for $p \ge \frac{n}{2s}$ and $\liminf_{i \to + \infty} \lambda_k(\beta_i) > - \infty$, then $\lambda_k(\beta) > - \infty$. 

\medskip

\begin{prop} \label{prop:lipeigen}
Let $p\geq \frac{n}{2s}$. For any $\beta \in \mathcal{A}_p$ such that $\beta>_{a.e}0$, there is an open neighborhood $W_\beta$ of $\beta$ in $\mathcal{A}_p$ such that $\lambda_k$ is Lipschitz in $W_\beta$.
\end{prop}

\begin{proof}
Let $\beta \in \mathcal{A}_p$ with $\beta>_{a.e}0$ and let $\beta_1, \beta_2 \in  \mathcal{A}_p$ be such that $\Vert \beta_i -\beta \Vert_{L^p} \leq \delta$ for $i=1,2$ and $\delta>0$ to be chosen. By proposition \ref{prop:uppersemicontinuity}, up to reducing $\delta$, we have for $i=1,2$ that 
\begin{equation}\label{eq:boundlambdakbetai} \lambda_k(\beta_i)\leq \lambda_k(\beta)+1 \end{equation}
Let $v,w\in H^s(M)$. We have, since $p \ge \frac{n}{2s}$,
\begin{equation} \label{eq:beta1minusbeta2}\begin{split}\left\vert B(\beta_1,v,w) - B(\beta_2,v,w) \right\vert & = \int_M \left( \beta_1-\beta_2 \right)v w dv_g \leq \Vert \beta_1-\beta_2  \Vert_{L^p} \Vert v \Vert_{L^{\frac{2p}{p-1}}} \Vert w \Vert_{L^{\frac{2p}{p-1}}} \\ 
& \leq C \Vert \beta_1-\beta_2  \Vert_{L^p} \Vert v \Vert_{H^s} \Vert w \Vert_{H^s} \end{split}\end{equation}
by a H\"older inequality and classical Sobolev inequalities with $\frac{2p}{p-1} \leq \frac{2n}{n-2s}$. We also have by Proposition \ref{prop:eqnormscompactness}, (1) that
\begin{equation}\label{eq:estnormhsbeta1} \Vert v \Vert_{H^s}^2 \leq C_\beta^2 \Vert v \Vert_{H^s_{\beta_1}}^2 = C_\beta^2 \left( \int_M v P_g^s v - \int_M v A_g^s v + \int_M \beta_i v^2 \right)  \end{equation}
for $i=1,2$, where $A_g^s$ is as in \eqref{eq:def:Ag}, and by Proposition \ref{prop:eqnormscompactness}, (3) that 
$$ \left\vert \int_M v A_g^s v \right\vert \leq \frac{1}{2 C_\beta^2} \int_M \vert \Delta_g^{\frac{s}{2}} v \vert^2 dv_g + D_\beta \int_M \beta_i v^2 dv_g $$
for some constant $D_\beta >0$ 
which is independent of $\delta$, and provided $\delta$ is chosen small enough to have $B_{L^p}(\beta,\delta) \subseteq V_\beta$, where $V_\beta$ is given by Proposition \ref{prop:eqnormscompactness}, (3). Then \eqref{eq:estnormhsbeta1} becomes
$$ \int_M \vert \Delta_g^{\frac{s}{2}} v \vert^2 dv_g + \int_M v^2 dv_g \leq C_\beta^2 \int_M v P_g^s v dv_g + \frac{1}{2} \int_M \vert \Delta_g^{\frac{s}{2}} v \vert^2 dv_g + C_\beta^2 D_\beta \int_M \beta_i v^2dv_g,  $$
so that there exist constants $A_\beta>0 $ and $B_\beta>0$ such that
\begin{equation} \label{eq:estHsbetai} \Vert v \Vert_{H^s}^2 \leq A_\beta \int_M v P_g^s v dv_g + B_\beta \int_M \beta_i v^2. \end{equation}
Now, let $0<\ep \leq 1$ and $V_1 \in \mathcal{G}_k^{\beta_1}(H^s)$ be such that
$$ \max_{v\in V_1 \setminus\{0\}} \mathcal{R}(\beta_1,v) \leq \lambda_k(\beta_1) + \ep.  $$
This implies that $\int_M v P_g^s v dv_g \le (\lambda_k(\beta) + \ep) Q(\beta_1, v)$ for all $v \in V_1$. Combining the latter with \eqref{eq:boundlambdakbetai}, \eqref{eq:beta1minusbeta2} and \eqref{eq:estHsbetai}, we obtain that for $v,w \in V_1$
\begin{equation} \label{eq:keyestbeta1minusbeta2} \left\vert B(\beta_1,v,w) - B(\beta_2,v,w) \right\vert
 \leq K_\beta \Vert \beta_1 - \beta_2 \Vert_{L^p} Q(\beta_1,v)^{\frac{1}{2}}Q(\beta_1,w)^{\frac{1}{2}} \end{equation}
 for some constant $K_\beta >0$ which does not depend on $\delta$.
 Choosing $\delta = \delta_0 (2K_\beta)^{-1}$ we thus have, for all $v,w \in V_1$,
\begin{equation*}
\left\vert B(\beta_1,v,w) - B(\beta_2,v,w) \right\vert \leq \delta_0 Q(\beta_1,v)^{\frac{1}{2}}Q(\beta_1,w)^{\frac{1}{2}} \end{equation*}
and it is clear that $V_1 \in \mathcal{G}_k^{\beta_2}(H^s)$ for $\delta_0$ small enough. We now test $V_1$ in the variational characterization of $\lambda_k(\beta_2)$ and we use \eqref{eq:keyestbeta1minusbeta2} to get
\begin{equation*}
\begin{split} \lambda_k(\beta_2) & \leq \max_{v\in V_1 \setminus\{0\}} \mathcal{R}(\beta_2,v) 
\\ & \leq \big(\lambda_k(\beta_1) + \ep \big) \left(1 + 2 K_\beta \Vert \beta_1 - \beta_2 \Vert_{L^p}\right) \end{split} \end{equation*}
provided $\delta_0$ is small enough.  Letting $\ep \to 0$ gives 
$$ \lambda_k(\beta_2) - \lambda_k(\beta_1) \le  2 K_\beta \lambda_k(\beta_1)\Vert \beta_1 - \beta_2 \Vert_{L^p}.$$ 
Repeating the same argument and exchanging $\lambda_k(\beta_1)$ and $\lambda_k(\beta_2)$ completes the proof of the proposition.
\end{proof}

\medskip

Let $\beta \in \mathcal{A}_p$ for some $ p \ge \frac{n}{2s}$. We let in what follows 
\begin{equation} \label{def:mbeta}
m(\beta) = \min \Big \{ k \in \N^*; \lambda_k(\beta) > -\infty  \Big\}.
\end{equation}
We recall the definition of the integers $k_+$ and $k_-$ defined in the introduction: 
$$k_- = \max \big\{k \geq 1,  \la_k(g)<0 \big\} \; \hbox{ and } k_+= \min  \big \{k \geq 1, \la_k(g)>0 \big\}.$$
If $\beta \in \mathcal{A}_{\frac{n}{2s}}$, using definition \eqref{eq:def:lambdak} it is easily seen that 
\begin{equation} \label{eq:enplus:kplus}
\lambda_k(\beta)  \ge 0 \quad   \text{ if } k \ge k_+ \text{ and } \beta >0 \text{ a.e.},  \\
\end{equation}
\begin{equation} \label{eq:enplus:kmoins}
\lambda_k(\beta)  \le 0 \quad  \text{ if } k \le k_+-1. 
\end{equation}
As a consequence, $m(\beta) \le k_+$.  In the next result we show that generalised eigenfunctions exist for $k \ge m(\beta)$ and give general conditions ensuring the continuity of generalised eigenvalues. 

\begin{prop} \label{prop:deficontinuityeigen} 
Let $s \in \mathbb{N}^*, 2s <n$ and let $P_g^s$ be the GJMS operator of order $2s$ in $(M,g)$. We assume that $P_g^s$ satisfies \eqref{eq:unique:continuation}. Let $p \ge \frac{n}{2s}$ and $\beta \in \mathcal{A}_p$. We have the following properties
\begin{enumerate}
\item Let $k \in \N^*$, and $b_i \in \mathcal{A}_p$ such that $b_i \to 0 $ as $i\to +\infty$ in $L^p(M)$. Then $\lambda_k(\beta+b_i)\to \lambda_k(\beta)$ in $\mathbb{R} \sqcup \{-\infty\}$ as $i\to +\infty$.
\item There is a sequence of eigenfunctions $\{\vp_k \}_{k\geq m(\beta)}$ in $H^s(M)$ such that for all $k, \ell \ge m(\beta)$ we have $Q(\beta,\vp_k,\vp_\ell) = \delta_{k,\ell}$ and
$$ P_g^s \vp_k = \lambda_k(\beta)\beta \vp_k \quad \text{ in } M. $$
\item We have $\lambda_k(\beta) \to +\infty$ as $k\to +\infty$. For any fixed $k \geq m(\beta)$, the space $E_k(\beta)$ of $k$-th eigenfunctions is equal to
$$ E_k(\beta) = \text{Span} \Big \{  \vp_i ; \lambda_i(\beta) = \lambda_k(\beta) \Big \} $$
and is a finite-dimensional subspace of $H^s(M)$. Furthermore, if there exists $\vp \in H^s(M) \backslash \{0\}$ and $\lambda \in \R$ such that $P_g^s \vp = \lambda \beta \vp$ then $\lambda = \lambda_k(\beta)$ for some $k \ge m(\beta)$ and $\vp \in E_k(\beta)$. 
\item Assume that $m(\beta)\geq 2$. Then for all $m\in \N^*$ and $\epsilon >0$, there are functions $\vp_1^{\epsilon,m},\cdots, \varphi_{m(\beta)-1}^{\epsilon,m}$ that satisfy
 the following: for any $ 1 \leq i,j \leq m(\beta)-1$ and any  $m(\beta) \leq k \leq m$,
$$
 \mathcal{R}(\beta,\vp_i^{\epsilon,m})  \leq -\frac{1}{\epsilon} \text{ and } \left \{
 \begin{aligned}
B(\beta,\vp_i^{\epsilon,m},\vp_j^{\epsilon,m}) & = \delta_{i,j}, \\
\Gamma(\vp_i^{\epsilon,m},\vp_j^{\epsilon,m}) & = \mathcal{R}(\beta,\vp_i^{\epsilon,m}) \delta_{i,j} \\
  B(\beta,\vp_k, \vp_i^{\epsilon,m}) & = 0 \\
 \end{aligned}\right. . $$
\end{enumerate}
\end{prop}
As a simple remark, note that it easily follows from  \eqref{eq:def:lambdak} that the sequence $(\lambda_k(\beta))_{k \ge m(\beta)}$ is nondecreasing. 

\begin{proof}
We prove Proposition \ref{prop:deficontinuityeigen} in several steps. 

\medskip

\textbf{Step 1:} let $k \ge 1$. We first claim that $\lambda_k(\beta + \epsilon) \to \lambda_k(\beta)$ as $\epsilon \to 0$. As a consequence, we prove that (2) and (4) hold. 

\medskip

\textbf{Proof of Step 1:} 
Since $\beta+\ep >0$ everywhere, we have by Proposition \ref{prop:wellposedmu_1} that $\lambda_1(\beta + \ep ) > -\infty$ and that there exits generalised eigenfunctions $(\vp_i^\epsilon)_{i\geq 1}$ associated to $\lambda_i(\beta + \epsilon)$, normalised as 
$$B(\beta+\epsilon,\vp_i^\epsilon,\vp_j^\epsilon) = \delta_{i,j}$$ 
for $i,j \geq 1$.  For simplicity we will let, in the rest of this proof, $\lambda_i^{\epsilon} = \lambda_i(\beta+\epsilon)$ for any $i \ge 1$. Assume first that $k \le m(\beta)-1$. 
Then $\lambda_k(\beta) = -\infty$, and $\lambda_k(\beta + \ep) \to -\infty$ as a consequence of Proposition \ref{prop:uppersemicontinuity}. 
We now assume that $k\ge m(\beta)$. We will prove that there exists $\varphi_k \in H^s(M)$ with $Q(\beta, \varphi_k) = 1$ and $- \infty < \nu_k \le \lambda_k(\beta)$ such that the following holds: 
\begin{equation} \label{eq:eq:contieigen0}
P_g^s \vp_k = \nu_k \beta \vp_k \quad \text{ in } M. 
\end{equation}
Assume first that $k \ge k_+$. Then by \eqref{eq:enplus:kplus} we have $\lambda_k^\ep \geq 0$ for all $\ep >0$, and hence by Proposition \ref{prop:uppersemicontinuity} we obtain $\lambda_k(\beta)\geq 0$. Since $\beta + \epsilon \to \beta$ in $L^p(M)$ as $\epsilon \to 0$ for any $p \ge \frac{n}{2s}$, with $\beta\neq 0$, lemma \ref{lem:mainlemma} applies and shows that the family $(\vp_k^\epsilon)_{\epsilon >0}$ is bounded in $H^s(M)$. Up to the extraction of a subsequence, $\vp_k^\epsilon$ weakly converges  in $H^s(M)$ and strongly converges in $L^2_\beta(M)$ and in $H^{s-1}(M)$ to some function $\vp_k$ as $\ep \to 0$. Up to a subsequence, the previous analysis also shows that $\lambda_k^\epsilon \to \nu_k$ as $\epsilon \to 0$ for some real number $\nu_k$ satisfying $0 < \nu_k \leq \lambda_k(\beta)$. For any $\epsilon >0$, $\vp_k^{\epsilon}$ satisfies 
$$P_g^s \vp_k^{\epsilon} = \lambda_k^{\epsilon}(\beta + \ep) \vp_k^{\epsilon} \quad \text{ in } M, $$
and passing the latter to the weak limit as $\epsilon \to 0$ shows that \eqref{eq:eq:contieigen0} is satisfied when $k \ge k_+$. The relation $Q(\beta+\ve, \vp_k^{\ep} )=1$ also passes to the limit and shows that $Q(\beta, \vp_k) = 1$, hence $\vp_k \neq 0$. If $k, \ell \ge k_+$ the relation $B(\beta+\epsilon,\vp_k^\epsilon,\vp_\ell^\epsilon) = \delta_{k,\ell}$ similarly passes to the limit and shows that $ B(\beta,\vp_k,\vp_\ell) = \delta_{k,\ell}$.  In particular, for $\ep$ small enough,
\begin{equation} \label{eq:dimeigenfunctions}\dim_\beta\left(\text{Span} \left( \vp_i^\ep \right)_{k_+ \leq  i \leq k} \right) = \dim\left(\text{Span} \left( \vp_i^\ep \right)_{k_+ \leq  i \leq k} \right) = k-k_++1. \end{equation}
We now assume that $m(\beta) \leq k\leq k_+-1$. Here the difficulty is to prove that
\begin{equation} \label{eq:provelimsup} \limsup_{\ep\to 0} \vert \lambda_k^\ep \vert < + \infty. \end{equation}
We would like to test $V_\ep = \text{Span}\left(\vp_1^\ep ,\cdots \vp_k^\ep\right)$ in the variational characterization of $\lambda_k(\beta)$. However one could have $\dim_\beta \left( V_\ep\right) < k$. To avoid this problem we set a small perturbation for $t \in \R$:
$$ V_\ep^t =  \text{Span}\left(\vp_i^\ep + t \vp_{k_+ -1 +i}^\ep ; i \in \{1,\cdots,k\} \right). $$
We have that 
$$P(t) := \det\left( \left(\beta^{\frac{1}{2}} \left(\vp_i^\ep + t \vp_{k_+ -1 +i}^\ep \right) \right)_{1\leq i \leq k} \vert \left(\beta^{\frac{1}{2}}\vp_{k_+ -1 + i}^\ep\right)_{1\leq i \leq k} \right)$$
is a well defined polynomial of degree $k$ since $\left(\beta^{\frac{1}{2}}\vp_{k_+ -1 + i}\right)_{1\leq i \leq k}$ is a free family by \eqref{eq:dimeigenfunctions} and we have $\lim_{t\to +\infty} \frac{P(t)}{t^k} = 1$. Since $P$ has finitely many roots we can let $t _\ep \leq 1$ be such that $P(t_\ep)\neq 0$ and we can test $V_\ep^{t_\ep}$ in the definition  \eqref{eq:def:lambdak}  of $\lambda_k(\beta)$.  Let $(\alpha_i^{\ep})_{1 \le i \le k}$ be such that 
$$ v_{\ep}:=  \sum_{i=1}^k \alpha_i^\ep \left(\vp_i^\ep + t_\ep \vp_{k_+ -1 +i}^\ep\right)$$
attains $\max_{v \in V_\ep^{t_\ep} \backslash \{0\}} \frac{G(v)}{Q(\beta,v)} $. Then 
$$ \lambda_k(\beta) \leq \max_{v \in V_\ep^{t_\ep}} \frac{G(v)}{Q(\beta,v)} = \frac{\sum_{i=1}^k \left(\alpha_i^\ep\right)^2\left( \lambda_i^\ep + t_\ep^2 \lambda_{k_+ +i-1}^\ep \right) }{ 1 + t_\ep^2 - Q(\ep,v)} \leq  \frac{\lambda_k^\ep + t_\ep^2 \lambda_{k_+ +k-1}^\ep  }{ 1 + t_\ep^2 - Q(\ep,v)} , $$
where we used that for any $i,j\geq 1$ and $\ep >0$, $\Lambda(\vp_i^\ep,\vp_i^\ep) = \lambda_i^\ep \delta_{i,j}$. Then
$$\lambda_k^\ep \geq \lambda_k(\beta)\left( 1-Q(\ep,v^{\ep}) \right) - t_{\ep}^2 \lambda_{k_++k-1}^\ep \geq \min\left(\lambda_k(\beta),0\right) - \lambda_{k_++k-1}^\ep, $$
where to obtain the last inequality we used \eqref{eq:enplus:kmoins}. The right-hand side of the previous expression is uniformly bounded from below as $\ep \to 0$ since $\limsup_{\ep\to 0} \lambda_{k_++k-1}^\ep \leq \lambda_{k_++k-1}(\beta)$. This proves  \eqref{eq:provelimsup} for $m(\beta) \leq k\leq k_+-1$. With  \eqref{eq:provelimsup} we can proceed as in the case $k\geq k_+$: we apply lemma \ref{lem:mainlemma} to the sequence $(\vp_k^\ep)_{\ep >0}$ and obtain, up to the extraction of a subsequence, that $\vp_k^\ep \rightharpoonup \vp_k$ in $H^s(M)$  and $\lambda_k^\ep \to \nu_k$ for some $-\infty < \nu_k \le \lambda_k(\beta)$ as $\ep \to 0$. Passing the eigenvalue equation satisfied by $\vp_k^{\ep}$ to the weak limit as $\ep \to 0$ similarly proves \eqref{eq:eq:contieigen0} for $k \leq k_+-1$. 

\smallskip

We now prove (4). For $m \geq m(\beta)$ we let
$$ H_m = \Big \{ v \in H^s(M) ; \forall m(\beta)\leq i \leq m,  Q(\beta,\vp_i,v) = 0  \Big \} $$
where, for $i \ge m(\beta)$, $\vp_i$ is given by \eqref{eq:eq:contieigen0}. We claim that
\begin{equation} \label{eq:mapproximation-infty} \inf_{V \in \mathcal{G}^\beta_{m(\beta)-1}(H_m)} \max_{v\in V\setminus \{0\}} \mathcal{R}(\beta,v) = -\infty .\end{equation}
Indeed, by definition of $m(\beta)$ and by \eqref{eq:def:lambdak} we can let, for any $\epsilon >0$, $V \in \mathcal{G}_{m(\beta)-1}^\beta(H^s(M))$ be such that 
\begin{equation} \label{eq:contieigen1}
 \max_{v\in V\setminus\{0\}} \mathcal{R}(\beta,v) \leq -\frac{1}{\epsilon} 
 \end{equation}
Let
$$ V_m = \left\{ v - \pi_m(v) ; v\in V \right\} $$
where
$$\pi_m(v) = \sum_{i=m(\beta)}^m B(\beta,v,\vp_i) \frac{\vp_i}{Q(\beta,\vp_i)} $$
is the orthogonal projection of $V$ on $ \text{Span} \big \{ \vp_i, m(\beta)\leq i \leq m \big \} $ with respect to $Q(\beta,\cdot)$. We observe that, up to assuming that $\epsilon$ is small enough, we can assume that  $\dim_\beta(V_m) = \dim_\beta(V) = m(\beta)-1$. Indeed, if this were not the case, a linear combination of the family $(\vp_i)_{m(\beta)\leq i \leq m}$ would belong to $V$, which would contradict \eqref{eq:contieigen1} for $\ep$ small enough. Straightforward computations using the definition of $\vp_i$ show that, for any $v \in V$,
$$ B \big( \beta, v- \pi_m(v), \pi_m(v) \big) = \Gamma \big( \beta, v - \pi_m(v), \pi_m(v) \big) = 0$$
holds, so that 
$$ \mathcal{R}(\beta, v - \pi_m(v) ) = \frac{G(v) - G(\pi_m(v)) }{Q(\beta,v) - Q(\beta,\pi_m(v))} \leq \frac{-\frac{1}{\epsilon} Q(\beta,v) - \lambda_{m(\beta)} Q(\beta,\pi_m(v))}{Q(\beta,v)-Q(\beta,\pi_m(v))} \to -\infty$$
as $\epsilon \to 0$. This proves \eqref{eq:mapproximation-infty}. Let now $V_{\epsilon,m} \in \mathcal{G}^\beta_{m(\beta)-1}(H_m)$ be such that 
$$ \max_{v\in V_{\epsilon,m}\setminus \{0\}} \mathcal{R}(\beta,v) \leq - \frac{1}{\epsilon} $$
By an orthodiagonalization of $\Gamma$ with respect to $Q(\beta,\cdot)$ on $V_{\epsilon,m}$, we obtain a family $\vp_1^{\epsilon,m},\cdots,\vp_{m(\beta)-1}^{\epsilon,m}$ which satisfies property (4).

\smallskip 

We are now in position to conclude the proof of (1) when $b_i = \epsilon_i$ and  $\epsilon_i \to 0$ is a sequence of positive constants, and conclude the proof of (2).  Let $k \ge m(\beta)$ and let
$$V_{\epsilon}  = \text{Span} \Big \{ \vp_1^{\epsilon,k},\cdots,\vp_{m(\beta)-1}^{\epsilon,k},\vp_{m(\beta)},\cdots, \vp_{k} \Big \}. $$
The variational characterization \eqref{eq:def:lambdak} of $\lambda_k(\beta)$ together with \eqref{eq:eq:contieigen0} show that 
\begin{equation} \label{eq:contieigen2} \begin{aligned}
\lambda_k(\beta) & \leq \max_{\alpha \in \mathbb{S}^{k-1}} \frac{\sum_{i=1}^{m(\beta)-1}\alpha_i^2 \mathcal{R}(\beta,  \vp_i^{\epsilon,m})+  \sum_{i=m(\beta)}^k \alpha_i^2 \mathcal{R}(\beta,  \vp_i)}{\sum_{i=1}^{m(\beta)-1}\alpha_i^2 Q(\beta,  \vp_i^{\epsilon,m})+  \sum_{i=m(\beta)}^k \alpha_i^2 Q(\beta,  \vp_i)} \\
& \le  \max_{\alpha \in \mathbb{S}^{k-1}}  \sum_{i=m(\beta)}^m \alpha_i^2 \mathcal{R}(\beta,  \vp_i) \leq \nu_k.
\end{aligned} 
\end{equation}
Therefore, $\nu_k = \lambda_k(\beta)$ and (1) holds for $b_i = \ep_i$. Coming back to \eqref{eq:eq:contieigen0}, this also finishes the proof of point (2).

\medskip

\textbf{Step 2:} We prove property (3).

\medskip

\textbf{Proof of Step 2:} First, recall that by Proposition \ref{prop:wellposedmu_1} $E_k(\beta)$ is orthogonal to $K_{\beta}$ defined in \eqref{def:Kbeta} for the bilinear form $\Gamma(\beta, \cdot, \cdot)$. As a consequence, $B(\beta, \cdot, \cdot)$ defines a scalar product on $E_k(\beta)$. We first prove that the unit ball of $E_k(\beta)$ with respect to the norm $\sqrt{Q(\beta,\cdot)}$ is compact. Indeed, let $(v_i)_{i \ge0}$ be a sequence of generalised eigenfunctions associated to $\lambda_k(\beta)$. By Lemma \ref{lem:mainlemma}, since $\beta \neq 0$ is fixed, $(v_i)_{i \ge0}$ is bounded in $H^s(M)$. Up to the extraction of a subsequence, $v_i$ converges to some $v$ weakly in $H^s(M)$ and strongly in $L^2_{\beta}(M)$. Passing to the weak limit, it is easily seen that $v$ still satisfies $P_g^s v = \lambda_k(\beta) \beta v$ in $M$. For all $i \ge 0$ we thus have $P_g^s(v_i - v) = \lambda_k(\beta) \beta (v_i-v)$, so that integrating against $v_i-v$ and using \eqref{eq:def:Ag} shows that $ \int_M \big|\Delta_g^{\frac{s}{2}} (v_i-v) \big|^2 dv_g = o(1)$  as $i \to + \infty$. Thus the unit ball of  $E_k(\beta)$ for $\sqrt{Q(\beta,\cdot)}$ is compact and $E_k(\beta)$ is finite dimensional. 

Assume now by contradiction that the sequence $(\lambda_k(\beta))_{k\ge m(\beta)}$ is bounded. A similar argument to the one we just used, again relying on Lemma \ref{lem:mainlemma}, then shows that the unit ball of $E = \overline{\bigoplus_{k\geq m(\beta)} E_k(\beta)}$ with respect to the norm $\sqrt{Q(\beta,\cdot)}$ is compact. This shows that $E$ is finite-dimensional, which is impossible since the family $(\vp_k)_{k\ge m(\beta)}$ is free by point (2) of Proposition  \ref{prop:deficontinuityeigen}. Hence $\lim_{k \to +\infty} \lambda_k(\beta) = + \infty$. 

Assume now that there are $\vp\in H^s(M) \backslash \{0\}$ and $\lambda\in \R$ such that $P_g^s \vp = \lambda \beta \vp$ in $M$ and assume that $\lambda \not \in \{\lambda_k(\beta)\}_{k \ge m(\beta)}$. We let $k \in \N^*$ be such that $\lambda_{k-1}(\beta)<  \lambda < \lambda_k(\beta) $. We let 
$$ \tilde{H}_k = \Big \{ v \in H^s(M) ; Q(\beta, \vp, v) = 0 \text{ and}, \forall m(\beta)\leq i \leq k-1,  Q(\beta,\vp_i,v) = 0  \Big \} .$$
Arguing as in the proof of point (4) above it is easily seen that, for any $\epsilon >0$, there exists a family $\tilde{\vp}_1^{\epsilon}, \cdots, \tilde{\vp}_{m(\beta)-1}^{\epsilon}$ of functions in $H^s(M)$ satisfying, for any $ 1 \leq i,j \leq m(\beta)-1$ and any  $m(\beta) \leq \ell \leq k-1$,
$$
 \mathcal{R}(\beta,\tilde{\vp}_i^{\epsilon})  \leq -\frac{1}{\epsilon} \text{ and } \left \{
 \begin{aligned}
B(\beta,\tilde{\vp}_i^{\epsilon},\tilde{\vp}_j^{\epsilon}) & = \delta_{i,j}, \\
\Gamma(\tilde{\vp}_i^{\epsilon},\tilde{\vp}_j^{\epsilon}) & = \mathcal{R}(\beta,\tilde{\vp}_i^{\epsilon}) \delta_{i,j} \\
  B(\beta,\vp_\ell, \tilde{\vp}_i^{\epsilon}) & = 0 \\
  B(\beta, \vp,  \tilde{\vp}_i^{\epsilon})& = 0
 \end{aligned}\right. . $$
We let
$$V_{\epsilon}  = \text{Span} \Big \{ \tilde{\vp}_1^{\epsilon},\cdots,\tilde{\vp}_{m(\beta)-1}^{\epsilon},\vp_{m(\beta)},\cdots, \vp \Big \}. $$
Then $\dim_{\beta} V_\ep = k$ and mimicking the computations in \eqref{eq:contieigen2}, the variational characterisation \eqref{eq:def:lambdak} of $\lambda_k(\beta)$ then shows that $\lambda_k(\beta) \le \lambda$, which is a contradiction. Hence $\lambda = \lambda_m(\beta)$ for some $m \ge m(\beta)$, and this also proves that for all $k \ge m(\beta)$ we have 
$$ E_k(\beta) = \text{Span} \Big \{  \vp_i ; \lambda_i(\beta) = \lambda_k(\beta) \Big \}, $$
and concludes the proof of (3). 

\medskip

\textbf{Step 3:} we conclude the proof of (1). 

\medskip 

\textbf{Proof of Step 3:} we let $(b_i)_{i \ge 0}$ be a sequence in $\mathcal{A}_p$ such that $b_i \to 0$ in $L^p(M)$ as $ i \to + \infty$. We claim that 
\begin{equation} \label{semicontinuite:m}
m(\beta+ b_i) \le m(\beta)
\end{equation} 
for all $i$ large enough. Assume first that $\lambda_{m(\beta)}(\beta) \ge 0$. Then, by \eqref{eq:def:lambdak}, any $V \in \mathcal{G}_k^{\beta}(H^s(M))$ contains a nonzero element $v$ such that $G(v) \ge 0$. Then, since $b_i \ge 0$, 
$$\max_{v \in V \backslash \{0\}} \frac{G(v)}{Q(\beta + b_i, v)} \ge0$$ 
and hence $\lambda_{m(\beta)}(\beta+b_i) \ge 0$. Assume now that $\lambda_{m(\beta)} <0$. Proposition~\ref{prop:uppersemicontinuity} then shows that $\lambda_{m(\beta)}(\beta + b_i) < 0$ for $i$ large enough. In the definition \eqref{eq:def:lambdak} of $\lambda_k(\beta + b_i)$ the infimum may therefore only be taken over subsets $V \in \mathcal{G}_k^{\beta}(H^s(M))$ such that $\max_{v \in V \backslash \{0\}} \mathcal{R}(\beta, v)  <0$. For a such a $V$ we have, since $b_i \ge 0$, 
$$\max_{v \in V \backslash \{0\}} \frac{G(v)}{Q(\beta + b_i, v)} \ge \max_{v \in V \backslash \{0\}} \frac{G(v)}{Q(\beta,v)},$$
and taking the infimum over $V$ yields $\lambda_{m(\beta)}(\beta + b_i) \ge \lambda_{m(\beta)}(\beta)$. This proves \eqref{semicontinuite:m}. 
With \eqref{semicontinuite:m}, and by points (2) and (3) of Proposition~\ref{prop:deficontinuityeigen}, for any $i \ge 0$ there thus exists a sequence $(\varphi_k^i)_{k \ge m(\beta)}$ of generalised eigenfunctions associated to $\beta + b_i$ and satisfying $P_g^s \vp_k^i = \lambda_k(\beta + b_i) (\beta + b_i) \vp_k^i$ in $M$ for every $k \ge m(\beta)$ and $i \ge 0$. The arguments in the proof of \eqref{semicontinuite:m} shows that for every $k \ge m(\beta)$ the sequence $(\lambda_k(\beta + b_i))_{i \ge 0}$ is uniformly bounded and, up to a subsequence, converges to some $\nu_k \in \R$ with $\nu_k \le \lambda_k(\beta)$. We can now mimic the arguments in the proof of Step $1$ to show that $( \vp_k^i)_{i \ge0}$ converges, up to a subsequence, to $\vp_k \in H^s(M)$ satisfying $P_g^s \vp_k = \nu_k \beta \vp_k$ in $M$. That $\nu_k = \lambda_k(\beta)$ now follows from point (3) of Proposition~\ref{prop:deficontinuityeigen}.
\end{proof}

We end this section with the following proposition. For $p \ge \frac{n}{2s}$ and $\beta \in \mathcal{A}_p$ we let 
$$ \bar{\lambda}_k^p(\beta) = \lambda_k(\beta) \Vert \beta \Vert_{L^p}. $$

\begin{prop} \label{prop:preliminaryvarset}
Let $s \in \mathbb{N}^*, 2s<n$ and let $P_g^s$ be the GJMS operator of order $2s$ in $(M,g)$. We assume that $P_g^s$ satisfies \eqref{eq:unique:continuation}. Let $p\geq \frac{n}{2s}$. 

Let $k \leq k_-$. Then $\bar{\lambda}_k^p$ is upper semi-continuous and
\begin{equation} \label{eq:supnegativeAp} \sup_{\beta \in C^\infty(M); \beta>0} \bar{\lambda}_k^p(\beta) = \sup_{\beta \in \mathcal{A}_p} \bar{\lambda}_k^p(\beta) = \sup_{\beta \in \mathcal{A}_p , \Vert \beta\Vert_{L^p} \geq 1} \bar{\lambda}_k^p(\beta) < 0. \end{equation}
Let $k \geq k_+$. Then $\bar{\lambda}_k^p$ is lower semi-continuous and
\begin{equation} \label{eq:infpositiveAp} \inf_{\beta \in C^\infty(M); \beta>0} \bar{\lambda}_k^p(\beta) = \inf_{\beta \in \mathcal{A}_p} \bar{\lambda}_k^p(\beta) = \inf_{\beta \in \mathcal{A}_p , \Vert \beta\Vert_{L^p} \geq 1} \bar{\lambda}_k^p(\beta) > 0. \end{equation}
\end{prop}
This result will be used in the next section, in particular in the proof of Proposition \ref{prop:PS}, when we will investigate the variational theory for eigenvalue functionals. 

\begin{rem} \label{rem:continuite:vp:pos}
Notice that if $k\geq k_+$, $\lambda_k$ is continuous in $\mathcal{A}_p$ by Proposition \ref{prop:uppersemicontinuity} and Proposition \ref{prop:preliminaryvarset}.
\end{rem}

\begin{proof}
We first assume that $k \leq k_-$. Upper semi-continuity of $\bar{\lambda}_k^p$ is given by Proposition \ref{prop:uppersemicontinuity}. We now prove \eqref{eq:supnegativeAp}. We first have by Proposition \ref{prop:deficontinuityeigen} (1) that 
$$ \sup_{\beta \in \mathcal{A}_p} \bar{\lambda}_k^p(\beta) = \sup_{\beta \in \mathcal{A}_p, \beta >_{a.e} 0} \bar{\lambda}_k^p(\beta) .$$
We then use that  $\{ \beta \in C^\infty(M), \beta > 0 \} $ is dense in the set $\{ \beta \in \mathcal{A}_p, \beta >_{a.e} 0 \}$ with respect to the $L^p$ distance and Proposition \ref{prop:lipeigen} to obtain the first equality in \eqref{eq:supnegativeAp}. The second equality comes from the observation that $\bar{\lambda}_k^p(\nu \beta) = \bar{\lambda}_k^p(\beta)$ for any $\nu >0$. The strict inequality in \eqref{eq:supnegativeAp} is a consequence of lemma \ref{lem:mainlemma} (2) applied to a maximising sequence of eigenfunctions for $\bar{\lambda}_k^p$ (which are given, for instance, by Proposition \ref{prop:deficontinuityeigen}, (2)).

We now assume that $k \geq k_+$. 
The proof of \eqref{eq:infpositiveAp} follows the same lines than in the case $k \leq k_-$ and we just have to prove that $ \bar{\lambda}_k^p(\beta)$ is lower semi-continuous. 
Let $(\beta_i)_{i\geq 0}$ be a sequence in $\mathcal{A}_p$ such that $\beta_i \to \beta$ in $\mathcal{A}_p$ as $i \to + \infty$. For $k_+ \leq j, \ell \leq k$, by Proposition \ref{prop:deficontinuityeigen} (2), there exist eigenfunctions $\vp_i^j$ associated to $\lambda_j(\beta_i) \geq 0$ such that $B(\beta_i,\vp_i^j,\vp_i^\ell) = \delta_{j,\ell}$. By Proposition \ref{prop:uppersemicontinuity}, since $\beta_i$ strongly converges to $\beta$ in $L^p(M)$, $\left(\lambda_j(\beta_i)\right)_{i \geq 0}$ is bounded from above. Then Lemma \ref{lem:mainlemma} applies and since $\beta \neq 0$, shows that $(\vp_i^j)_{i\geq 0}$ is bounded in $H^s(M)$ for $k \le j \le k_+$. Up to the extraction of a subsequence, $(\vp_i^j)_{i\geq 0}$ weakly converges to $v_j$ in $H^s(M)$. Passing to the weak limit in the equation satisfied by $\vp_i^j$ we get that 
$$ P_g^s v_j = \nu_j \beta v_j \quad \text{ in } M,$$
where up to the extraction of a subsequence, $\nu_j = \lim_{i\to +\infty} \lambda_j(\beta_i) \leq \lambda_j(\beta)$. 
By Proposition \ref{prop:eqnormscompactness}, (2), we may pass the relations  $B(\beta_i,\vp_i^j,\vp_i^{\ell}) = \delta_{j,\ell}$ to the limit as $i \to + \infty$ and obtain $B(\beta,v_j,v_\ell) = \delta_{j,\ell}$. 
We now claim that there exists a family $(\vp_1, \cdots, \vp_{k_+-1})$ of functions in $H^s(M)$ that satisfy 
  $$ \begin{aligned}
  &   \int_M \vp_i P_g^s \vp_i dv_g \le 0 \quad \text{ for all } 1 \le i \le k_+-1 \text{ and } \\
  &   \int_M \vp_i  P_g^s \vp_j dv_g = \Big(\int_M \vp_i P_g^s \vp_i dv_g\Big) \delta_{ij}  \quad \text{ for all } 1 \le i,j \le k,
    \end{aligned} $$ 
and such that $(\vp_1, \cdots, \vp_{k_+-1}, v_{k_+}, \cdots, v_k)$  is orthonormal with respect to $Q(\beta,\cdot)$. By Proposition \ref{prop:deficontinuityeigen} (2), we let $(\vp_{m(\beta)},\cdots, \vp_{k_+ -1})$ be generalised eigenfunctions associated to $\left(\lambda_{j}(\beta)\right)_{m(\beta)\leq j \leq k_+-1}$ such that $(\vp_{m(\beta)},\cdots, \vp_{k_+ -1},v_{k_+},\cdots, v_k)$ is orthonormal with respect to $Q(\beta,\cdot)$. By Proposition \ref{prop:deficontinuityeigen} (4), applied for $m=k$ and $\ep = -1$, we obtain an orthonormal family $(\vp_{1},\cdots, \vp_{k_+ -1},v_{k_+},\cdots, v_k)$ with respect to $Q(\beta,\cdot)$ so that the expected properties hold. We now let 
$$ V = \text{Span} \big\{ \vp_1, \cdots, \vp_{k_+-1}, v_{k_+}, \cdots, v_k\big\}, $$
which is an admissible subspace in the definition \eqref{eq:def:lambdak} of $\lambda_k(\beta)$. Simple computations then show that $\lambda_k(\beta) \le \nu_k$, which shows that $\lambda_k$ is lower semi-continuous in $\mathcal{A}_p$ when $k\geq k_+$. The lower semi-continuity of $ \bar{\lambda}_k^p(\beta)$ then easily follows. 
\end{proof}

\section{Variational theory for generalised eigenvalues} \label{sec:theorievariationnellevp}

In this section we derive the Euler-Lagrange equations associated to local extremals of renormalised eigenvalue functionals. We also prove the existence of almost-extremals for such functionals as well as the existence of true extremals for a subcritical approximation of \eqref{defLambdak}. 

\subsection{First variation of generalised eigenvalues}
We aim at computing the directional derivatives of eigenvalues $\lambda_k(\beta)$ at $\beta \in \mathcal{A}_p$ for variations of the form $\beta +t b \in \mathcal{A}_p$,  with $t \ge 0$ and $ b \in \mathcal{A}_p$. If $\beta \in \mathcal{A}_p$ for some $p \ge \frac{n}{2s}$ we introduce the following notations: for $k \ge m(\beta)$, where $m(\beta)$ is as in \eqref{def:mbeta},
\begin{equation} \label{def:petitgrand:i}
\begin{aligned} 
i(k) & = \min\{ i \in \N^* ; \lambda_i(\beta) = \lambda_k(\beta) \} \\
I(k) & = \max\{ i \in \N^* ; \lambda_i(\beta) = \lambda_k(\beta) \}.
\end{aligned}
\end{equation}
The integers $i(k)$ and $I(k)$ do depend on $\beta$ but for simplicity, and since no confusion shall arise in the following, we will omit the dependence on $\beta$. If $\mathcal{R}$ is as in \eqref{eq:definitions} we denote by $\mathcal{R}_\beta(\beta,v)$ the partial differential of $\mathcal{R}$ with respect to the first variable at $(\beta,v)$, which is given by 
\begin{equation} 
\label{eq:def:rbetah} 
\mathcal{R}_{\beta}(\beta,v) (h) = -\mathcal{R}(\beta,v) \frac{\int_M h v^2 dv_g}{\int_M \beta v^2 dv_g}. 
\end{equation}

\begin{prop} \label{prop:firstderivative}
Let $(M,g)$ be a Riemannian manifold of dimension $n\geq 3$, $s \in  \N^*$ such that $n>2s$ and $p\geq \frac{n}{2s}$. Let $\beta \in \mathcal{A}_p$. We assume that $\beta>0$ or that $P_g^s$ satisfies \eqref{eq:unique:continuation}. Let $k \ge m(\beta)$ and $b\in \mathcal{A}_p$. Then
\begin{equation} \label{eq:firstderivativeminmax}
\begin{split} \lim_{t \searrow 0 } \frac{\lambda_k(\beta+tb) - \lambda_k(\beta)}{t} = & \min_{ V \in \mathcal{G}_{k-i(k)+1}(E_k(\beta)) } \max_{v \in V\setminus \{0\}} \mathcal{R}_\beta(\beta,v)(b) \\ 
= & \max_{ V \in \mathcal{G}_{I(k)-k+1}(E_k(\beta)) } \min_{v \in V\setminus \{0\}} \mathcal{R}_\beta(\beta,v)(b)
\end{split}
\end{equation}
\end{prop}

\begin{rem}
Proposition~\ref{prop:firstderivative} allows us to compute right derivatives of eigenvalues at points $\beta$ that can be at the boundary of the set $\mathcal{A}_p$. This generalises the previous results in \cite{AmmannHumbert, GurskyPerez} where, in the special case of $P_g^1 =\Delta_g + \frac{n-2}{4(n-1)} S_g$,  left and right derivatives along variations of the form $(1+th)\beta$  were computed. The right derivative along general variations $\beta +tb$ that we cover here provides more information and was first observed and used in \cite{Petrides2,Petrides5, Petrides4} for Laplace eigenvalues in dimension 2. This result will be crucial for the computation of the Euler-Lagrange equation of optimisers and also for almost optimisers.
\end{rem}

\begin{proof} First, the second equality in \eqref{eq:firstderivativeminmax} is a consequence of the explicit form \eqref{eq:def:rbetah} of $\mathcal{R}_\beta$ and of min-max formulae for the quotients of a quadratic form by a positive definite quadratic form on finite-dimensional spaces. We thus only have to prove the first equality in \eqref{eq:firstderivativeminmax}. Notice that from Proposition \ref{prop:deficontinuityeigen}, we have that $\lambda_k(\beta+tb) \to \lambda_k(\beta)$ as $t \searrow 0$, hence $\lambda_k(\beta+tb) > - \infty$ for all $t$ small enough when $k \ge m(\beta)$. For $t \ge0$ we denote by 
$$\vp_{i(k)}^t,\cdots, \vp_{I(k)}^t$$
a family of $Q(\beta+tb,\cdot)$-orthonormal generalised eigenfunctions associated to the eigenvalues $\lambda_{i(k)}(\beta+tb) \leq \cdots \leq \lambda_{I(k)}(\beta+tb)$, that we denote for simplicity by 
$$\lambda_{i(k)}^t \leq \cdots \leq \lambda_{I(k)}^t.$$
By Proposition \ref{prop:deficontinuityeigen} and by the definition \eqref{def:petitgrand:i} of $i(k)$ and $I(k)$ these families  all converge to $\lambda_k = \lambda_k(\beta)$ as $t\to 0$.
By Lemma \ref{lem:mainlemma}, since $\beta + t b \to \beta \neq 0$, $(\vp_i^t)_t$ is bounded in $H^s(M)$ as $t \to 0$ for all $i(k)\leq i \leq I(k)$, so that up to the extraction of a subsequence as $t \to 0$, $\vp_i^t$ converges to $\vp_i$ weakly in $H^s(M)$, strongly in $H^{s-1}(M)$ and in $L^2_\beta(M)$ by Proposition \ref{prop:eqnormscompactness}, (2). As a consequence, $ Q(\beta+tb , \vp_i^t - \vp_i) \to 0 $ as $t\to 0$, and passing to the weak limit the equation satisfied by $\vp_i^t$ and passing to the limit the relation $B(\beta+tb,\vp_i^t,\vp_j^t) = \delta_{i,j}$, we obtain
$$ P_g^s \vp_i = \lambda_k \beta \vp_i \quad \text{ and } \quad  B(\beta,\vp_i,\vp_j) = \delta_{i,j} $$
for $i(k)\leq i,j\leq I(k)$. Integrating the equation with respect to $\vp_i$ proves that
\begin{equation*}
\begin{split} \int_M \vp_i \Delta_g^s \vp_i dv_g & = -\int_M \vp_i A_g^s \vp_i dv_g + \lambda_k Q(\beta,\vp_i) 
\\ & = -\int_M \vp_i^t A_g^s \vp_i^t dv_g + \lambda_i^t Q(\beta+tb,\vp_i^t) + o(1) \\
& = \int_M \vp_i^t \Delta_g^s \vp_i^t dv_g  + o(1)
\end{split} \end{equation*}
as $t \to 0$, where $A_g^s$ is as in \eqref{eq:def:Ag}, so that $\vp_i^t$ converges strongly to $\vp_i$ in $H^s(M)$ as $t \to 0$.

\medskip 

For $i(k) \leq i \leq I(k)$ we let $R_i^t = \vp_i^t - \pi_k(\vp_i^t) $ where for $v\in H^s(M)$ we have let
$$ \pi_k(v) =  \sum_{i=i(k)}^{I(k)} B(\beta,v,\vp_i) \vp_i $$
be the orthogonal projection of $v$ on $E_k(\beta)$ with respect to $Q(\beta,\cdot)$. We have, for every $t \ge 0$,
\begin{equation} \label{eqRit}
 P_g^s R_i^t  - \lambda_k \beta R_i^t = \lambda_i^t (\beta + tb) \vp_i^t - \lambda_k \beta \vp_i^t = (\lambda_i^t - \lambda_k) \beta \vp_i^t + \lambda_i^t tb \vp_i^t . 
 \end{equation}
For $i(k) \le i \le I(k)$ we let, for $t >0$, 
\begin{equation} \label{eq:defialphait} \alpha_i^t = \left\vert \lambda_i^t - \lambda_k \right\vert + t + \Vert R_i^t \Vert_{H^s} \end{equation}
and
\begin{equation} \label{eq:defitilderit} \tilde{R}_i^t = \frac{R_i^t}{\alpha_i^t} \hspace{5mm} \tau_i^t = \frac{t}{\alpha_i^t} \hspace{5mm} \delta_i^t = \frac{\lambda_i^t - \lambda_k}{\alpha_i^t}. \end{equation}
Up to the extraction of a subsequence as $t\to 0$, we have
$$ \tilde{R}_i^t \to \tilde{R}_i \text{ weakly in } H^s(M), \hspace{5mm} \tau_i^t \to \tau_i, \hspace{5mm} \delta_i^t \to \delta_i,$$
where $\tau_i$ and $\delta_i$ belong to $[0,1]$. Passing to the weak limit in \eqref{eqRit} we obtain
\begin{equation} \label{eq:limritilde} P_g^s \tilde{R}_i - \lambda_k \beta  \tilde{R}_i = \delta_i \beta \vp_i + \tau_i \lambda_k b \vp_i. \end{equation}
Dividing \eqref{eqRit} by $\alpha_i^t$ and substracting \eqref{eq:limritilde} shows that $ \tilde{R}_i^t - \tilde{R}_i$ satisfies 
$$ \big( P_g^s - \lambda_k \beta \big) \big(  \tilde{R}_i^t - \tilde{R}_i \big) = o(1) \quad \text{ in } H^{-s}(M) \quad \text{ as } t \to 0. $$
Standard elliptic theory shows independently that for any $f \in E_k(\beta)^{\perp_Q(\beta, \cdot)}$ we have $\Vert f \Vert_{H^s} \le C \Vert \big( P_g^s - \lambda_k \beta \big) f \Vert_{H^{-s}}$ for some $C>0$. Since by definition $ \tilde{R}_i^t - \tilde{R}_i  \in E_k(\beta)^{\perp_Q(\beta, \cdot)}$ we get that, up to passing to a subsequence, $(\tilde{R}_i^t)_t$ converges strongly to $\tilde{R}_i$ as $t\to 0$ in $H^s(M)$ and with \eqref{eq:defialphait}, and \eqref{eq:defitilderit} we have 
\begin{equation} \label{eq:sumRitildedeltatau} \Vert \tilde{R}_i \Vert_{H^s} + \vert \delta_i \vert + \tau_i = 1. \end{equation}
Now, we integrate \eqref{eq:limritilde} against $\vp_i$ and we obtain that
\begin{equation} \label{eq:eqriintegrated}
\delta_i Q(\beta,\vp_i) + \tau_i \lambda_k \int_{M} b \vp_i^2 dv_g = 0.
\end{equation}
Assume by contradiction that $\tau_i = 0$, then by \eqref{eq:eqriintegrated} and since $Q(\beta,\vp_i) = 1$, we have  $\delta_i=0$ so that  by \eqref{eq:limritilde}, $\tilde{R}_i \in E_k(\beta) \cap E_k(\beta)^{\perp_{Q(\beta,\cdot)}} = \{0\}$. This contradicts \eqref{eq:sumRitildedeltatau}. Therefore $\tau_i \neq 0$ and, by \eqref{eq:eqriintegrated}, 
$$ \frac{\delta_i}{\tau_i} = \frac{-\lambda_k\int_M b\vp_i^2 dv_g }{Q(\beta,\vp_i)} $$
Integrating \eqref{eq:limritilde} against $\vp_j$ for $j\neq i$, we obtain that
$$ \lambda_k \int_M b \vp_i \vp_j = 0 .$$
so that $\vp_{i(k)},\cdots,\vp_{I(k)}$ are both an orthonormal basis for $Q(\beta,\cdot)$ and an orthogonal family for the bilinear form $(\vp,\psi) \mapsto  -\lambda_k \int_M b\vp \psi dv_g $ over $E_k(\beta)$. Since by construction we have $\frac{\delta_i^t}{\tau_i^t}  = \frac{\lambda_i^t - \lambda_i}{t}$ the family $\big(\frac{\delta_{i}^t }{\tau_{i}^t}\big)_{i(k) \le i \le I(k)}$ is non-decreasing in $i$ for every fixed $t >0$. Passing to the limit yields $\frac{\delta_{i(k)}}{\tau_{i(k)}}\leq \cdots \leq \frac{\delta_{I(k)}}{\tau_{I(k)}}$ and classical min-max formulae for orthonormal diagonalization give
$$ \frac{\delta_i}{\tau_i} = \min_{ V \in \mathcal{G}_{i-i(k)+1}(E_k(\beta)) } \max_{v \in V\setminus \{0\}} \frac{-\lambda_k\int_M bv^2 dv_g }{Q(\beta,v)} .$$
Since the right-hand side is independent of the choice of the subsequence as $t\to 0$, we obtain that the directional derivative exists and
$$ \lim_{t\searrow 0} \frac{\lambda_i^t-\lambda_i}{t} = \lim_{t\searrow 0} \frac{\delta_i^t}{\tau_i^t} = \frac{\delta_i}{\tau_i} .$$
Since $\mathcal{R}(\beta,v) = \lambda_k$ when $v \in V \in \mathcal{G}_{i-i(k)+1}(E_k(\beta))$ this proves \eqref{eq:firstderivativeminmax}
\end{proof}

\subsection{Euler-Lagrange equation for local extremals of renormalised eigenvalue functionals}

Let $p \geq \frac{n}{2s}$ be fixed and let $\beta \in \mathcal{A}_p$. Let $k \ge m(\beta)$, where $m(\beta)$ is defined in \eqref{def:mbeta}. We define the following volume-renormalised version of $\lambda_k(\beta)$:
\begin{equation} \label{def:lambdabar}
 \bar{\lambda}_k^p(\beta) = \lambda_k(\beta) \Vert \beta \Vert_{L^p}. 
 \end{equation}
We recall that the integers $k_{-}$ and $k_{+}$ in this paper are respectively defined as follows: 
\begin{equation} \label{defkk}
k_- = \max \{k \geq 1,  \la_k(g)<0 \}  \quad \hbox{ and } \quad k_+= \min\{k \geq 1, \la_k(g)>0 \},
\end{equation}
where $\la_k(g)$ denote the eigenvalues of the GJMS operator $P_g^s$ over $H^s(M)$. If $\lambda_1(g) \ge 0$ we use the convention $k_{-} = 0$. We recall that $i(k) = i(k, \beta)$ and $I(k) = I(k,\beta)$ are defined in \eqref{def:petitgrand:i}.  We recall that we wish, in this paper, to maximise negative eigenvalues of $P_g^s$ and minimise positive eigenvalues of $P_g^s$ in $[g]$ (the setting is that of Proposition~\ref{prop:preliminaryvarset}). The following result establishes the Euler-Lagrange equation satisfied by local extremals of $\bar{\lambda}_k^p $ over $\mathcal{A}_p$ (if they exist):

\begin{prop} \label{euler_minimizer}
Let $s \in \mathbb{N}^*, 2s <n$, and let $P_g^s$ be the GJMS operator of order $2s$ in $(M,g)$. Let $p \geq \frac{n}{2s}$ and let $\beta \in \mathcal{A}_p$ be such that $\Vert \beta \Vert_{L^p} = 1$. We assume that $\beta >0$ or that $P_g^s$ satisfies \eqref{eq:unique:continuation}.  
\begin{itemize}
\item Assume that $k \ge k_+$ and that $\beta$ is a local minimum of $\bar{\lambda}^p_k$. Then, for any subspace $V \in \mathcal{G}_{k-i(k)+1}(E_k(\beta))$, there exists an integer $i(k) \leq r \leq k$, a $Q(\beta,\cdot)$-orthonormal family $(v_{r},\cdots,v_k)$ of $V$, and there exist positive numbers $d_r,\cdots,d_k$ satisfying $\sum_{i=r}^k d_i = 1$ such that 
$$   \beta^{p-1} = \sum_{i=r}^k d_i v_i^2 \quad \text{ a. e. in } M.$$
\item  Assume that $k \le k_-$ and that $\beta$ is a local maximum of $\bar{\lambda}^p_k$. Then, for any subspace $V \in \mathcal{G}_{k-i(k)+1}(E_k(\beta))$, there exists an integer $I(k) \geq s \geq k$, a $Q(\beta,\cdot)$-orthonormal family $(v_{k},\cdots,v_s)$ of $V$, and there exist positive numbers $d_k,\cdots,d_s$ satisfying $\sum_{i=k}^s d_i = 1$ such that 
\begin{equation*}
 \beta^{p-1} = \sum_{i=k}^sd_i v_i^2 \quad \text{ a. e. in } M.
 \end{equation*}
\end{itemize}
\end{prop}
Recall that if $i(k) \le r \le k$ or $k \le r \le I(k)$ we have $\lambda_r(\beta) = \lambda_k(\beta)$ by \eqref{def:petitgrand:i}, and that if $V \in \mathcal{G}_{k-i(k)+1}(E_k(\beta))$ and $v \in V$ we have $P_g^s v = \lambda_k(\beta) \beta v$ in $M$. In the cases $s=1$ and $s=2$ weaker forms of Proposition~\ref{euler_minimizer} had previously been obtained in \cite{GurskyPerez, Perez_Ayala_2022}.

\begin{rem} \label{rem:continuite}
A simple but important consequence of Proposition \ref{euler_minimizer} is the following: if $p \ge \frac{n}{2s}$ and $\beta \in \mathcal{A}_{p}$ is an extremal for  $\bar{\lambda}_k^p$ for some $k \ge 1$, then $\beta \in C^{0,\alpha}(M)$ for some $0 < \alpha < 1$. We prove this in the case $k \ge k_+$ but the proof in the case $k \le k_-$ is identical. Assume that $\beta \in \mathcal{A}_p$, with $\Vert \beta \Vert_{L^p} = 1$, is a local minimum of $\bar{\lambda}_k^p$. Then by Proposition~\ref{euler_minimizer} there are $i(k) \le r \leq k$, a $Q(\beta,\cdot)$-orthonormal family $(v_{r},\cdots,v_k)$ of $E_k(\beta)$, and positive numbers $d_r,\cdots,d_k$ satisfying $\sum_{i=r}^k d_i = 1$ such that 
\begin{equation}  \label{lieu:annulation:beta} 
  \beta^{p-1} = \sum_{i=r}^k d_i v_i^2 \quad \text{ a. e. in } M.
  \end{equation}
Since $p \ge \frac{n}{2s}$, $\beta \in L^{\frac{n}{2s}}(M)$, every $v_i$ satisfies $P_g^s v_i = \lambda_k(\beta) \beta v_i$ in $M$, and standard regularity theory (see for instance Mazumdar \cite[Theorem 5]{Mazumdar1}) shows that $v_i \in L^q(M)$ for all $q >1$ and $r \le i \le k$. This in turn implies that $\beta  \in L^q(M)$ for all $q >1$, and standard bootstrap arguments then show that $v_i \in C^1(M)$, hence $\beta \in C^{0,\alpha}(M)$ for some $0 < \alpha < 1$. 
\end{rem}

\begin{rem}
Under the assumptions of Remark \ref{rem:continuite}, and in the special case where $s=1$, we also obtain that if the set $\{\beta = 0\}$ is non-empty it is of measure zero. Indeed, since every $v_i$, $r  \le i \le k$, satisfies $P_g^s v_i = \lambda_k(\beta) \beta v_i$ in $M$ and $Q(\beta, v_i) = 1$, classical unique continuation results (see e.g. \cite[Theorem 1.7]{HardtSimon}) show that $v_i$ vanishes on a set of measure zero in $M$. It then follows from \eqref{lieu:annulation:beta} that the set $\{\beta=0\}$, if it is non-empty, is of measure zero. Determining \emph{a priori} whether an extremal $\beta$ for  $\bar{\lambda}_k^p$ for some $k \ge 1$ vanishes seems difficult. In \cite[Theorem 1.3]{GurskyPerez}, and for $s=1$, the authors construct examples of manifolds $(M,g)$ where $k_- \ge 2$ and where $\beta \equiv 1$ achieves $\Lambda_2^1(M,[g]) <0$ -- or, in other words, where $g$ itself is the maximal metric for $\Lambda_2^1(M,[g])$. By contrast, in Proposition~\ref{prop:espace1dim} below and for any $s \ge 1$, we describe special situations -- corresponding to the particular case $r=k$ in \eqref{lieu:annulation:beta}, and occurring both for positive and negative generalised eigenvalues -- where $\{\beta = 0\} \neq \emptyset$, and where $\beta g$ does not define a smooth metric in $g$. This behavior was first observed in \cite{AmmannHumbert} when $s=1$. 
\end{rem}

\begin{proof}[Proof of Proposition \ref{euler_minimizer}] We only prove the proposition for the local minimum, which corresponds to the case $k \ge k_+$.  The proof for the case $k \le k_{-}$ is similar. Let $p'= \frac{p}{p-1}$, so that $1 <p' \le \frac{n}{n-2s}$. We set $\bar{\lambda}_k = \bar{\lambda}_k^p$ in this proof.  We assume that $\beta \in \mathcal{A}_p$ is a local minimum of $\bar{\lambda}_k$ with $\Vert \beta \Vert_{L^p} = 1$. Let $V \in \mathcal{G}_{k-i(k)+1}(E_k(\beta))$. We first prove that
\begin{equation} \label{eq:HBnonempty}
\begin{split} \{ \zeta \in L^{p'}(M) , & \zeta \geq 0 \text{ a.e } \}  \cap   
 co \left\{ \bar{\lambda}_k(\beta) \left( \beta^{p-1} - v^2 \right) ; v \in V, Q(\beta,v) = 1 \right\} \neq \emptyset ,
\end{split} \end{equation}
where $co$ denotes the convex hull. Assume by contradiction that this intersection is empty. Since $V$ is finite-dimensional it is easily seen that $ \{ \zeta \in L^{p'}(M) , \zeta \geq 0 \text{ a.e } \}$ and $co \left\{ \bar{\lambda}_k(\beta) \left( \beta^{p-1} - v^2 \right) ; v \in V, Q(\beta,v) = 1 \right\}$ are, respectively, closed and compact subsets of $(L^{p'}(M), \Vert \cdot \Vert_{L^{p'}})$. Since they are assumed to be disjoint the geometric Hahn-Banach theorem shows that there is $b \in L^p(M)$ such that
\begin{equation} \label{eq:HB1} \forall \zeta \in L^{p'}(M), \zeta \geq 0, \int_M \zeta b dv_g \geq 0 \end{equation}
and
\begin{equation} \label{eq:HB2} \forall \zeta \in  
co \left\{ \bar{\lambda}_k(\beta) \left( \beta^{p-1} - v^2 \right) ; v \in V, Q(\beta,v) = 1 \right\} , \int_M \zeta b dv_g < 0. \end{equation}
\eqref{eq:HB1} implies that $b \geq 0$ a.e and \eqref{eq:HB2} that $b \neq 0$, and hence  $b \in \mathcal{A}_p$. 
By Proposition \ref{prop:firstderivative} and the Leibniz rule for products of derivatives we have
$$ \lim_{t\to 0} \frac{\bar{\lambda}_k(\beta+tb) - \bar{\lambda}_k(\beta)}{t} \leq \max_{v\in V} \left( \bar{\lambda}_k(\beta)\int_{M} b (\beta^{p-1} - v^2)dv_g \right) $$
Then, there is $v \in V$ such that $Q(\beta,v) = 1$ and 
$$ \lim_{t\to 0} \frac{\bar{\lambda}_k(\beta+tb) - \bar{\lambda}_k(\beta)}{t} \leq  \bar{\lambda}_k(\beta)\int_{M} b (\beta^{p-1} - v^2)dv_g. $$ 
The latter is negative by  \eqref{eq:HB2} applied to the function $\zeta = \bar{\lambda}_k(\beta)(\beta^{p-1}-v^2)$, which is impossible since $\beta$ is a local minimum of $\bar{\lambda}_k$. This proves \eqref{eq:HBnonempty}.  Therefore for any $V \in \mathcal{G}_{k-i(k)+1}(E_k(\beta))$, there are $v_1,\cdots,v_\ell \in V$ such that
\begin{equation} \label{eq:eulerlagrange} \sum_{i=1}^\ell Q(\beta,v_i) = 1 \quad \text{ and }\quad \bar{\lambda}_k(\beta)(\beta^{p-1}-\sum_{i=1}^\ell v_i^2) \geq 0. \end{equation}
Integrating \eqref{eq:eulerlagrange} against $\beta$ gives, since $\Vert \beta \Vert_{L^p} = 1$,
$$ \bar{\lambda}_k(\beta) \ge \bar{\lambda}_k(\beta) \sum_{i=1}^\ell Q(\beta,v_i) =  \bar{\lambda}_k(\beta) . $$
Hence the latter inequality is an equality and by \eqref{eq:eulerlagrange}, $\bar{\lambda}_k(\beta)(\beta^{p}- \beta \sum_{i=1}^\ell v_i^2) \geq 0$ a.e implies that
$$\bar{\lambda}_k(\beta)\left( \beta^p - \beta \sum_{i=1}^\ell v_i^2 \right) = 0  \quad \text{ a.e. in } M.$$
Since $ \bar{\lambda}_k(\beta) >0$ by Proposition~\ref{prop:preliminaryvarset}, \eqref{eq:eulerlagrange} implies that $\sum_{i=1}^{\ell} v_i^2 = 0$ a.e. in $\{ \beta = 0 \}$.  We therefore obtain that 
$\bar{\lambda}_k(\beta)\left(\beta^{p-1}-\sum_{i=1}^\ell v_i^2\right) = 0$ a.e. in $M$, that is that $ 0 \in  co \left\{ \bar{\lambda}_k(\beta) \left( \beta^{p-1} - v^2 \right) ; v \in V, Q(\beta,v) = 1 \right\} $. This is true for any linear subspace $V \in \mathcal{G}_{k-i(k)+1}(E_k(\beta))$, hence 
\begin{equation*} 0\in   \bigcap_{V \in \mathcal{G}_{k - i(k) + 1}(E_k(\beta))} co \left\{ \bar{\lambda}_k(\beta) \left( \beta^{p-1} - v^2 \right) ; v \in V, Q(\beta,v) = 1 \right\}.
\end{equation*}
The proposition then follows from Lemma \ref{mixinglemma} below since for any linear subspace $V\subset E_k(\beta)$ of dimension $\ell$ we have
\begin{equation} 
\begin{split} co & \left\{ \bar{\lambda}_k(\beta) \left( \beta^{p-1} - v^2 \right) ; v \in V, Q(\beta,v) = 1 \right\} \\ & = \left\{ \bar{\lambda}_k(\beta) \left( \beta^{p-1} - \sum_{i=1}^\ell d_i v_i^2 \right) ; \sum_{i=1}^\ell d_i = 1 \text{ and } (v_i) \in O_{Q(\beta,\cdot)}(V)  \right\}
\end{split} \end{equation}
where $O_{Q(\beta,\cdot)}(V)$ is the set of orthonormal bases of $V$ with respect to the scalar product $Q(\beta,\cdot)$.
\end{proof}

For the sake of completeness, we give the proof of the following lemma of linear algebra that we used in the proof of Proposition~\ref{euler_minimizer} (see the mixing lemma in \cite{PetridesTewodrose}, lemma 2.1 for a more general setting):
\begin{lemme} \label{mixinglemma} 
Let $F \subseteq L^2_{\beta}(M)$ be a linear subspace of dimension $\ell$ and $O_{Q(\beta,\cdot)}(F)$ be the set of orthonormal bases of $F$ with respect to the scalar product $Q(\beta,\cdot)$. We have
\begin{equation} 
\begin{split} co  \left\{  v^2  ; v \in F, Q(\beta,v) = 1 \right\}  = \left\{ \sum_{i=1}^\ell d_i v_i^2  ; \sum_{i=1}^\ell d_i = 1 \text{ and } (v_i)_{1 \le i \le \ell} \in O_{Q(\beta,\cdot)}(F)  \right\}
\end{split} \end{equation}
\end{lemme}
\begin{proof}
The right inclusion is obvious and we prove the left one. Let $(v_\alpha)_{1 \le \alpha \le N}$, $v_\alpha \in F$, be such that $Q(\beta, v_\alpha) = 1$ for all $\alpha$ and let $(t_\alpha)_{1 \le \alpha \le N}$ be nonnegative numbers such that $\sum_{\alpha = 1}^N t_\alpha =1$.  
Let $(e_1,\cdots,e_\ell) \in O_{Q(\beta,\cdot)}(F) $. Then, for $1 \le \alpha \le N$, we can write $v_\alpha = \sum_{i=1}^\ell c_i^\alpha e_i$ with $\sum_{i=1}^\ell \left(c_i^\alpha\right)^2 = 1$ and we have
$$  \sum_{\alpha=1}^N t_\alpha v_\alpha^2 = \sum_{\alpha=1}^N \sum_{1\leq i,j \leq \ell} t_\alpha c_i^\alpha c_j^\alpha e_i e_j = \sum_{1\leq i,j \leq \ell} A_{i,j} e_ie_j $$
where we have set
$$ A_{i,j} = \sum_{\alpha=1}^N t_\alpha c_i^\alpha c_j^\alpha. $$
$A$ is a non-negative symmetric matrix. Let $d_1,\cdots,d_\ell$ and $O \in \mathbb{O}_\ell(\R)$ be such that
$$ O^T A O = diag(d_1,\cdots,d_\ell). $$ 
Then
$$ \sum_{\alpha=1}^N t_\alpha v_\alpha^2 = \sum_{1\leq i,j,k \leq \ell}  O_{ik} d_k O_{jk} e_i e_j. $$
Let, for $1 \le k \le \ell$, $v_k = \sum_{i=1}^\ell O_{i,k} e_i$. Straightforward computations show that $(v_1,\cdots,v_\ell) \in O_{Q(\beta,\cdot)}(F)$ and that
$$ \sum_{\alpha=1}^N t_\alpha v_\alpha^2 = \sum_{k=1}^\ell d_k v_k^2 $$
Finally we have that $\sum_{k=1}^\ell d_k = tr(A) = \sum_{\alpha=1}^N  t_\alpha   \sum_{i=1}^\ell \left(c_i^\alpha\right)^2 = 1 $ and the proof of the lemma is complete.
\end{proof}

\subsection{Unboundedness of renormalised eigenvalue functionals}
The following result, based \cite{AmmannJammes}, shows that positive renormalised eigenvalue functionals are not bounded from above and that renormalised negative eigenvalues are not bounded from below:

\begin{prop} \label{prop:unbounded}
We have
\begin{equation} \label{eq:unbounded}
\begin{aligned}
& \inf_{\be \in C^\infty(M), \be>0}  \lambda_k(\beta) \Vert \beta \Vert_{L^{\frac{n}{2s}}}  = -\infty & \text{ if } k \le k_{-} \\
& \sup_{\be \in C^\infty(M), \be>0}  \lambda_k(\beta) \Vert \beta \Vert_{L^{\frac{n}{2s}}} = +\infty & \text{ if } k \ge k_{+}. 
\end{aligned}
\end{equation}
\end{prop} 
Proposition~\ref{prop:unbounded} justifies a posteriori the choice of the extremisation problem \eqref{defLambdak}. This result was first proven in \cite{AmmannJammes} for positive eigenvalues. 

\begin{proof}
Since, for any $\beta \in \mathcal{A}_{\frac{n}{2s}}$, the sequence $(\lambda(\beta))_{k \ge m(\beta)}$ is nondecreasing we only need to prove \eqref{eq:unbounded} for $k=k_-$ and $k=k_+$, respectively. A general result of this type has been proven in \cite[Theorem 1.3, Theorem 4.1]{AmmannJammes}. We check here that the assumptions of the results in \cite{AmmannJammes} are satisfied. First, relation \eqref{eq:confinv} shows that $P_g^s$ is a conformally covariant elliptic operator of order $2s$ and of bidegree $((n-2s)/2, (n+2s)/2)$ in the sense of \cite[Definition 2.2]{AmmannJammes}.
Let now $g_*$ denote the product metric in $ \mathbb{R} \times \mathbb{S}^{n-1}$: if $g_0$ is the round metric of $\mathbb{S}^{n-1}$ it is given by 
$$ g_{*} = dt^2  \oplus g_0 $$
where $t$ is the variable in $\R$. This metric defines a metric laplacian $\Delta_{g_*} = -\partial_t^2 + \Delta_{g_0}$. A recent result of Case-Malchiodi \cite[Theorem 1.1]{CaseMalchiodi} shows that since $( \mathbb{R} \times \mathbb{S}^{n-1}, g_*)$ is a product of Einstein manifolds, $P_{g_*}$ explicitly writes as 
$$ P_{g_{*}}^s = \begin{cases} \prod_{i=0}^{\frac{s-2}{2}} D_{s-1-2i} \text{ if } s \text{ is even} \\  \left( \Delta_{g_*} + \frac{(n-2)^2}{4} \right) \prod_{i=0}^{\frac{s-3}{2}} D_{s-1-2i} \text{ if } s \text{ is odd}, \end{cases}$$
where $D_j$ is defined as
\begin{equation*}\begin{split} D_j & = \Delta_{g_*}^2 + 2\left( \left(\frac{n-2}{2}\right)^2 \Delta_{g_*} + j^2 \left( \Delta_g + \partial_t^2 \right) \right) + \left(\left(\frac{n-2}{2}\right)^2 - j^2\right)^2 \\
& = \left(\Delta_{g_*} + \frac{(n-2)^2}{4} - j^2 \right)^2 + 4 j^2 \Delta_g. 
\end{split}\end{equation*}
Since $n >2s$ we have, for $0 \le j \le \frac{s-2}{2}$ if $s$ is even or $0 \le j \le \frac{s-3}{2}$ is $s$ is odd, $\frac{(n-2)^2}{4} - j^2  > 0$. Straightforward computations, since all the coefficients of the operators that appear in the product are positive and since all the operators $\Delta_{g}$, $\partial_t$, $\Delta_{g_*}$ commute, show that for $u \in H^s(\mathbb{S}^{n-1} \times \R)$ we have
$$ \int_{\mathbb{S}^{n-1} \times \R} u P_{g_*}^s u dv_{g_*} \geq  \int_{\mathbb{S}^{n-1} \times \R} \vert \Delta_{g_*}^{\frac{s}{2}} u \vert^2 dv_{g_*} + c \int_{\mathbb{S}^{n-1} \times \R}  u^2 dv_{g_*}  $$
where
$$ c = \begin{cases}\prod_{i=0}^{\frac{s-2}{2}} \left( \frac{(n-2)^2}{4} - (s-1-2i)^2 \right)^2 \\ \frac{(n-2)^2}{4} \prod_{i=0}^{\frac{s-3}{2}} \left( \frac{(n-2)^2}{4} - (s-1-2i)^2 \right)^2 \end{cases} .$$
In particular $P_{g_*}^s: H^k(\mathbb{S}^{n-1} \times \R) \to L^2(\mathbb{S}^{n-1} \times \R)$ is coercive and thus invertible and hence $P_g^s$ is invertible in $\mathbb{S}^{n-1} \times \R$ in the sense of \cite[Definition 2.4]{AmmannJammes} and its essential spectrum is contained in $[\sigma, + \infty)$ for some $\sigma >0$ (here $\sigma$ denotes the bottom of the essential spectrum). 

We may thus apply the analysis of \cite{AmmannJammes} that we briefly recall. Let $\chi \in C^\infty_c( [0,2) )$ be such that $0\leq \chi \leq 2$ and $\chi = 1$ in $[0,1]$. Let $2 \delta < i_g(M)$ and $\xi \in M$ be fixed. We define, for $x \in M$,
$$ \beta(x)^{\frac{1}{2s}} = 1- \chi \big(d_g(\xi,x) \big) + \chi \big( d_g(\xi,x)\big) \frac{1}{d_g(\xi,x)} $$
and, for $\epsilon >0$ small enough,
$$ \beta_\ep(x)^{\frac{1}{2s}} = \left(1-\chi\left(\frac{d_g(x,\xi)}{\ep}\right)\right)\beta(x) + \frac{1}{\epsilon} \chi\left(\frac{d_g(x,\xi)}{\ep}\right). $$
Clearly, $\beta_\ep \to \beta$ in $C^0_{loc}(M \backslash \{\xi\})$. It is proven in \cite[Section 3]{AmmannJammes} that $\beta^{\frac{1}{s}}g$ defines a complete asymptotically cylindrical metric in $M \backslash \{\xi\}$ and that the essential spectrum of $P_{\beta^{\frac{1}{s}}g}^s$ in $L^2(M \backslash \{\xi\}, \beta^{\frac{1}{s}}g)$ is the same as the essential spectrum of $P_{g_*}^s$ in $L^2(\mathbb{S}^{n-1} \times \R, g_*)$, and is as such contained in $[\sigma, +\infty) \subset (0, + \infty)$. Since $P_{\beta^{\frac{1}{s}}g}^s$ is self-adjoint its spectrum in the interval $(-\sigma, \sigma)$ is discrete and, when it exists,  we will denote by $\lambda_{\beta,+}$ the smallest positive isolated eigenvalue of $P_{\beta^{\frac{1}{s}}g}^s$ (resp. by $\lambda_{\beta,-}$ the largest negative isolated eigenvalue of $P_{\beta^{\frac{1}{s}}g}^s$). We now claim that
\begin{equation} \label{eq:limite spectre}
 \limsup_{\ep \to 0} \lambda_{k_-}(\beta_\ep) < 0 \quad \text{ and } \quad  \liminf_{\ep \to 0} \lambda_{k_+}(\beta_\ep)>0. 
\end{equation}
It is easily seen that $\Vert \beta_\ep \Vert_{L^{\frac{n}{2s}}} \to + \infty$ as $\ep \to 0$ by definition of $\beta_\ep$, so that \eqref{eq:unbounded} follows from \eqref{eq:limite spectre}. We now prove \eqref{eq:limite spectre} for $k=k_-$ but the proof for $k=k_+$ is identical. If first  $  \limsup_{\ep \to 0} \lambda_{k_-}(\beta_\ep)  \le - \sigma$ then \eqref{eq:limite spectre} is obvious. We may thus assume that $ - \sigma < \liminf_{\ep \to 0} \lambda_{k_-}(\beta_\ep)  \le 0$. Then \cite[Theorem 4.1]{AmmannJammes} applies and shows that, as $\ep \to 0$, we have $\lambda_{k_-}(\beta_\ep) \to \lambda_{\beta,-} < 0$, and \eqref{eq:limite spectre} follows. 
\end{proof}

\subsection{Existence of almost optimisers of renormalised eigenvalue functionals}

Let $p \ge \frac{n}{2s}$ be fixed and $k \ge 1$ be an integer. In this subsection and in the next one we investigate the following variational problem: 
\begin{equation} \label{Lasouscritique}
\La_k^{s,p}(M,[g])=\left \{
\begin{aligned}
& \sup_{\be \in \mathcal{A}_p}  \bar{\lambda}_k^p(\beta) & \text{ if } k \le k_{-} \\
& \inf_{\be \in \mathcal{A}_p}  \bar{\lambda}_k^p(\beta) & \text{ if } k \ge k_{+}, 
\end{aligned}
\right. 
\end{equation}
where $k_{-}$ and $k_+$ are defined by \eqref{defkk} and where $\bar{\lambda}_k^p$ is as in \eqref{def:lambdabar}. In the special case where $p=\frac{n}{2s}$ we let 
$$\La_k^s(M,[g]) =\Lambda_k^{s,\frac{n}{2s}}(M,[g]),$$
where $\La_k^s(M,[g])$ is the variational problem introduced in \eqref{defLambdak}.  We prove in this subsection that almost-optimisers for $\La_k^{s,p}(M,[g])$ always exist, in a sense that we describe. We first state our results for almost-minimisers of $\bar{\lambda}_k^{p}$ when $k \ge k_+$:

\begin{prop} \label{prop:PS} Let $s \in \mathbb{N}^*$, $2s <n$, and let $P_g^s$ be the GJMS operator of order $2s$ in $(M,g)$. Assume that $P_g^s$ satisfies \eqref{eq:unique:continuation}.  Let $p \geq \frac{n}{2s}$ and $k\geq k_+$ be fixed. Define $p' = \frac{p}{p-1}$. For any $\epsilon >0$ there is $\beta_\epsilon \in \mathcal{A}_p$ with $1 \leq \Vert \beta_\epsilon \Vert_{L^p} \leq 1 + \epsilon$ such that
\begin{equation*} 
 \bar{\lambda}_k^p(\beta_{\epsilon}) \leq  \Lambda_k^{s,p}(M,[g]) + \epsilon^2 
 \end{equation*}
and there is an integer $k_+ \le \ell_\epsilon \leq k$, an orthonormal family $(v_{\ell_\epsilon}^\epsilon,\cdots,v_k^\epsilon) $ with respect to $Q(\beta_\epsilon,\cdot)$ of eigenfunctions in $E_k(\beta_\epsilon)$, positive numbers  $d_{\ell_\epsilon},\cdots, d_k$ such that $\sum_{i=\ell_\epsilon}^k d_i^\epsilon = 1$ and a sequence $(f_\epsilon)_{\epsilon}$ of functions in $ L^{p'}(M)$ such that  $ \Vert f_\epsilon \Vert_{L^{p'}} \leq \epsilon $ and 
\begin{equation*} 
\bar{\lambda}^p_k(\beta_\epsilon) \sum_{i=\ell_\epsilon}^k d_i^\epsilon \left(v_i^\epsilon\right)^2 \leq \bar{\lambda}^p_k(\beta_\epsilon)\beta_\epsilon^{p-1} + f_\epsilon  \quad \text{ a.e. in } M.
\end{equation*}
\end{prop}
The corresponding version for almost-maximisers when $k \le k_{-}$ is as follows: 

\begin{prop} \label{prop:PS2}
Let $s \in \mathbb{N}^*$, $2s <n$, and let $P_g^s$ be the GJMS operator of order $2s$ in $(M,g)$. Assume that $P_g^s$ satisfies \eqref{eq:unique:continuation}. Let $p \geq \frac{n}{2s}$ and $k\leq k_-$ be fixed. Define $p' = \frac{p}{p-1}$. For any $\epsilon >0$ there is $\beta_\epsilon \in \mathcal{A}_p$ with $1 \leq \Vert \beta_\epsilon \Vert_{L^p} \leq 1 + \epsilon$ such that
\begin{equation*} 
 \bar{\lambda}_k^p(\beta_{\epsilon}) \geq  \Lambda_k^{s,p}(M,[g]) - \epsilon^2 
 \end{equation*}
and there is an integer $k \le \ell_\epsilon \leq k_-$, an orthonormal family $(v_{k}^\epsilon,\cdots,v_{\ell_\epsilon}^\epsilon) $ with respect to $Q(\beta_\epsilon,\cdot)$ of eigenfunctions in $E_k(\beta_\epsilon)$, positive numbers  $d_k, \cdots, d_{\ell_\epsilon}$ such that $\sum_{i=k}^{\ell_\epsilon} d_i^\epsilon = 1$ and a sequence $(f_\epsilon)_{\epsilon}$ of functions in $ L^{p'}(M)$ such that  $ \Vert f_\epsilon \Vert_{L^{p'}} \leq \epsilon $ and 
\begin{equation*} 
 \bar{\lambda}^p_k(\beta_\epsilon) \sum_{i=k}^{\ell_\epsilon} d_i^\epsilon \left(v_i^\epsilon\right)^2 \geq \bar{\lambda}^p_k(\beta_\epsilon)\beta_\epsilon^{p-1} - f_\epsilon   \quad \text{ a.e. in } M.
\end{equation*}
\end{prop}

\begin{proof} We only prove Proposition \ref{prop:PS} for $k\geq k_+$, since the proof of Proposition \ref{prop:PS2} in the case $k \le k_{-}$ is similar. We also set, for simplicity, $\bar{\lambda}_k = \bar{\lambda}_k^p$ during the proof. Let $\tilde{\beta}_\epsilon \in \mathcal{A}_p$ be such that 
$$ \bar{\lambda}_k(\tilde{\beta}_{\epsilon}) \leq  \Lambda_k(M,[g]) + \epsilon^2 \text{ and } \Vert \tilde{\beta}_\epsilon \Vert_{L^p} = 1. $$
By Remark~\ref{rem:continuite:vp:pos}, $\bar{\lambda}_k$ is a continuous map on the complete metric space 
$$\mathcal{A}_p' = \{ \beta \in \mathcal{A}_p ; \Vert \beta \Vert_{L^p}\geq 1 \}$$
endowed with the distance associated to $\Vert \cdot \Vert_{L^p}$. 
By Ekeland's variational principle, for any $\epsilon >0$, there is $\beta_\epsilon \in \mathcal{A}_p'$ such that $ \bar{\lambda}_k(\beta_{\epsilon}) \leq \bar{\lambda}_k(\tilde{\beta}_{\epsilon})$, $ \Vert \beta_{\epsilon} - \tilde{\beta}_{\epsilon} \Vert \leq \epsilon$ and 
\begin{equation} \label{eq:ekeland} \forall \beta \in \mathcal{A}_p', \quad  \bar{\lambda}_k(\beta) - \bar{\lambda}_k(\beta_\epsilon) \geq - \epsilon\Vert \beta-\beta_\epsilon \Vert_{L^p} .\end{equation}
We immediately deduce that $1 \leq \Vert \beta_\epsilon \Vert_{L^p} \leq 1 +\epsilon$ and that $\bar{\lambda}_k(\beta_{\epsilon}) \leq  \Lambda_k(M,[g]) + \epsilon^2$.
For $b \in \mathcal{A}_p$ and $t \ge 0$ we choose $\beta = \beta_\epsilon + t b$ in  \eqref{eq:ekeland}.Letting $t \to 0$ we obtain
\begin{equation}\label{eq:ekelandderivative}\forall b \in \mathcal{A}_p, \quad \lim_{t\to 0} \frac{\bar{\lambda}_k(\beta_\epsilon+tb)- \bar{\lambda}_k(\beta_\varepsilon)}{t} \geq - \epsilon \Vert b \Vert_{L^p}.\end{equation}
We now let $V \in F \in \mathcal{G}_{k - i(k) + 1}(E_k(\beta_{\epsilon})$ be a linear subspace. We claim that 
\begin{equation} \label{eq:HBnonemptyekeland}
\begin{split}  \{ \zeta \in L^{p'}(M) , \zeta \geq 0 \text{ a.e } \} \cap  K \neq \emptyset
\end{split} \end{equation}
where we have let 
$$ K = 
co \left\{ \bar{\lambda}_k(\beta_{\epsilon}) \left( \beta_{\epsilon}^{p-1} - v^2 \right) + f ; v \in F, Q(\beta_{\epsilon},v) = 1, \Vert f \Vert_{L^{p'}} \leq \epsilon \right\}. $$
Assume by contradiction that the intersection in \eqref{eq:HBnonemptyekeland} is empty. Since $ \{ \zeta \in L^{p'}(M) , \zeta \geq 0 \text{ a.e } \}$ is closed and $K$ is weakly compact in $L^{p'}(M)$, the geometric Hahn-Banach theorem shows that there is $b \in L^p(M)$ such that
\begin{equation} \label{eq:HB1ekeland} \forall \zeta \in L^{p'}(M), \zeta \geq 0, \int_M \zeta b dv_g \geq 0 \end{equation}
and there is $\delta>0$ such that
\begin{equation} \label{eq:HB2ekeland} \forall \zeta \in  K , \int_M \zeta b dv_g \leq - \delta. \end{equation}
First, \eqref{eq:HB1ekeland} implies that $b \geq 0$ a.e, and then \eqref{eq:HB2ekeland} shows that $b \neq 0$, so that $b \in \mathcal{A}_p$. 
By Proposition \ref{prop:firstderivative} we have 
$$ \lim_{t\to 0} \frac{\bar{\lambda}_k(\beta_\epsilon+tb) - \bar{\lambda}_k(\beta_\epsilon)}{t} \leq \max_{v\in V} \left( \bar{\lambda}_k(\beta_\epsilon)\int_{M} b (\beta_\epsilon^{p-1} - v^2)dv_g \right), $$
so that there is $v \in V$ such that $Q(\beta_\epsilon,v) = 1$ and 
$$ \lim_{t\to 0} \frac{\bar{\lambda}_k(\beta_\epsilon+tb) - \bar{\lambda}_k(\beta_\epsilon)}{t} \leq \int_M \bar{\lambda}_k(\beta_\epsilon)(\beta_\epsilon^{p-1}-v^2)b dv_g.  $$
Using \eqref{eq:ekelandderivative}, we obtain that
$$ \bar{\lambda}_k(\beta_\epsilon) \int_M (\beta_\epsilon^{p-1}-v^2)b dv_g \geq - \epsilon \Vert b \Vert_{L^p}. $$
We let $ f = \ep b^{p-1} \Vert b \Vert_{L^p}^{1-p} $. Then $f\in L^{p'}(M)$, $\Vert f \Vert_{L^{p'}} = \ep $ and we get
$$ \int_M f b dv_g = \ep \int_M b^{p} \Vert b \Vert_{L^p}^{1-p} = \ep \Vert b \Vert_{L^p}, $$
so that
$$ \int_M \left( \bar{\lambda}_k(\beta_\epsilon)(\beta_\epsilon^{p-1}-v^2) + f \right) b dv_g \geq 0. $$
This contradicts \eqref{eq:HB2ekeland} for the function $\zeta = \bar{\lambda}_k(\beta_\epsilon)(\beta_\epsilon^{p-1}-v^2)+f$. Therefore \eqref{eq:HBnonemptyekeland} holds and  
Proposition~\ref{prop:PS} follows from lemma \ref{mixinglemma} since
\begin{equation*} 
\begin{split} & co  \left\{ \bar{\lambda}_k(\beta_{\epsilon}) \left( \beta_{\epsilon}^{p-1} - v^2 \right) + f ; v \in V, Q(\beta_{\epsilon},v) = 1, \Vert f \Vert_{L^{p'}} \leq \epsilon \right\} \\ & = \left\{ \bar{\lambda}_k(\beta) \left( \beta^{p-1} - \sum_{i=1}^\ell d_i v_i^2 \right) + f ; \sum_{i=1}^\ell d_i = 1 \text{ and } (v_i) \in O_{Q(\beta,\cdot)}(V), \Vert f \Vert_{L^{p'}} \leq \epsilon \right\}
\end{split} \end{equation*}
where $O_{Q(\beta,\cdot)}(V)$ is the set of orthonormal bases of $V$ with respect to the scalar product $Q(\beta,\cdot)$.
\end{proof}

\begin{rem}
Notice that Proposition \ref{prop:PS} and Proposition \ref{prop:PS2} rely on Ekeland's variational principle. One can refer to \cite{Petrides5, Petrides4} for a similar use of this principle in dimension 2. 
\end{rem}

\subsection{Existence of optimisers of the subcritical approximations}

In this subsection we consider problem \eqref{Lasouscritique} when $p > \frac{n}{2s}$. Under this assumption the problem becomes subcritical, since the higher integrability of $\beta$ allows us to recover the compactness of almost-optimising sequences of generalised eigenfunctions. We prove the following result:

\begin{prop}\label{subcritical_theorem}
Let $s \in \mathbb{N}^*$, $2s <n$, and let $P_g^s$ be the GJMS operator of order $2s$ in $(M,g)$. Assume that $P_g^s$ satisfies \eqref{eq:unique:continuation}. Let  $p>\frac{n}{2s}$. Then, for any $k \le k_{-}$ or $k \ge k_{+}$,  $\La_k^{s,p}(M,[g])$ defined by \eqref{Lasouscritique} is attained. In other words, there exists $\beta \in \mathcal{A}_p$ such that 
 \begin{equation*} 
\bar{\lambda}_k^p(\beta) = \La_k^{s,p}(M,[g]) =  \left \{
\begin{aligned}
& \sup_{\be \in \mathcal{A}_p}  \bar{\lambda}_k^p(\beta) & \text{ if } k \le k_{-} \\
& \inf_{\be \in \mathcal{A}_p}  \bar{\lambda}_k^p(\beta) & \text{ if } k \ge k_{+},
\end{aligned} \right. 
\end{equation*}
\end{prop}
where $\bar{\lambda}_k^p(\beta)$ is given by \eqref{def:lambdabar}.

\begin{proof}
We prove Proposition \ref{subcritical_theorem} in the case $k \ge k_+$ since the proof when $k \le k_{-}$ is identical. By proposition \ref{prop:PS} there exists, for any $\epsilon >0$, some $\beta_\epsilon \in \mathcal{A}_{p}$, with $\Vert \beta_\epsilon \Vert_{L^p} = 1 + O(\epsilon)$, such that  
\begin{equation} \label{eq:vp:souscrit:1}
\bar{ \lambda}_k^p(\beta_{\epsilon}) \leq  \Lambda_k^{s,p}(M,[g]) + \epsilon^2 
 \end{equation}
holds, and there is an integer $k_+ \leq \ell_\epsilon \leq k$, there is a family $(v_{\ell_\epsilon}^\epsilon, \cdots, v_{k}^\epsilon) $ in $E_k(\beta_\epsilon)$, orthonormal with respect to $Q(\beta_\epsilon,\cdot)$, there are positive numbers $d_{\ell_\epsilon},\cdots, d_k$ with $\sum_{i=k}^{\ell_\epsilon} d_i^\epsilon = 1$ and there exists a family of functions $f_\epsilon \in L^{p'}(M)$, with $p' = \frac{p}{p-1}$, such that $ \Vert f_\epsilon \Vert_{L^{p'}} \leq \epsilon $ and
\begin{equation} \label{eq:vp:souscrit:2} 
\bar{\lambda}^p_k(\beta_\epsilon) \sum_{i=\ell_\epsilon}^k d_i^\epsilon \left(v_i^\epsilon\right)^2 \leq \bar{\lambda}^p_k(\beta_\epsilon)\beta_\epsilon^{p-1} + f_\epsilon 
 \end{equation}
 a.e. in $M$. Up to passing to a subsequence we may assume that $\ell_\epsilon = \ell$ is constant. We let $i \in \{\ell, \dots, k\}$. The function $v_i^\epsilon$ satisfies $Q(\beta_\epsilon, v_i^{\epsilon}) = 1$ and 
 $$ P_g^s v_i^{\epsilon} = \lambda_k(\beta_\epsilon) \beta_{\epsilon} v_i^{\epsilon} \quad \text{ in } M .$$
 Lemma \ref{lem:mainlemma} then applies and shows, since $p > \frac{n}{2s}$, that $(v_i^{\epsilon})_{\epsilon >0}$ is bounded in $H^s(M)$. Up to a subsequence as $\ep \to 0$ we may thus assume that there is $v_i \in H^s(M)$ such that $v_i^{\epsilon}$ converges to $v_i $ weakly in $H^s(M)$ and strongly in $H^{s-1}(M)$. Let $\beta$ be the weak limit of $(\beta_i)_{i \ge0}$ in $L^{p}(M)$. Then $\Vert \beta \Vert_{L^p} \le 1$ and it is easily seen that $v_i $ satisfies 
 $$ P_g^s v_i  =\nu_k \beta v_i \quad \text{ in } M ,$$
 where we have let $\nu_k = \lim_{\epsilon \to 0} \lambda_k (\beta_{\epsilon})$.  Lemma \ref{lem:mainlemma} also shows that $\nu_k >0$ and by standard elliptic theory and since $p > \frac{n}{2s}$ (see again \cite[Theorem 5]{Mazumdar1}), we have $v_i \in L^\infty(M)$. We now claim that $v_i^{\epsilon}$ strongly converges to $v_i$ in $H^s(M)$, as $\ep \to 0$, up to taking a subsequence. Indeed, $v_i^{\epsilon} - v_i$ satisfies
  \begin{equation} \label{eq:vp:souscrit:3} \begin{aligned}
 P_g^s(v_i^{\epsilon} - v_i) & = \big(\lambda_k (\beta_\epsilon) -\nu_k \big) \beta_\epsilon v_i^{\epsilon} + \nu_k \beta_\epsilon (v_i^{\epsilon} - v_i) \\
 & + \nu_k  (\beta_\epsilon - \beta) v_i
 \end{aligned} 
 \end{equation}
 in $M$. Since $\beta_{\epsilon} - \beta$ weakly converges  to $0$ in $L^p(M)$ and $v_i^{\epsilon} - v_i$ strongly converges  to $0$ in $L^{p'}(M)$ by Sobolev's embedding, we have 
$$ \left| \int_M  (\beta_\epsilon - \beta) v_i (v_i^{\epsilon} - v_i) dv_g \right| \le \Vert v_i \Vert_{L^\infty} \int_M |\beta_\epsilon - \beta| |v_i^{\epsilon} - v_i| dv_g = o(1) $$
as $\epsilon \to 0$. Independently, and since $p > \frac{n}{2s}$, we have $p' = \frac{p}{p-1} < \frac{n}{n-2s}$.  Sobolev's embeddings then show that the sequence  $\big( (v_i^{\ep} - v_0)^2 \big)_{\ep >0}$ strongly converges to $0$ in $L^{p'}(M)$. Since $(\beta_\ep)_{\ep}$ is bounded in $L^{p}(M)$ we then have
$$ \int_M  \nu_k \beta_\epsilon (v_i^{\epsilon} - v_i)^2 dv_g = o(1)$$
as $\ep \to 0$. Integrating \eqref{eq:vp:souscrit:3} against $v_i^{\epsilon} - v_i$ thus gives, by \eqref{eq:def:Ag} and since $\lambda_k (\beta_\epsilon) \to \nu_k$, 
$$ \int_M  \big|\Delta_g^{\frac{s}{2}} (v_i^{\epsilon} - v_i) \big|^2 dv_g  = o(1) $$
as $\epsilon \to 0$, which shows that $v_i^{\ep} \to v_i$ in $H^s(M)$ as $\ep \to 0$. 

\medskip

Since $v_i^{\ep} \to v_i$ in $H^s(M)$, and since $(v_{\ell}^\epsilon, \cdots, v_{k}^\epsilon)$ is $Q(\beta_\ep, \cdot)$-orthonormal, we obtain at the limit that $(v_\ell, \cdots, v_k)$ is $Q(\beta, \cdot)$-orthonormal.  We may now pass equation \eqref{eq:vp:souscrit:2} to the weak limit as $\epsilon \to 0$. Up to taking a subsequence we may assume that $d_i^{\ep} \to d_i$ for all $\ell \le i \le k$ with  $\sum_{i=k}^\ell d_i = 1$. 
Since $(v_i^{\epsilon})^2$ strongly converges to $v_i^2$ in $L^{p'}(M)$ as 
$\ep \to 0$ we have that $\left(\sum_{i=\ell}^{k} d_i^\ep \left(v_i^\ep\right)^2\right)^{\frac{1}{p-1}}$
 strongly converges to $\left(\sum_{i=\ell}^{k} d_i v_i ^2\right)^{\frac{1}{p-1}}$ in $L^p(M)$. 
Since pointwise inequalities are preserved under weak convergence,  and since $\nu_k >0$ we obtain that
\begin{equation} \label{eq:vp:souscrit:4} 
  \sum_{i=\ell}^{k} d_i v_i ^2 \le  \beta^{p-1} \quad \text{ a.e. in } M.
  \end{equation}
Multiplying the latter inequality by $\beta$ and integrating over $M$ then shows that 
$$ 1 = \int_M   \sum_{i=\ell}^{k} d_i \beta v_i ^2 dv_g \le \int_M \beta^{p}dv_g  \le 1. $$
All these inequalities are thus equalities, so that $\Vert \beta \Vert_{L^{p}} = 1 = \lim_{\epsilon \to 0} \Vert \beta_\epsilon \Vert_{L^{p}}$ and therefore, since $p>1$ and $L^{p}(M)$ is uniformly convex, $\beta_\epsilon$ strongly converges to $\beta$ in $L^{p}(M)$ as $\epsilon \to 0$. 
By lower semi-continuity of positive eigenvalues in $\mathcal{A}_p$ (Proposition \ref{prop:preliminaryvarset}), we obtain $\nu_k \leq \lambda_k(\beta)$. Passing \eqref{eq:vp:souscrit:1} to the limit as $\ep \to 0$ finally shows that $ \Lambda_k^{s,p}(M,[g]) \le \nu_k$, and since $\Vert \beta \Vert_{L^p} = 1$ we thus obtain that 
$$ \bar{\lambda}_k^p(\beta) \le \Lambda_k^{s,p}(M,[g]) $$
which concludes the proof of Proposition \ref{subcritical_theorem}.
\end{proof}

We conclude this subsection by a simple result that will be used in the proof of Theorem \ref{theo:vp:positives} in Section \ref{proof_theorems}:

\begin{prop} \label{CVLambdap}
Assume that $k \le k_-$ or $k \ge k_+$ and let $s \in \mathbb{N}^*$, $2s < n$. As $p \to \frac{n}{2s}$ from above we have 
$$ \La_k^{s,p}(M,[g])  \to  \La_k^{s}(M,[g]).$$
\end{prop}

\begin{proof}
Let $p > \frac{n}{2s}$ and $\beta \in L^p(M)$. By H\"older's inequality we have $\Vert \beta \Vert_{L^{\frac{n}{2s}}} \le \Vert \beta \Vert_{L^p} \text{Vol}_g(M)^{\frac{2sp-n}{np}}$. Since $\text{Vol}_g(M)^{\frac{2sp-n}{np}} = 1 +O\big(\big| p - \frac{n}{2s}\big|\big)$ as $p \to \frac{n}{2s}$ this implies 
$$\begin{aligned}
&  \La_k^{s}(M,[g]) \le  \La_k^{s,p}(M,[g])\Big[ 1 + O\big(\big| p - \frac{n}{2s}\big|\big)\Big] & \text{ if } k \ge k_{+} \\
&  \La_k^{s}(M,[g]) \ge  \La_k^{s,p}(M,[g])\Big[1 + O\big(\big| p - \frac{n}{2s}\big|\big)\Big] & \text{ if } k \le k_{-},
\end{aligned} $$
from which we deduce that 
$$\begin{aligned}
&  \La_k^{s}(M,[g]) \le  \liminf_{p \to \frac{n}{2s}} \La_k^{s,p}(M,[g]) & \text{ if } k \ge k_{+} \\
&  \La_k^{s}(M,[g]) \ge  \limsup_{p \to \frac{n}{2s}} \La_k^{s,p}(M,[g])& \text{ if } k \le k_{-}.
\end{aligned} $$
For the other inequality we only treat the $k \ge k_+$ case, since the case $k \le k_-$ is similar. Let $\ep >0$ and let $\beta$ be a positive and continuous function in $M$ such that $\lambda_k(\beta) \Vert \beta \Vert_{L^\frac{n}{2s}} \le   \La_k^{s}(M,[g]) + \ep$. By definition, $\La_k^{s,p}(M,[g]) \le \lambda_k(\beta)\Vert \beta \Vert_{L^p}$. Since $\beta >0$ in $M$ we have $\Vert \beta \Vert_{L^p} \to \Vert \beta \Vert_{L^{\frac{n}{2s}}}$ as $p \to \frac{n}{2s}$, and thus 
$$   \limsup_{p \to \frac{n}{2s}}  \La_k^{s,p}(M,[g]) \le \lambda_k(\beta) \Vert \beta \Vert_{L^{\frac{n}{2s}}} \le  \La_k^{s}(M,[g]) + \ep. $$
Letting $\ep \to 0$ proves that $   \limsup_{p \to \frac{n}{2s}}  \La_k^{s,p}(M,[g]) \le  \La_k^{s}(M,[g])$. 
\end{proof}

\subsection{The first generalised eigenvalue of a coercive operator}

We conclude this section by investigating the first generalised eigenvalue when the operator  $P_g^s$ is coercive in $H^s(M)$. We recall that $P_g^s$ is coercive if $v \in H^s(M) \mapsto \int_M v P_g^s v \, dv_g$ is equivalent to the $H^s(M)$ norm or, equivalently, if there exists $C >0$ such that 
\begin{equation} \label{def:coercivite}
\Vert u \Vert_{H^s}^2 \le C \int_M v P_g^s v \, dv_g  \quad \text{ for any } u \in H^s(M).
\end{equation} 
This is the case if $(M,g)$ is the round sphere $(\mathbb{S}^n, g_0)$ and, more generally, if $(M,g)$ is an Einstein manifold $(M,g)$ of positive scalar curvature: this follows for instance from  \eqref{factor:Pg}. We prove that in this case the invariant $\Lambda_1^s(M,[g])$ defined in \eqref{defLambdak} coincides with the Rayleigh quotient of $P_g^s$:

\begin{prop} \label{prop:Lambda1:Yamabe}
Let $s \in \mathbb{N} \backslash \{0\}$ and $(M,g)$ be a closed Riemannian manifold of dimension $n > 2s$. We let $P_g^s$ be the GJMS operator of order $2s$, and we assume that $P_g^s$ is coercive. Then  
\begin{equation} \label{yamabelambda1}
\Lambda_1^s(M,[g]) = \inf_{u \in C^\infty(M) \backslash \{0\}} \frac{\int_{M} u P_{g}^s u dv_{g}}{\left(\int_{M} |u|^{\frac{2n}{n-2s}} dv_{g} \right)^{\frac{n-2s}{n}}}.
\end{equation}
\end{prop}
Remark that if $P_g^s$ is coercive we have $\ker(P_g^s) = \{0\}$, so that \eqref{eq:unique:continuation} is automatically satisfied. The coercivity assumption is equivalent to requiring that  $\lambda_1(P_g^s) >0$. 

\begin{proof}
We let $I(M,[g])$ denote the right-hand side in \eqref{yamabelambda1}. Let $\beta \in \mathcal{A}_{\frac{n}{2s}}$, with $\Vert \beta \Vert_{L^{\frac{n}{2s}}} = 1$. By \eqref{eq:def:lambdak} it is easy to see that 
\begin{equation} \label{yamabelambda12}
 \lambda_1(\beta) = \inf_{ v \in H^s(M), \beta^{\frac12}v \not \equiv 0} \frac{\int_M v P_g^sv dv_g}{\int_M \beta v^2 dv_g}.  
 \end{equation}
If $v$ satisfies $ \beta^{\frac12}v \not \equiv 0$, H\"older's inequality gives $ \int_{M} \beta v^2dv_g \le \Big(\int_{M} |v|^{\frac{2n}{n-2s}}dv_g \Big)^{\frac{n-2s}{n}}$, so that $ \lambda_1(\beta) \ge I(M,[g])$. Taking the infimum over $\beta \in \mathcal{A}_{\frac{n}{2s}}$ with $\Vert \beta \Vert_{L^{\frac{n}{2s}}} = 1$ yields $\Lambda_1^s(M,[g]) \ge I(M,[g])$. For the reverse inequality, we let $v \in C^\infty(M)$ with $v >0$, and let $\beta = v^{\frac{4s}{n-2s}}$. By \eqref{yamabelambda12} we have 
$$\begin{aligned}
\lambda_1(\beta) \Vert \beta \Vert_{L^{\frac{n}{2s}}}  \le \frac{\int_M v P_g^sv dv_g}{\int_M v^{\frac{2n}{n-2s}} dv_g} \left( \int_M v^{\frac{2n}{n-2s}} dv_g \right)^{\frac{2s}{n}}    =  \frac{\int_{M} v P_{g}^s v dv_{g}}{\left(\int_{M} v^{\frac{2n}{n-2s}} dv_{g} \right)^{\frac{n-2s}{n}}}.
\end{aligned} $$ 
Taking the infimum over all $v >0$ and then over all $\beta \in \mathcal{A}_{\frac{n}{2s}}$ yields $\Lambda_1^s(M,[g]) \le I(M,[g])$.
\end{proof}

We recall the definition of the best constant for Sobolev's inequality in $\R^n$ given by \eqref{defKns}:
\begin{equation*}
 \frac{1}{K_{n,s}^2} = \inf_{u \in C^\infty_c(\R^n) \backslash \{0\}} \frac{\int_{\R^n} \left| \Delta_\xi^{\frac{s}{2}} u \right|^2 dv_{\xi}}{\left(\int_{\R^n} |u|^{\frac{2n}{n-2s}} dv_{\xi} \right)^{\frac{n-2s}{n}}}.
 \end{equation*}
We recall (see \cite{LiebSharpConstants}) that $K_{n,s}^{-2}$ is attained by a $n+2$-dimensional family of smooth functions which, up to translation, dilation and multiplication by a nonzero constant, coincide with  
$$ U_0(x) = \left( 1+ |x|^2 \right)^{- \frac{n-2s}{2}}  \quad \text{ for } x \in \R^n. $$
Remark that for any function $u \in C^\infty_c(\R^n)$, $\int_{\R^n} \left| \Delta_\xi^{\frac{s}{2}} u \right|^2 dx = \int_{\R^n} u (\Delta_\xi)^{s}u dv_{\xi} = \int_{\R^n} u P_{\xi}^s u dv_\xi$.  If $g_0$ denotes the round metric on $\mathbb{S}^n$, using the stereographic projection and the conformal invariance property of $P_{g_0}^s$ shows that we also have 
$$ \frac{1}{K_{n,s}^2} = \inf_{u \in C^\infty(\mathbb{S}^n) \backslash \{0\}} \frac{\int_{\mathbb{S}^n} u P_{g_0}^s u dv_{g_0}}{\left(\int_{\mathbb{S}^n} |u|^{\frac{2n}{n-2s}} dv_{g_0} \right)^{\frac{n-2s}{n}}}. $$
Proposition \ref{prop:Lambda1:Yamabe} then shows that
\begin{equation} \label{eq:Lambda1:sphere}
\Lambda_1^s(\mathbb{S}^n, [g_0]) = K_{n,s}^{-2}
\end{equation}
and we again obtain with the result of \cite{LiebSharpConstants} that $\Lambda_1^s(\mathbb{S}^n, [g_0]) $ is attained for every $s \in \mathbb{N} \backslash \{0\}$. Together with \eqref{eq:Lambda1:sphere}, assumption \eqref{conc_points} in the statement of Lemma \ref{convergence_appendix} can be reinterpreted in the following way: if $(\be_i)_{i \ge0}$ is a sequence of nonnegative functions that is uniformly bounded in $L^{\frac{n}{2s}}(M)$ its set of {\em set of concentration points} is defined as 
\begin{equation} \label{conc_points_2}
A=  \Biggl\{ x \in M \,\Big|\, \forall \de>0, \;  \limsup_{i \to + \infty} \int_{B_g(x,\de)}
  \be_i^{\frac{n}{2s}}  \,dv_g > \frac12 \Lambda_1^s(\mathbb{S}^n)^{\frac{n}{2s}} \Biggr\}. 
  \end{equation}
As we will see in the proof of Theorem \ref{theo:vp:positives} in Section \ref{proof_theorems}, condition \eqref{conc_points_2} ensures that the invariants $\Lambda_k^s(M,[g])$, when they are not attained, satisfy rigid quantification properties.

\section{Maximisation of negative eigenvalues: proof of Theorem \ref{theo:vp:negatives}} \label{sec:vpnegatives}

In this section we show that negative eigenvalues can be always be maximised in a fixed conformal class. Let $s \in \mathbb{N} \backslash \{0\}$ and $(M,g)$ be a closed Riemannian manifold of dimension $n > 2s$. We let $P_g^s$ be the GJMS operator of order $2s$, and we assume that it satisfies \eqref{eq:unique:continuation}. Let $k \le k_{-}$ be an integer. 
The goal of this section is to prove Theorem \ref{theo:vp:negatives} i.e. to prove that  $\Lambda_k^s(M,[g])<0$ and $\Lambda_k(M,[g])$ is attained at a generalised metric $\tilde g$. We recall that $\Lambda_k^s(M,[g])$ is defined in \eqref{defLambdak}. 

\begin{proof}[Proof of Theorem \ref{theo:vp:negatives}]
The proof follows the same lines than the proof of Proposition \ref{subcritical_theorem}, with a few complications due to the criticality of the exponent $p = \frac{n}{2s}$. We will prove that the compactness of sequences of generalised eigenfunctions in $H^s(M)$ follows this time from the negativity of the eigenvalues. By Proposition~\ref{prop:preliminaryvarset} we have 
$$ \Lambda_k^s(M,[g]) = \sup_{\beta \in \mathcal{A}_{\frac{n}{2s}}} \lambda_k(\beta) \Vert \beta \Vert_{L^\frac{n}{2s}}. $$
Since $k \le k_{-}$ we have $\lambda_k(\beta) <0$ for any $\beta \in \mathcal{A}_{\frac{n}{2s}}$.  We first prove that $\Lambda_k^s(M,[g])<0$. Let for this $(\beta_i)_{i \ge0}$, $\beta_i \in \mathcal{A}_{\frac{n}{2s}}$, $\Vert \beta_i \Vert_{L^{\frac{n}{2s}}}=1$ be a sequence satisfying $ \lambda_k(\beta_i) \to \Lambda_k(M,[g])$. By Proposition \ref{prop:deficontinuityeigen} we can let $(\vp_i)_{i\ge0}$ be an associated sequence of generalised eigenvectors satisfying $Q(\beta_i, \vp_i) = 1$ and 
$$ P_g^s \vp_i = \lambda_k(\beta_i) \beta_i \vp_i \quad \text{ in } M. $$
Since $\lambda_k(\beta_i) <0$, Lemma \ref{lem:mainlemma} applies and shows that 
$$\Lambda_k^s(M,[g]) = \limsup_{i \to + \infty} \lambda_k(\beta_i) <0.$$ 
We now prove that $\Lambda_k^s(M,[g])$ is attained. We apply for this Proposition \ref{prop:PS}: for any $\epsilon >0$, it shows that there exists $\beta_\epsilon \in \mathcal{A}_{\frac{n}{2s}}$, $\Vert \beta_\epsilon \Vert_{L^\frac{n}{2s}} = 1 + O(\epsilon)$, such that  
\begin{equation} \label{eq:vp:neg:1}
 \lambda_k(\beta_{\epsilon}) \geq  \Lambda_k^s(M,[g]) - \epsilon^2 
 \end{equation}
holds, and there is an integer $k \le \ell_\epsilon \leq k_-$, there is a family $(v_{k}^\epsilon,\cdots,v_{\ell_\epsilon}^\epsilon) $ in $E_k(\beta_\epsilon)$, orthonormal with respect to $Q(\beta_\epsilon,\cdot)$, there exist positive numbers $d_{\ell_\epsilon},\cdots, d_k$ with $\sum_{i=k}^{\ell_\epsilon} d_i^\epsilon = 1$ and there exists a family of functions $f_\epsilon \in L^{\frac{n}{n-2s}}(M)$ such that $ \Vert f_\epsilon \Vert_{L^{\frac{n}{n-2s}}} \leq \epsilon $ and
\begin{equation} \label{eq:vp:neg:2} 
 \lambda_k(\beta_\epsilon) \sum_{i=k}^{\ell_\epsilon} d_i^\epsilon \left(v_i^\epsilon\right)^2 \geq \lambda_k(\beta_\epsilon)\beta_\epsilon^{\frac{n-2s}{2s}} - f_\epsilon 
 \end{equation}
 a.e. in $M$. Up to passing to a subsequence we may assume that $\ell_\epsilon = \ell$ is constant. We let $i \in \{k, \dots, \ell\}$. The function $v_i^\epsilon$ satisfies $Q(\beta_\epsilon, v_i^{\epsilon}) = 1$ and 
 $$ P_g^s v_i^{\epsilon} = \lambda_k(\beta_\epsilon) \beta_{\epsilon} v_i^{\epsilon} \quad \text{ in } M .$$
 Lemma \ref{lem:mainlemma} then applies and shows, since $\lambda_k(\beta_\epsilon) <0$, that $(v_i^{\epsilon})_{\epsilon >0}$ is bounded in $H^s(M)$. Up to a subsequence we may thus assume that there is $v_i \in H^s(M)$ such that $v_i^{\epsilon}$ converges to $v_i $ weakly in $H^s(M)$ and strongly in $H^{s-1}(M)$. Let $\beta$ be the weak limit of $(\beta_i)_{i \ge0}$ in $L^{\frac{n}{2s}}(M)$. Then $\Vert \beta \Vert_{L^{\frac{n}{2s}}} \le 1$ and it is easily seen that $v_i $ satisfies 
 \begin{equation} \label{eq:vp:neg:22}
  P_g^s v_i  =\nu_k \beta v_i \quad \text{ in } M ,
  \end{equation}
 where we have let $\nu_k = \lim_{\epsilon \to 0} \lambda_k(\beta_{\epsilon})$. By Lemma \ref{lem:mainlemma} we again have $\nu_k <0$. By \eqref{eq:vp:neg:22}, classical regularity results for higher-order elliptic equations (see again Mazumdar \cite[Theorem 5]{Mazumdar1}) show that $v_i \in L^q(M)$ for any $1 \le q < + \infty$.  
 
 \medskip
 
 We first prove that $v_i^{\epsilon}$ strongly converges to $v_i$ in $H^s(M)$ as $\epsilon \to 0$. The function $v_i^{\epsilon} - v_i$ satisfies
  \begin{equation} \label{eq:vp:neg:3} \begin{aligned}
 P_g^s(v_i^{\epsilon} - v_i) & = \big(\lambda_k(\beta_\epsilon) -\nu_k \big) \beta_\epsilon v_i^{\epsilon} + \nu_k \beta_\epsilon (v_i^{\epsilon} - v_i) \\
 & + \nu_k  (\beta_\epsilon - \beta) v_i
 \end{aligned} 
 \end{equation}
 in $M$. First, since $\lambda_k(\beta_\epsilon) \to \nu_k$ as $\epsilon \to 0$, H\"older's inequality shows that 
 $$ \int_M \big(\lambda_k(\beta_\epsilon) -\nu_k \big) \beta_\epsilon v_i^{\epsilon}(v_i^{\epsilon} - v_i) dv_g = o(1) $$
 as $\ep \to 0$. Let now $R>0$ be fixed. We write 
 $$\begin{aligned} 
  \int_M  (\beta_\epsilon - \beta) v_i (v_i^{\epsilon} - v_i) dv_g & = \int_{\{ |v_i| \le R \}}  (\beta_\epsilon - \beta) v_i (v_i^{\epsilon} - v_i) dv_g \\
  & +  \int_{\{ |v_i| \ge R \}}  (\beta_\epsilon - \beta) v_i (v_i^{\epsilon} - v_i) dv_g.
  \end{aligned} $$ 
  On the one hand, since $\beta_{\epsilon} - \beta$ converges weakly to $0$ in $L^{\frac{n}{2s}}(M)$ and $v_i^{\epsilon} - v_i$ converges strongly to $0$ in $L^{\frac{n}{n-2s}}(M)$ by Sobolev's embedding, we have 
$$ \left| \int_{\{ |v_i| \le R \}}   (\beta_\epsilon - \beta) v_i (v_i^{\epsilon} - v_i) dv_g \right| \le R \int_M |\beta_\epsilon - \beta| |v_i^{\epsilon} - v_i| dv_g = o(1) $$
as $\epsilon \to 0$. On the other hand, since $(v_i^{\ep})_{\ep >0}$ is bounded in $H^s(M)$ and $(\beta_\ep)_{\ep >0}$ is bounded in $L^{\frac{n}{2s}}(M)$, H\"older's inequality shows that 
$$ \left|  \int_{\{ |v_i| \ge R \}}  (\beta_\epsilon - \beta) v_i (v_i^{\epsilon} - v_i) dv_g \right| \le C \left( \int_{\{ |v_i| \ge R \}} |v_i|^{\frac{2n}{n-2s}} dv_g\right)^{\frac{n-2s}{2n}} .$$
Combining these two estimates and letting $R \to + \infty$ thus gives 
$$ \limsup_{\ep \to + \infty} \left|   \int_M  (\beta_\epsilon - \beta) v_i (v_i^{\epsilon} - v_i) dv_g \right|  =0. $$
With the previous arguments, integrating \eqref{eq:vp:neg:3} against $v_i^{\epsilon} - v_i$ we thus obtain, by \eqref{eq:def:Ag}, 
$$ \int_M  \big|\Delta_g^{\frac{s}{2}} (v_i^{\epsilon} - v_i) \big|^2 dv_g - \nu_k \int_M \beta_\epsilon (v_i^{\epsilon} - v_i)^2dv_g = o(1) $$
as $\epsilon \to 0$. Since $\nu_k < 0$ and $\beta_\epsilon \ge 0$ this shows that  $v_i^{\epsilon}$ strongly converges to $v_i$ in $H^s(M)$. The relations $B(\beta_\epsilon, v_i^{\epsilon}, v_j^{\epsilon}) = \delta_{i,j}$, for any $k \le i,j \le \ell$, now pass to the limit as $\epsilon \to 0$ and show that $B(\beta, v_i, v_j) = \delta_{i,j}$, which shows in particular that $\beta \neq 0$. 

\medskip

We may now pass equation \eqref{eq:vp:neg:2} to the weak limit as $\epsilon \to 0$. Up to taking a subsequence, we let $d_i = \lim_{\epsilon \to 0} d_i^{\epsilon}$, so that we again have $\sum_{i=k}^\ell d_i = 1$.  Since $(v_i^{\epsilon})^2$ strongly converges to $v_i^2$ in $L^{\frac{n}{n-2s}}(M)$,
 then $\left(\sum_{i=k}^{\ell} d_i^\ep \left(v_i^\ep\right) ^2 \right)^{\frac{2s}{n-2s}}$ converges strongly to  $\left(\sum_{i=k}^{\ell} d_i v_i^2\right)^{\frac{2s}{n-2s}}$ in $L^{\frac{n}{2s}}(M)$. Then, since pointwise inequalities are preserved under weak convergence, and since $\nu_k < 0$, we obtain that
$$ \beta^{\frac{n-2s}{2s}} \ge  \sum_{i=k}^{\ell} d_i v_i ^2  \quad \text{ a.e. in } M. $$
Multiplying the latter inequality by $\beta$ and integrating over $M$ then shows that 
$$ 1 = \int_M   \sum_{i=k}^{\ell} d_i \beta v_i ^2 dv_g \le \int_M \beta^{\frac{n}{2s}}dv_g  \le 1. $$
All these inequalities are thus equalities, so that $\Vert \beta \Vert_{L^{\frac{n}{2s}}} = 1 = \lim_{\epsilon \to 0} \Vert \beta_\epsilon \Vert_{L^{\frac{n}{2s}}}$ and therefore, since $n > 2s$ and $L^{\frac{n}{2s}}(M)$ is uniformly convex, $\beta_\epsilon$ strongly converges to $\beta$ in $L^{\frac{n}{2s}}(M)$ as $\epsilon \to 0$. 
We can thus apply Proposition \ref{prop:uppersemicontinuity}, and we get that $\nu_k \le \lambda_k(\beta)$. Using \eqref{eq:vp:neg:1} this finally gives
$$ \lambda_k(\beta) \ge \Lambda_k^s(M,[g]) .$$
By definition of $\Lambda_k^s(M,[g])$ and since $\Vert \beta \Vert_{\frac{n}{2s}} = 1$ we thus have $\lambda_k(\beta) = \Lambda_k^s(M,[g])$, which shows that $\Lambda_k^s(M,[g])$ is attained at $\beta$. This concludes the proof of Theorem \ref{theo:vp:negatives}. 
\end{proof}

\section{Proof of Theorem \ref{theo:vp:positives}} \label{proof_theorems}

We prove in this section the main result of this paper, Theorem \ref{theo:vp:positives}. Unlike in the case of negative eigenvalues, the situation for positive ones is much more complicated since bubbling phenomena may occur. We prove Theorem \ref{theo:vp:positives} in several steps. As a first result, we prove that the following large inequality always holds:
\begin{prop} \label{prop_ineg_large}
Let $s \in \mathbb{N} \backslash \{0\}$ and $(M,g)$ be a closed Riemannian manifold of dimension $n > 2s$. We let $P_g^s$ be the GJMS operator of order $2s$. Let $k \geq k_+$ be an integer, where $k_+$ is defined in \eqref{defkk}. Then 
 $$\La_k^s(M,[g]) \leq  X_k^s(M,[g]),$$ 
where $X_k^s(M,[g])$ is as in \eqref{definition_X}.
\end{prop}
We then prove that equality holds if $\La_k^s(M,[g]) $ is \emph{not} attained. Following Proposition~\ref{prop:preliminaryvarset} we will say that $\Lambda_k^s(M,[g])$ is not attained if for every $\beta \in \mathcal{A}_{\frac{n}{2s}}$ we have $\La_k^s(M,[g]) < \lambda_k(\beta) \Vert \beta \Vert_{L^{\frac{n}{2s}}}$. We recall that if $g_0$ is the round metric in $\mathbb{S}^n$ we write $\Lambda_k^s(\mathbb{S}^n, [g_0]) = \Lambda_k^s(\mathbb{S}^n)$.  The main result that we prove in this section is the following:
\begin{theo} \label{prop:bubbling}
Let $s \in \mathbb{N} \backslash \{0\}$ and $(M,g)$ be a closed Riemannian manifold of dimension $n > 2s$. We let $P_g^s$ be the GJMS operator of order $2s$ and we assume that $P_g^s$ satisfies \eqref{eq:unique:continuation}. Let $k \geq k_+$ be an integer, where $k_+$ is defined in \eqref{defkk}. Assume that $\La_k^s(M,[g])$ is not attained. Then:
\begin{itemize}
\item Either there exist $\ell_0 \in \{k_+, \cdots, k-1\}$ and $\ell_1 \in \{1, \cdots, k-1\}$ with $\ell_0 + \ell_1 = k$ such that $\Lambda_{\ell_0}^s(M,[g])$ is attained and
$$ \La_k^s(M,[g])^{\frac{n}{2s}} \ge  \La_{\ell_0}^s(M,[g])^{\frac{n}{2s}}  + \La_{\ell_1}^s(\mathbb{S}^n)^{\frac{n}{2s}} $$
\item or 
$$  \La_k^s(M,[g]) \ge  \La_{k-{k_+}+1}^s(\mathbb{S}^n). $$
\end{itemize}
Furthermore, in the case where $(M,g)$ is the round sphere $(\mathbb{S}^n, g_0)$, only the first alternative occurs. 
\end{theo}
Note that the first case can only occur if $k \ge k_++1$. Theorem \ref{prop:bubbling} implies Theorem \ref{theo:vp:positives} as we now show:

\begin{proof}[Proof of Theorem \ref{theo:vp:positives} assuming Theorem \ref{prop:bubbling}]
First, the inequality $\Lambda_k^s(M,[g]) \le X_k^s(M,[g])$ follows from Proposition \ref{prop_ineg_large}. We now assume that $\Lambda_k^s(M,[g]) < X_k^s(M,[g])$. We proceed by contradiction and assume that $\La_k^s(M,[g])$ is \emph{not} attained at a generalised metric. In this case Theorem \ref{prop:bubbling} applies.

Assume first that $\La_k^s(M,[g])^{\frac{n}{2s}} \ge  \La_{\ell_0}^s(M,[g])^{\frac{n}{2s}}  + \La_{\ell_1}^s(\mathbb{S}^n)^{\frac{n}{2s}}$ for some indexes $\ell_0 \in \{k_+, \cdots, k-1\}$ and $\ell_1 \in \{1, \cdots, k-1\}$ such that $\ell_0 + \ell_1 = k$ and such that $\Lambda_{\ell_0}^s(M,[g])$ is attained. If  $\La_{\ell_1}^s(\mathbb{S}^n)$ is also attained then, by definition of $X_k^s(M,[g])$ in \eqref{definition_X}, we have $\La_k^s(M,[g]) \ge X_k^s(M,[g])$, which is a contradiction. Thus $\La_{\ell_1}^s(\mathbb{S}^n)$ is not attained and we may then apply again Theorem \ref{prop:bubbling} to the $\ell_1$-th generalised eigenvalue for the round sphere $(\mathbb{S}^n,g_0)$. In this case only the first alternative occurs, and since $k_+=1$ for $(\mathbb{S}^n, g_0)$, we obtain that
$$ \La_{\ell_1}^s(\mathbb{S}^n)^{\frac{n}{2s}} \ge  \La_{\ell_2}^s(\mathbb{S}^n)^{\frac{n}{2s}}  + \La_{\ell_3}^s(\mathbb{S}^n)^{\frac{n}{2s}} $$
for some $\ell_2 \ge 1$ and $1 \le \ell_3 \le \ell_1-1$, where $ \La_{\ell_2}^s(\mathbb{S}^n)$ is attained. If $\La_{\ell_3}^s(\mathbb{S}^n)$ is attained then we have again shown that $\La_k^s(M,[g]) \ge X_k^s(M,[g])$ which leads to a contradiction. Hence $\La_{\ell_3}^s(\mathbb{S}^n)$ is not attained. We may thus repeat this argument by induction: after a finite number of iterations, we reach a contradiction since $\Lambda_1^s(\mS^n)$ is attained (this follows from \eqref{eq:Lambda1:sphere} and the result in \cite{LiebSharpConstants}). 

Assume then that $\La_k^s(M,[g]) \ge  \La_{k-{k_+}+1}^s(\mathbb{S}^n)$. If $\La_{k-{k_+}+1}^s(\mathbb{S}^n)$ is attained then we again have $\La_k^s(M,[g]) \ge X_k^s(M,[g])$ by definition of $X_k^s(M,[g])$, which is a contradiction. Otherwise the proof follows from an inductive application of Theorem \ref{prop:bubbling} to the conformal eigenvalues of $(\mathbb{S}^n, [g_0])$ as in the previous case. This concludes the proof of \ref{theo:vp:positives}. 
\end{proof}

\noindent We prove Proposition \ref{prop_ineg_large} and Theorem \ref{prop:bubbling} in the next subsections.

\subsection{Proof of Proposition \ref{prop_ineg_large}.}

We prove a slightly more general result. We define 
\begin{equation} \label{defYks}
Y_k^s(M,[g]) = \inf \left\{ \left( \La_{\ell_0}^s(M,[g])^{\frac{n}{2s}} + \La_{\ell_1}^s(\mS^n)^{\frac{n}{2s}} +  \cdots + \La_{\ell_r}^s(\mS^n)^{\frac{n}{2s}} \right)^{\frac{2s}{n}} \right\}
\end{equation}
where the infimum is taken over the set of  all $r,\ell_0,\cdots,\ell_r \in \mN$ such that 
\begin{enumerate}
\item  $\ell_0 \in \{ 0 \} \cup \{ k_+ ,\cdots, k - 1 \}$ where by convention $\Lambda_0^s(M,[g])= 0$,
 \item $\ell_0  + \cdots + \ell_r =k$ if $\ell_0 \geq k_+ $ and $\ell_0+ \cdots + \ell_r= k-k_+ +1$ if $\ell_0 =0$.
 \end{enumerate}
In other words, $Y_k^s(M,[g])$ is a rougher version than $X_k^s(M,[g])$ defined in  \eqref{definition_X}: it is defined similarly to $X_k^s(M,[g])$ but without assuming that $ \La_{\ell_0}^s(M,[g])$ or the $\La_{\ell_i}^s(\mS^n)$ are attained. As a consequence, we easily have
\begin{equation} \label{YleqX}
 Y_k^s(M,[g]) \leq X_k^s(M,[g]). 
\end{equation}
In this subsection we prove the following result:

\begin{prop} \label{prop_ineg_large2}
Let $s \in \mathbb{N} \backslash \{0\}$ and $(M,g)$ be a closed Riemannian manifold of dimension $n > 2s$. We let $P_g^s$ be the GJMS operator of order $2s$. Let $k \geq k_+$ be an integer, where $k_+$ is defined in \eqref{defkk}. Then 
 $$\La_k^s(M,[g]) \leq  Y_k^s(M,[g]).$$ 
\end{prop}
Proposition \ref{prop_ineg_large} then follows from \eqref{YleqX} and from Proposition \ref{prop_ineg_large2}. 

\begin{proof}
Let  $r,\ell_0,\cdots,\ell_r \in \mN$ be such that 
 \begin{equation} \label{toprovemuleq}
  \ell_0 + \cdots + \ell_r =k
 \end{equation}
if $\ell_0 \ge k_+$ and 
\begin{equation} \label{toprovprop}
 \ell_1 + \cdots + \ell_r =k-k_++1 
\end{equation}
if $\ell_0 = 0$.  We fix $\ep >0$ and we let  $\be_0 \in L^{\frac{n}{2s}}(M) \backslash \{0\}, \beta_0 \ge 0$ a.e.,  and $ \be_1,\cdots,\be_r \in L^{\frac{n}{2s}}(\mathbb{S}^n)  \backslash \{0\}, \beta_i \ge 0$ a.e. for all $1 \le i \le r$ satisfy the following: 
\begin{itemize} 
\item For all $1 \le i \le r$, 
 \begin{equation} \label{almostmin1}
\inf_{V \in \mathcal{G}^{\be_{i}}_{\ell_i}(H^s(\mS^n))} \max_{v\in V\setminus\{0\}} \mathcal{R}(\be_i,v) \leq  \La_{\ell_i}^s(\mS^n) + \ep.
 \end{equation}
 \item if \eqref{toprovemuleq} holds, that is if $\ell_0 \ge k_+$, then 
 \begin{equation} \label{almostmin3}
 \inf_{V \in  \mathcal{G}^{\be_0}_{\ell_0}(H^s(M))} \max_{v\in V\setminus\{0\}} \mathcal{R}(\be_0,v)  \leq  \La_{\ell_0}^s(M,[g]) + \ep  
 \end{equation} 
 \end{itemize}
In the case where $\ell_0 =0$, we let $\beta_0 \equiv \text{Vol}(M,g)^{-\frac{2s}{n}}$. Without loss of generality, and by Proposition~\ref{prop:preliminaryvarset}, we can assume that $\be_i>0$ a.e. for any $1 \le i \le r$ and that $\Vert \be_i \Vert_{L^{\frac{n}{2s}}}  = 1$. 
 By Proposition~\ref{prop:deficontinuityeigen},  for all $i\in \{1, \cdots ,r \}$ and  $j \in \{1,\cdots, \ell_i \}$, there exists   $v_j^i \in H^s(\mS^n)$ and 
$\la_j^i \in \R$, with $\lambda_1^i \leq \cdots \leq  \la_{\ell_i}^i$ such that for all $1 \leq p,q  \leq \ell_i$ we have $\int_{\mS^n} \be_i v^i_{p} v^i_{q}   \,dv_{g_0} = \delta_{pq}$ and  $P_{g_0}^s v_p^i= \la_p^i \be_i v_p^i$ in $\mS^n$. Furthermore, for any $1 \le i \le r$, if we let $V_i=\hbox{span}\{v^i_1 \cdots, v^i_{\ell_i}\}$, then  $V_i \in \mathcal{G}^{\be_i}_{\ell_i}(H^s(\mS^n))$ and we have 
$$ \max_{v\in V_i\setminus\{0\}} \mathcal{R}(\be_i,v) = \inf_{V' \in \mathcal{G}^{\be_i}_{\ell_i}(H^s(\mS^n))}  \max_{v\in V' \setminus\{0\}} \mathcal{R}(\be_i,v), $$
as well as $\mathcal{R}(\be_i ,v_1^i) = \min_{v \in V_i} \mathcal{R}(\be_i,v)$ and $\mathcal{R}(\be_i,v^i_{\ell_i}) = \max_{v \in V_i} \mathcal{R}(\be_i,v)$. Similarly, if \eqref{toprovemuleq} holds (which implies $\ell_0 \ge k_+$),  for any $j \in \{1,\cdots, \ell_0 \}$, there exists   $v_j^0 \in H^s(M)$ and  $\la_j^0 = \lambda_j(\beta_0) \in \R$, with $\lambda_1^0 \leq \cdots \leq  \la_{\ell_0}^0$ such that for all $1 \leq p,q  \leq \ell_0$ we have, $\int_{M} \be_0 v^0_{p} v^0_{q}   \,dv_{g} = \delta_{pq}$,  $P_{g}^s v_p^0= \la_p^0 \be_0 v_p^0$ in $M$ and 
$$ \max_{v\in V_0\setminus\{0\}} \mathcal{R}(\be_0,v) = \inf_{V' \in \mathcal{G}^{\be_0}_{\ell_0}(H^s(M))}  \max_{v\in V' \setminus\{0\}} \mathcal{R}(\be_0,v) $$
where we have let $V_0=\hbox{span}\{v^0_1 \cdots, v^0_{\ell_0}\}$. In particular by \eqref{almostmin1} and \eqref{almostmin3} we have
\begin{equation} \label{laleq}
 \la_{\ell_0}^0 \leq \La_{\ell_0}^s(M,[g]) + \ep  \; \hbox{ and }  \la_{\ell_i}^i \leq \La_{\ell_i}^s(\mS^n)+\ep \quad \text{ for all } 1 \le i \le r.
\end{equation}
Let $p \in \mS^n$ be fixed. The manifolds $(\mS^n \setminus \{p\}, g_0)$ and $(\mR^n, \xi)$ are conformally equivalent. Precisely, if $\pi$ denote the stereographic projection that sends $-p$ to $0$ we have  $(\pi^{-1})^* g_0(x) = U(x)^{\frac{4s}{n-2s}} \xi$, where $\xi$ is the Euclidean metric in $\R^n$ and where we have let $ U(x) = \left(\frac{2}{1+|x|^2}\right)^{\frac{n-2s}{2s}}$. For $1 \le i \le r$ and $1 \le j \le \ell_i$ we define
$$ y_j^i (x) =  \frac{v_j^i }{U}\big( \pi^{-1}(x) \big) \quad \text{ and } \quad \tilde{\beta}_i(x) = U(x)^{-\frac{4s}{n-2s}} \beta_i(\pi^{-1}(x)) \quad \text{ for any } x \in \R^n.$$
By the conformal invariance property \eqref{eq:confinv} of $P_{g_0}$, each function $y_j^i$ belongs to $D^{s,2}(\R^n)$ (see  \eqref{normDs2} below
 for the definition of $D^{s,2}(\R^n)$) so that $\int_{\R^n} |\Delta_\xi^{\frac{s}{2s}} y_j^i |^2 \, dv_\xi < + \infty$, and they satisfy 
\begin{equation} \label{yij}
 P^s_{\xi} y_j^i = \la_j^i \tilde{\be}_i y_j^i 
\end{equation}
in $\mR^n$ and 
\begin{equation} \label{yij2} 
\int_{\mR^n} \tilde{\be}_i y_{j_1}^i y_{j_2}^i \, dv_\xi = \delta_{j_1 j_2} \quad  \hbox{ and } \quad  \int_{\mR^n} \tilde{\be}_i ^{\frac{n}{2s}}\,  dv_\xi =1.
\end{equation}
For $a>0$ and for any  $1 \le i \le r$ we now let 
$$z_j^{i,a}= a^{\frac{n-2s}{2}}  y_j^i(ax) \quad \text{ and } \quad \be_{i,a}= a^{2s} \tilde{\be}_i(ax) $$
in $\R^n$. As one easily checks, $z_j^{i,a}$ and $\be_{i,a}$ still satisfy \eqref{yij} and \eqref{yij2}. We now let $x_1,\cdots,x_r$ be distinct points in $M$. Let $\de >0$ be small and let $\eta_i \in C^{\infty}(M)$ a cut-off function equal to $1$ in $B_g(x_i,\de)$ and $0$ outside $B_g(x_i,2 \de)$. If $\de$ is small enough the supports of the $\eta_i$ are disjoint. We now define
\begin{equation} \label{defmui}
 \mu_i = \la_{\ell_i}^i \quad \text{ for all } 1 \le i \le r \quad \text{ and } \quad \mu_0=  \left \{ \begin{aligned} &  \la_{\ell_0}^0 &\text{ if } \ell_0 \ge k_+ \\
 & 0 & \text{ if } \ell_0 = 0 \end{aligned} \right. .
\end{equation}
and we let, for $a >0$,
$$\be_a= \mu_0 \beta_0 + \sum_{i= 1}^r \mu_i \eta_i \beta_{i,a}\circ \exp_{x_i}^{-1}$$
where $\exp_{x_i}$ denotes the exponential chart at $x_i$ for the metric $g$. We also let
 $$V^{i,a}_j = \eta_i z^{i,a}_j \circ  \exp_{x_i}^{-1} \quad \text{ for } 1 \le i \le r, 1 \le j \le \ell_i . $$
In case $\ell_0 \ge k_+$ we let $V_j^{0,a} = v^0_j$ for $1 \le j \le \ell_0$. If $\ell_0 = 0$ we let $\vp_1, \cdots, \vp_{k_+-1}$ be the $k_+-1$ first eigenfunctions of $P_g^s$ in $M$ 
 and we let $ V_j^{0,a} = \vp_j$ for $1 \le j \le k_+-1$. With these definitions and since the support of the $\eta_i$ are disjoint, it is easy to check  that for all $0 \le i,i' \le r$ and $1 \le j \le \ell_i, 1 \le j' \le \ell_{i'}$ ($1 \le j \le k_+-1$ if $\ell_0 = 0$) we have
\begin{equation} \label{Vij1} 
\begin{aligned} 
  \lim_{a \to +\infty} \int_{M} \be_{a}^{\frac{n}{2s}} \, dv_g & =  \sum_{i=0}^r \left(\mu_i \right)^{\frac{n}{2s}}  \\
\lim_{a \to + \infty} \int_{M} \be_a V_{j}^{i,a} V_{j'}^{i',a} \, dv_g &  = \mu_i  \delta_{ii'} \delta_{jj'} \\
\lim_{a \to + \infty} \int_M P_g^s V_{j}^{i,a} V_{j'}^{i',a}  \, dv_g&  = \la_{j}^i \delta_{ii'} \delta_{jj'}.
\end{aligned}
\end{equation}
For $a$ large enough we define 
$$V_a= \left \{ \begin{aligned}
&  \span\big\{V_{i,a}^j | i \in \{0,\cdots,r\}, j \in \{0,\cdots,\ell_i\} \big\} & \text{ if } \ell_0 \ge k_+ \\
&  \span\big\{V_{i,a}^j, V_{0,a}^p | i \in \{1,\cdots,r\}, j \in \{0,\cdots,\ell_i\}, p \in \{1, \cdots, k_+-1\} \big\} & \text{ if } \ell_0 = 0 \\
 \end{aligned} \right. . $$
 By \eqref{Vij1} and by the definition of $\ell_0, \ell_1, \cdots, \ell_r$ we have  $\dim_{\beta_a} V_a = k$ so that $V_a \in \mathcal{G}_k^{\beta_a} (H^s(M))$. In particular, 
we get that 
\begin{equation} \label{muleq1}
 \La_k(M,[g]) \leq  \max_{v \in V_a} \mathcal{R} (\be_a , v) \Vert \beta_a \Vert_{L^{\frac{n}{2s}}}. 
\end{equation}
First, using \eqref{laleq}, \eqref{defmui} and \eqref{Vij1} we have 
$$\lim_{a \to +\infty}  \int_M \beta_a^{\frac{n}{2s}}\, dv_g 
 \leq   \La_{\ell_0}^s(M,[g])^{\frac{n}{2s}} + \La_{\ell_1}^s(\mS^n)^{\frac{n}{2s}} +  \cdots +\La_{\ell_r}^s(\mS^n)^{\frac{n}{2s}} + (r+1) \ep,$$
where we used the convention that $\Lambda_0^s(M,[g]) = 0$. Let now $\al_{a,ij} \in \mR$, where the indexes $i$ and $j$ are as in the definition of $V_a$, be chosen so such that $\sum_{i,j} \alpha_{a,ij}^2 = 1$ and 
$$ v_a = \sum_{i,j} \al_{a,ij} V_j^{i,a}$$
attains the maximum in the right-hand side of \eqref{muleq1}. Up to a subsequence as $a \to + \infty$ we can assume that $\alpha_{a,ij} \to \alpha_{ij}$, with again $\sum_{i,j} \alpha_{i,j}^2 = 1$. Using \eqref{defmui} and  \eqref{Vij1} we have 
$$\lim_{a \to +\infty} \int_M \beta_a v_a^2 \, dv_g = \sum_{i,j}  \al_{ij}^2 \mu_i  ,$$ 
while using \eqref{defmui} and since $\la_j^i \leq \mu_i$ we get  
$$\begin{aligned} 
  \lim_{a \to +\infty} \int_M v_a P_g^s v_a \, dv_g& =   \sum_{i,j}\al_{ij}^2 \la_j^i \le  \sum_{i,j}  \al_{ij}^2 \mu_i  .  \\
 \end{aligned}$$
 Combining these estimates into \eqref{muleq1} finally gives 
 $$\La_k^s(M,[g])^{\frac{n}{2s}} \leq  \La_{\ell_0}^s(M,[g])^{\frac{n}{2s}} + \La_{\ell_1}^s(\mS^n)^{\frac{n}{2s}} +  \cdots +\La_{\ell_r}^s(\mS^n)^{\frac{n}{2s}} + (r+1) \ep.$$
 Since $r,\ell_0,\cdots,\ell_r$ and $\ep$ are arbitrary, we obtain that 
 \begin{equation} \label{muleqX'}
\La_k^s(M,[g]) \leq Y_k^s(M,[g])
 \end{equation}
 which concludes the proof of Proposition \ref{prop_ineg_large}.
\end{proof}

\subsection{Proof of Theorem \ref{prop:bubbling}.}

In this subsection we let $k \ge k_+$ be fixed and we assume that $\La_k^s(M,[g])$ is not attained. We prove Theorem \ref{prop:bubbling} by proving an energy-quantification identity for a suitable minimising sequence for $\La_k^{s}(M,[g])$ that we construct.

For any $p > \frac{n}{2s}$, we consider the subcritical approximation $\La_k^{s,p}(M,[g])$ of $\La_k^s(M,[g])$ given by \eqref{Lasouscritique}. Proposition \ref{subcritical_theorem} shows that, for each $p > \frac{n}{2s}$, $\La_k^{s,p}(M,[g])$ is attained: that is, there exists $\beta_p \in L^p(M)$, with $\Vert \beta_p \Vert_{L^p} = 1$, such that $ \lambda_k(\beta_p) =  \La_k^{s,p}(M,[g])$. Proposition \ref{euler_minimizer} then shows there exists an integer $k_+ \le \ell_p \leq k$ with $\lambda_{\ell_p}(\beta_p) = \lambda_{k}(\beta_p)$, a $Q(\beta_p,\cdot)$-orthonormal family $(v_{\ell_p,p},\cdots,v_{k,p})$ in $E_k(\beta_p)$ and there exist positive numbers $d_{\ell_p,p} ,\cdots,d_{k,p}$ satisfying $\sum_{i=\ell_p}^k d_{i,p}  = 1$ such that 
\begin{equation} \label{mainthp:1}
  \beta_p^{p-1} = \sum_{i=\ell_p}^k d_{i,p} v_{i,p}^2 \quad \text{ a. e. in } M.
\end{equation}
Up to passing to a subsequence as $p \to \frac{n}{2s}$ from above we may assume that $\ell_p$ is constant and equal to $\ell$. If $k = k_+$, we have $\ell = k$. By Proposition \ref{prop:deficontinuityeigen} we can let $(v_{k_+,p}, \cdots, v_{\ell-1,p})$ be a family of eigenfunctions associated to $\lambda_{k_+}(\beta_p), \cdots, \lambda_{\ell-1}(\beta_p)$ such that the family $(v_{k_+,p}, \cdots, v_{k,p})$ is $Q(\beta_p,\cdot)$-orthonormal. We thus have 
\begin{equation}  \label{mainthp:2}
P_g^s v_{i,p} = \lambda_{i,p} \beta_p v_{i,p} \quad \text{ in } M 
\end{equation}
and 
\begin{equation} \label{mainthp:3}
\int_M \beta_p v_{i,p} v_{j,p} \,dv_g = \delta_{ij}. 
\end{equation}
for all $k_+ \le i,j \le k$, where we have let $\lambda_{i,p} = \lambda_i(\beta_p)$ for simplicity. We perform in what follows an asymptotic analysis of the sequences $(\beta_p)_p$ and $(v_{k_+,p}, \cdots, v_{k,p})_p$ as $p \to \frac{n}{2s}$ from above. 

\medskip

First, up to passing to a subsequence, we can assume that the $(d_{i,p})_{\ell \le i \le k}$ and the $(\la_{i,p})_{k_+ \le i \le k}$ converge, respectively to $d_i \in [0,1]$ and $\la_i  \in [0,\La_k^s(M,[g])]$, as $p \to \frac{n}{2s}$. By the assumption $\sum_{i=\ell}^k d_i = 1$ there exists at least one $i \in \{\ell, \cdots, k \}$ such that $d_i >0$. Proposition \ref{CVLambdap} also implies that 
\begin{equation} \label{cvlambdalimite}
\la_k= \lim_{p \to \frac{n}{2s}} \la_{k,p} =  \La_k^s(M,[g]).
\end{equation}
An important consequence of Lemma \ref{lem:mainlemma} is that for all $k_+ \le i \le k$
\begin{equation} \label{lai>0}
 \la_i>0, \quad \text{ and in particular }  \quad \La_k^s(M,[g]) >0.
\end{equation}
 Since, for each $k_+ \le i \le k$, the sequence $(\lambda_{i,p})_{p}$ is bounded, Lemma \ref{lem:mainlemma} applies and shows that for any $k_+ \le i \le k$ the following dichotomy occurs along a subsequence as $p \to \frac{n}{2s}$:
\begin{itemize}
 \item either $(v_{i,p})_p$ is bounded in $H^s(M)$, 
 \item or $(v_{i,p})_p$ is unbounded in $H^s(M)$  and we have
 \begin{equation} \label{mainthp:4}
  v_{i,p} = \gamma_{i,p} \big(K + o(1) \big) \quad \text{ as } p \to \frac{n}{2s},
  \end{equation}
  where $\gamma_{i,p} = \Vert v_{i,p} \Vert_{H^s} \to + \infty$, $K$ is a nonzero element of $\ker(P_g^s)$ and $o(1)$ denotes a sequence that strongly converges to $0$ in $H^s(M)$ as $p \to \frac{n}{2s}$. In this case $v_{i,p}$ also  decomposes as 
\begin{equation} \label{starstar}
 v_{i,p} = k_{i,p} + \tilde{w}_{i,p}, 
 \end{equation}
where $\tilde{w}_{i,p} \in \ker(P_g^s)^{\perp}$ is the orthogonal projection of $v_{i,p}$ onto $\ker(P_g^s)$ for the $H^s(M)$-scalar product. Lemma \ref{lem:mainlemma} also states that the sequence $(\tilde{w}_{i,p})_p$ is bounded in $H^s(M)$  and  that $\tilde{w}_{i,p} \rightharpoonup 0$ in $H^s(M)$  as $p \to \frac{n}{2s}$.
\end{itemize}

\medskip

Remark that, again by Lemma \ref{lem:mainlemma}, the second alternative can only occur if $\ker(P_g^s) \neq \{0\}$. We prove Theorem \ref{prop:bubbling} in a series of lemmas. Unbounded families of eigenfunctions are the main obstacle in passing to the limit the Euler-Lagrange relation \eqref{mainthp:1} as $p \to \frac{n}{2s}$. We separate the indexes in $\{\ell, \cdots, k\}$ appearing in \eqref{mainthp:1} in two, according to the limit of $d_i$: we write $ \{\ell, \cdots, k\} = I \sqcup J$, where 
\begin{equation} \label{defIandJ}
I = \big \{ i \in\{\ell, \cdots, k\}, d_i >0  \big \} \quad \text{ and } \quad J = \big \{ i \in\{\ell, \cdots, k\}, d_i =0  \big \}.
\end{equation}
The first result that we prove establishes that the eigenfunctions $v_{i,p}$ in \eqref{mainthp:1} with $i \in I$ are uniformly bounded in $H^s(M)$:
\begin{lemme} \label{vkqbounded} 
Let $i \in \{\ell, \cdots, k\}$, where $\ell$ is as in \eqref{mainthp:1}. If $i \in I$, that is if $d_i >0$, there exists $C >0$ independent of $p$ such that $\Vert v_{i,p} \Vert_{H^s} \le C$ for all $p > \frac{n}{2s}$. 
\end{lemme}
Recall that the relations $d_i \ge0$ and  $\sum_{i=\ell}^k d_i  =1$ ensures that $I \neq \emptyset$. 

\begin{proof}
Let $i \in I$. Equality \eqref{mainthp:1} shows that, almost everywhere in $M$, we have 
\begin{equation} \label{euler_eq_w2}
  \frac{d_i}{2} (v_{i,p})^2 \leq \beta_p^{p-1}
\end{equation}
which shows, since $\Vert \beta_p \Vert_{L^p} = 1$, that $(v_{i,p})_{p}$ is uniformly bounded in $L^2(M)$. Assume by contradiction that $(v_{i,p})_p$ is not bounded in $H^s(M)$. Then \eqref{mainthp:4} shows that 
$$ \Vert v_{i,p} \Vert_{L^2}^2 = \gamma_{i,p}^2 \big( \Vert Z \Vert_{L^2}^2 + o(1) \big), $$
which is a contradiction since $Z \neq 0$. 
\end{proof}
We next observe that only the sequences $(v_{i,p})_p$ for $i \in I$ contribute to \eqref{mainthp:1}:
\begin{lemme} \label{lemme:tildebeta}
For $p > \frac{n}{2s}$, we define 
\begin{equation} \label{mainthp:5}
\tilde{\beta}_p = \Big( \beta_p^{p-1} - \sum_{i \in J}d_{i,p} v_{i,p}^2 \Big)^{\frac{1}{p-1}}. 
\end{equation}
Then 
\begin{equation} \label{mainthp:6}
\Vert \tilde{\beta}_p - \beta_p \Vert_{L^p} \to 0 
\end{equation}
as $p \to \frac{n}{2s}$ from above. 
\end{lemme}

\begin{proof}
First, $\tilde{\beta}_p$ is well-defined since $\beta_p^{p-1} - \sum_{i \in J}d_{i,p} v_{i,p}^2 \ge 0$ a.e. by \eqref{mainthp:1}. Similarly, \eqref{mainthp:1} also shows that $0 \le \tilde{\beta}_p \le \beta_p$ a.e. in $M$. As a consequence, we have 
\begin{equation} \label{mainthp:7}
 \int_M |\beta_p - \tilde{\beta}_p|^{p}dv_g \le 2^{p-1} \int_M \beta_p^{p-1} |\beta_p - \tilde{\beta}_p| dv_g. 
\end{equation} 
We now estimate $|\beta_p - \tilde{\beta}_p|$ pointwise. Since $p >1$, we recall that the following inequality holds for all real numbers $a \ge0$ and $0 \le b \le a$:
$$ \Big|\big(a^{p-1} - b^{p-1}\big)^{\frac{1}{p-1}} - a \Big| \le C a^{2-p}b^{p-1} $$
for some positive constant $C$ that is independent of $p$. Applying the latter to $a = \beta_p$ and $b = \big(\sum_{i \in J}d_{i,p} v_{i,p}^2\big)^{\frac{1}{p-1}}$ shows that we have, a.e. in $M$, 
$$\beta_p^{p-2} |\beta_p - \tilde{\beta}_p| \le C \sum_{i \in J}d_{i,p} v_{i,p}^2.$$
Using \eqref{mainthp:7} finally gives, since $Q(\beta_p, v_{i,p}) = 1$ for all $p$,
$$ \begin{aligned}
 \int_M |\beta_p - \tilde{\beta}_p|^{p}dv_g & \le  C\int_M \sum_{i \in J}d_{i,p}\beta_p v_{i,p}^2 dv_g  = C \sum_{i \in J}d_{i,p} \to 0
\end{aligned}$$
as $p \to \frac{n}{2s}$ by definition of $J$ in \eqref{defIandJ}. This proves the lemma.
\end{proof}
Using Lemma \ref{lemme:tildebeta}, the Euler-Lagrange relation \eqref{mainthp:1} now rewrites as 
\begin{equation} \label{mainthp:8}
  \tilde{\beta}_p^{p-1} = \sum_{i\in I} d_{i,p} v_{i,p}^2 \quad \text{ a. e. in } M,
\end{equation}
where $\tilde{\beta}_p$ is given by \eqref{mainthp:5}. Relation \eqref{mainthp:8} is of course reminiscent of \eqref{mainthp:1}, but in \eqref{mainthp:8} every $v_{i,p}$ in the r.h.s is now bounded in $H^s(M)$. Since $\| \beta_p \|_{L^p}=1$, there exists  $\beta \in L^{\frac{n}{2s}}(M)$ such that $\beta_p \rightharpoonup \beta$ weakly in $L^{\frac{n}{2s}}(M)$ as $p \to \frac{n}{2s}$, and by \eqref{mainthp:6} we also have $\tilde{\beta}_p\rightharpoonup \beta$ weakly in $L^{\frac{n}{2s}}(M)$. For any $k_+ \le i \le k$ we define a new family of functions as follows: 
\begin{equation} \label{defwip}
w_{i,p} = \left \{
\begin{aligned}
& v_{i,p} & \text{ if } (v_{i,p})_p \text{ is bounded in } H^s(M) \\ 
& \tilde{w}_{i,p} 
 & \text{ otherwise }, 
\end{aligned}
\right.
\end{equation}
where $\tilde{w}_{i,p}$ is defined in \eqref{starstar}. For any $k_+ \le i \le k$, the sequence $(w_{i,p})_p$ is in particular bounded in $H^s(M)$, and we denote by $w_i$ its weak limit in $H^s(M)$ as $p \to \frac{n}{2s}$. If $i$ is such that $(v_{i,p})_p$ is bounded in $H^s(M)$, we denote by $v_i$ its weak limit. Again by Lemma \ref{lem:mainlemma} we thus have 
$$ w_i = \left \{ \begin{aligned} 
& v_i & \text{ if } (v_{i,p})_p \text{ is bounded in } H^s(M) \\
& 0 & \text{ otherwise}
 \end{aligned}  \right. .$$
We define the {\em set of concentration points of $(\be_p)_p$} by 
$$A=  \Biggl\{ x \in M \,\Big|\, \forall \de>0, \;  \limsup_{p \to \frac{n}{2s}} \int_{B_g(x,\de)} \beta_p^{\frac{n}{2s}} dv_g\geq \frac14 \left(\frac{\Lambda_1^s(\mathbb{S}^n)}{\Lambda_k^s(M,[g])}\right)^{\frac{n}{2s}} \Biggl\}. $$
It is a finite set. As a first result we prove that each sequence $(w_{i,p})_p$ strongly converges to $w_i$ outside of $A$:

\begin{lemme}
As $p \to \frac{n}{2s}$, we have 
\begin{equation} \label{mainthp:9}
 w_{i,p} \to w_i \quad \text{ strongly in } H^s_{loc}(M \backslash A)
\end{equation}
and 
\begin{equation} \label{mainthp:10}
\beta_p \to \beta \quad \text{ strongly in } L^{\frac{n}{2s}}_{loc}(M \backslash A). 
\end{equation}
Also, if there exists $i \in \{k_+, \cdots, k \}$ such that $(v_{i,p})_p$ is not bounded in $H^s(M)$, then $\beta \equiv 0$. 
\end{lemme}

\begin{proof}
Let $i$ be such that $(v_{i,p})_p$ is bounded in $H^s(M)$. Then $w_{i,p} = v_{i,p}$ and $v_{i,p} \to v_i$ in $H^s_{loc}(M \backslash A)$ by Lemma \ref{convergence_appendix}, which proves \eqref{mainthp:9} in this case. The set of such indexes $i$ is not empty, since it contains indexes $i \in I$ by Lemma \ref{vkqbounded}, where $I$ is given by \eqref{defIandJ}. Using \eqref{mainthp:9} together with \eqref{mainthp:6} and \eqref{mainthp:8} then proves \eqref{mainthp:10}.  

We now assume that there exists $i\in \{\ell, \cdots, k \}$ such that $(v_{i,p})_p$ is not bounded in $H^s(M)$. Then Lemma \ref{lem:mainlemma} implies that $\beta \equiv 0$ and \eqref{mainthp:10} shows that $\beta_p \to 0$ in $L^{\frac{n}{2s}}_{loc}(M \backslash A)$. Let $\eta \in C^\infty_c(M\backslash A)$ with $\Vert \eta \Vert_{L^\infty} \le 1$. By \eqref{mainthp:2} and \eqref{defwip}, and since $w_{i,p} \rightharpoonup 0$ in $H^s(M)$, $\eta w_{i,p}$ satisfies
$$ P_g^s (\eta w_{i,p})  = \lambda_{i,p} \beta_p \eta v_{i,p} + o(1) \quad \text{ in } M, $$
where $\Vert o(1) \Vert_{H^{-s}} \to 0$ as $p \to \frac{n}{2s}$. By \eqref{eq:inv:noyau} and since $L^{\frac{2n}{n+2s}}(M)$ continuously embeds into $H^{-s}(M)$ by H\"older and Sobolev inequalities we have
$$ \Vert w_{i,p} \Vert_{H^s} \le C \Vert \beta_p \eta v_{i,p} \Vert_{L^{\frac{2n}{n+2s}}} + o(1) $$
for some $C>0$. 
By H\"older's inequality and \eqref{mainthp:3} we have
$$ 
\begin{aligned} 
\int_M\big( \beta_p \eta v_{i,p} \big)^{\frac{2n}{n+2s}} dv_g  & \le \left(\int_M (\eta \beta_p)^{\frac{n}{2s}} dv_g \right)^{\frac{2s}{n+2s}} \to 0 
\end{aligned} $$
as $p \to \frac{n}{2s}$, since $\beta_p \to 0$ in $L^{\frac{n}{2s}}_{loc}(M \backslash A)$. Combining these inequalities proves that $\Vert w_{i,p} \Vert_{H^s} \to 0$ as $p  \to \frac{n}{2s}$ from above and concludes the proof of the Lemma. 
\end{proof}
 A consequence of our assumption that $\Lambda_k^s(M,[g])$ is not attained is that there exist concentration points: 

\begin{lemme} \label{lem:Anonvide}
We have $A \neq \emptyset$. 
\end{lemme}

\begin{proof}
Assume that $A = \emptyset$. Then, by \eqref{mainthp:10}, $\beta_p \to \beta$ strongly in $L^{\frac{n}{2s}}(M)$, so that $\Vert \beta \Vert_{L^{\frac{n}{2s}}} = 1$. A consequence of Lemma   \ref{lem:mainlemma} is then that every sequence $(v_{i,p})_p$, for $k_+ \le i \le k$, is bounded in $H^s(M)$. If we denote by $(v_i)_{k_+ \le i \le k }$ the respective weak limits in $H^s(M)$, relations \eqref{mainthp:2} and \eqref{mainthp:3} strongly pass to the limit as $p \to \frac{n}{2s}$ and show that $( v_i)_{k_+ \le i \le k}$ is an orthonormal family, for $Q(\beta, \cdot)$, of eigenfunctions associated to eigenvalues $(\lambda_i)_{k_+\le i \le k}$ of the generalised metric $\beta g$. Here we used that $\beta \in \mathcal{A}_{\frac{n}{2s}} \mapsto \lambda_k(\beta)$ is continuous since $k \ge k_+$ (see Remark~\ref{rem:continuite:vp:pos}), which implies that $ \lambda_k = \lim_{p \to \frac{n}{2s}} \lambda_k(\beta_p) = \lambda_k(\beta)$. By \eqref{cvlambdalimite} this shows that $\lambda_k(\beta) = \Lambda_k^s(M,[g])$ and hence that $ \Lambda_k^s(M,[g])$ is attained which contradicts the assumptions of Theorem  \ref{prop:bubbling}. 
\end{proof}

In the following we distinguish two cases depending on the weak limit $\beta$.

\subsection{Proof of Theorem \ref{prop:bubbling}: the case $\beta \not \equiv 0$.} If $\beta \not \equiv 0$, for any $k_+ \le i \le k$ the sequence $(v_{i,p})_p$ is bounded in $H^s(M)$ by Lemma \ref{lem:mainlemma}. The weak limits $v_i$ then satisfy 
\begin{eqnarray} \label{eqwlim}
P_g^s v_i  =  \la_i \beta v_i \quad \text{ in } M.
\end{eqnarray}
We now introduce a new family of eigenfunctions that will simplify the analysis. We let $\la_{k_+}',\cdots, \la_m'$ be the distinct values of $\la_{k_+} ,\cdots,\la_k$ i.e. 
$$\{ \la_{k_+}',\cdots, \la_m' \}= \{\la_{k_+},\cdots,\la_k\}$$
and $\la'_i \not=\la'_j$ if $i \not= j$. 
For $i \in \{k_+,\cdots,m\}$, denote by $S_i$ set of indexes $j \in \{k_+,\cdots, k\}$ such that $\la_j = \la'_i$. 
We define $V_{i,p}= \span_{j \in S_i} \{v_{j,p}\}$ and $V_p = \oplus_{i=k_+}^m V_{i,p}$. In particular, $V_p = \text{Span} \big\{ v_{k_+,p}, \cdots, v_{k,p} \big\}$. 
The bilinear form
$$\begin{aligned}
B(\beta_p,h,k)& = \int_M \beta_p h k dv_g 
\end{aligned} $$
defines a scalar products in $V_p$ (and thus on each $V_{i,p}$), and $(v_{j,p})_{j \in S_i}$ is an orthonormal basis of $V_{i,p}$ for $B(\beta_p, \cdot, \cdot)$. By standard results of basic algebra, for any $i \in \{k_+, \cdots, m\}$, there exists an orthonormal basis $(u_{j,p})_{j \in S_i}$ of $V_{i,p}$ for $B(\be_p, \cdot,\cdot)$ which is orthogonal for $B(\beta, \cdot, \cdot)$. By doing this in each $V_{i,p}$ we get an orthonormal basis $(u_{k_+,p},\cdots,u_{k,p})$ of $V_p$ for the scalar product $B(\be_p, \cdot,\cdot)$ which is orthogonal for $B(\beta, \cdot, \cdot)$ i.e. 

\begin{equation} \label{vivj}
\begin{aligned}
 \int_M \beta_p  u_{i,p} u_{j,p} \,dv_g & = \delta_{ij} \text{ for all } k_+ \le i,j \le k \quad \text{ and }\\
  \int_M \beta u_{i,p} u_{j,p} dv_g& = 0 \quad  \text{ for all } k_+ \le i,j \le k , i \neq j.  
\end{aligned}
\end{equation}
Assume that $j \in S_i$. Then there exist some real numbers $(\al_{r,p})_{r \in S_i}$ satisfying $\sum_{r \in S_i} \al_{r,p}^2 =1$ such that  $u_{j,p} = \sum_{r \in S_i}\al_{r,p} v_{r,p}$. 
Therefore, using \eqref{mainthp:2}, 
$$P_g^s u_{j,p} = \beta_p  \sum_{r \in S_i} \la_{r,p} \al_{r,p} v_{r,p}.$$  
Since by construction $\la_{j,p} $ tends to $\la_j=\la'_i$ as $p \to \frac{n}{2s}$ we have 
\begin{equation} \label{eqvq}
 P_g^s u_{j,p} = \la_j \be_p  u_{j,p} + \be_p f_{j,p}  
\end{equation}
where $f_{j,p}= \sum_{r \in S_i} (\la_{r,p}- \la_j) \al_{r,p} v_{r,p} $ satisfies
\begin{equation} \label{eqeq} 
\lim_{p \to \frac{n}{2s}} \Vert f_{j,p} \Vert_{H^s} = 0 \quad \text{ and hence } \quad  \lim_{p \to \frac{n}{2s}} \int_M \be_p (f_{j,p})^2 dv_g= 0. 
\end{equation}
Up to passing to a subsequence we can assume that for each $k_+ \le i \le k$, $(u_{i,p})_p$ weakly converges to $u_i \in H^s(M)$ as $p \to \frac{n}{2s}$. Passing \eqref{vivj} and \eqref{eqvq} to the limit yields
\begin{equation} \label{vivjlim}
  \int_M \beta u_{i} u_{j} dv_g = 0 \quad  \text{ for all } k_+ \le i \neq j \le k 
\end{equation}
and 
\begin{equation} \label{eqvlim}
 P_g^s u_j = \la_j \beta u_j.   
\end{equation}
We also have, thanks to \eqref{mainthp:9} and since $v_{i,p} = w_{i,p}$ for all $k_+ \le i \le k$,
\begin{equation} \label{convlim}
 u_{i,p} \to u_i \quad \text{ in } H^s(M \backslash A)
\end{equation}
as $p \to \frac{n}{2s}$. We will work from now on with the sequences $(u_{j,p})_p$ instead of $(v_{j,p})_p$. We define in what follows the {\em weight } of $M \setminus A$ by 
\begin{equation} \label{defom0}
\om_0= \sharp \{i \in  \{k_+,\cdots,k\} | \beta^{\frac{1}{2}} u_i \not\equiv 0\}.
\end{equation}
We prove energy-quantification of $\Lambda_k^s(M,[g])$ in two steps. We first estimate the energy contained in $M$ after passing to the weak limit: 
\begin{lemme} \label{mu(M)}
Assume that $\omega_0>0$. Then 
$$\La_k^s(M,[g]) \left( \int_M \beta^{\frac{n}{2s}} dv_g \right)^{\frac{2s}{n}} \geq \La_{\omega_0+k_+-1}^s (M,[g]).$$
\end{lemme}

\begin{rem}
 We will prove later that, if $\omega_0 >0$, equality holds in the previous inequality and that $\La_{\omega_0+k_+-1}^s(M,[g])$ is attained. 
\end{rem}

\begin{proof}
Let $k_+ \leq i_1 < \cdots <i_{\omega_0} \le k $ be such that $\beta^{\frac{1}{2}} u_{i_r} \not\equiv 0$ for $r = 1 \dots \omega_0$. From \eqref{vivjlim} we know that $\big( \beta^{\frac12} u_{i_1},\cdots, \beta^{\frac12} u_{i_{\omega_0}} \big)$ is free, and by \eqref{eqvlim} each $u_j$ is a eigenfunction associated to a positive eigenvalue for the generalised metric $\beta g$. Arguing as in the proof of Proposition \ref{prop:preliminaryvarset}  there is a family $(\vp_1, \cdots, \vp_{k_+-1})$ of functions in $H^s(M)$ that satisfy 
  $$ \begin{aligned}
  &   \int_M \vp_i P_g^s \vp_i dv_g \le 0 \quad \text{ for all } 1 \le i \le k_+-1 \text{ and } \\
  &   \int_M \vp_i  P_g^s \vp_j dv_g = \Big(\int_M \vp_i P_g^s \vp_i dv_g\Big) \delta_{ij}  \quad \text{ for all } 1 \le i,j \le k_+
    \end{aligned} $$ 
and such that $(\vp_1, \cdots, \vp_{k_+-1}, u_{i_1}, \cdots, u_{\omega_0})$  is orthonormal with respect to $Q(\beta,\cdot)$. We may thus let 
$$ V = \text{Span} \big\{ \vp_1, \cdots, \vp_{k_+-1}, u_{i_1}, \cdots, u_{\omega_0}\big\}. $$
Then $\dim_{\beta} V =\omega_0 + k_+-1$ and $V$ is an admissible subspace in the definition \eqref{eq:def:lambdak} of $\lambda_{\omega_0 + k_+-1}(\beta)$. Using again \eqref{eqvlim} we then obtain that 
\begin{equation} \label{eq:lambda:bulk}
 \lambda_{\omega_0+k_+- 1}(\beta) \le \sup_{v \in V \setminus \{0\}} \mathcal{R}(\beta,v) \le \lambda_k =   \La_k^s(M,[g]),
 \end{equation}
where $\mathcal{R}(\beta,v)$ is as in \eqref{eq:definitions}. Going back to the definition \eqref{Lasouscritique} of $\Lambda_{\omega_0 + k_+-1}(M,[g])$ we thus obtain  
$$ \La_{\omega_0+k_+-1}^s(M,[g]) \le  
\Lambda_k^s(M,[g])  \left( \int_M \beta^\frac{n}{2s} dv_g \right)^{\frac{2s}{n}}$$
which proves the lemma. 
\end{proof}

We now estimate the energy which is lost \emph{via} bubbling: 
\begin{lemme}  \label{mu(S)}
We have 
 $$\La_k^s (M,[g]) \left(1- \int_M \beta^{\frac{n}{2s}} dv_g\right)^{\frac{2s}{n}}  \geq \Lambda_{k-\omega_0 -k_++1}^s(\mS^n).$$
\end{lemme}
\begin{rem} We have
\begin{equation} \label{intbeta}
0 < \int_M \beta^{\frac{n}{2s}} dv_g< 1.
\end{equation} Indeed, $\beta \not \equiv 0$, and since $L^{\frac{n}{2s}}(M)$ is uniformly convex, the assumption $\Vert \beta \Vert_{L^{\frac{n}{2s}}} = 1$ would imply that $\beta_p$ strongly converges to $\beta$ in $L^{\frac{n}{2s}}(M)$, which would contradict Lemma \ref{lem:Anonvide}. Note also that lemma \ref{mu(S)} does not require the assumption $\omega_0 >0$. 
\end{rem}

\begin{proof}
Let $\delta >0$ small enough so that all the balls $B_g(x,\delta)$, for $x \in A$, are disjoint. Let $\eta \in C^\infty_c([0, + \infty))$ be such that $\eta \equiv 1$ in $[0, \delta]$, $\eta \equiv 0$ in $[2\delta, + \infty)$ and $|\eta^{(r)}| \leq \frac{C}{\de^r}$ for any $r \ge 1$, where $C$ does not depend on $\de$. Let $\Gamma \subseteq \{k_+, \cdots, k \}$ be the set of indexes $i$ such that $\beta^{\frac{1}{2}} u_i \equiv 0$. By definition of $\omega_0$ in \eqref{defom0} we have 
\begin{equation} \label{indicesgamma}
 \sharp \Gamma= k-\omega_0 -k_++1.
\end{equation}
Define, for $i \in \Gamma$, 
\begin{equation} \label{defzip}
z_{i,p}= \chi u_{i,p}, 
\end{equation}
where $u_{i,p}$ is as in \eqref{vivj} and where we have let $\chi =  \sum_{x \in A} \eta \big( d_g(x, \cdot) \big) $. As a first observation, we claim that the following holds: 
if $i,j \in \Gamma$ and $\delta >0$ is fixed we have 
\begin{equation} \label{step11}
\lim_{p \to \frac{n}{2s}} \int_M \beta_p z_{i,p} z_{j,p} dv_g = \delta_{ij}
\end{equation}
and 
\begin{equation} \label{step12}
\lim_{p \to \frac{n}{2s}}\int_M  z_{j,p} P_g^sz_{i,p} dv_g =\lambda_i \delta_{ij}   + \ve(\delta)
\end{equation}
as $p \to \frac{n}{2s}$, where $\ve(\delta)$ denotes a real function such that $\lim_{\delta \to 0} \ve(\delta) = 0$.   

\begin{proof}[Proof of \eqref{step11} and \eqref{step12}.]
Let $i \in \Gamma$. By \eqref{vivj}, \eqref{convlim} and since $\beta^{\frac12} u_i \equiv 0$ we have 
\begin{equation} \label{step1eq1}
\begin{aligned}
\int_{\bigcup_{x \in A} B_g(x,\delta)} \beta_p u_{i,p} u_{j,p} dv_g  & = \delta_{ij} - \int_{M \backslash \bigcup_{x \in A} B_g(x,\delta)} \beta_p u_{i,p} u_{j,p} dv_g \\
& = \delta_{ij}   + o(1) \\
\end{aligned}
\end{equation}
as $p \to \frac{n}{2s}$. Independently, since $z_{i,p} = u_{i,p}$ in $ \bigcup_{x \in A} B_g(x, \delta)$, and again by \eqref{convlim} and since $\beta^{\frac12} u_i \equiv 0$, we have
$$ \begin{aligned}
\int_M \beta_p z_{i,p} z_{j,p} dv_g = \int_{\bigcup_{x \in A} B_g(x, \delta)} \beta_p u_{i,p} u_{j,p} dv_g + o(1) 
\end{aligned} $$
as $p \to \frac{n}{2s}$, so that \eqref{step11} follows from \eqref{step1eq1}. To prove \eqref{step12} we now observe that 
$$ P_g^s z_{i,p} = P_g^s \big( \chi u_{i,p} \big) = \chi P_g^s u_{i,p} +  O \big( \sum_{r+t \le 2s, r \ge 1, t \ge 0} |\nabla ^r \chi| |\nabla^t u_{i,p}| \big), $$
so that by definition of $\chi$ we have 
\begin{equation} \label{step1eq2}
\begin{aligned}
 \int_M z_{j,p} P_g^s & z_{i,p}  dv_g  = \int_M \chi^2 u_{j,p}  P_g^s u_{i,p} dv_g \\
 & + O \Big(\int_{ \bigcup_{x \in A}B_g(x,2\delta) \backslash B_g(x, \delta)}  \sum_{r+t \le 2s, r \ge 1, t \ge 0} \delta^{-r} |\nabla^t u_{i,p}| |u_{j,p}| dv_g \Big). 
\end{aligned} 
 \end{equation}
If $r,t$ are fixed integers with $r \ge1, t \ge 0$ and $r+t \le 2s$ we have, using the definition of $\chi$,  \eqref{convlim} and H\"older's inequality, that 
$$ \begin{aligned}
\delta^{-r}& \int_{ \bigcup_{x \in A}B_g(x,2\delta) \backslash B_g(x, \delta)}  |\nabla^t u_{i,p}| |u_{j,p}| dv_g \\
& = \delta^{-r} \int_{ \bigcup_{x \in A}B_g(x,2\delta) \backslash B_g(x, \delta)}|\nabla^t u_{i}| |u_{j}| dv_g + o(1) \\
& \le \delta^{-r}\sum_{x \in A}   \Vol_g \big(B_g(x, 2\delta)\big)^{\frac{4s - 2t}{2n}} \left(\int_{B_g(x,2\delta)} |\nabla^t  u_{i}|^{\frac{2n}{n-2s + 2t}} dv_g \right)^{\frac{n-2s+2t}{2n}} \\
& \times \left(\int_{B_g(x,2\delta) \backslash B_g(x, \delta)} |u_j|^{\frac{2n}{n-2s}} dv_g  \right)^{\frac{n-2s}{n}} + o(1) \\
& \le C \sum_{x \in A}  \left(\int_{B_g(x,2\delta) \backslash B_g(x, \delta)} |u_j|^{\frac{2n}{n-2s}} dv_g  \right)^{\frac{n-2s}{n}} + o(1)\\
&  \le \ve(\delta) + o(1)
\end{aligned} $$
as $p \to \frac{n}{2s}$, where we used that $2s-t \ge r$. Independently, using \eqref{vivj}, \eqref{eqvq} and \eqref{eqeq}, we get that 
$$  \int_M \chi^2 P_g^s u_{i,p} u_{j,p} dv_g =  \int_M \chi^2 \lambda_j \beta_p u_{i,p} u_{j,p} dv_g + o(1) . $$
Again by \eqref{convlim} and since $\beta^{\frac12} u_i \equiv 0$ we then get that 
$$ \begin{aligned}
 \int_M \chi^2 \lambda_j \beta_p u_{i,p} u_{j,p} dv_g & = \int_{\bigcup_{x \in A} B_g(x,\delta)} \lambda_j \beta_p u_{i,p} u_{j,p} dv_g  + o(1) \\
 & = \lambda_j \delta_{ji} + o(1) 
 \end{aligned} $$
 as $p \to \frac{n}{2s}$, where the last line follows from \eqref{step1eq1}. Combining the previous equations in \eqref{step1eq2} proves  \eqref{step12}.
\end{proof}
We now denote the points in $A$ by $x_1, \cdots, x_r$. We let $y_1, \cdots, y_r \in \mR^n$ be $r$ distinct points of $\mR^n$ and let $\delta >0$ be small enough so that all the balls $B_{\xi}(y_t, 3\delta)$, $1\le t \le r$, are disjoint, where we have denoted by $\xi$ the Euclidean metric. For any $1 \le t \le r$ we let $\phi_t: B_g(x_t, 2\delta) \to \R^n$ be a local chart that defines normal coordinates in $M$ and which sends $z_t$ to $y_t$. For $\delta$ small enough we may assume that $\phi_t\big( B_g(x_t,2 \de) \big) \subset B_\xi(y_t, 3 \de)$.  We consider the map 
\begin{equation} \label{defmapphi}
\phi: \cup_{t=1}^r B_g(x_t, 2\de) \to \cup_{t=1}^r \phi_t\big( B_g(x_t,2 \de) \big)
\end{equation}
defined by $\phi(y) =  \phi_t(y)$ if $y \in B_g(y_t,2 \delta)$. We define, for any $i \in \Gamma$,
$$Z_{i,p}= z_{i,p} \circ \phi^{-1} \; \hbox{ and } \Theta_p= \beta_p \circ \phi^{-1},$$
where $z_{i,p}$ is defined in \eqref{defzip}. $Z_{i,p}$ is well-defined in $\R^n$ since $z_{i,p}$ is compactly supported in $\bigcup_{i=1}^r B_g(x_i,2 \delta)$, and our convention here is that we extend $\Theta_p$ by zero outside of the image of $\phi$. We now claim that the following holds: 
\begin{equation} \label{step21}
\lim_{p \to \frac{n}{2s}} \int_{\mR^n} \Theta_p Z_{i,p} Z_{j,p} dv_\xi = \delta_{ij} + \ve(\delta)
\end{equation}
and 
\begin{equation} \label{step22}
\lim_{p \to \frac{n}{2s}}  \int_{\mR^n} P_\xi^s (Z_{i,p}) Z_{i,p} dv_\xi  =\lambda_i  \delta_{ij} + \ve(\delta),
\end{equation}
where as before $\ve(\delta)$ denotes a real function such that $\lim_{\delta \to 0} \ve(\delta) = 0$.   

\begin{proof}[Proof of \eqref{step21} and \eqref{step22}]
Since $\vp_t$ defines normal coordinates in each ball $B_g(x_t, 2\delta)$, $1 \le t \le r$, we have 
\begin{equation} \label{elementvolume}
\big[(\vp^{-1})^*g\big]_{ij} = \xi_{ij} + O(\delta) \quad \text{ and } \quad dv_{(\vp^{-1})^*g} = \big(1+ O(\delta) \big) dv_{\xi}
\end{equation} in  $\cup_{t=1}^r \phi_t\big( B_g(x_t,2 \de) \big)$, so that \eqref{step21} follows from the $L^{\frac{2n}{n-2s}}(M)$-boundedness of $z_{i,p}$,  \eqref{step11}, \eqref{elementvolume} and a simple change of variables. To prove \eqref{step22} we first recall that $P_{\xi}^s = (\Delta_{\xi})^s$. We express $P_g^s$ in $B_g(x_y, 2 \delta)$ in the normal coordinates defined by $\vp_t$: for any smooth function $u$ compactly supported in $\phi_t\big( B_g(x_t,2 \de) \big) \subset \R^n$ we have
\begin{equation} \label{eq:ajout:1}
 P_{(\vp^{-1})^*g} u  = (\Delta_{(\vp^{-1})^*g} )^s u + Tu, 
 \end{equation}
where $T$ is a differential operator of order at most $2s-1$ whose coefficients converge in $C^0_{loc}(\R^n)$ as $\delta \to 0$ (see for instance \cite{Juhl} or \cite[Proposition 2.1]{MazumdarVetois}). We let in the following 
$$ U_\delta = \cup_{t=1}^r \phi_t\big( B_g(x_t,2 \de) \big).$$
Since the sequence $(z_{i,p})_p$ is uniformly bounded in $H^s(M)$, and since $Z_{i,p}$ is supported in $U_\delta$, a change of variables with \eqref{elementvolume} shows that there exists $C >0$ independent of $p$ and $\delta$ such that $ \Vert \Delta_\xi^{\frac{s}{2}} Z_{i,p} \Vert_{L^2(\R^n)}\le C$
for all $p$. Sobolev's embedding then shows that 
\begin{equation} \label{eq:ajout:2}
 \Vert Z_{i,p} \Vert_{H^s(\R^n)} \le C 
 \end{equation}
for some $C >0$ independent of $p$ and $\delta$. If $Z_i$ denotes the weak limit of $(Z_{i,p})_p$ in $H^s(\R^n)$, which is supported in $U_\delta$, compact embeddings show that 
$Z_{i,p} \to Z_i$ strongly in $H^{s-1}_0(U_\delta)$ as $p \to \frac{n}{2s}$, so that 
$$ \int_{\R^n} T Z_{i,p} Z_{j,p} dv_\xi =  \int_{U_\delta} T Z_{i} Z_{j} dv_\xi + o(1)  = \ve(\delta) + o(1)$$
as $p \to \frac{n}{2s}$. Independently, by \eqref{elementvolume}, \eqref{eq:ajout:2} and \cite[Lemma 9.1]{Mazumdar1}, we have 
$$ \int_{U_\delta} Z_{j,p}  (\Delta_{(\vp^{-1})^*g} )^s Z_{i,p} dv_{(\vp^{-1})^*g} = \int_{U_\delta} Z_{j,p}  \Delta_{\xi}^s Z_{i,p} dv_\xi + \ve(\delta) $$
as $\delta \to 0$. Using again \eqref{elementvolume} and \eqref{eq:ajout:1} then gives, since $P_\xi^s = \Delta_\xi^s$,
$$ \begin{aligned}
\int_M P_g^sz_{i,p} z_{j,p} dv_g & =  \sum_{i=1}^r \int_{\phi_t\big( B_g(x_t,2 \de) \big)} P_{(\vp^{-1})^*g} Z_{i,p} Z_{j,p} dv_{\xi} + \ve(\delta) \\
& =  \int_{\R^n}  Z_{j,p}P_\xi^s Z_{i,p}  dv_\xi+  \ve(\delta) + o(1) 
\end{aligned} $$
as $p \to \frac{n}{2s}$. Equation \eqref{step22} then follows from \eqref{step12}.
 \end{proof}
We can now conclude the proof of Lemma \ref{mu(S)}. We let $W_p = \span_{i \in \Gamma}  \{ Z_{i,p}\}$. For fixed $\delta >0$ small enough, and as $p$ tends to $\frac{n}{2s}$, \eqref{step21} shows that $\dim_{\Theta_p} W_p =  \sharp \Gamma =  k-\omega_0 -k_++1$ by \eqref{indicesgamma}. Hence $W_p$ is an admissible subspace to compute $\Lambda_{k-\omega_0 -k_++1}^s (\mathbb{S}^n)$ and we have  
\begin{equation} \label{step31}
\La_{k-\omega_0-k_++1}^s(\mS^n) \leq \max_{v \in W_p \backslash \{0\}} \mathcal{R}(\Theta_p,v) \left(\int_{\R^n} \Theta_p^{\frac{n}{2s}} dv_\xi \right)^{\frac{2s}{n}},
\end{equation}
where $\mathcal{R}$ is as in \eqref{eq:definitions}. We implicitly used here that $(\mathbb{S}^n \backslash \{p\}, g_0)$ is conformally diffeomorphic to $(\R^n, \xi)$ for any $p \in \mS^n$ via stereographic projection. Recall that the functions $\Theta_p$ and $Z_{i,p}$ are compactly supported in $\R^n$ so there are no integrability conditions at infinity to verify.  Let $(\al_{i,p})_{ i \in \Gamma}$ be such that $\sum_{i \in \Gamma} \al_{i,p}^2=1$ and such that $Z_p: = \sum_{i \in \Gamma} \al_{i,p} Z_{i,p}$ achieves $\max_{v \in W_p \backslash \{0\}} \mathcal{R}(\Theta_p,v)$. Then \eqref{step21} and \eqref{step22} show that 
\begin{equation} \label{step32}
\max_{v \in W_p \backslash \{0\}} \mathcal{R}(\Theta_p,v) = \frac{\int_{\mR^n} Z_p P_\xi^s Z_p dv_\xi}{\int_{\mR^n} \Theta_p Z_p^2 dv_\xi}  \leq \La_k^s(M,[g] )+ \ve(\delta).
\end{equation}
By definition of $\Theta_p$ and \eqref{step22} we have 
$$ \int_{\mR^n} \Theta_p^{\frac{n}{2s}}  dv_\xi = \sum_{i=1}^r \int_{B^g(x_i,2\de)} \be_p^{\frac{n}{2s}} dv_g+ \ve(\delta) = 1 - \int_{M \backslash \bigcup_{x \in A} B_g(x, 2 \delta)} \be_p^{\frac{n}{2s}} dv_g  + \ve(\delta). $$ 
From \eqref{mainthp:10} we thus obtain, since $\beta \in L^{\frac{n}{2s}}(M)$, 
\begin{equation} \label{step33}
\lim_{p \to \frac{n}{2s}}\int_{\mR^n} \Theta_p^{\frac{n}{2s}}  dv_\xi  = 1- \int_{M}\be^{\frac{n}{2s}} dv_g + \ve(\delta).
\end{equation}
Combining \eqref{step32} and \eqref{step33} in \eqref{step31} and letting $\delta \to 0$ concludes the proof of Lemma \ref{mu(S)}.
\end{proof}

We are now in position to prove Theorem \ref{prop:bubbling} when $\beta \not \equiv 0$. Assume first that $\omega_0=0$, where $\omega_0$ is as in \eqref{defom0}. 
Lemma \ref{mu(S)} then shows that $   \La_k^s(M)\geq \Lambda_{k-k_+ +1}^s(\mS^n)$ and concludes the proof. If now $\omega_0 >0$, summing the conclusions of Lemmas  \ref{mu(M)} and \ref{mu(S)} shows that 
$$\La_k^s(M,[g])\geq (\La_{\omega_0 + k_+-1}^s(M,[g])^{\frac{n}{2s}}  + \La_{k-\omega_0-k_+ +1}^s(\mS^n)^{\frac{n}{2s}})^{\frac{2s}{n}}. $$ 
The right-hand side is by definition larger than the invariant $Y_k^s(M,[g])$ (given by \eqref{defYks}) hence we have $\Lambda_k^s(M,[g]) \ge Y_k^s(M,[g])$. Proposition \ref{prop_ineg_large2} shows that the reverse inequality also holds. We thus have 
$$\La_k^s(M,[g])=  (\La_{\omega_0 + k_+-1}^s  (M,[g])^{\frac{n}{2s}}  + \La_{k-\omega_0-k_+ +1}^s(\mS^n)^{\frac{n}{2s}})^{\frac{2s}{n}}. $$ 
As a consequence, the inequality of Lemma \ref{mu(M)} is also an equality. Since $\beta \not \equiv 0$ we obtain by using \eqref{eq:lambda:bulk} that $\La_{\omega_0+ k_+ -1}^s (M,[g])$ is attained by $\be$. This concludes the proof of Theorem \ref{prop:bubbling} in the case where $\beta \not \equiv 0$ and when $(M,g)$ is not conformally diffeomorphic to the round sphere $(\mS^n, g_0)$. 

\medskip

We conclude this subsection by proving that in the case of the standard sphere we can always assume that $\beta \not \equiv 0$. In the following lemma we use the notations introduced in the previous paragraphs
\begin{lemme}  \label{om0sphere}
Assume that $(M,g)$ is conformally diffeomorphic to the round sphere $(\mathbb{S}^n, g_0)$. Without loss of generality we can always assume that $\beta \not \equiv 0$ and that $\om_0 >0$, where $\omega_0$ is as in \eqref{defom0}. 
\end{lemme}

\begin{proof}
We keep the same notations than before. We recall that $P_{g_0}^s$ is coercive (see e.g. \eqref{factor:Pg}) and hence $\ker (P_{g_0}^s) = \{0\}$ and $k_+ = 1$. As the discussion before \eqref{mainthp:9} shows, this implies that all the sequences $(v_{i,p})_p$, for $1 = k_+ \le i \le k$, are bounded in $H^s(M)$. We also have $\beta_p \in C^{0,\alpha}(M)$ for some $0 < \alpha <1$, and this follows from \eqref{mainthp:1} and Remark \ref{rem:continuite}. We let $V_p = \text{Span} \big\{v_{1,p} , \cdots, v_{k,p}\big \}$, so that 
\begin{equation} \label{lakep=} 
 \la_{k,p} = \max_{v \in V_p \backslash \{0\}}  \mathcal{R}(\beta_p, v),
\end{equation}
where $ \mathcal{R}(\beta_p, v)$ is as in \eqref{eq:definitions}. Up to composing $\beta_p$ and $(v_{i,p})_{1 \le i \le k}$ with an isometry of $(\mathbb{S}^n, g_0)$ we can assume that, for any $p$, $\beta_p$ attains its maximum at the north pole $N$. Let $\pi_N : \mS^n \backslash \{-N\} \to \mR^n$  
be the stereographic projection of pole $-N$, that sends $N$ to $0$. Let $\mu_p >0$ be a positive number to be fixed  and let  $h_p : \mR^n \to \mR^n$ be the homothecy defined by $h_p(x)= \mu_p x$. We define $\vp_p = \pi_N^{-1} \circ h_p \circ \pi_N$ and, if $u \in L^{\frac{2n}{n-2s}}(\mathbb{S}^n)$ is any function, 
$$\theta_p(u) = \mu_p^{\frac{n-2s}{2}} u \circ \phi_p. $$
It is well-known that $\phi_p$ is a conformal diffeomorphism of $\mS^n$, and therefore it preserves the quadratic form $G$ defined in \eqref{eq:definitions}. A simple change of variables with \eqref{lakep=} thus gives
\begin{equation} \label{transf_u}
\lambda_{k,p} = \max_{v \in V_p \backslash \{0\}}  \mathcal{R}(\beta_p, v) = \max_{v \in \theta_p(V_p) \backslash \{0\}}  \mathcal{R}\big(\mu_p^{2s} \beta_p\circ \vp_p, v \big)  .
\end{equation}
We let in what follows, for $1 \le i \le k$, $\bar{v}_{i,p} = \theta_p(v_{i,p})$ and $\bar{\beta}_p = \mu_p^{2s} \beta_p\circ \vp_p$.  Using \eqref{mainthp:2} it is straightforward to check that $\bar{v}_{i,p}$ satisfies
\begin{equation} \label{transf_u2}
P_{g_0}^s \bar{v}_{i,p} = \lambda_{i,p} \bar{\beta}_p \bar{v}_{i,p} \quad \text{ in } \mathbb{S}^n.
\end{equation}
We choose
$$\mu_p= \|\beta_p \|_{L^\infty}^{-\frac{1}{2s}}.$$
With this choice of $\mu_p$ the function $\bar{\beta}_p$ is bounded in $L^\infty(\mathbb{S}^n)$ and attains its maximum, of value $1$, at the north pole $N$. We let $\bar{\beta}$ be the weak limit of $(\bar{\beta}_p)_p$ in $L^{\frac{n}{2s}}(M)$, up to a subsequence as $p \to \frac{n}{2s}$. Since $v \in H^s(\mathbb{S}) \mapsto \int_{\mathbb{S}^n} v P_{g_0}^s v dv_{g_0}$ is equivalent to the standard $H^s(\mathbb{S}^n)$ norm and is invariant under conformal transformations, the sequences $(\bar{v}_{i,p})_p$, for $1 \le i \le k$, remain bounded in $H^s(\mathbb{S}^n)$. If we denote their weak limits by $(\bar{v}_i)_{1 \le i \le k}$, \eqref{transf_u2} and standard elliptic theory show that $\bar{v}_{i,p} \to v_i$ in $C^{2s-1}(\mathbb{S}^n)$. 
 We claim that  there exists $i \in \{1,\cdots,k\}$ such that 
 $$\bar{\beta}^{\frac12} \bar{v}_i \not \equiv0.$$
 By \eqref{defom0} this will conclude the proof of the Lemma. To see this, we remark that relation \eqref{mainthp:1} rewrites as 
 $$ \mu_p^{n-2s-2s(p-1)} \bar{\beta}_p^{p-1} = \sum_{i=\ell}^k d_{i,p} \bar{v}_{i,p}^2 \quad \text{ everywhere in } \mathbb{S}^n. $$
This is now a pointwise relation since both $\bar{\beta}_p$ and $\bar{v}_{i,p}$ are continuous in $\mathbb{S}^n$. By evaluating this relation at the north pole $N$ we obtain
\begin{equation}  \label{transf_u3}
 \sum_{i=\ell}^k d_{i,p} \bar{v}_{i,p}(N)^2 =\mu_p^{n-2s-2s(p-1)} . 
 \end{equation}
The condition $\Vert \beta_p \Vert_{L^p} = 1$ implies that $\Vert \beta_p \Vert_{L^\infty} \not \to 0$, and hence there exists $C >0$ independent of $p$ such that $\mu_p \le C$ for all $p$. Since $p > \frac{n}{2s}$ we have $n-2s-2s(p-1) <0$, so that 
$$ \mu_p^{n-2s-2s(p-1)} \ge C^{n-2s-2s(p-1)} = 1 + o(1) $$ 
as $p \to \frac{n}{2s}$. Passing \eqref{transf_u3} to the limit as $p \to \frac{n}{2s}$ then gives 
$$   \sum_{i\in I}^k d_{i} \bar{v}_{i}(N)^2 \ge 1 $$
 where $I$ is as in \eqref{defIandJ}, so that there exists $i_0 \in I$ such that $\bar{v}_{i_0} (N) \neq 0$, and hence $\bar{v}_{i_0} \not \equiv 0$. Now, equation \eqref{transf_u2} also passes to the limit and shows that 
 $$ P_{g_0}^s \bar{v}_{i_0} = \lambda_{i_0} \bar{\beta} \bar{v}_{i_0} \quad \text{ in } \mathbb{S}^n, $$
where $\lambda_{i_0} >0$ by \eqref{lai>0}. Since $ \bar{v}_{i_0} \not \equiv 0$ and since $P_{g_0}^s$ is coercive, integrating by parts shows that $\int_{\mathbb{S}^n} \bar{\beta}\bar{v}_{i_0}^2 dv_{g_0} >0$, and hence that $\bar{\beta}^{\frac12} \bar{v}_{i_0} \not \equiv0$.
\end{proof} 

Lemma \ref{om0sphere} concludes the proof of Theorem \ref{prop:bubbling} in the case where $\beta \not \equiv 0.$

\subsection{Proof of Theorem \ref{prop:bubbling}: the case $\beta \equiv 0$.} We now assume that $\beta \equiv 0$. In this case we cannot guarantee anymore that every sequence $(v_{i,p})_p$, $k_+ \le i \le k$, is bounded in $H^s(M)$. To overcome this issue we work with the sequences $(w_{i,p})_p$ introduced in \eqref{defwip}, that are bounded in $H^s(M)$. We define, for $i \in \{k_+, \cdots, k\}$, 
\begin{equation} \label{defzip2}
z_{i,p}= \chi w_{i,p}, 
\end{equation}
where $w_{i,p}$ is as in \eqref{defwip} and where we have let $\chi =  \sum_{x \in A} \eta \big( d_g(x, \cdot) \big) $. As a first result we show that the family $\{ z_{i,p}, k_+ \le i \le k \}$ is free in $L^2_{\beta_p}(M)$:

\begin{lemme} \label{lem:independent}
Let $V'_p= \span\{z_{i,p}, k_+ \le i \le k \}$. Then, for $p$ close enough to $\frac{n}{2s}$, we have $\dim_{\beta_p} V_p' = k-k_++1$. 
\end{lemme}

\begin{proof}
Assume by contradiction that there is a nontrivial family $(a_{i,p})_{k_+ \le i \le k}$ such that 
$$\sum_{i = k_+}^k  \be_p^{\frac{1}{2}} a_{i,p} z_{i,p}=0 \quad \text{ in } M.$$
Since $z_{i,p} = \chi w_{i,p}$ the latter rewrites as 
$$  \sum_{i = k_+}^k  \be_p^{\frac{1}{2}} a_{i,p} w_{i,p} = \sum_{i = k_+}^k a_{i,p}  (1- \chi) \be_p^{\frac{1}{2}} w_{i,p}. $$
Going back to \eqref{starstar} we write $w_{i,p} = v_{i,p} - k_{i,p}$, with $k_{i,p} \in \ker(P_g^s)$ and with the convention that $k_{i,p} \equiv 0$ if the sequence $(v_{i,p})_p$ was bounded in $H^s(M)$. We thus get 
$$  \sum_{i = k_+}^k  \be_p^{\frac{1}{2}} a_{i,p} v_{i,p}=  \be_p^{\frac{1}{2}}K_ p  + \sum_{i = k_+}^k a_{i,p}  (1- \chi) \be_p^{\frac{1}{2}} w_{i,p}, $$
where we have let $ K_p = \sum_{i = k_+}^k  a_{i,p} k_{i,p} \in \ker(P_g^s)$. 
Let $A_p = \sum_{i=k_+}^k |a_{i,p}|$. We integrate the latter against $\be_p^{\frac{1}{2}}v_{i_0,p}$ for some fixed $i_0 \in \{k_+, \cdots, k\}$. First, integrating \eqref{mainthp:2} against $K_p$ shows  that $\int_M \beta_p K_p v_{i_0,p} dv_g = 0$. Second, Since $\beta_p \to 0$ in $L^{\frac{n}{2s}}_{loc}(M \backslash A)$ we have $(1-\chi) \beta_p^{\frac12} w_{i,p} \to 0$ in $L^2(M)$ for all $i \in \{k_+, \cdots, k\}$. We thus obtain, with \eqref{mainthp:3}, 
$$ a_{i_0,p} = o (A_p) $$
as $p \to \frac{n}{2s}$. Summing over all $i_0$ gives $A_p = o(A_p)$ which proves that all the $a_{i,p}$ are zero for $p$ close enough to $\frac{n}{2s}$, which is a contradiction.
\end{proof}

We now claim that the following relations hold: for any $k_+ \le i,j \le k$,
\begin{equation} \label{bnonborne1}
\lim_{p \to \frac{n}{2s}}\left|  \int_M \beta_p z_{i,p} z_{j,p} dv_g -  \int_M \beta_p w_{i,p} w_{j,p} dv_g \right| = 0
\end{equation}
and 
\begin{equation} \label{bnonborne2}
\lim_{p \to \frac{n}{2s}}\left| \int_M P_g^sz_{i,p} z_{j,p} dv_g -  \int_M P_g^s w_{i,p} w_{j,p} dv_g  \right| = \ve(\delta)
\end{equation}
as $p \to \frac{n}{2s}$, where $\ve(\delta)$ denotes a real function such that $\lim_{\delta \to 0} \ve(\delta) = 0$. 

\begin{proof}[Proof of \eqref{bnonborne1} and \eqref{bnonborne2}]
First, by \eqref{mainthp:10} we have $\beta_p \to 0$ in $L^{\frac{n}{2s}}_{loc}(M \backslash A)$. As a consequence, and since all the sequences $(z_{i,p})_p$ and $(w_{i,p})_p$, for $k_+ \le i \le k$, are bounded in $H^s(M)$, we have 
$$  \begin{aligned} 
\int_M \beta_p z_{i,p} z_{j,p} dv_g & = \int_{\bigcup_{x\in A} B_g(x,\delta)} \beta_p w_{i,p} w_{j,p} dv_g + o(1) \\
& =  \int_{M} \beta_p w_{i,p} w_{j,p} dv_g + o(1) ,
\end{aligned} $$ 
which proves \eqref{bnonborne1}. Arguing as in the proof of \eqref{step12}  we similarly obtain that 
$$\begin{aligned}
 \int_M P_g^sz_{i,p} z_{j,p} dv_g & =  \int_M \chi^2 P_g^s w_{i,p} w_{j,p} dv_g + o(1) + \ve(\delta)  \\
 & = \int_M \lambda_{i,p} \chi^2 \beta_p v_{i,p} w_{j,p} dv_g + o(1) + \ve(\delta)\\
 \end{aligned}  $$
as $p \to \frac{n}{2s}$, where we used \eqref{defwip} and \eqref{mainthp:2} to write that $P_g^s w_{i,p} = \lambda_{i,p} \beta_p v_{i,p}$. Observe now that
$$\begin{aligned}
\int_M \lambda_{i,p} \chi^2 \beta_p v_{i,p} w_{j,p} dv_g& =  \int_M \lambda_{i,p} \beta_p v_{i,p} w_{j,p} dv_g  \\
& + \lambda_{i,p} \int_M (1- \chi^2)  \beta_p v_{i,p} w_{j,p} dv_g \\
 & = \int_M P_g^s w_{i,p} w_{j,p} dv_g +  o(1),
 \end{aligned} $$ 
where in the last line we used \eqref{mainthp:3} and the definition of $\chi$ to write that
 $$ \begin{aligned}
\left|  \int_M (1- \chi^2)  \beta_p v_{i,p} w_{j,p} dv_g \right| &\le \int_{M \backslash \bigcup_{x \in A} B_g(x, \delta)} \beta_p |v_{i,p}| | w_{j,p}|  dv_g  \\
& \le \left( \int_{M \backslash \bigcup_{x \in A} B_g(x, \delta)} \beta_p w_{j,p}^2 \right)^{\frac12} = o(1) 
 \end{aligned} $$
 since $\beta_p \to 0$ in $L^{\frac{n}{2s}}_{loc}(M \backslash A)$. This proves \eqref{bnonborne2}.
\end{proof}

As in the case $\beta \not \equiv 0$, we denote the points in $A$ by $x_1, \cdots, x_r$, we let $y_1, \cdots, y_r \in \mR^n$ be $r$ distinct points of $\mR^n$, we let $\delta >0$ be small enough so that all the balls $B_{\xi}(y_t, 3\delta)$ are disjoint and we let $\vp$ be as in \eqref{defmapphi}. For any $k_+ \le i \le k$ we define
$$Z_{i,p}= z_{i,p} \circ \phi^{-1} \; \hbox{ and } \Theta_p= \beta_p \circ \phi^{-1},$$
where $z_{i,p}$ is defined this time as in  \eqref{defzip2}. $Z_{i,p}$ is well-defined in $\R^n$ since $z_{i,p}$ is compactly supported in $\bigcup_{i=1}^r B_g(x_i,2 \delta)$, and we again extend $\Theta_p$ by zero outside of the image of $\phi$. We now let 
$$W_p' = \span_{k_+ \le i \le k }  \big\{ Z_{i,p} \big\}. $$
By Lemma \ref{lem:independent} and by the definition of $Z_{i,p}$ and $\Theta_p$ we have $\dim_{\Theta_p} W_p' = k-k_++1$, so that $W_p'$ is an admissible subspace to compute $\Lambda_{k-\omega_0 -k_++1}^s (\mathbb{S}^n)$. By H\"older's inequality we have
$$\left( \int_{\R^n} \Theta_p^{\frac{n}{2s}} dv_\xi\right)^{\frac{2s}{n}} \le \text{Vol}_g(M)^{1 - \frac{n}{2sp}} \Vert \beta_p \Vert_{L^p} = 1+ o(1) $$
as $p \to \frac{n}{2s}$, and thus
\begin{equation} \label{stepbzero1}
\La_{k-k_++1}^s(\mS^n) \leq \big( 1 + o(1) \big) \max_{v \in W_p' \backslash \{0\}} \mathcal{R}(\Theta_p,v), 
\end{equation} 
where $\mathcal{R}$ is as in \eqref{eq:definitions}. For any $p > \frac{n}{2s}$ let $(\al_{i,p})_{k_+ \le i \le k}$ be such that $\sum_{i=k_+}^k \al_{i,p}^2=1$ and such that 
$$Z_p= \sum_{i =k_+}^k \al_{i,p} Z_{i,p}$$
achieves the maximum of the right-hand side of \eqref{stepbzero1}. By \eqref{elementvolume}, \eqref{bnonborne1} and since all the sequences $(Z_{i,p})_p$ are bounded in $L^{\frac{2n}{n-2s}}(M)$  we have 
$$\begin{aligned}
 \int_{\R^n} \Theta_p Z_p^2 dv_{\xi} & \ge \big(1 + \ve(\delta) \big) \int_M \beta_p \big(  \sum_{i =k_+}^k \al_{i,p} z_{i,p} \big)^2 dv_g  \\
 & \ge \int_M \beta_p \big(  \sum_{i =k_+}^k \al_{i,p} w_{i,p} \big)^2 dv_g + o(1) + \ve(\delta)
  \end{aligned} $$
as $p \to \frac{n}{2s}$, for some positive constant $C$ independent of $p$. Using \eqref{defwip} we can write 
$$ \sum_{i =k_+}^k \al_{i,p} w_{i,p} = \sum_{i =k_+}^k  \al_{i,p} v_{i,p}  + K_p,$$ 
for some $K_p \in \ker(P_g^s)$. Integrating \eqref{mainthp:2} against $K_p$ gives that for all $i$, $\int_M \be_p v_{i,p} K_p dv_g=0$. Using \eqref{mainthp:3} we then obtain 
$$ \begin{aligned} \int_M \beta_p \big(  \sum_{i =k_+}^k \al_{i,p} w_{i,p} \big)^2 dv_g = 1+ \int_M \be_p K_v^2 dv_g \geq 1, \end{aligned} $$
so that in the end 
\begin{equation} \label{stepbzero2}
  \int_{\R^n} \Theta_p Z_p^2 dv_{\xi} \ge 1 + o(1) + \ve(\delta) 
\end{equation}
holds true as $p \to \frac{n}{2s}$. Independently, arguing as in the proof of \eqref{step22}, we have 
\begin{equation} \label{stepbzero3}
\lim_{p \to \frac{n}{2s}} \left|  \int_{\mR^n} P_\xi^s (Z_{i,p}) Z_{i,p} dv_\xi -  \int_{M} P_g^s (z_{i,p}) z_{i,p} dv_g\right| = \ve(\delta).
\end{equation}
Using \eqref{bnonborne2} and since $P_g^s w_{i,p}=P_g^s v_{i,p}$ and $\int_M \be_p v_{i,p} K_p dv_g=0$, it holds that 
\begin{equation*} 
\begin{aligned} 
 & \int_{M}  \left(  \sum_{i =k_+}^k  \al_{i,p} z_{i,p} \right)  P_g^s  \left(  \sum_{i =k_+}^k  \al_{i,p} z_{i,p} \right) dv_g \\
 & = \int_M  \left(  \sum_{i =k_+}^k  \al_{i,p} v_{i,p} \right) P_g^s\left(\sum_{i =k_+}^k  \al_{i,p} v_{i,p}\right) dv_g + o(1) + \ve(\delta) \\
 &  \leq \Lambda_k^s(M,[g]) + o(1) + \ve(\delta),
\end{aligned} 
\end{equation*}
where the last inequality follows from \eqref{mainthp:2}, \eqref{mainthp:3} and \eqref{cvlambdalimite}. Combining with \eqref{stepbzero3} then gives 
\begin{equation}  \label{stepbzero4}
\begin{aligned} 
 \int_{\mR^n} Z_{p} P_\xi^s Z_{p}  dv_\xi   \leq \Lambda_k^s(M,[g]) + o(1) + \ve(\delta) 
\end{aligned} 
\end{equation}
as $p \to \frac{n}{2s}$. Plugging \eqref{stepbzero2} and \eqref{stepbzero4} in \eqref{stepbzero1} finally proves that 
$$ \La_{k-k_++1}^s(\mS^n) \leq \Lambda_k^s(M,[g]) + o(1) + \ve(\delta) $$
as $p \to \frac{n}{2s}$. Letting $p \to \frac{n}{2s}$ and then $\delta \to 0$ concludes the proof of Theorem \ref{prop:bubbling} in the case $\beta \equiv 0$. $\square$
 
\begin{rem}
The asymptotic analysis that we perform in the proof of Theorem \ref{prop:bubbling} heavily relies on the choice of the minimising sequence $(\beta)_p$ obtained as a minimum of the subcritical invariants $\Lambda_k^{s,p}(M,[g])$. It is unlikely to hold true for a general minimising sequence. 
\end{rem}

\section{Coercive operators and the case of the standard sphere $(\mathbb{S}^n, g_0)$ }  \label{sec:vpsphere}

In this Section we apply Theorem \ref{theo:vp:positives} to the case of the standard sphere $(\mathbb{S}^n, g_0)$, and we prove Theorems \ref{lambda2:sphere} and \ref{gamma}. We first recall some generalities concerning positivity properties of the operators $P_g^s$. We let $(M,g)$ be a Riemannian manifold, $n >2s$ and we let $P_g^s$ be the GJMS operator of order $2s$ that we assume to be coercive as in \eqref{def:coercivite}. We will say in what follows that $P_g^s$ satisfies \emph{the pointwise comparison principle} if, for any $u \in C^{2s}(M)$ such that $P_g^s u \ge 0$, then either $u\equiv 0$ or $u >0$ in $M$. If $P_g^s$ is coercive we have $\ker (P_g^s) = \{0\}$ and we can thus let $G_g^s$ be the Green's function of $P_g^s$. It is the unique smooth function in $M \times M \backslash \{ (x,x), x \in M\}$ that, for any $x \in M$,  satisfies 
 $$ P_g^s G_g^s(x, \cdot) = \delta_x$$
 in $M$.  For $u \in C^{2s}(M)$ we then have, for any $x \in M$,
  $$ u(x) = \int_M G_g^s(s,\cdot) P_g^su \, dv_g. $$
It is easily seen that $P_g^s$ satisfies the pointwise comparison principle if and only if, for any $x \in M$, $G_g^s(x, \cdot)$ is positive in $M \backslash \{x\}$. An easy example of manifolds $(M,g)$ for which $P_g^s$ satisfies the pointwise comparison principle are Einstein manifolds with positive scalar curvature: by \eqref{factor:Pg}, $P_g^s$ indeed decomposes as a product of second-order operators that each satisfy the point-wise comparison principle, and the result follows by induction. The round metric on the sphere falls in this category. When $s=1$ the coercivity of $P_g^s = \Delta_g + \frac{n-2}{4(n-1)}S_g$ is equivalent to the pointwise comparison principle by the maximum principle. When $s \ge 2$, however, it is well-known that coercivity alone does not entail the pointwise comparison property in general. Determining whether $P_g^s$ satisfies the pointwise comparison principle is a difficult question in general.
 
  For operators satisfying the positive comparison principle we prove the following result that complements Proposition \ref{prop:deficontinuityeigen}:
\begin{prop} \label{prop:vp1:simple}
Let $(M,g)$ be a closed Riemannian manifold, $n >2s$ and let $P_g^s$ be the GJMS operator of order $2s$. We assume that $s=1$ or that $P_g^s$ is coercive and satisfies the pointwise comparison principle. Let $0< \alpha <1$ be fixed. Then, for any  $\beta \in C^{0,\alpha}(M) \backslash \{0\}$ with $\beta \ge 0$ in $M$, we have 
$$ \lambda_1(\beta) < \lambda_2(\beta) \quad \text{ and } \quad \dim E_1(\beta) = 1,$$
where $ \lambda_k(\beta)$ and $E_1(\beta)$ are as in \eqref{eq:def:lambdak} and Proposition \ref{prop:deficontinuityeigen}. In addition, if $\vp_1$ is a generator of $E_1(\beta)$, then either $\vp_1 >0$ in $M$ or $\vp_1 <0$ in $M$. 
\end{prop}


Remark that if $P_g^s$ is coercive it satisfies in particular $\ker(P_g^s) = \{0\}$. Under the assumptions of Proposition \ref{prop:vp1:simple}, property \eqref{eq:unique:continuation} is therefore satisfied. 

\begin{proof}
Let $\beta \in C^{0,\alpha}(M) \backslash \{0\}$ with $\beta \ge 0$, for some $0 < \alpha < 1$. The definition of $\lambda_1(\beta)$ in  \eqref{eq:def:lambdak} shows that 
$$ \lambda_1(\beta)  = \inf_{v \in H^s(M), \beta^{\frac12} v \neq 0} \frac{\int_M v P_g^sv \, dv_g}{\int_M \beta v^2 \, dv_g}. $$ 
Elements in $E_1(\beta)$ are exactly the functions $\vp \in H^s(M) \cap L^2_{\beta}(M) \backslash \{0\}$ attaining the infimum in the right-hand side. We first claim that if $\vp$ is nonzero and attains $\lambda_1(\beta)$, then
\begin{equation} \label{eq:signe:constant}
 \text{ either } \vp >0 \text{ in } M \text{ or } \vp < 0 \text{ in } M. 
\end{equation}
We prove \eqref{eq:signe:constant}: let $ \vp \in H^s(M)$ with $\beta^{\frac12} \vp \neq 0$ that attains $\lambda_1(\beta)$. It satisfies in particular 
$$ P_g^s \vp = \lambda_1(\beta) \beta \vp \quad \text{ in } M, $$
and standard elliptic theory thus shows that $\vp \in C^{2s,\alpha}(M)$. Assume first that $P_g^s$ is coercive and satisfies the pointwise comparison principle. We let $\psi \in C^{2s}(M)$ be the unique solution of 
$$ P_g^s \psi = |P_g^s \vp| \quad \text{ in } M. $$
The existence of such a $\psi$ follows from the coercivity of $P_g^s$, and $\psi \in C^{2s,\alpha}(M)$ by standard elliptic theory. By definition we have $P_g^s(\psi \pm \vp) \ge 0$, so that the pointwise comparison property of $P_g^s$ shows that $\psi \ge |\vp|$ everywhere in $M$. Note that  $\beta^{\frac12} \psi \neq 0$, otherwise we would  have $\beta^{\frac12} \vp \equiv 0$. Hence $\psi \neq 0$, and the pointwise comparison principle again shows that $\psi >0$ everywhere in $M$. We now prove that $\psi$ again achieves $\lambda_1(\beta)$: we indeed have 
$$ \begin{aligned}
\lambda_1(\beta) \le \frac{\int_M \psi P_g^s \psi \, dv_g}{\int_M \beta \psi^2 \, dv_g} & = \lambda_1(\beta) \frac{ \int_M \psi  \beta | \vp|  \, dv_g}{\int_M \beta \psi^2 \, dv_g}  \le \lambda_1(\beta) \left( \frac{ \int_M \beta \vp^2 \, dv_g}{\int_M \beta \psi^2 \, dv_g} \right)^{\frac12}  \le \lambda_1(\beta),
\end{aligned} $$ 
where we used H\"older's inequality and since $\psi \ge |\vp|$ everywhere. This shows that $\psi$ attains $\lambda_1(\beta)$. In addition, all these inequalities are equalities: the equality case in H\"older's inequality then shows that $\psi = |\vp|$ everywhere. Since $\psi >0$ in $M$ this proves \eqref{eq:signe:constant} in this case. Assume now that $s=1$. Then up to replacing $\vp$ by $|\vp$, which still attains $\lambda_1(\beta)$, we may assume that $\vp \ge0$ a.e. in $M$. Then \eqref{eq:signe:constant} simply follows from the maximum principle. 

We now prove that $E_1(\beta)$ is one-dimensional. Let $\vp \in E_1(\beta) \backslash \{0\}$. Then $\vp \in C^{2s}(M)$ and up to changing $\vp$ in $-\vp$, we may assume that $\vp >0$ in $M$ by \eqref{eq:signe:constant}. Let $\theta \in E_1(\beta)$, let $x_0 \in M$ and let $t = \frac{\theta(x_0)}{\vp(x_0)}$. The function $\theta - t \vp$ is still in $E_1$ and, since it vanishes at $x$, it is zero everywhere in $M$ by \eqref{eq:signe:constant}. Thus $\theta = t \vp$ and $E_1(\beta)$ is one-dimensional, which also proves that $\lambda_1(\beta) < \lambda_2(\beta)$ thanks to Proposition \ref{prop:deficontinuityeigen}. 
\end{proof}

If $p \ge \frac{n}{2s}$ and $\beta \in \mathcal{A}_p$ we recall the definition of $ \bar{\lambda}_k^p(\beta)$ introduced in \eqref{def:lambdabar}: $ \bar{\lambda}_k^p(\beta) = \lambda_k(\beta) \Vert \beta \Vert_{L^p}$. The following result complements Proposition \ref{euler_minimizer} and establishes situations in which optimisers of $ \bar{\lambda}_k^p(\beta)$ are simple: 

\begin{prop} \label{prop:espace1dim}
Let $(M,g)$ be a closed Riemannian manifold of dimension $n >2s$ and let $P_g^s$ be the GJMS operator of order $2s$. Let $k \ge 1$, $p \ge\frac{n}{2s}$ and $\beta \in \mathcal{A}_p$. Assume that $P_g^s$ satisfies \eqref{eq:unique:continuation} and that one of the following assumptions is satisfied:
\begin{enumerate}
\item $k = k_-$ and $\beta$ is a local maximum of $ \bar{\lambda}_k^p$
\item $k = k_+$ and $\beta$ is a local minimum of $ \bar{\lambda}_k^p$
\item $k=2$, $\beta$ is a local minimum of $ \bar{\lambda}_k^p$ and we assume that either [$s=1$] or [$P_g^s$ is coercive and satisfies the pointwise comparison principle], 
\end{enumerate}
Then $E_k(\beta)$ has dimension $1$ and if  $\vp \in E_k(\beta)$ is such that  $\int_M \beta \vp^2 \, dv_g = 1$, we have 
$$ \beta^{p-1} = \vp^2.$$
As a consequence, $\vp$ solves   
$$ P_g^s \vp = \lambda_k(\beta) |\vp|^{\frac{2}{p-1}} \vp \quad \text{ in } M.$$
Furthermore, this $\vp$ is a truly sign-changing function, in the sense that $\vp_+ \neq 0$ and $\vp_- \neq 0$, in the following cases: 
\begin{itemize}
\item  If assumption (3) is satisfied
\item If $s=1$ and assumptions (1) or (2) are satisfied with, respectively, $k_- \ge2$ or $k_+ \ge 2$. 
\end{itemize}
\end{prop}
Recall that if $P_g^s$ is coercive we have $k_+ = 1$. In the particular case $p = \frac{n}{2s}$, and under the assumptions of Proposition \ref{prop:espace1dim}, the unique $\vp \in E_k(\beta)$ with $Q(\beta, \vp) = 1$ satisfies 
$$ P_g^s \vp =  \lambda_k(\beta)|\vp|^{\frac{4s}{n-2s}} \vp \quad \text{ in } M.$$
In the case $s=1$, and for negative eigenvalues, this was already observed in \cite[Remark 6.1]{GurskyPerez}.

\begin{proof}
We only prove Proposition \ref{prop:espace1dim} when $k \ge k_+$, since the case $k = k_-$ follows from the same arguments. Assume thus that $k \ge k_+$ and, without loss of generality, that $\beta \in \mathcal{A}_p$ is a local minimum of $ \bar{\lambda}_k^p$ with $\Vert \beta \Vert_{L^p} = 1$. Under the assumptions of Proposition \ref{prop:espace1dim}, Proposition \ref{euler_minimizer} applies: it shows that for any linear subspace $V \in \mathcal{G}_{k-i(k)+1}(E_k(\beta))$ there exists $k_+ \le r \le k $ with $\lambda_r(\beta) = \lambda_k(\beta)$, a $Q(\beta,\cdot)$-orthonormal family $(v_{r},\cdots,v_k)$ of $V$ and there exist positive numbers $d_r,\cdots,d_k$ satisfying $\sum_{i=r}^k d_i = 1$ such that 
\begin{equation} \label{eq:EL:dim1}
  \beta^{p-1} = \sum_{i=r}^k d_i v_i^2 \quad \text{ a. e. in } M.
  \end{equation}
Assume first that $k = k_+$. Then $i(k) = k = k_+$ and any linear subspace $V \in \mathcal{G}_{k-i(k)+1}(E_k(\beta)) = \mathcal{G}_{1}(E_k(\beta))$ is spanned by a nonzero function $v \in H^s(M)$ with $Q(\beta, v) = 1$. Equation \eqref{eq:EL:dim1} then applies to $V = \mathbb{R}v$ and shows that $r = k_+$, $d_{r} =1$ and $v_r = v$, so that 
$\beta = |v|^{\frac{2}{p-1}}$ a.e. in $M$. Since $P_g^s v = \lambda_{k_+}(\beta) \beta v$ we obtain that $P_g^s v = \lambda_k(\beta) |v|^{\frac{2}{p-1}} v$ in $M$. 
We have thus shown that 
$\beta^{p-1} = v^2$ for all $v \in E_k(\beta)$ with $\int_M \beta v^2 dv_g=1$. Assume by contradiction that $E_k(\beta)$ has dimension greater or equal than $2$, then  we  can choose $v_i$ ($i=1,2$) in $E_{k}(\be)$ such that 
$$\int_M \beta v_i v_j dv_g = \de_{ij}.$$ 
We obtain that for all $t$ close to $0$ 
$$ \frac{(v_1+t v_2)^2}{ \int_M \beta  (v_1+t v_2)^2 dv_g} = \be^{p-1}.$$ 
Expanding in $t$ shows that $v_1v_2 = 0$ in $M$ which contradicts $v_i^2 = \beta^{p-1}$. 

\medskip

Assume now that $[s=1]$ or $[P_g^s$ is coercive, satisfies the pointwise comparison principle] and $k=2$. Let $\beta \in \mathcal{A}_p, \Vert \beta \Vert_{L^p} = 1$ be a local minimum of $ \bar{\lambda}_2^p$. By the coercivity of $P_g^s$ we have $k_+=1$, and  Proposition \ref{prop:vp1:simple} applies and  shows that $\lambda_1(\beta) < \lambda_2(\beta)$. Thus $i(2)=2$ and any linear subspace $V \in \mathcal{G}_{2-i(2)+1}(E_2(\beta)) = \mathcal{G}_{1}(E_2(\beta))$ is spanned by a nonzero function $v \in H^s(M)$ with $Q(\beta, v) = 1$. Equation \eqref{eq:EL:dim1} then applies to $V = \mathbb{R}v$ and again shows that $\beta = |v|^{\frac{2}{p-1}}$ a.e. in $M$. We conclude that $\dim E_2(\beta) = 1$ as in the previous case. 

\medskip

It remains to prove that, under the assumptions of the last part of Proposition~\ref{prop:espace1dim}, the unique generator $\vp \in E_k(\beta)$ with $Q(\beta, v) = 1$ is a truly sign-changing function. Assume by contradiction that $v \ge 0$ or $v \le 0$ a.e. in $M$. Let $\vp_1 >0$ be a nonzero generator of $E_1(\beta)$ which is given by Proposition \ref{prop:vp1:simple}. Again by Proposition \ref{prop:vp1:simple} we have $\lambda_1(\beta) < \lambda_2(\beta)$, so that $\vp_1$ and $v$ are $B(\beta, \cdot, \cdot)$-orthogonal, that is $0 = \int_{M} \beta \vp_1 v \, dv_g = \int_M \vp_1 v^{\frac{p+1}{p-1}} \, dv_g$. Since $\vp_1 >0$ in $M$ this implies that $v \equiv 0$, which is a contradiction since $\int_M |v|^{\frac{2p}{p-1}} \, dv_g = \int_M \beta^{p}\, dv_g =  1$.
\end{proof}

With Proposition \ref{prop:espace1dim} we can now prove Corollaries \ref{corol:k:moins} and \ref{corol:k:plus}:

\begin{proof}[Proof of Corollaries \ref{corol:k:moins} and \ref{corol:k:plus}]
We only prove the case $k = k_+$ since the case $k = k_-$ is similar. Under the assumptions of Corollary \ref{corol:k:plus}, $\Lambda_{k_+}^s(M,[g])$ is attained. If $ \beta \in L^{\frac{n}{2s}}(M) \backslash \{0\}$, with $\Vert \beta \Vert_{L^{\frac{n}{2s}}} = 1$, attains  $\Lambda_{k_+}^s(M,[g])$. Proposition \ref{prop:espace1dim} then applies and shows that $\dim E_{k+}(\beta) = 1$ and that, if we denote by $\vp$ the unique function in $E_k(\beta)$ with $Q(\beta, \vp) = 1$, we have $\beta = |\vp|^{\frac{4s}{n-2s}}$. Also, $\vp$ satisfies 
$$ P_g^s \vp =  \Lambda_{k_+}^s(M,[g])|\vp|^{\frac{4s}{n-2s}} \vp \quad \text{ in } M.$$
That $\vp$ changes sign when $s=1$ and $k_+ \ge 2$ also follows from  Proposition \ref{prop:espace1dim}.
\end{proof} 

Similarly, we are now in position to prove Theorem \ref{lambda2:sphere}:

\begin{proof}[Proof of Theorem \ref{lambda2:sphere}]
First, Theorem \ref{theo:vp:positives} shows that 
\begin{equation} \label{contra:sphere:0}
\Lambda_2^s(\mathbb{S}^n)^{\frac{n}{2s}} \le 2 \Lambda_1^s(\mathbb{S}^n)^{\frac{n}{2s}}. 
\end{equation}
We proceed by contradiction and assume that 
\begin{equation} \label{contra:sphere}
\text{ either } \quad \Lambda_2^s(\mathbb{S}^n)^{\frac{n}{2s}} < 2 \Lambda_1^s(\mathbb{S}^n)^{\frac{n}{2s}} \quad \text{ or } \quad \Lambda_2^s(\mathbb{S}^n) \quad \text{ is attained. } 
\end{equation}
Again by Theorem \ref{theo:vp:positives}, assumption \eqref{contra:sphere} implies that $\Lambda_2^s(\mathbb{S}^n)$ is attained at some function $\beta \in \mathcal{A}_{\frac{n}{2s}}$. Since $P_{g_0}^s$ is coercive and satisfies the pointwise comparison principle by \eqref{factor:Pg}, case $(3)$ of Proposition \ref{prop:espace1dim} applies and shows that there exists a sign-changing $\vp \in H^s(\mathbb{S}^n)$, with $\int_{\mathbb{S}^n} |\vp|^{\frac{2n}{n-2s}} \, dv_{g_0} = 1$, such that 
\begin{equation} \label{contra:sphere:1bis} 
P_{g_0}^s \vp = \Lambda_2^s(\mathbb{S}^n) |\vp|^{\frac{4s}{n-2s}} \vp \quad \text{ in } \mathbb{S}^n.
\end{equation}
Integrating against $\vp$ then shows, together with \eqref{contra:sphere:0}, that 
\begin{equation} \label{contra:sphere:2}
\int_{\mathbb{S}^n} \vp P_{g_0}^s \vp \, dv_{g_0} = \Lambda_2^s(\mathbb{S}^n) \le 2^{\frac{2s}{n}} \Lambda_1^s(\mathbb{S}^n)^{\frac{n}{2s}}.
\end{equation}
We now use the conformal invariance of $P_{g_0}$. We let $N \in \mathbb{S}^n$ and let $\pi_N$ denote the stereographic projection that sends $-N$ to $0$. It is well-known that $(\pi_N^{-1})^* g_0(x) = U(x)^{\frac{4s}{n-2s}} \xi$, where $\xi$ is the Euclidean metric in $\R^n$ and where we have let $ U(x) = \left(\frac{2}{1+|x|^2}\right)^{\frac{n-2s}{2s}}$. We define, for $x \in \R^n$, 
$$\tilde{\vp}(x) =  \Lambda_2^s(\mathbb{S}^n)^{\frac{n-2s}{4s}}\frac{\vp}{U}\big( \pi_N^{-1}(x) \big).$$ Using \eqref{contra:sphere:1bis} and \eqref{contra:sphere:2}, the conformal invariance property \eqref{eq:confinv} of $P_{g_0}$ shows that $\tilde{\vp}$ satisfies 
\begin{equation} \label{contra:sphere:3}
\begin{aligned}
\Delta_\xi^s \tilde{\vp} & = |\tilde \vp|^{\frac{4s}{n-2s}}  \tilde{\vp} \quad \text{ in } \R^n \quad \text{ and }\\
\int_{\R^n} \big|\Delta_\xi^{\frac{s}{2}} \tilde{\vp} \big|^{2} \, dx & = \int_{\mathbb{S}^n} \vp P_{g_0}^s \vp dv_{g_0} =   \Lambda_2^s(\mathbb{S}^n) ^{\frac{n}{2s}} \le 2 K_{n,s}^{-\frac{n}{s}},
\end{aligned}
\end{equation}
where $K_{n,s}$ is the optimal Sobolev constant defined by \eqref{defKns} and where we used that $K_{n,s}^{-2} = \Lambda_1^s(\mathbb{S}^n)$, which follows from \eqref{eq:Lambda1:sphere}. Equation \eqref{contra:sphere:3} then contradicts the result of Lemma \ref{lemme:energie:Rn} in the appendix. Hence \eqref{contra:sphere} cannot occur, which together with \eqref{contra:sphere:0} concludes the proof of Theorem \ref{lambda2:sphere}.
\end{proof}




We now prove Theorem \ref{gamma}:

\begin{proof}[Proof of Theorem \ref{gamma}]
Let $s \ge 1$ and assume that $n \ge 2s+5$. Recall that by \eqref{factor:Pg} $P_{g_0}^s$ writes as 
\begin{equation} \label{gamma:0}
 P_{g_0}^{s} = \prod_{j=1}^s \Big( \Delta_{g_0} + \frac{(n+2j-2)(n-2j)}{4} \Big). 
 \end{equation}
We proceed by contradiction and assume that, for any $3 \le k \le n+1$, $\Lambda_k^s(\mathbb{S}^n)$ is not attained. It then follows from Theorem \ref{theo:vp:positives} for $k=n+1$ that 
$$\La_{n+1}^s(\mS^n)^{\frac{n}{2s}}= \La_{\ell_0}^s(\mS^n)^{\frac{n}{2s}} + \cdots+ \La_{\ell_r}^s(\mS^n)^{\frac{n}{2s}}$$
for some indexes $\ell_0,\cdots,\ell_r \ge 1$ such that $\ell_0 + \cdots + \ell_r=n+1$ and such that $\La_{\ell_i}^s(\mS^n)$ is attained for all $0 \le i \le r$. By Theorem \ref{lambda2:sphere} and by our contradiction assumption we then have $r=n$ and $\ell_i=1$ for all $i$ which implies 
that 
\begin{equation} \label{gamma:1} 
\La_{n+1}^s(\mS^n)= (n+1)^{\frac{2s}{n}} \La_1^s(\mS^n) = (n+1)^{\frac{2s}{n}} K_{n,s}^{-2},
\end{equation}
where $K_{n,s}$ is as in \eqref{defKns} and where we again used \eqref{eq:Lambda1:sphere}. It is shown in \cite{LiebSharpConstants} that the explicit value of $K_{n,s}$ is given by 
$$ K_{n,s}^{-2} = 2^{2s} \pi^s \frac{\Gamma\left( \frac{n+2s}{2}\right)}{\Gamma\left( \frac{n-2s}{2}\right)} \left( \frac{\Gamma\left( \frac{n}{2}\right)}{\Gamma(n)} \right)^{\frac{2s}{n}}, $$
where $\Gamma$ denotes the Gamma function. It is shown in \cite{AubinSobolev, Talenti} that $K_{n,1}^{-2} =  \frac{n(n-2)}{4} \omega_n^{\frac{2}{n}}$ where $\omega_n$ is the volume of $(\mathbb{S}^n, g_0)$. 
We deduce from the latter and standard properties of the Gamma function that 
$$\omega_n^{\frac{2s}{n}}=2^{2s}\pi^s   \left( \frac{\Gamma\left( \frac{n}{2}\right)}{\Gamma(n)} \right)^{\frac{2s}{n}},$$
so that $K_{n,s}^{-2} = \frac{\Gamma\left( \frac{n+2s}{2}\right)}{\Gamma\left( \frac{n-2s}{2}\right)} \omega_n^{\frac{2s}{n}}$. Standard properties of the Gamma function finally show that 
\begin{equation} \label{gamma:2} 
K_{n,s}^{-2} =  \frac{\omega_n^{\frac{2s}{n}}}{2^{2s}} \prod_{j=-s}^{s-1} (n+2j) . 
\end{equation}
We now estimate the left-hand side of \eqref{gamma:1} from above. First, the constant function $1$ is an eigenfunction of $P_{g_0}^s$ associated to the eigenvalue $ \frac{1}{2^{2s}} \prod_{j=1}^s (n+2j-2)(n-2j)$ by \eqref{gamma:0}. The coordinate functions $x_i$ of $\mR^{n+1}$, for $1 \le i \le n+1$, when restricted to $\mS^n$ are eigenfunction of $\Delta_{g_0}$ associated to the eigenvalue $n$. As a consequence, and again by \eqref{gamma:0}, they are also eigenfunctions of $P_{g_0}^s$ associated to the eigenvalue 
\begin{equation} \label{def:lambda0:kn}
 \lambda_0 =  \frac{1}{2^{2s}}  \prod_{j=1}^s \Big( 4n + (n+2j-2)(n-2j) \Big) = \frac{1}{2^{2s}}  \prod_{j=1}^s (n-2j+2)(n+2j)  . 
 \end{equation}
By definition of $\La_{n+1}^s(\mS^n)$ we thus have $\lambda_0 \omega_n^{\frac{2s}{n}} \ge \La_{n+1}^s(\mS^n)$. Using \eqref{gamma:1}, \eqref{gamma:2} and the expression of $\lambda_0$, the latter inequality becomes 
\begin{equation} \label{gamma:3}
\frac{n+2s}{n-2s}  = \frac{1+ \frac{2s}{n}}{1 - \frac{2s}{n}}\ge (n+1)^{\frac{2s}{n}}. 
\end{equation} 
For $\alpha \in [0,1)$ we define 
$$ F(\alpha) = \frac{1+\alpha}{1-\alpha} - (n+1)^{\alpha}. $$
Simple considerations show that $F$ is continuous and that $F >0$ for $\alpha$ close to $1$. As $\alpha \to 0$ we have $F(\alpha ) = \big( 2 - \ln(n+1) \big) \alpha + o(\alpha)$, so that $F(0)=0$ and  $F(\alpha)<0$ for $\alpha >0$ small, since the assumption $n \ge 2s+5$ implies in particular $n \ge 7$ and thus $\ln(n+1) >2$. As a consequence, $F$ vanishes in $(0,1)$. We denote by $\alpha_0$ the smallest positive zero of $F$ in $(0,1)$. Inequality  \eqref{gamma:3}  shows that $\frac{2s}{n} \ge \alpha_0$. Straightforward computations show that, for all  $\alpha \in (0,1)$,
\begin{equation} \label{gamma:4}
\begin{aligned}
 F'(\alpha) & = \frac{2 - \ln(n+1)(1-\alpha^2)}{(1-\alpha)^2} + \ln(n+1) F(\alpha) \quad  \text{ and } \\
 F''(\alpha) & = \frac{4 - \ln(n+1)^2(1-\alpha)(1-\alpha^2)}{(1-\alpha)^3} + \ln(n+1)^2 F(\alpha). 
 \end{aligned} 
 \end{equation}
Since $F<0$ in $[0, \alpha_0)$ we have $F'(\alpha_0) \ge 0$, which implies $1-\alpha_0^2 \le \frac{2}{\ln(n+1)}$. The latter implies, since $\alpha_0 >0$, that $1- \alpha_0 < \frac{2}{\ln(n+1)}$ and thus that $4 - \ln(n+1)^2(1-\alpha_0)(1-\alpha_0^2) >0$, which in turn implies that $F''(\alpha_0) >0$. Thus $F$ changes sign at $\alpha_0$ and becomes positive in a right-neighbourhood of $\alpha_0$. Any $\alpha > \alpha_0$ now satisfies $\alpha^2 > \alpha_0^2 \ge 1 - \frac{2}{\ln(n+1)}$, so that \eqref{gamma:4} shows that $F$ is increasing in $(\alpha_0, 1]$ and hence that $\alpha_0$ is the only zero of $F$ in $(0,1)$.

Assume first that $n=7$, in which case $s=1$. Since $F(\frac13) = 0$ we have $\alpha_0 = \frac13$, and inequality  \eqref{gamma:3} shows that $\frac{2}{n} \ge \frac13$ which is a contradiction. We may now assume that $n \ge 8$ (and $n \ge 2s+5$). We have 
$$ F\big (1 -  \frac{4}{n} \big) = \frac{n-2}{2} - (n+1)^{1-\frac{4}{n}} \le 0,$$
which in turn implies that $\alpha_0 \ge 1 - \frac{4}{n}$. Inequality \eqref{gamma:3} then shows that 
$$\frac{2s}{n} \ge \alpha_0 \ge 1 - \frac{4}{n}, $$
which contradicts the assumption $n \ge 2s+5$. This proves that there exists $k \in \{3, \cdots, n+1\}$ such that $\Lambda_k^s(\mathbb{S}^n)$ is attained.

\medskip

We let $k \in \{3, \cdots, n+1\}$ be such that $\Lambda_k^s(\mS^n)$ is attained. We proceed by contradiction and assume that it is attained by the round metric. By this we mean that $\La_k^s(\mS^n)=  \la_k(P_{g_0}^s) \omega_n^{\frac{2s}{n}}$, or in other words that $\Lambda_k^s(\mS^n)$ is equal to $\lambda_k(\beta)$, where $\beta_0$ is the constant function equal to $\omega_n^{- \frac{2s}{n}}$ that satisfies $\Vert \beta_0 \Vert_{\frac{n}{2s}} = 1$. We first observe that the eigenfunctions of $P_{g_0}^s$ in $L^2(\mathbb{S}^n)$ are still given by the restriction to $\mathbb{S}^n$ of homogeneous harmonic polynomials in $\R^{n+1}$. Indeed, these polynomials are still eigenfunctions of $P_{g_0}^s$ by \eqref{gamma:0}, and they form a Hilbert basis of $L^2(\mathbb{S}^n)$ since they are the eigenfunction of $\Delta_{g_0}$. They therefore exhaust all the eigenfunctions of $P_{g_0}^s$ in  $L^2(\mathbb{S}^n)$ by the spectral theorem. The coordinate functions $x_i$ of $\mR^{n+1}$, for $1 \le i \le n+1$, when restricted to $\mS^n$ are the eigenfunction of $\Delta_{g_0}$ associated to the least positive eigenvalue $n$. Since $k \le n+1$ they thus also span the eigenspace associated to $\lambda_k(\beta_0)$ and we have 
$$\Lambda_k^s(\mS^n) = \prod_{j=1}^s \left( n + \frac{(n+2j-2)(n-2j)}{4} \right)  \omega_n^{\frac{2s}{n}}.$$
Since $\beta_0$ attains $\La_k^s(\mS^n)$, Proposition \ref{euler_minimizer} applies and shows that there exist nonnegative numbers $d_1  \cdots, d_{n+1} $, satisfying  $\sum_{i=1}^{n+1} d_i = 1$ and at most $k-1$ of which are non-zero, such that 
$$\beta_0 = \om_n^{-\frac{n-2s}{n}} = \sum_{i=1}^{n+1} d_i x_i^2 \quad \text{ a. e. in } \mS^n.$$
Since the left-hand side is constant over $\mS^n$ and $k-1 < n+1$ we obtain a contradiction. 
\end{proof}

Let  $k$ be the integer in the statement of Theorem~\ref{gamma}. The inequality $3 \le k \le n+1$ is valid for any $n \ge 2s+5$, and it can be easily verified that it is sharp for low values of $n$ and $s$. As $n$ gets large, however, we obtain a much better bound on $k$ that does not depend on $s$:

\begin{prop} \label{gamma:meilleur}
Let $s \ge 1$ be fixed and assume that $n \ge 2s + 5$. Let 
$$k_n = \min \Big \{ k  \in \{3, \cdots, n+1\}, \quad \Lambda_k^s(\mathbb{S}^n) \text{ is attained } \Big \}.$$ As $n \to + \infty$, we have 
$$ \limsup_{n \to + \infty} k_n \le \big[ e^2 \big] + 1$$
where $[\cdot ]$ denotes the integer part.  
\end{prop}

\begin{proof}
We adapt the argument in the proof of Theorem~\ref{gamma} by keeping the same notations. By definition of $k_n$ and by Theorem \ref{lambda2:sphere}, $\La_{k_n-1}^s(\mS^n)$ is not attained. It then follows from Theorem \ref{theo:vp:positives}  applied to $k=k_n-1$ and from Theorem \ref{lambda2:sphere} that 
$$\La_{k_n-1}^s(\mS^n)^{\frac{n}{2s}}= (k_n-1)^{\frac{2s}{n}} K_{n,s}^{-2},$$
where $K_{n,s}$ is as in \eqref{defKns} and where we again used \eqref{eq:Lambda1:sphere}. Since $k_n \le n+1$ by Theorem~\ref{gamma} we again have $\La_{k_n-1}^s(\mS^n) \le  \La_{n+1}^s(\mS^n)\le \lambda_0 \omega_n^{\frac{2s}{n}}$  where $\lambda_0$ is given by \eqref{def:lambda0:kn}. Arguing as in the proof of Theorem~\ref{gamma} we thus obtain 
\begin{equation*} 
(k_n-1)^{\frac{2s}{n}} \le \frac{n+2s}{n-2s} \quad \text{ and thus } \quad k_n \le 1 +\left( \frac{n+2s}{n-2s}\right)^{\frac{n}{2s}}. 
\end{equation*} 
Expanding the right-hand side as $n \to + \infty$ yields
$$ k_n \le 1 +e^2 + O\big( \frac{1}{n} \big) $$
and proves the result.
\end{proof}

\section{Strict monotonicity of the sequence $(\Lambda_k^s(M,[g]))_k$ }  \label{sec:monotonie}

In this short section we prove Theorem~\ref{incr_prop}. We obtain this result by induction, as a very natural consequence of our main existence result, Theorem~\ref{theo:vp:positives}, and of the variational theory for the eigenvalue functionals that we developed in Section~\ref{sec:theorievariationnellevp}. We first prove the following Lemma:

\begin{lemme} \label{tech_lem} Let $\Omega$ be a non-empty set and $K$ a field. We denote by $\mathcal{F}(\Omega,K)$ the vector space of $K$-valued  functions on $\Om$. Let  $\ell \geq 1$. Assume that there exists 
\begin{itemize}
    \item  $\beta \in  \mathcal{F}(\Omega,K)$; 
    \item a subspace $F \subset  \mathcal{F}(\Omega,K)$ of dimension $\ell+1$ such that for all subspace $V \subset F$ of dimension $\ell$, there exists a finite family $f_1,\cdots,f_r \in V$ such that 
    $$\beta=\sum_{i=1}^{r} f_i^2.$$ 
\end{itemize}
Then $\beta=0$. 
\end{lemme} 
\begin{proof}
Let $x \in \Om$ and let $e:F \to K$ be the evaluation at $x$ defined for all $f \in F$ by $e(f)=f(x)$. 
It is a linear map whose kernel has dimension  $\ell$ if $e\not=0$ and $\ell+1$ otherwise. In any case, we can choose $V \subset \ker(e)$ of dimension $\ell$.
From the assumptions of Lemma \ref{tech_lem}, there exists $f_1,\cdots,f_r \in V$ such that 
$$\beta=f_1^2 + \cdots +f_r^2.$$ 
By construction, it holds that for  all $f \in V$, $v(x)=0$. In particular, for all $i=1,\cdots,r$, $f_i(x)=0$ which implies 
$\beta(x)=0$. Since $x$ is arbitrary, this proves Lemma \ref{tech_lem}.
\end{proof}

When applied to generalised eigenfunctions, Lemma~\ref{tech_lem} yields the following result:

\begin{lemme} \label{corstrictk+1k} Let $(M,g)$ be a closed Riemannian manifold of dimension $n >2s$ and let $P_g^s$ be the GJMS operator of order $2s$. Assume that $P_g^s$ satisfies \eqref{eq:unique:continuation}. We assume that $\Lambda_k^s(M,[g])$ is attained at some $\beta \in L^{\frac{n}{2s}}(M) \backslash \{0\}, \beta \ge 0$. Then $\lambda_{k+1}(\beta) > \lambda_k(\beta)$ if $k \geq k_+$ and $\lambda_{k-1}(\beta) < \lambda_k(\beta)$ if $k \leq k_-$
\end{lemme}

We recall that $k_-$ and $k_+$ are defined in \eqref{defkk}. They depend on $(M,[g])$ but for the sake of clarity we simply denote them by $k_-, k_+$. 

\begin{proof}
We only prove the result for $k\geq k_+$. The other case is similar. Up to renormalizing $\beta$, we may assume that  $\Vert \beta \Vert_{L^{\frac{n}{2s}}}=1$ so that $\lambda_k(\beta) = \Lambda_{k}^s(M,[g])$. We assume by contradiction that $\lambda_{k+1}(\beta) = \lambda_k(\beta)$. We let $i(k)$ be given by  \eqref{def:petitgrand:i}, so that 
$$i(k)= \min\left\{ r \geq k_+ ; \la_r (\beta)= \la_{k}(\beta)\right\}.$$
By Proposition~\ref{prop:deficontinuityeigen} we can choose a  linearly independent family $(v_{i(k)},\cdots,v_{k+1}) \in H^s(M)$ of generalised eigenfunctions such that for all $i$, 
$$P_g^s v_i = \la \beta v_i$$ 
where $\la=\la_i(\beta)$ for $i=\{i(k), \cdots, k+1\}.$ Set $F= \span\{v_{i(k)},\cdots,v_{k+1}\}$. Let $\ell=k+1-i(k) \ge 1$ so that $\dim(F)=\ell+1$. 
Let $V$ be any $\ell$-dimensional subspace of $F$ and choose a $Q(\beta, \cdot)$-orthonormal family $v'_1,\cdots,v'_{\ell}$ in $V$. Then since  $\Lambda^s_{k}(M,[g]) = \la_{k}(\beta)=\la$ and since 
for all $1 \le i \le \ell$, $$P_g^s v'_i = \la \beta v'_i$$
we get from Proposition \ref{euler_minimizer} that there exists a family 
$f_1,\cdots,f_\ell \in V$ such that 
\begin{equation} \label{strictecroissance:1}
\beta^{\frac{n}{2s}-1}= f_1^2 + \cdots+f_\ell^2 \quad \text{a.e. in } M.
\end{equation}
The arguments given in Remark~\ref{rem:continuite} show, since $\beta$ attains $\lambda_k(\beta)$, that $\beta$ and the functions $(f_i)_{1 \le i \le \ell}$ appearing in \eqref{strictecroissance:1} are continuous in $M$. Equality \eqref{strictecroissance:1} therefore becomes a pointwise equality between functions in $M$, and Lemma \ref{tech_lem}  applies: it shows that that $\beta=0$, which contradicts the assumption $\Vert \beta \Vert_{L^{\frac{n}{2s}}} =1$.
\end{proof}

We are in position to prove Theorem \ref{incr_prop}

\begin{proof} [Proof of Theorem \ref{incr_prop}]
We proceed by contradiction and assume that there exists a closed Riemannian manifold $(M,g)$ of dimension $n >2s$ such that $P_g^s$ satisfies \eqref{eq:unique:continuation} and $k$ such that 
\begin{equation}\label{incr_prop:0}
\Lambda^s_k(M,[g])=\Lambda^s_{k+1}(M,[g])\not=0.
\end{equation} 
Then by Corollary \ref{Lambdak=0}, it holds that $k+1 \leq k_-$ or $k \geq k_+$. \\ 

We first assume that $k+1 \leq k_-$. By Theorem \ref{theo:vp:negatives}, $\Lambda_{k}^s(M,[g])$ is attained at some $\beta \in L^{\frac{n}{2s}}(M) \backslash \{0\}, \beta \ge 0$, with $\Vert \beta \Vert_{L^{\frac{n}{2s}}}=1$. By definition of $\Lambda^s_{k+1}(M,[g])$ we have
$$\Lambda^s_{k+1}(M,[g]) \geq \la_{k+1}(\beta) \geq \la_{k}(\beta)=\Lambda^s_{k}(M,[g])=\Lambda^s_{k+1}(M,[g]). $$ 
Then, $\la_{k}(\beta) =\la_{k+1}(\beta)$ and $\Lambda^s_{k+1}(M,[g])$ is attained at $\beta$, which contradicts Lemma \ref{corstrictk+1k}. 

From now on, we assume that $k\geq k_+$. Assume first that $\Lambda_{k+1}^s(M,[g])$ is attained at some $\beta \in L^{\frac{n}{2s}}(M) \backslash \{0\}, \beta \ge 0$, with $\Vert \beta \Vert_{L^{\frac{n}{2s}}}=1$. We have, again by definition of $\Lambda^s_{k}(M,[g])$,
\begin{equation} \label{incr_prop:1}
\Lambda^s_{k}(M,[g]) \leq \la_{k}(\beta) \leq \la_{k+1}(\beta)=\Lambda^s_{k+1}(M,[g])=\Lambda^s_{k}(M,[g]). 
\end{equation}
This implies at once that $\Lambda^s_{k}(M,[g]) = \la_{k}(\beta)$, and thus that $\Lambda^s_{k}(M,[g])$ is also attained, and that  $\la_{k}(\beta) =\la_{k+1}(\beta)$. This again contradicts Lemma \ref{corstrictk+1k}. 

We may thus finally assume that $\Lambda_{k+1}^s(M,[g])$ is not attained. By Theorem \ref{theo:vp:positives}, there exist integers $\ell_0, \cdots,\ell_r$ such that 
\begin{enumerate}
\item  $\ell_0 \in \{ 0 \} \cup \{ k_+  ,\cdots, k \}$, where by convention we let $\Lambda_0^s(M,[g])= 0$;
 \item $\ell_0  + \cdots + \ell_r =k+1$ if $\ell_0 \geq k_+ $ and $\ell_1+ \cdots + \ell_r= k-k_+ +2$ if $\ell_0 =0$ ; 
 \item $\La_{\ell_i}(\mS^n)$, $1 \le i \le r$ are attained, and $\La_{\ell_0}^s(M,[g])$ is attained if $\ell_0 >0$.   
 \end{enumerate} 
 and 
\begin{equation} \label{lamell0}
 \Lambda^s_{k+1}(M,[g])^{\frac{n}{2s}}= \La_{\ell_0}^s(M,[g])^{\frac{n}{2s}} + \La_{\ell_1}^s(\mS^n)^{\frac{n}{2s}} +  \cdots + 
 \La_{\ell_r}^s(\mS^n)^{\frac{n}{2s}} .\end{equation}
Since $\ell_0 < k+1$ we have $r\geq 1$, and in particular $\ell_1 \geq 1$ and thus $\La_{\ell_1'}^s(\mS^n) >0$. Combining \eqref{incr_prop:0} and \eqref{lamell0} we thus obtain 
$$ \Lambda_k^s(M,[g])^{\frac{n}{2s}} =   \La_{\ell_0}^s(M,[g])^{\frac{n}{2s}} + \La_{\ell_1}^s(\mS^n)^{\frac{n}{2s}} +  \cdots + 
 \La_{\ell_r}^s(\mS^n)^{\frac{n}{2s}} > \La_{\ell_0}^s(M,[g])^{\frac{n}{2s}}, $$
 which in particular implies that 
 \begin{equation} \label{incr_prop:2}
\ell_0 \in \{ 0 \} \cup \{ k_+  ,\cdots, k -1\}.
\end{equation}
For any $0 \le i \le $ we now define 
$$\ell_0' = \ell_0, \quad \ell_1'=\ell_1-1 \quad \text{ and } \quad \ell_i'=\ell_i \quad \text{ if } 2 \le i \le r, $$ 
so that by \eqref{incr_prop:2} we have 
\begin{enumerate}
\item  $\ell_0' \in \{ 0 \} \cup \{ k_+  ,\cdots, k-1 \}$ ;
 \item $\ell_0'  + \cdots + \ell_r' =k$ if $\ell_0 \geq k_+ $ and $\ell_1'+ \cdots + \ell_r'= k-k_+ +1$ if $\ell_0 =0$. 
 \end{enumerate} 
 We deduce from Proposition \ref{prop_ineg_large2} that 
$$\Lambda^s_{k}(M,[g])^{\frac{n}{2s}} 
 \leq  \La_{\ell_0'}^s(M,[g])^{\frac{n}{2s}} + \La_{\ell_1'-1}^s(\mS^n)^{\frac{n}{2s}} +  \cdots + 
 \La_{\ell_r'}^s(\mS^n)^{\frac{n}{2s}}.$$
 Together with \eqref{incr_prop:0} and \eqref{lamell0} this implies that $\La_{\ell_1'}^s(\mS^n) \geq \La_{\ell_1}^s(\mS^n)$ from which we deduce that
\begin{equation} \label{eq:ell1} \La_{\ell_1 - 1}^s(\mS^n) = \La_{\ell_1}^s(\mS^n). \end{equation}
If $\Lambda_{\ell_1}^s(\mathbb{S}^n)$ were attained, we would get a contradiction as in \eqref{incr_prop:1}  by using \eqref{eq:ell1} and Lemma \ref{corstrictk+1k}. Thus $\Lambda_{\ell_1}^s(\mathbb{S}^n)$ is not attained, which contradicts the definition of $\ell_1$. This completes the proof of Theorem \ref{incr_prop}.
\end{proof}

\section{Estimating the invariants $\Lambda_k^s(M,[g])$} \label{sec:fonctionstest}

In this section we perform test-function computations to prove that, in some cases, the strict inequality in Theorem \ref{theo:vp:positives} is satisfied, and thus that the invariant $\Lambda_k^s(M,[g])$ is attained at a generalised metric. As a consequence we prove Theorems~\ref{prop:kplus:atteint} and \ref{noncf}.

\subsection{Proof of Theorem \ref{prop:kplus:atteint}}

We first introduce some notations that will be needed throughout this section. Let $\xi \in M$ be a fixed point. If $(M,g)$ is locally conformally flat, up to a conformal change of metric we can assume that $g$ is flat in a neighbourhood of $\xi$. If $(M,g)$ is not locally conformally flat, up to replacing the metric $g$ by a conformal metric $\tilde g \in [g]$ we may assume, by the celebrated  conformal normal coordinates results of \cite{Cao, Gunther, LeeParker} that $g$ satisfies 
\begin{equation} \label{eq:coord:conformes}
 \det g(x) = 1 \quad \text{ for all } x \in U,
 \end{equation}
where $U$ is an open neighbourhood of $\xi$ in $M$. Relation \eqref{eq:coord:conformes} has several fundamental consequences: it implies for instance that 
$$\text{Ric}(\xi) = 0, \quad  S_g(\xi) = 0 \text{ and } \quad \Delta_g S_g(\xi) = \frac{1}{6} |W_g(\xi)|_g^2,$$
 where $W_g$ is the Weyl tensor of $(M,g)$. More generally, \eqref{eq:coord:conformes} implies that all symmetrised covariant derivatives of $\text{Ric}_g$ at $\xi$ vanish to an arbitrary order. We refer to \cite{LeeParker} for a proof of these facts. We assume in this subsection that $\ker (P_g^s) = \{0\}$. Then we can let $G_g^s$ be its Green's function. It is the unique smooth function in $M \times M \backslash \{ (x,x), x \in M\}$ that satisfies, for $x \in M$, 
 $$ P_g^s G_g^s(x, \cdot) = \delta_x$$
  in $M$. Note that we do not assume here that $G_{g}^s(x, \cdot)$ is positive in $M \backslash \{x\}$, that is, we do not assume here that $P_g^s$ satisfies the pointwise comparison principle. Assume that one of the following conditions is satisfied: 
  \begin{equation} \label{eq:hyp:masse} \left \{
  \begin{aligned}
& 2s+1 \le n \le 2s+3 \text{ and } \eqref{eq:coord:conformes} \text{ is satisfied in a neighbourhood of } \xi, \text{ or } \\   
& (M,g) \text{ is flat in a neighbourhood of } \xi \text{ and } n \ge 2s + 4.
  \end{aligned}
  \right. \end{equation}
 Then the Green's function $G_g^s(\xi, \cdot)$ has the following expansion in exponential coordinates for $g$ at $\xi$:
 \begin{equation} \label{eq:def:mass}
 G_g^s(\xi, x) = b_{n,s} d_g(\xi,x)^{2s-n} + m(\xi) + o(1) \quad \text{ as } x \to \xi,
 \end{equation}
 where $b_{n,s}$ is an explicit constant given by 
 $$b_{n,s}^{-1} = 2^{k-1} (k-1)! \prod_{i=1}^s (n-2i) \omega_{n-1}$$
  and $\omega_{n-1}$ is the area of the standard sphere (see \cite{LeeParker} for $k=1$ and \cite{Michel} for $k \ge 2$ for a proof of this statement). The real number $m(\xi)$ appearing in \eqref{eq:def:mass} is called the mass of $G_g^s$ at $\xi$. For $x \neq \xi$ we let 
  \begin{equation} \label{eq:def:robin}
  h_{s,\xi}(x) = G_g^s(\xi, x) - b_{n,s} d_g(\xi,x)^{2s-n},
  \end{equation}
  which is a continuous function in $M$. 
 
\medskip

\begin{proof}[Proof of Theorem \ref{prop:kplus:atteint}]
Let $\xi \in M$ and $\delta>0$. Throughout this proof we will assume that $g$ is either flat or satisfies \eqref{eq:coord:conformes} in $B_g(\xi, 3 \delta)$, where $3 \delta < i_g(M)$ and $i_g(M)$ denotes the injectivity radius of $(M,g)$. As explained above this can always be achieved up to replacing $g$ by a metric conformal to $g$. Recall that we assume that $\ker(P_g^s) = \{0\}$. For $u \in H^s(M) \backslash \{0\}$ we let 
$$ I(u) = \frac{\int_M u P_g^su dv_g}{\left( \int_M |u|^{\frac{2n}{n-2s}} dv_g \right)^{\frac{n-2s}{n}}}. $$
The proof of Theorem~\ref{prop:kplus:atteint} follows from test-function computations. Since we cover the case of a general $k_+ \ge 1$, however, our test-function computations are more involved than the usual Rayleigh quotient estimates.  In accordance with the assumptions of Theorem~\ref{prop:kplus:atteint} we will distinguish two cases in this proof: 
\begin{itemize}
\item the case where $n \ge 2s+4$ and $(M,g)$ is not locally conformally flat
\item the case where \eqref{eq:hyp:masse} holds and the mass is positive somewhere. 
\end{itemize}
In the following the point $\xi$ will always be assumed to satisfy 
\begin{equation} \label{assumptions:xi}
\left \{ \begin{aligned}
 |W_g(\xi)|_g &>0  &\quad \text{ if } n \ge 2s + 4 \text{ and } (M,g) \text{ is not l.c.f } \\
 m(\xi) & >0  \quad & \text{ if } \eqref{eq:hyp:masse} \text{ holds. } 
\end{aligned} \right. 
\end{equation}

\medskip

\noindent \textbf{Step $1$ : Choice of the test-functions.} Let $B: \R^n \to \R$ be given by 
\begin{equation} \label{eq:bulle:std}
 B(x) = \left( 1 + \frac{|x|^2}{\Gamma_{n,s}} \right)^{- \frac{n-2s}{2}} \quad \text{ for } x \in \R^n, 
 \end{equation}
where $\Gamma_{n,s} = \Big(\Pi_{j=-s}^{s-1} (n+2j)  \Big)^{\frac{1}{s}}$. It is well-known (see \cite{LiebSharpConstants, MazumdarVetois}) that $B$ satisfies 
\begin{equation} \label{eq:Bbulle}
 P_{\xi}^s B = \Delta_\xi^s B = B^{\frac{n+2s}{n-2s}} \quad \text{ in } \R^n 
 \end{equation}
and that, up to a multiplication, translation and scaling, $B$ attains $K_{n,s}$ in \eqref{defKns}. In particular, $\int_{\R^n} B^{\frac{2n}{n-2s}}dx = K_{n,s}^{-\frac{n}{s}}$. We let $\chi \in C_c^\infty(B(0,3 \delta))$ satisfy $\chi \equiv 1$ in $B(0, \delta)$. For any $\epsilon >0$ and any $x \in M$ we define
\begin{equation}  \label{eq:defBeps}
 B_\ep(x) = \chi \big( d_g(\xi, x) \big) \ep^{- \frac{n-2s}{2}} B\Big( \frac{1}{\ep} \exp_{\xi}^{-1}(x) \Big), 
 \end{equation}
where $\exp_{\xi}$ denotes the exponential map for $g$ at $\xi$. For $\ep >0$ small enough, we now let 
\begin{itemize}
\item If $n \ge 2s+4$ and $(M,g)$ is not locally conformally flat: 
\begin{equation} \label{eq:defve}
\tilde{u}_\ep  = B_\ep + \left(1-\chi\big( d_g(\xi, x) \big)\right)\ep^{\frac{n-2s}{2}}  
\end{equation}
\item If $2s+1 \le n \le 2s+3$ or $(M,g)$ is locally conformally flat:
\begin{equation} \label{eq:defve2}
\tilde{u}_\ep  = B_\ep + \frac{\Gamma_{n,s}^{\frac{n-2s}{2}}}{b_{n,s}}\ep^{\frac{n-2s}{2}} \Big[ \chi \big( d_g(\xi, x) \big) h_{s,\xi} + \big( 1 - \chi \big( d_g(\xi, x) \big)\big)G_g^s(\xi, \cdot) \Big],
\end{equation} 
where $b_{n,s}, \Gamma_{n,s}$ and $h_{s,\xi}$ are as in \eqref{eq:def:mass}, \eqref{eq:def:robin} and \eqref{eq:bulle:std}. 
\end{itemize}
Equation \eqref{eq:def:mass} shows in particular that $\tilde{u}_{\ve}$ is smooth in $M$. First, straightforward computations using \eqref{eq:coord:conformes}, \eqref{eq:defve} and \eqref{eq:defve2} show that 
\begin{equation} \label{eq:kplus:4}
\int_{M} |\tilde{u}_{\ve}|^{\frac{2n}{n-2s}}dv_g = K_{n,s}^{-\frac{n}{s}} + O(\ep^{n})
\end{equation} 
as $\ep \to 0$. For any $\ep >0$ we now define
  \begin{equation} \label{eq:kplus:3}
 u_\ve = \Vert \tilde{u}_{\ep} \Vert_{L^{\frac{2n}{n-2s}}}^{-1} \tilde{u}_\ep,
   \end{equation} 
   where $\tilde{u}_\ep$ is given by \eqref{eq:defve} or \eqref{eq:defve2}. It was proven in \cite[Proposition 2.1, Proposition 3.1 ]{MazumdarVetois} that the following expansions hold true as $\ep \to 0$: 
 \begin{itemize}
 \item If  $n \ge 2s+4$ and $(M,g)$ is not locally conformally flat at $\xi$, 
\begin{equation} \label{eq:kplus:1}
I(u_\ep) =K_{n,s}^{-2} - C |W_g(\xi)|_g^2 \left\{  
\begin{aligned}
& \ep^4 \ln \frac{1}{\ep} + O(\ep^4) & \text{ if } n = 2s+4 \\
& \ep^4  + o(\ep^4) & \text{ if } n \ge 2s+5 \\
\end{aligned}\right. 
\end{equation}
\item If $2s+1 \le n \le 2s+3$ or $(M,g)$ is locally conformally flat,
\begin{equation} \label{eq:kplus:1bis}
I(u_\ep) =K_{n,s}^{-2} - C m(\xi) \ep^{n-2s} + o \big( \ep^{n-2s} \big).
\end{equation}
\end{itemize}
In both \eqref{eq:kplus:1} and \eqref{eq:kplus:1bis}, $C = C(n,s)$ is a fixed positive constant. These expansions heavily rely on the assumption \eqref{eq:coord:conformes} in order to obtain a precise expansion of $P_g^s$ around $\xi$ using Juhl's formulae, see e.g.  \cite[Step 2.1]{MazumdarVetois}.  For $1 \le \ell \le k_+-1$ we let $\lambda_\ell = \lambda_{\ell}(g)$ be the $k_+-1$ first eigenvalues of $P_g^s$ and we let $\vp_1, \dots, \vp_{k_+-1}$ be the associated eigenvectors. They satisfy:  
 \begin{equation} \label{eq:kplus:2} 
 \begin{aligned}   
 P_g^s \vp_\ell & =  \lambda_\ell \vp_\ell \quad \text{ in } M, \\
  	\int_{M} P_g^s \vp_p \vp_q  dv_g & = \lambda_p \delta_{p, q} \text{ and }  \int_M \varphi_{p} \varphi_{q} dv_g =  \delta_{p, q} \quad \text{ for } p, q\geq 1.
 \end{aligned} 
 \end{equation}
Straightforward computations using \eqref{eq:defve}, \eqref{eq:kplus:4}  and \eqref{eq:kplus:3} show that
\begin{equation} \label{eq:kplus:5}
\begin{aligned}
\int_{M} |u_{\ep}|^{\frac{4s}{n-2s}} u_\ep \varphi_{\ell}\, dv_g & = O \big(\ep^{\frac{n-2s}{2}} \big) , \\
\int_{M} |u_{\ep}|^{\frac{n+2s}{n-2s}}\varphi_{\ell}\, dv_g & = O \big(\ep^{\frac{n-2s}{2}} \big) \quad \text{ and } \\ 
\int_{M} \vp_\ell u_\ve dv_g & = O \big(\ep^{\frac{n-2s}{2}} \big) 
\end{aligned}
\end{equation}
as $\ep \to 0$. When $(M,g)$ is locally conformally flat or of small dimension we improve these estimates: 

\begin{lemme}
Assume that $2s+1 \le n \le 2s+3$ or that $(M,g)$ is locally conformally flat. Let $u_\ep$ be given by \eqref{eq:kplus:3}. There is a positive constant $C_0 = C_0(n,s)$ such that, as $\ep \to 0$,
\begin{equation} \label{eq:kplus:5bis}
\begin{aligned}
\int_{M} |u_{\ep}|^{\frac{4s}{n-2s}} u_\ep \varphi_{\ell}\, dv_g & = C_0 K_{n,s}^{\frac{n+2s}{2s}}  \varphi_{\ell}(\xi) \ep^{\frac{n-2s}{2}}  + o \big(\ep^{\frac{n-2s}{2}} \big) \quad \text{ and } \\
\int_{M} \vp_\ell u_\ve \, dv_g  &= \frac{1}{\lambda_\ell}C_0 K_{n,s}^{\frac{n-2s}{2s}} \varphi_{\ell}(\xi) \ep^{\frac{n-2s}{2}}  + o \big(\ep^{\frac{n-2s}{2}} \big).
\end{aligned}
\end{equation}  
\end{lemme}
\begin{proof}
First, straightforward computations using \eqref{eq:defve2} show that 
\begin{equation*}
\int_{M} |u_{\ep}|^{\frac{4s}{n-2s}} u_\ep \varphi_{\ell}\, dv_g  = \Big( K_{n,s}^{\frac{n+2s}{2s}} \int_{\R^n} B^{\frac{n+2s}{n-2s}} \, dv_\xi \Big) \varphi_{\ell}(\xi) \ep^{\frac{n-2s}{2}}  + o \big(\ep^{\frac{n-2s}{2}} \big)
\end{equation*}
as $\ep \to 0$. Independently, the definition of $\tilde{u}_{\ve}$ in \eqref{eq:defve2} shows that there exists a positive constant $C$ independent of $\ep$ such that $\ep^{- \frac{n-2s}{2}} \tilde{u}_{\ve}(x) \le \frac{C}{d_g(\xi, x)^{n-2s}}$ for any $x \neq \xi$ and that 
$$ \ep^{- \frac{n-2s}{2}} \tilde{u}_{\ve}(x) \to b_{n,s}^{-1}\Gamma_{n,s}^{\frac{n-2s}{2}}G_g^s(\xi, x) $$
pointwise in $M \backslash \{\xi\}$ as $\ep \to 0$. Dominated convergence then shows that 
\begin{equation*}
\begin{aligned}
\int_{M} \vp_\ell u_\ve \, dv_g & = b_{n,s}^{-1}\Gamma_{n,s}^{\frac{n-2s}{2}} K_{n,s}^{\frac{n-2s}{2 s}} \ep^{\frac{n-2s}{2}} \int_M G_g^s(\xi, x)   \vp_\ell \, dv_g   + o \big(\ep^{\frac{n-2s}{2}} \big) \\
& = \frac{\Gamma_{n,s}^{\frac{n-2s}{2}}}{b_{n,s} \lambda_\ell}K_{n,s}^{\frac{n-2s}{2s}} \ep^{\frac{n-2s}{2}}  \varphi_\ell(\xi) + o \big(\ep^{\frac{n-2s}{2}} \big),
\end{aligned} 
\end{equation*}
where the last line follows from \eqref{eq:kplus:2} and a representation formula. To conclude the proof of the lemma we thus only need to show that 
$$  \int_{\R^n} B^{\frac{n+2s}{n-2s}} \, dv_\xi =  \frac{\Gamma_{n,s}^{\frac{n-2s}{2}}}{b_{n,s}}. $$
This easily follows by writing down a representation formula for $B$ in \eqref{eq:bulle:std} as in \cite[Lemma 2.1]{Premoselli13}. Recall that $x \mapsto b_{n,s} |x|^{2s-n}$ is the fundamental solution of $\Delta_\xi^s$ in $\R^n$. By \eqref{eq:Bbulle}, for a fixed $x \in \R^n$ and for $R > 2 |x|$ we have
$$\begin{aligned} 
B(x) &= \int_{\R^n} \frac{b_{n,s}}{|x-y|^{n-2s}} B^{\frac{n+2s}{n-2s}}(y) \, dv_\xi \\
& =  \int_{B_{\xi}(0,R)} \frac{b_{n,s}}{|x-y|^{n-2s}}  B^{\frac{n+2s}{n-2s}}(y) \, dv_\xi + O\Big( \frac{1}{R^{2s} |x|^{n-2s}} \Big).
\end{aligned} $$
Multiplying by $|x|^{n-2s}$ on both sides and letting $|x| \to + \infty$ then $R \to + \infty$ proves the claim.
\end{proof}

\begin{rem}
The arguments in the proof of \eqref{eq:kplus:5bis} also show that  
$$\int_{M} |u_{\ep}|^{\frac{n+2s}{n-2s}}\varphi_{\ell}\, dv_g  = C_0 K_{n,s}^{\frac{n+2s}{2s}}  \varphi_{\ell}(\xi) \ep^{\frac{n-2s}{2}}  + o \big(\ep^{\frac{n-2s}{2}} \big) $$
where $C_0$ is the same constant than in \eqref{eq:kplus:5bis}.
\end{rem}

\medskip

\noindent \textbf{Step $2$ : choice of the $k_+$-dimensional subspace.}
We finally let, for any $\ep >0$,
\begin{equation*}
\beta_\ep = |u_{\ep}|^{\frac{4s}{n-2s}} + \ep^{n+1}.
\end{equation*}
Remark that $\beta_\ep \in C^0(M)$ and that $\beta_\ep >0$ in $M$. By definition of $u_\ep$ in \eqref{eq:kplus:3} we have 
 \begin{equation} \label{norm_beta_ep}
     \Vert \beta_\ep \Vert_{L^{\frac{n}{2s}}} = 1 + o(\ep^n)
 \end{equation} 
as $\ep \to 0$. We will prove that, for $\ep$ small enough,
\begin{equation} \label{eq:kplus:6}
\lambda_{k_+}(\beta_\ep)    \Vert \beta_\ep \Vert_{L^{\frac{n}{2s}}}<  K_{n,s}^{-2} = \Lambda_{1}^s(\mathbb{S}^n)
\end{equation}
holds, where the last equality follows from \eqref{eq:Lambda1:sphere}. By definition of $\Lambda_{k_+}^{s}(M,[g])$ this will imply that $\Lambda_{k_+}^s(M,[g]) < \Lambda_1^s(\mS^n)$ and conclude the proof of Theorem \ref{prop:kplus:atteint}. 
For $\ep >0$ we let 
$$ V_\ep = \text{Span} \big( \vp_{1}, \dots, \vp_{k_+-1}, u_\ep\big), $$
where $u_\ep$ is given by \eqref{eq:kplus:3} and the $\varphi_i$ are as in \eqref{eq:kplus:2}. We first claim that
 \begin{equation} \label{dimension:espace:test}
 \dim_{\beta_\ep} V_\ep = k_+
 \end{equation}  for $\ep$ small enough. Indeed, since $\beta_\ep >0$,  we have $\beta_\ep^{\frac12} \vp_\ell \neq 0$ for all $1 \le \ell \le k_{+}-1$. Proving that $(\vp_1, \dots, \vp_{k_+-1}, u_\ep)$ is free in $L^2_{\beta_\ep}(M)$ thus amounts to prove that $(\vp_1, \dots, \vp_{k_+-1}, u_\ep)$ is free in $L^2(M)$. Since $( \vp_{1}, \dots, \vp_{k_+-1})$ is free in $L^2(M)$ we just need to show that $u_\ep \not \in \text{Span}\big\{ \vp_1, \dots, \vp_{k_+-1}\big\}$. Assume by contradiction that $u_\ep = \sum_{i=1}^{k_+-1}\mu_{i,\ep} \vp_i$ for some $\mu_{i,\ep} \in \R$. Integrating against $\vp_i$ for all $1 \le i \le k_+-1$ and using \eqref{eq:kplus:2} and  \eqref{eq:kplus:5} shows that $\mu_{i,\ep} = o(1)$ as $\ep \to 0$ for all $i$, which is a contradiction with the condition $\Vert u_\ep \Vert_{L^{\frac{2n}{n-2s}}} = 1$. This proves \eqref{dimension:espace:test}. The subspace $V_\ep$ is thus admissible in the definition of $\lambda_{k_+}(\beta_\ep)$ given by \eqref{eq:def:lambdak}: for any $\ep >0$, we can thus let $\alpha_\ep \in \mathbb{S}^{k_+-1}$ be such that 
 \begin{equation*} 
 	v_\ve=\sum_{\ell=1}^{k_+-1}\alpha_{\ell,\ve} \varphi_{\ell}+\alpha_{k_+,\ve}u_{\ve}
 \end{equation*}
satisfies
 \begin{equation}\label{eq:kplus:7}
 	\mathcal{R}(\beta_\ep, v_\ep) = \max_{v \in V_\ep \backslash \{0\}} \mathcal{R}(\beta_\ep, v). 
	\end{equation}
We are going to prove that 
 \begin{equation}\label{eq:kplus:71}
 	\mathcal{R}(\beta_\ep, v_\ep)  \Vert \beta_\ep \Vert_{L^{\frac{n}{2s}}} < K_{n,s}^{-2}
	\end{equation}
for $\ep$ small enough, which will prove \eqref{eq:kplus:6}. Without loss of generality we will assume from now on that the right-hand side of \eqref{eq:kplus:7} is nonnegative, otherwise Theorem \ref{prop:kplus:atteint} is true. 
Straightforward computations using \eqref{eq:kplus:2} show that
\begin{equation} \label{test:est1}
  \int_{M} v_\ep P_g^s v_\ep dv_g = \sum_{\ell=1}^{k_+-1} \lambda_\ell \alpha_{\ell,\ep}^2+ \alpha_{k_+, \ep}^2 I(u_\ep) 
+  2 \alpha_{k_+,\ep}\sum_{\ell=1}^{k_+-1} \lambda_\ell \alpha_{\ell,\ep} \int_{M} \vp_\ell u_\ep dv_g.
\end{equation}
We similarly have, using the definition of $\beta_\ep$,
\begin{equation} \label{test:est2:1.5}
\begin{aligned}
	\int_{M} \beta_\ep v_{\ep}^2\, dv_g & =  \int_{M} |u_\ep|^{\frac{4s}{n-2s}} \big( \sum_{\ell=1}^{k_+-1} \alpha_{\ell,\ep} \vp_\ell \big)^2dv_g  + \\
	&+ \alpha_{k_+, \ep}^2 + 2\alpha_{k_+,\ep}\sum_{\ell=1}^{k_+-1}\alpha_{\ell,\ep} \int_{M} \varphi_{\ell}|u_{\ep}|^{\frac{4s}{n-2s}} u_\ep\, dv_g + o(\ep^n),
	\end{aligned}
	\end{equation} 
	and thus 
	\begin{equation} \label{test:est2:2}
	\int_{M} \beta_\ep v_{\ep}^2\, dv_g
	 \ge  \alpha_{k_+, \ep}^2 + 2\alpha_{k_+,\ep}\sum_{\ell=1}^{k_+-1}\alpha_{\ell,\ep} \int_{M} |u_{\ep}|^{\frac{4s}{n-2s}} u_\ep \varphi_{\ell}\, dv_g + o(\ep^n). 
\end{equation}
Up to passing to a subsequence $(\ep_m)_{m\ge 0}$ with $\ep_m \to 0$ as $m \to + \infty$, we can assume that $\alpha_{\ell,\ep_m}$ converges towards $\alpha_{\ell,0}$ for any $1 \le \ell \le k_+$, that still satisfy 
$$\sum_{\ell=1}^{k_+} \alpha_{\ell,0}^2 = 1.$$ 
In the following, for the sake of clarity, we will omit the subscript $m$, and will keep denoting the sequences $\alpha_{\ell,\ep_m}, u_{\ep_m}, v_{\ep_m}, \dots$ by $\alpha_{\ell,\ep}, u_\ep, v_\ep, \dots$. 

As a first observation, we remark that we can assume that $\alpha_{k_+,0}^2 = 1$. Indeed, assume first that $\alpha_{k_+,0} = 0$. Since $\ker (P_g^s) = \{0\}$ by assumption we have $\lambda_{\ell} <0$ for all $\ell \le k_+-1$, so that 
$$  \int_{M} v_\ep P_g^s v_\ep dv_g  \to  \sum_{\ell=1}^{k_+-1} \lambda_\ell \alpha_{\ell,0}^2 <0.$$
as $\ep \to 0$. The quantity $\mathcal{R}(\beta_\ep, v_\ep)$ is thus nonpositive for $\ep$ small enough and since $\int_{M} \beta_\ep v_{\ep}^2 dv_g >0$ this proves \eqref{eq:kplus:71} in this case. Assume now that $0 <\alpha_{k_+,0}^2 < 1$, so that $\sum_{\ell=0}^{k_+-1} \alpha_{\ell,0}^2 >0$. Using  \eqref{eq:kplus:1}, \eqref{eq:kplus:1bis}, \eqref{eq:kplus:5},  \eqref{norm_beta_ep}, \eqref{test:est1} 
\eqref{test:est2:2} we obtain that 
$$ \limsup_{\ep \to 0} \mathcal{R}(\beta_\ep, v_\ep) \|\be_\ep\|_{L^{\frac{n}{2s}}} \le \frac{\sum_{\ell=1}^{k_+-1} \lambda_{\ell} \alpha_{\ell,0}^2 + \alpha_{k_+,0}^2 K_{n,s}^{-2}}{\alpha_{k_+,0}^2} < K_{n,s}^{-2}  $$
as $\ep \to 0$, since $\lambda_\ell < 0$ for any $0 \le \ell \le k_+-1$. This proves \eqref{eq:kplus:71} in this case. We may therefore assume in the rest of the proof of Theorem~\ref{prop:kplus:atteint} that 
$$\alpha_{k_+,0}^2 = 1 \quad \text{ and } \quad \alpha_{\ell,0} = 0 \quad \text{ for all } \quad  1 \le \ell \le k_+-1.$$ 

\medskip

\textbf{We first assume that  $n \ge 2s+4$ and $(M,g)$ is not locally conformally flat.} Recall that $\xi$ is chosen so that $W_g(\xi) \neq 0$, where $W_g$ is the Weyl tensor of $(M,g)$. For each $1 \le \ell \le k_+-1$ fixed, two situations may occur along the subsequence we are considering as $\ep \to 0$: 
\begin{equation} \label{dichotomy}
 \text{ either } \quad \frac{|\alpha_{\ell,\ep}| }{\ep^{\frac{n-2s}{2}}} \to + \infty \quad \text{ or } \quad  |\alpha_{\ell,\ep}| = O(\ep^{\frac{n-2s}{2}}).
 \end{equation}
We let $0 \le L \le k_+-1$ be the number of indexes $\ell \in \{1, \dots, k_+-1\}$ satisfying the first case in \eqref{dichotomy}. $L$ can be assumed to be constant up to passing to a subsequence. Up to relabeling the indexes we may assume that $\alpha_{1, \ep}, \dots, \alpha_{L,\ep}$ satisfy the first case in \eqref{dichotomy} and we may thus write that 
\begin{equation} \label{assumption:beta}
 \ep^{-\frac{n-2s}{2}}  \sum_{\ell=1}^L |\alpha_{\ell,\ep}| \to + \infty  \quad \text{ and } \quad   \sum_{\ell = L+1}^{k_+-1} |\alpha_{\ell,\ep}| = O (\ep^{\frac{n-2s}{2}})
 \end{equation}
 as  $\ep \to 0$, where it is intended that the first sum is empty if $L=0$ and the second sum is empty if $L= k_+-1$. . Using  \eqref{eq:kplus:5} together with \eqref{norm_beta_ep} and  \eqref{assumption:beta} we can write that  
$$ \begin{aligned}
&\alpha_{k_+,\ep}\sum_{\ell=1}^{k_+-1}  \alpha_{\ell,\ep}\lambda_\ell \int_{M} \vp_\ell u_\ep dv_g \\
& = O \Big( \ep^{\frac{n-2s}{2}} \sum_{\ell=1}^{k_+-1} |\alpha_{\ell,\ep}| \Big) = o \Big( \sum_{\ell=1}^{L} \alpha_{\ell,\ep}^2 \Big) + O\big( \ep^{n-2s} \big). 
\end{aligned} $$
Inserting the latter in \eqref{test:est1} we obtain
\begin{equation}
\begin{aligned} \label{test:est4}
  \int_{M} v_\ep P_g^s v_\ep dv_g  =   \alpha_{k_+, \ep}^2 I(u_\ep)  + \sum_{\ell=1}^{L} \big( \lambda_\ell + o(1) \big) \alpha_{\ell,\ep}^2 + O\big( \ep^{n-2s} \big)\\
\end{aligned}
\end{equation}
as $\ep \to 0$. Similarly, \eqref{eq:kplus:5} and \eqref{assumption:beta} show that 
$$  2\alpha_{k_+,\ep}\sum_{\ell=1}^{k_+-1}\alpha_{\ell,\ep} \int_{M} \varphi_{\ell,\ep}|u_{\ep}|^{\frac{4s}{n-2s}} u_\ep \, dv_g  = o \Big( \sum_{\ell=1}^{L} \alpha_{\ell,\ep}^2 \Big) + O(\ep^{n-2s}),$$
so that 
 equation \eqref{test:est2:2} becomes 
\begin{equation} \label{test:est2}
\begin{aligned}
	\int_{M}& \be_\ep v_{\ep}^2\, dv_g  \ge \alpha_{k_+, \ep}^2 + o \Big( \sum_{\ell=1}^{L} \alpha_{\ell,\ep}^2 \Big) + O\big( \ep^{n-2s} \big).
\end{aligned} 
\end{equation}
Combining \eqref{norm_beta_ep}, \eqref{test:est4} and \eqref{test:est2} shows, since $\alpha_{k,0}^2=1$, that 
$$ \begin{aligned}
\mathcal{R}(\beta_\ep, v_\ep)  \Vert \beta_\ep \Vert_{L^{\frac{n}{2s}}} & \le I(u_\ve)  + \sum_{\ell=1}^{L} \big( \lambda_\ell + o(1) \big) \alpha_{\ell,\ep}^2 + O\big( \ep^{n-2s} \big)\\
& \le I(u_\ve)   + O\big( \ep^{n-2s} \big) 
\end{aligned} $$ 
since $\lambda_{\ell}<0$ for all $1 \le \ell \le k_+-1$. Together with \eqref{eq:kplus:1}  this proves \eqref{eq:kplus:71} when $n \ge 2s+4$ and $(M,g)$ is not locally conformally flat.

\medskip

\textbf{We now assume that $2s+1 \le n \le 2s+3$ or $(M,g)$ is locally conformally flat.} We recall that  $ \xi$ is chosen so that $m(\xi) > 0$, where $m$ is given by \eqref{eq:def:mass}. 
We proceed as in the test-function computations of \cite{CAP2}. First, using \eqref{eq:kplus:5bis}, equation \eqref{test:est1} becomes
$$
\begin{aligned}
  \int_{M} v_\ep P_g^s v_\ep dv_g & = \sum_{\ell=1}^{k_+-1} \lambda_\ell \alpha_{\ell,\ep}^2+ \alpha_{k_+, \ep}^2 I(u_\ep) \\
& +  2 C_0 K_{n,s}^{\frac{n-2s}{2s}} \alpha_{k_+,\ep}\sum_{\ell=1}^{k_+-1}  \alpha_{\ell,\ep}   \varphi_{\ell}(\xi)  \cdot \ep^{\frac{n-2s}{2}} + o \big(\ep^{\frac{n-2s}{2}} \sum_{\ell=1}^{k_+-1}|\alpha_{\ell,\ep}| \big). \\
\end{aligned} 
$$
Similarly, and again by \eqref{eq:kplus:5bis} and since $|\alpha_{k_+, \ep}| \to 1$, \eqref{test:est2:2} becomes 
\begin{equation*}
\begin{aligned}
	\int_{M} \beta_\ep v_{\ep}^2\, dv_g & \ge  \alpha_{k_+, \ep}^2\Big[  1 + 2 C_0 K_{n,s}^{\frac{n+2s}{2s}}  \sum_{\ell=1}^{k_+-1}\alpha_{\ell,\ep}   \varphi_{\ell}(\xi)  \cdot \ep^{\frac{n-2s}{2}} + o \big(\ep^{\frac{n-2s}{2}} \sum_{\ell=1}^{k_+-1}|\alpha_{\ell,\ep}| \big)  \Big]\\
	\end{aligned} 
\end{equation*}
Combining the latter two equations and using \eqref{eq:kplus:1bis} thus shows that 
\begin{equation} \label{eq:kplus:8} 
\begin{aligned}
\mathcal{R}(\beta_\ep, v_\ep)  \le I(u_\ve)  + \sum_{\ell=1}^{k_+-1} \big( \lambda_\ell + o(1) \big) \alpha_{\ell,\ep}^2 + o \big( \ep^{n-2s} \big) + o \big( \ep^{\frac{n-2s}{2}} \sum_{\ell=1}^{k_+-1}|\alpha_{\ell,\ep}| \big)
\end{aligned}
\end{equation}
as $\ep \to 0$. We let again $L$ be as in \eqref{dichotomy} and we again have 
$$ \ep^{\frac{n-2s}{2}} \sum_{\ell=1}^{k_+-1} |\alpha_{\ell,\ep}|= o \Big( \sum_{i=1}^{L} \alpha_{\ell,\ep}^2 \Big) + O\big( \ep^{n-2s} \big). $$
Plugging the latter in \eqref{eq:kplus:8} then shows that 
\begin{equation*} 
\begin{aligned}
\mathcal{R}(\beta_\ep, v_\ep) &  \le I(u_\ve)  + \sum_{\ell=1}^{k_+-1} \big( \lambda_\ell + o(1) \big) \alpha_{\ell,\ep}^2 + o \big( \ep^{n-2s} \big) \\
& \le I(u_\ve)+ o \big( \ep^{n-2s} \big)
\end{aligned}
\end{equation*}
as $\ep \to 0$, since $\lambda_{\ell}<0$ for any $1 \le \ell \le k_+-1$. Using now \eqref{eq:kplus:1bis} and \eqref{norm_beta_ep} this proves \eqref{eq:kplus:71} when $2s+1 \le n \le 2s+3$ or when $(M,g)$ is locally conformally flat, and concludes the proof of Corollary \ref{prop:kplus:atteint}.
 \end{proof}

 We point out that in the case $2s+1 \le n \le 2s+3$, and unlike in the higher-dimensional case, \eqref{eq:kplus:71} follows from a subtle cancellation of the terms of order $\ep^{\frac{n-2s}{2}}$. 
 
 \subsection{Proof of Theorem \ref{noncf}}

We start this subsection by proving the following general result: 

\begin{prop} \label{prop:fonctions:test:atteint}
Let $(M,g)$ be a closed Riemannian manifold of dimension $n \ge 3$ and let $s \in \mathbb{N}^*$, $s < \frac{n}{2}$. Let $k \ge k_+$ and assume that $\Lambda_k^s(M,[g])$ is attained. Assume that $n \ge 2s+9$ and that $(M,g)$ is not locally conformally flat. Then 
\begin{equation} \label{muk+1}
 \La_{k+1}^s(M,[g])^{\frac{n}{2s}} <  \La_k^s(M,[g])^{\frac{n}{2s}}+\La_1^s(\mS^n)^{\frac{n}{2s}}.
\end{equation}
\end{prop}

\begin{proof}
We adapt the proof of \cite[Theorem 5.4]{AmmannHumbert} where the result is proven for $s=1$. Assume that $\beta \in \mathcal{A}_{\frac{n}{2s}}$ attains $\Lambda_k^s(M,[g])$. Without loss of generality we can assume that
\begin{eqnarray} \label{intv}
\int_M \beta^{\frac{n}{2s}} \,dv_g = 1.
\end{eqnarray}
By Remark \ref{rem:continuite}, $\beta$ is continuous in $M$. By Proposition \ref{prop:wellposedmu_1}, there is a family $(\vp_1,\cdots,\vp_k)$ of $H^s(M)$ which is orthonormal for the scalar product $B(\beta, \cdot, \cdot)$ defined in \eqref{eq:definitions} such  that for all  $1 \le i \le k$ 
$$P_g^s \vp_i = \la_i \beta \vp_i \quad \text{ in } M,$$
where we have let $\la_i=\la_i(\beta)$ and thus $\la_k= \La_k^s(M,[g])$. In particular, for any $\vp \in \text{Span} \{\vp_1, \dots, \vp_k\} $, it holds that 
\begin{equation} \label{estiv}
 \int_M \vp P^s_g \vp dv_g \leq \La_k^s(M,[g]). 
\end{equation}
We let in what follows $\xi \in M$ be such that $|W_g(\xi)|_g >0$. For $\ep >0$ we let $B_\ep$ be given by \eqref{eq:defBeps}. For all $\ep >0$ we  define 
$$u_\ep = \frac{B_\ep}{\Vert B_\ep \Vert_{L^\frac{2n}{n-2s}}} \quad \text{ in } M,$$
so that 
\begin{eqnarray} \label{intve}
\int_M u_{\ep}^{\frac{2n}{n-2s}} \,dv_g = 1.
\end{eqnarray}
We let, as in the proof of Corollary \ref{prop:kplus:atteint},
$$ I(u) = \frac{\int_M u P_g^su dv_g}{\left( \int_M |u|^{\frac{2n}{n-2s}} dv_g \right)^{\frac{n-2s}{n}}}. $$
Since $n \ge 2s+9$ by assumption, \eqref{eq:kplus:1} shows that 
\begin{equation} \label{eq:klarge:1}
I(u_\ep) = \Lambda_1^s(\mathbb{S}^n)- C |W_g(\xi)|_g^2  \ep^4 + o(\ep^4) 
\end{equation}
as $\ep \to 0$, where we used \eqref{eq:Lambda1:sphere} and the equality $K_{n,s}^{-2} = \Lambda_1^s(\mathbb{S}^n)$ and where $C$ is a positive constant. We let in what follows:
$$ \beta_\ep =\left( I(u_{\ep})^{\frac{n-2s}{4s}}  u_{\ep} + \La_k^s(M,[g])^{\frac{n-2s}{4s}} \be^{\frac{n-2s}{4s}} \right)^{\frac{4s}{n-2s}} + \ep^{n+1}$$
and
$$V_\ep= \span \big\{ \vp_1,\cdots, \vp_k, u_\ep \big\} . $$
If $f$ is any continuous function in $M$ straightforward computations show that 
\begin{equation}  \label{eq:klarge:2}
\begin{aligned}
\int_{M} u_{\ep}^{\frac{n+2s}{n-2s}}f\, dv_g & = O  \big(\ep^{\frac{n-2s}{2}} \big) \quad \text{ and } \quad  
\int_{M}  u_\ve f \, dv_g  & = O  \big(\ep^{\frac{n-2s}{2}} \big).
\end{aligned}
\end{equation}
As a consequence, since $\beta$ and $\vp_1, \cdots, \vp_k$ are continuous in $M$,  and since $\beta_\ep>0$ in $M$, arguing as in  the proof of \eqref{dimension:espace:test} shows that
$$ \dim_{\beta_\ep} V_\ep = k+1, $$
so that 
\begin{equation} \label{eq:klarge:2bis}
\La_{k+1}^s(M,[g]) \leq \max_{v\in V_\ep \setminus \{0\}}\mathcal{R}(\beta_\ep,v) \Vert \beta_\ep \Vert_{L^\frac{n}{2s}}.
\end{equation}
For any $\ep >0$, we can let $\alpha_\ep$ and $(\gamma_{1,\ep}, \cdots, \gamma_{k,\ep})$ satisfy 
$$ \alpha_\ep^2 + \sum_{i=1}^k \gamma_{i,\ep}^2 = 1$$
be such that 
 \begin{equation*} 
 	v_\ve=\sum_{i=1}^{k}\gamma_{i,\ve} \vp_i +\alpha_\ep u_{\ve}
 \end{equation*}
satisfies
 \begin{equation*}
 	\mathcal{R}(\beta_\ep, v_\ep) = \max_{v \in V_\ep \backslash \{0\}} \mathcal{R}(\beta_\ep, v). 
	\end{equation*}
Straightforward computations using \eqref{estiv} and \eqref{eq:klarge:2} show that 
\begin{equation}   \label{eq:klarge:3}
\begin{aligned}
  \int_{M} v_\ep P_g^s v_\ep dv_g & \le \Lambda_k^s(M,[g]) \sum_{i=1}^{k}\gamma_{i,\ep}^2+ \alpha_{\ep}^2 I(u_\ep) + O  \big(\ep^{\frac{n-2s}{2}} \big) .
\end{aligned} 
\end{equation}
Independently, we have $\beta_\ep  \geq I(u_\ep) u_\ep^{\frac{4s}{n-2s}}$ and $\beta_\ep  \geq \La_k^s(M,[g]) \beta$ by definition of $\beta_\ep$. Therefore, using \eqref{intve} and since $(\vp_1, \cdots, \vp_k)$ is $Q(\beta,\cdot)$-orthonormal, we have 
\begin{equation}   \label{eq:klarge:4}
\begin{aligned}
\int_M \beta_\ep v_\ep^2  \,dv_g & =   \int_{M} \beta_\ep \big( \sum_{i=1}^{k} \gamma_{i, \ep} \vp_i \big)^2dv_g \\
	&+ \alpha_{\ep}^2 \int_M \beta_\ep u_\ep^2 \, dv_g + 2\alpha_{\ep}\sum_{i=1}^{k}\gamma_{i,\ep} \int_{M} \beta_\ep \varphi_{\ell} v_i u_{\ep}\, dv_g  \\
& \geq \al_\ep^2 I(u_\ep)  + \La_k^s(M,[g]) \sum_{i=1}^k \gamma_{i,\ep}^2  + 2\alpha_{\ep}\sum_{i=1}^{k}\gamma_{i,\ep} \int_{M} \beta_\ep \vp_i u_{\ep}\, dv_g. 
\end{aligned}
\end{equation}
We estimate the last integral in \eqref{eq:klarge:4}. By \eqref{eq:klarge:1} and since $\beta$ is continuous in $M$ we have 
$$\beta_\ep  \leq C (u_\ep^{\frac{4s}{n-2s}} +1)$$
for some positive constant $C$ which does not depend on $\ep$. As a consequence, staightforward computations show that, for any $1 \le i \le k$,  
$$\left| \int_M \beta_\ep \vp_i u_\ep dv_g \right| \leq C \left( \int_M u_\ep^{\frac{n+2s}{n-2s}} dv_g + \int_M u_\ep dv_g  \right) = O  \big(\ep^{\frac{n-2s}{2}} \big) $$
as $\ep \to 0$, where the last line follows from \eqref{eq:klarge:2}. Going back to \eqref{eq:klarge:4} we thus obtain
\begin{equation}   \label{eq:klarge:5}
\begin{aligned}
\int_M \beta_\ep v_\ep^2  \,dv_g & \geq \al_\ep^2 I(u_\ep)  + \La_k^s(M,[g]) \sum_{i=1}^k \gamma_{i,\ep}^2 + O  \big(\ep^{\frac{n-2s}{2}} \big).
\end{aligned}
\end{equation}
Since we assumed $n \ge 2s+9$ we have in particular $\ep^{\frac{n-2s}{2}}  = o(\ep^4)$ as $\ep \to 0$. Combining \eqref{eq:klarge:3} and \eqref{eq:klarge:5} thus gives 
\begin{equation} \label{eq:klarge:6}
 \mathcal{R}(\beta_\ep, v_\ep)\le 1 + o(\ep^4)
 \end{equation}
as $\ep \to 0$.  
Lemma 5.7 in \cite{AmmannHumbert} asserts that for any $\al>2$, there is a  $C >0$ such that 
$$|a+b|^{\al} \leq a^{\al} + b^{\al} + C ( a^{\al -1}b  +  a b^{\al -1})$$
for all $a,b >0$. Applying this to $\alpha = \frac{2n}{n-2s}$ and using \eqref{intv} and \eqref{intve} gives:   
$$ \begin{aligned} 
 \int_M \be_\ep^{\frac{n}{2s}} \,dv_g &= \int_M \big(\be_\ep^{\frac{n-2s}{4s}}\big)^{\frac{2n}{n-2s}} \,dv_g \\
& \leq  I(u_{\ep})^{\frac{n}{2s}}  + \La_k^s(M,g)^{\frac{n}{2s}}  \\ 
&  + C  \left(\int_M u_{\ep}^{\frac{n+2s}{n-2s}} \beta^{\frac{n-2s}{4s}} \,dv_g + \int_M u_{\ep}\beta^{\frac{n+2s}{4s}}
   \,dv_g \right) +o(\ep^n) \\
   & \le I(u_{\ep})^{\frac{n}{2s}}  + \La_k^s(M,g)^{\frac{n}{2s}} +  O  \big(\ep^{\frac{n-2s}{2}} \big),
\end{aligned}$$ 
where the last inequality again follows from \eqref{eq:klarge:2}. Combining the latter with \eqref{eq:klarge:2bis} and  \eqref{eq:klarge:6} finally shows, using \eqref{eq:klarge:1}, that 
$$ \La_{k+1}^s(M,[g])^{\frac{n}{2s}} \leq \Lambda_1^s(\mathbb{S}^n)^{\frac{n}{2s}} + \La_k^s(M,g)^{\frac{n}{2s}} - C' |W_g(\xi)|_g^2  \ep^4 +  o(\ep^4) $$
for some positive constant $C'$. Since $|W_g(\xi)|_g >0$, choosing $\ep$ small enough proves \eqref{muk+1} and concludes the proof of the Proposition.
\end{proof}

We are now in position to prove Theorem  \ref{noncf}.

%
%
\begin{proof}[Proof of Theorem \ref{noncf}]
Let $k_n$ be given by Theorem \ref{gamma}. First, since $P_g^s$ is coercive we have $k_+=1$. We prove by induction on $k \in \{1,\cdots,k_n-1\}$ that $\La_k^s(M,[g])$ is attained. Since $n \ge 2s+9$ and $(M,g)$ is not locally conformally flat, $\La_{1}^s(M,[g])$ is attained by Theorem \ref{prop:kplus:atteint}. We now let $k \leq k_n-2 $ and assume that $\La_{k'}^s(M,[g])$ is attained for every  $k' \in  \{1,\cdots,k\}$. We proceed by contradiction and assume that $\La_{k+1}^s(M,[g])$ is not attained. We claim that the following holds:
\begin{equation} \label{xk+1}
 X_{k+1}^s(M,[g])^{\frac{n}{2s}}= \La_k^s(M,[g])^{\frac{n}{2s}}+\Lambda_1^s(\mS^n)^{\frac{n}{2s}}.
\end{equation}
Assume for the moment that \eqref{xk+1} is proven. Since $\La_{k}^s(M,[g])$ is attained by assumption, Proposition \ref{prop:fonctions:test:atteint} applies and shows that 
$$  \Lambda_{k+1}^s(M,[g])^{\frac{n}{2s}}< \La_k^s(M,[g])^{\frac{n}{2s}}+\Lambda_1^s(\mS^n)^{\frac{n}{2s}}.$$ 
Together with \eqref{xk+1} this implies that  $  \Lambda_{k+1}^s(M,[g]) <  X_{k+1}^s(M,[g])$, and Theorem  \ref{theo:vp:positives} yields that $\La_{k+1}^s(M,[g])$ is attained, a contradiction. 

We thus only need to prove \eqref{xk+1}. For this we let  $r,\ell_0,\cdots,\ell_r \in \mN$ be such that 
\begin{enumerate}
 \item $\ell_0  + \cdots + \ell_r =k+1$ ; 
 \item $\La_{\ell_0}^s(M,[g])$ and $\La_{\ell_i}^s(\mS^n)$ are attained  
\end{enumerate}
and such that 
$$X_{k+1}^s(M,[g])^{\frac{n}{2s}}=  \La_{\ell_0}^s(M,[g])^{\frac{n}{2s}} + \La_{\ell_1}^s(\mS^n)^{\frac{n}{2s}} +  \cdots + \La_{\ell_{r}}^s(\mS^n)^{\frac{n}{2s}} .$$
For any $i \in \{1, \cdots, r\}$ we have $1 \le \ell_i \le k+1 \le k_n-1$ by assumption. Since $\La_{\ell_{r}}^s(\mS^n)$ is attained, by Theorem \ref{gamma} we have $\ell_i=1$ for all $1 \le i \le r$ and hence
\begin{equation} \label{Xk:nonatteint}
 X_{k+1}^s(M,[g])^{\frac{n}{2s}}=  \La_{\ell_0}^s(M,[g])^{\frac{n}{2s}} + r \La_{1}^s(\mS^n)^{\frac{n}{2s}} . 
 \end{equation}
Using Proposition \ref{prop_ineg_large2} we have, independently, 
$$\La_k^s(M,[g])^{\frac{n}{2s}} \le  \La_{\ell_0}^s(M,[g])^{\frac{n}{2s}} + (r-1) \La_{1}^s(\mS^n)^{\frac{n}{2s}}, $$
which together with \eqref{Xk:nonatteint} shows that 
$$ X_{k+1}^s(M,[g])^{\frac{n}{2s}} \ge \La_k^s(M,[g])^{\frac{n}{2s}} +  \La_{1}^s(\mS^n)^{\frac{n}{2s}}. $$
Since we assumed that $ \La_k^s(M,[g])$ is attained the previous inequality is an equality by definition of $X_{k+1}^s(M,[g])$, which proves \eqref{xk+1} and concludes the proof of Theorem \ref{noncf}.
\end{proof}

\subsection{Results when $P_g^s$ has a nonzero kernel.}

In this last subsection we finally prove Theorem \ref{prop:kplus:atteint:noyau}.

%

\begin{proof}[Proof of Theorem \ref{prop:kplus:atteint:noyau}]
We keep the same notations that were used in the proof of Theorem \ref{prop:kplus:atteint}. For $\ep >0$, we still define $B_\ep$ as in \eqref{eq:defBeps}, where $\xi$ is chosen so that $|W_g(\xi)|_g >0$, and we let 
$$\tilde{u}_\ep = B_\ep +  \left(1-\chi\big( d_g(\xi, x) \big)\right)\Gamma_{n,s}^{\frac{n-2s}{2}} \ep^{\frac{n-2s}{2}} d_g(\xi, x)^{2s-n}$$ 
and 
$$u_\ep = \frac{\tilde{u}_\ep}{\Vert \tilde{u}_\ep \Vert_{L^\frac{2n}{n-2s}}},$$
so that $u_\ep$ satisfies $\Vert u_\ep \Vert_{L^\frac{2n}{n-2s}} = 1$ and $u_\ep >0$ in $M$. We let $\vp_1, \dots, \vp_{k_+-1}$ be the $k_+-1$ first eigenvectors of $P_g^s$. They satisfy 
\eqref{eq:kplus:2}, and again \eqref{eq:kplus:5} holds true. We notice that pointwise, it holds that 
$$\lim_{\ep \to 0} \ep^{-\frac{n-2s}{2}}u_\ep(x) = \Gamma_{n,s}^{\frac{n-2s}{2}}  d_g(\xi, x)^{2s-n}.$$
As a consequence, for any function $f \in C^0(M)$, we have 
\begin{equation} \label{eq:kplus:5:1}
\int_{M} u_\ve^{\frac{4s}{n-2s}} f dv_g  = K_{n,s}^{2}\Gamma_{n,s}^2 \int_{M} d_g(\xi, \cdot)^{2s-n} f dv_g \cdot \ep^{2s} + o(\ep^{2s}) 
\end{equation}
as $\ep \to 0$. Independently, direct computations using \eqref{eq:kplus:1} show that 
\begin{equation} \label{eq:kplus:100}
I(u_\ep) =K_{n,s}^{-2} - C |W_g(\xi)|_g^2 \ep^4  + o(\ep^4) 
\end{equation}
as $\ep \to 0$, where we used that $n \ge 4s+5 > 2s+4$. As in the proof of Theorem \ref{prop:kplus:atteint} we let 
\begin{equation*}
\beta_\ep = u_{\ep}^{\frac{4s}{n-2s}}.
\end{equation*}
 By definition of $u_\ep$ we have $\be_\ep >0$ and 
 \begin{equation} \label{norm_beta_ep2}
     \Vert \beta_\ep \Vert_{L^{\frac{2n}{n-2s}}} = 1
 \end{equation}
 We again claim that, for $\ep$ small enough,
\begin{equation} \label{eq:kplus:6:noyau}
\lambda_{k_+}(\beta_\ep) \| \be_\ep\|_{L^{\frac{n}{2s}}} <  K_{n,s}^{-2} = \Lambda_1^s(\mathbb{S}^n)
\end{equation}
holds. By definition of $\Lambda_{k_+}^{s}(M,[g])$ and by \eqref{eq:Lambda1:sphere} this will conclude the proof of Proposition  \ref{prop:kplus:atteint:noyau}. We again let 
$$ V_\ep = \text{Span} \big( \vp_{1}, \dots, \vp_{k_+-1}, u_\ep\big), $$
so that mimicking the proof of \eqref{dimension:espace:test} again shows that $\dim_{\beta_\ep} V_\ep = k_+$ for $\ep$ small enough, and for any $\ep >0$, we let $\alpha_\ep \in \mathbb{S}^{k_+-1}$ be such that 
 \begin{equation*} 
 	v_\ve=\sum_{\ell=1}^{k_+-1}\alpha_{\ell,\ve} \varphi_{\ell}+\alpha_{k_+,\ve}u_{\ve}
 \end{equation*}
satisfies
 \begin{equation*} 
 	\mathcal{R}(\beta_\ep, v_\ep) = \max_{v \in V_\ep \backslash \{0\}} \mathcal{R}(\beta_\ep, v). 
	\end{equation*}
We are going to prove that 
 \begin{equation}\label{eq:kplus:71:noyau}
 	\mathcal{R}(\beta_\ep, v_\ep) \|\be_\ep\|_{L^{\frac{n}{2s}}} < K_{n,s}^{-2}
	\end{equation}
which implies \eqref{eq:kplus:6:noyau} for $\ep$ small enough. 

\medskip

First, the assumption $\ker(P_g^s)\neq \{0\}$ shows that $k_- < k_{+}$. For any $k_- < \ell < k_+$ we have $\lambda_{\ell}= 0$, so that 
\eqref{test:est1} rewrites as
\begin{equation} \label{spectre:eq2:noyau} 
  \int_{M} v_\ep P_g^s v_\ep dv_g  = \sum_{\ell=1}^{k_-} \lambda_\ell \alpha_{\ell,\ep}^2+ \alpha_{k_+, \ep}^2 I(u_\ep) 
  +2 \alpha_{k_+,\ep}\sum_{\ell=1}^{k_-} \lambda_\ell \alpha_{\ell,\ep} \int_{M} \vp_\ell u_\ep dv_g. 
\end{equation}
We use the convention that the first sum in the right-hand side of \eqref{spectre:eq2:noyau} vanishes if  $k_-=0$, that is, if every eigenvalue of $P_g^s$ is nonnegative. Let $0< 2\eta < | \lambda_{k_-}|$ be fixed. Using Young's inequality and \eqref{eq:kplus:5},   we have 
$$ \begin{aligned}
 2 \alpha_{k_+,\ep}\sum_{\ell=1}^{k_-}  \lambda_\ell \alpha_{\ell,\ep}  \int_{M} \vp_\ell u_\ep dv_g & 
= 2 \sum_{\ell=1}^{k_-} \big(\eta^{\frac{1}{2}} \al_{\ell,\ep} \big) \left( \eta^{-\frac{1}{2}}  \la_{\ell}  \int_{M} \vp_\ell u_\ep dv_g  \right) \\
&\le \eta \sum_{\ell=1}^{k_-} \alpha_{\ell,\ep}^2 + O(\ep^{n-2s} \alpha_{k_+,\ep}^2), 
\end{aligned}  $$ 
and inserting the latter in \eqref{spectre:eq2:noyau} gives 
 \begin{equation} \label{spectre:eq3:noyau} 
\begin{aligned}
  \int_{M} v_\ep P_g^s v_\ep dv_g & \le   \alpha_{k_+, \ep}^2 \Big( I(u_\ep) + O(\ep^{n-2s})\Big) +  \sum_{\ell=1}^{k_-}\big( \lambda_\ell + \eta \big) \alpha_{\ell,\ep}^2\\
  & \le \alpha_{k_+, \ep}^2 \Big( I(u_\ep) + O(\ep^{n-2s})\Big).
\end{aligned} 
\end{equation}
At this point, to prove \eqref{spectre:eq3:noyau}, we used no assumption on $n$ other than $n \ge 2s+5$. Recall that without loss of generality we can assume that the left-hand side in \eqref{spectre:eq3:noyau} is positive, otherwise \eqref{eq:kplus:71:noyau} is trivial. As a consequence, equation \eqref{spectre:eq3:noyau} shows in particular that we can assume that $\alpha_{k_+,\ep} \neq 0$ for all $\ep >0$. The main difference with respect to the case where $\ker(P_g^s) = \{0\}$ is that we cannot anymore assume, however, that $\lim_{\ep \to 0} \alpha_{k_+,\ep} \neq 0$: compensations between the numerator and the denominator of $\mathcal{R}(\beta_\ep, v_\ep)$ may occur via the coefficients  $\alpha_{k_-+1, \ep}, \dots, \alpha_{k_+-1, \ep}$ \emph{that do not appear} in \eqref{spectre:eq2:noyau}. We overcome this by proving better estimates on the denominator of $\mathcal{R}(\beta_\ep, v_\ep)$, which will rely on the more restrictive assumption on $n$. 

\medskip

We now bound from below the denominator of $\mathcal{R}(\beta_\ep, v_\ep)$. For $\ep >0$ we let $\bar{\alpha}_\ep = \big( \alpha_{1, \ep}, \cdots, \alpha_{k_+-1, \ep} \big)$. Equation \eqref{eq:kplus:5:1} shows that 
\begin{equation} \label{eq:minor:noyau}
 \int_M \beta_\ep \Big( \sum_{\ell=1}^{k_+-1} \alpha_{\ell,\ep} \vp_\ell \Big)^2 dv_g =  \big({}^{t} \bar{\alpha}_\ep D \bar{\alpha}_\ep \big) \cdot  \ep^{2s} + o(\ep^{2s}),
 \end{equation}
 where $D$ is the $(k_+-1)$-square matrix given by 
 $$ D_{k\ell} = K_{n,s}^{2} \Gamma_{n,s}^2 \int_M  \vp_k(x) \vp_\ell(x) d_g(\xi,x)^{2s-n} dv_g(x), \quad 1 \le k, \ell  \le k_+-1. $$
Since the family $(\vp_1, \cdots, \vp_{k_+-1})$ is free in $L^2(M)$ is it easily seen that $D$ is positive-definite: as a consequence, there exists $\nu = \nu(n,s,g) >0$ such that 
$$ {}^{t} \beta D \beta \ge \nu \sum_{\ell=1}^{k_+-1} \beta_{\ell}^2 \quad \text{ for any } \beta \in \R^{k_+-1}. $$
With the latter, \eqref{eq:minor:noyau} and \eqref{eq:kplus:5} in  \eqref{test:est2:1.5} shows that 
\begin{equation}  \label{spectre:eq1:2}
\begin{aligned}
	\int_{M} \beta_\ep  v_{\ep}^2\, dv_g & \ge \alpha_{k_+,\ep}^2 + \nu \sum_{\ell=1}^{k_+-1}\alpha_{\ell,\ep}^2 \cdot \ep^{2s} \\
	&+ O \Big(\ep^{\frac{n-2s}{2}}|\alpha_{k_+,\ep}|\sum_{\ell=1}^{k_+-1}|\alpha_{\ell,\ep}|  \Big), 
	\end{aligned} 
\end{equation}
where the constant in the $O(\cdot)$ term only depends on $n,s$ and $M$. Young's inequality now shows, since $n \ge 4s + 5$, that we have 
$$ \begin{aligned}
 \ve^{\frac{n-2s}{2}}|\alpha_{k_+,\ep}|\sum_{\ell=1}^{k_+-1}|\alpha_{\ell,\ep}| |\vp_\ell(\xi)| 
  &  = \sum_{\ell=1}^{k_+-1}| \big(  \nu^{\frac{1}{2}} \ep^{s} |\alpha_{\ell,\ep}| \big) \big( \nu^{-\frac{1}{2}} \alpha_{k_+,\ep}|\vp_\ell(\xi)|  \ep^{\frac{n-4s}{2}} \big)\\
  & \le \frac{\nu}{2} \ep^{2s} \sum_{\ell=1}^{k_+-1}\alpha_{\ell,\ep}^2 + o \big( \ep^{4} \alpha_{k_+,\ep}^2 \big).
 \end{aligned} $$
Plugging the latter in \eqref{spectre:eq1:2} finally shows that 
\begin{equation}  \label{spectre:eq1}
\begin{aligned}
	\int_{M} \beta_{\ep} v_{\ep}^2\, dv_g & \ge \big(1 +o(\ep^4) \big) \alpha_{k_+,\ep}^2 +   \frac{\nu}{2} \sum_{\ell=1}^{k_+-1}\alpha_{\ell,\ep}^2 \cdot  \ve^{2s} \ge  \big(1 +o(\ep^4) \big) \alpha_{k_+,\ep}^2. 
	\end{aligned} 
\end{equation}
Combining \eqref{spectre:eq1} with  \eqref{spectre:eq3:noyau} then yields, since $n \ge 4s+5$,
\begin{equation*} 
\lambda_{k_+}(\beta_\ep) \le I(u_\ve)   + o(\ep^4)
\end{equation*}
which, together with \eqref{eq:kplus:100} and \eqref{norm_beta_ep2}, proves \eqref{eq:kplus:71:noyau} and concludes the proof of Proposition \ref{prop:kplus:atteint:noyau}.
\end{proof}

\appendix

\section{A technical result}

Let $n \ge3$ and $s < \frac{n}{2}$. For  $u \in C^\infty_c(\R^n)$ we define 
\begin{equation} \label{normDs2}
 \Vert u \Vert_{D^{s,2}}^2 = \int_{\R^n} \big|\Delta_\xi^{\frac{s}{2}} u  \big|^{2} \, dv_\xi,
 \end{equation}
where we have let 
$$\Delta_\xi^{\frac{s}{2}} u = \left \{ 
\begin{aligned}
& \Delta_\xi^{m } u & \text{ if } s = 2m \text{ is even}, \\
& \nabla  \Delta_\xi^{m } u & \text{ if } s = 2m+1 \text{ is odd}
\end{aligned} \right. $$
and $\Delta_\xi = - \sum_{i=1}^n \partial_i^2$. It is easily checked that $\Vert \cdot \Vert_{D^{s,2}}$ is a norm. We define $D^{s,2}(\R^n)$ as the closure of $C^\infty_c(\R^n)$ with respect to $\Vert \cdot \Vert_{D^{s,2}}$ and we denote by $(u,v)_{D^{s,2}}$ the associated scalar product defined by $(u,v)_{D^{s,2}} = \int_{\R^n} \langle  \Delta_\xi^{\frac{s}{2}} u, \Delta_\xi^{\frac{s}{2}} v \rangle \, dv_\xi$. It follows from \eqref{defKns} that every function $u \in D^{s,2}(\R^n)$ satisfies the following Sobolev inequality: 
\begin{equation} \label{sob:Dms}
\Vert u \Vert_{\frac{2n}{n-2s}}^2 \le K_{n,s}^2 \Vert u \Vert_{D^{s,2}}^2, 
\end{equation}
where $K_{n,s}$ is as in \eqref{defKns}, and it follows from \cite{LiebSharpConstants} that equality in \eqref{sob:Dms} is achieved if and only if there exist $c \in \R$, $\mu > 0$ and $\xi \in \R^n$ such that 
$$u = c B_{\mu, \xi}, \quad \text{ where } \quad B_{\mu,\xi}(x) =  \mu^{- \frac{n-2s}{2}} B \left( \frac{x-\xi}{\mu} \right), $$
and $B$ is given by \eqref{eq:bulle:std}. With \eqref{eq:Bbulle} we see that  $B_{\mu,\xi}$ satisfies 
\begin{equation} \label{annexe:eqB}
 \Delta_\xi^s B_{\mu,\xi} = B_{\mu,\xi}^{\frac{n+2s}{n-2s}} \quad \text{ in } \R^n. 
 \end{equation}
For $u \in D^{s,2}(\R^n)$ we let $u_+ = \max(u,0)$ and $u_- = \min(u,0)$. We remark that when $s \ge 2$ these functions are not in $D^{s,2}(\R^n)$ in general. In this section we prove the following result that was used in the proof of Theorem \ref{gamma}.

\begin{lemme} \label{lemme:energie:Rn}
Let $n \ge 3$, $s < \frac{n}{2}$ and assume that $u \in D^{s,2}(\R^n)$ solves 
$$ \Delta_\xi^s u = |u|^{\frac{4s}{n-2s}} u \quad \text{ in } \R^n. $$
If $u$ changes sign, that is if $u_+ \neq 0$ and $u_- \neq 0$, then $\Vert u \Vert_{D^{s,2}}^2 > 2 K_{n,s}^{-\frac{n}{s}}$. 
\end{lemme}

\begin{proof}
That $\Vert u \Vert_{D^{s,2}}^2 \ge 2 K_{n,s}^{-\frac{n}{s}}$ is proven in \cite[Lemma 7.22]{GazzolaGrunauSweers}. We here explain how the arguments in the proof of \cite[Lemma 7.22]{GazzolaGrunauSweers} also yield the strict inequality. Let $u$ satisfy the assumptions of Lemma \ref{lemme:energie:Rn} and let $C = \{ v \in D^{s,2}(\R^n), v \ge 0 \text{ a.e.} \}$ and $C^* = \{ v \in D^{s,2}(\R^n), (v,w)_{D^{s,2}}\le 0 \text{ for all } w \in C \}$. Arguing as in the proof of  \cite[Lemma 7.22]{GazzolaGrunauSweers}, there exists a unique pair $(u_1, u_2) \in C \times C^*$ such that $u=u_1+u_2$ and $(u_1,u_2)_{D^{s,2}} = 0$, and that satisfies $\Vert u_i \Vert_{\frac{2n}{n-2s}} \ge K_{n,s}^{-\frac{n-2s}{2s}}$ for $i=1,2$, so that in particular $u_1$ and $u_2$ are non-zero. Since $u_1$ and $u_2$ are $D^{s,2}(\R^n)$-orthogonal we then have, by Sobolev's inequality:
$$ \begin{aligned}
 \Vert u \Vert_{D^{s,2}}^2   & = \Vert u_1 \Vert_{D^{s,2}}^2 + \Vert u_2 \Vert_{D^{s,2}}^2  \\
& \ge K_{n,s}^{-2} \big( \Vert u_1 \Vert_{\frac{2n}{n-2s}} ^2 + \Vert u_2 \Vert_{\frac{2n}{n-2s}} ^2 \big)  \\
& \ge 2 K_{n,s}^{-\frac{n}{s}}.
\end{aligned} $$
Assume by contradiction that $u$ satisfies $ \Vert u \Vert_{D^{s,2}}^2 = 2 K_{n,s}^{-\frac{n}{s}}$. Then all the previous inequalities are equalities: as a consequence we have $\Vert u_i \Vert_{\frac{2n}{n-2s}} = K_{n,s}^{-\frac{n-2s}{2s}}$ for $i=1,2$ and both $u_1$  and $u_2$ satisfy the equality case of \eqref{sob:Dms}. Thus, for $i=1,2$, we have $ u_i = c_i B_{\mu_i, \xi_i}$ in $\R^n$, for some $c_i \neq 0, \mu_i >0$ and $\xi_i \in \R^n$. This is a contradiction with $(u_1,u_2)_{D^{s,2}} = 0$ since, as can easily be checked using \eqref{annexe:eqB}, we have 
$$ \big ( B_{\mu_1, \xi_1}, B_{\mu_2, \xi_2}\big)_{D^{s,2}} = \int_{\R^n} B_{\mu_1, \xi_1}^{\frac{n+2s}{n-2s}} B_{\mu_2, \xi_2}\, dv_ \xi >0   $$
for all $\mu_1, \mu_2>0$  and $\xi_1, \xi_2 \in \R^n$.
\end{proof}

\bibliographystyle{amsplain}
\bibliography{biblio}

\end{document}